\documentclass[reqno,10pt]{amsart}

\usepackage{amsmath}
\usepackage{amsfonts}
\usepackage{amssymb}
\usepackage{amsthm}
\usepackage{newlfont}
\usepackage{graphicx}
\newtheorem{thm}{Theorem}[section]

\newtheorem{lem}[thm]{Lemma}
\newtheorem{prop}[thm]{Proposition}
\newtheorem{conj}[thm]{Conjecture}
\theoremstyle{definition}

\theoremstyle{remark}
\newtheorem{rem}[thm]{Remark}
\numberwithin{equation}{section}

\newcommand{\R}{\mathbb R}

\newcommand{\ed}{\end {document}}

\newcommand{\dpp}{{\Delta^{-1}\nabla^{\perp}}}
\newcommand{\haa}{{\sqrt{\log A}}} 

\title[illposedness of Euler in Sobolev]
{Strong Ill-posedness of the incompressible Euler equation in borderline
Sobolev spaces}
\author[J. Bourgain]{Jean Bourgain}
\address[J. Bourgain]{School of Mathematics, Institute for Advanced Study, Princeton, NJ 08544, USA}
\email{bourgain@math.ias.edu}

\author[D. Li]{Dong Li}
\address[D. Li]{Department of Mathematics, University of British Columbia, Vancouver BC Canada V6T 1Z2}%
\email{mpdongli@gmail.com}

\begin{document}
\begin{abstract}
For the $d$-dimensional incompressible Euler equation, the standard
energy method gives local wellposedness for initial velocity in
Sobolev space $H^s(\mathbb R^d)$, $s>s_c:=d/2+1$. The borderline
case $s=s_c$ was a folklore open problem. In this paper we consider
the physical dimensions $d=2,3$ and show that if we perturb any
given smooth initial data in $H^{s_c}$ norm, then  the corresponding
solution can have infinite $H^{s_c}$ norm instantaneously at $t>0$.
The constructed solutions are unique and even $C^{\infty}$-smooth in
some cases. To prove these results we introduce a new strategy:
\emph{large Lagrangian deformation induces critical norm inflation}.
As an application we also settle several closely related open
problems.

\end{abstract}

\maketitle

\section{Introduction}
The $d$-dimensional incompressible Euler equation takes the form
\begin{align} \label{V_usual}
 \begin{cases}
  \partial_t u + (u\cdot \nabla) u + \nabla p =0, \quad (t,x) \in \mathbb R \times \mathbb R^d, \\
  \nabla \cdot u =0, \\
  u \bigr|_{t=0}=u_0,
 \end{cases}
\end{align}
where $u=u(t,x)=(u_1(t,x),\cdots,u_d(t,x)):\, \mathbb R \times \mathbb R^d \to \mathbb R^d$
denotes the velocity of the fluid and $p=p(t,x):\, \mathbb R \times \mathbb R^d \to \mathbb R$ is the pressure.
The second equation $\nabla \cdot u=0$ in \eqref{V_usual} is usually called the incompressibility (divergence-free)
condition. By taking the divergence on both sides of the first equation in \eqref{V_usual}, one can recover the pressure
from the quadratic term in velocity by inverting the Laplacian in suitable functional spaces. Another way to eliminate
the pressure is to use the vorticity formulation. For this we will discuss separately the 2D and 3D case.
In 2D, introduce the scalar-valued vorticity function
\begin{align*}
\omega = -\partial_2 u_1 + \partial_1 u_2=\nabla^{\perp}\cdot u,\quad\nabla^{\perp} := (-\partial_2, \partial_1).
\end{align*}
By taking $\nabla^{\perp}\cdot$ on both sides of \eqref{V_usual}, we have the equation
\begin{align} \label{E_usual}
\begin{cases}
\partial_t \omega + u \cdot \nabla \omega =0,  \quad (t,x) \in \mathbb R\times \mathbb R^2,\\
u= \nabla^{\perp} \psi=(-\partial_2 \psi, \partial_1 \psi),  \;\; \Delta \psi= \omega, \\
\omega\big|_{t=0}= \omega_0.
\end{cases}
\end{align}
Under some suitable regularity assumptions, the second equations in \eqref{E_usual} can be written as
a single equation
\begin{align}
 u= \Delta^{-1} \nabla^{\perp} \omega, \label{uw}
\end{align}
which is the usual Biot--Savart law. Alternatively one can express
\eqref{uw} as a convolution integral
\begin{align*}
u=K*\omega, \quad K(x)= \frac 1 {2\pi} \cdot \frac {x^{\perp}}{|x|^2}.
\end{align*}
We can then rewrite \eqref{E_usual} more compactly as
\begin{align} \label{2D_vorticity_usual}
\begin{cases}
 \partial_t \omega + (\Delta^{-1} \nabla^{\perp} \omega \cdot \nabla) \omega =0, \\
 \omega\Bigr|_{t=0} =\omega_0.
 \end{cases}
\end{align}
We shall frequently refer to \eqref{2D_vorticity_usual} as the usual 2D Euler equation in vorticity formulation.
Note that \eqref{2D_vorticity_usual} is a transport equation which preserves all $L^p$, $1\le p\le \infty$ norm of the
vorticity $\omega$. In the 3D case the vorticity is vector-valued and given by
$$\omega = \operatorname{curl} u =\nabla \times u.$$
The 3D Euler equation in vorticity formulation has the form
\begin{align*}
 \begin{cases}
  \partial_t \omega + (u\cdot \nabla) \omega =(\omega \cdot \nabla) u, \\
  u=-\Delta^{-1} \nabla \times \omega, \\
  \omega \Bigr|_{t=0} =\omega_0.
 \end{cases}
\end{align*}
Note that  the second equation above is just the Biot--Savart law in
3D.  The expression $(\omega \cdot \nabla) u$ is often referred to
as the vorticity stretching term. It is one  of the main source of
difficulties in the wellposedness theory of 3D Euler.

There is by now an extensive literature on the wellposedness theory for Euler equations. We shall only mention a few
and refer to Majda-Bertozzi \cite{MB} and Constantin \cite{C07} for more extensive references.
 The papers of Lichtenstein \cite{L21} and Gunther \cite{Gun27} started the subject of
  local wellposedness in H\"older spaces $C^{k,\alpha}$ ($k\ge 1$, $0<\alpha<1$).
In \cite{Wolibner33} Wolibner obtained global solvability of classical (belonging to H\"older class)
solutions for 2D Euler (see Chemin \cite{Chemin_book} for a modern exposition).
In \cite{EM70} Ebin and Marsden proved the short time existence, uniqueness, regularity, and continuous dependence on
initial conditions for solutions of the Euler equation on general compact manifolds (possibly with
$C^{\infty}$ boundary). Their method is to topologize the
space of diffeomorphisms by Sobolev $H^s$, $s> d/2+1$ norms and then solve the geodesic equation using contractions.
In \cite{BB74} Bourguignon and Brezis  generalized $H^s$ to the case of $W^{s,p}$ for $s>d/p+1$.
In \cite{Kato72} Kato proved local wellposedness of $d$-dimensional Euler in $C_t^0H_x^m$
for initial velocity $u_0\in H^m(\mathbb R^d)$ with integer $m>d/2+1$.  Later Kato and Ponce
\cite{KP88} proved wellposedness in the general
Sobolev space $W^{s,p}(\mathbb R^d) =(1-\Delta)^{-s/2} L^p(\mathbb R^d) $ with real $s>d/p+1$ and $1<p<\infty$.
The key argument in \cite{KP88} is the following commutator estimate for the operator $J^s=(1-\Delta)^{s/2}$:
\begin{align}
\| J^s(fg)-f J^s g\|_{p} \lesssim_{d,s,p} \|D f\|_{\infty}  \|
J^{s-1} g\|_{p} +\| J^s f\|_p \| g\|_{\infty}, \qquad 1<p<\infty,\,
s\ge 0. \label{commutator_1}
\end{align}
To extend the local solutions globally in time, one can use
 the Beale-Kato-Majda criterion \cite{BKM84} which asserts that (here $s>d/2+1$)
 \begin{align*}
 \limsup_{t\to T^*} \| u(t,\cdot)\|_{H^s(\mathbb R^d)} = +\infty,
 \end{align*}
 if and only if
 \begin{align*}
 \limsup_{t\to T^*} \int_0^{t} \| \omega(s,\cdot) \|_{L^\infty(\mathbb R^d)} ds =+\infty.
 \end{align*}
By using this criterion and conservation of $\|\omega\|_{\infty}$ in 2D, one can immediately deduce
the global existence of Kato's solutions in dimension two. In \cite{Vishik98} (see also \cite{Vishik99})
 Vishik considered the borderline case $s=d/p+1$ and obtained global solvability for
the 2D Euler in Besov space $B^{ 2/p+1}_{p,1}$ with $1<p<\infty$. In
\cite{Chae04} Chae proved local existence and uniqueness of
solutions to $d$-dimensional
 Euler in critical Besov space (for velocity) $B^{d/p+1}_{p,1}(\mathbb R^d)$
with $1<p<\infty$. The local wellposedness in
$B^{1}_{\infty,1}(\mathbb R^d)$, $d\ge 2$ was settled by Pak and
Park in \cite{PP04}. Roughly speaking, all the aforementioned local
wellposedness results rely on finding a certain Banach space $X$
with the norm $\|\cdot\|_X$ such that (take $f=\nabla \times u$ and
$X=B^{d/p}_{p,1}$ for example)
\begin{enumerate}
\item If $f\in X$, then $ \| f \|_{L^{\infty}} + \| \mathcal R_{ij}
f \|_{L^{\infty}} \lesssim \| f \|_X$ ($\mathcal R_{ij}$ is the
Riesz transform);

\item Some version of a commutator estimate similar to
\eqref{commutator_1} holds in $X$.
\end{enumerate}
The above are essentially minimal conditions needed to close the
energy estimates. On the other hand, this type of scheme completely
breaks down for the natural borderline Sobolev spaces such as $H^{
d/2+1}$  (in terms of vorticity we have $X=H^{d/2}$) since both
conditions will be violated. In \cite{Ta10}, Takada constructed\footnote{Counterexamples for the case $s<d/p+1$
was also considered therein.} several counterexamples of Kato-Ponce-type commutator estimates in critical
Besov $B^{d/p+1}_{p,q}(\mathbb R^d)$ and Triebel-Lizorkin spaces $F^{d/p+1}_{p,q}(\mathbb R^d)$ for
various exponents $p$ and $q$ (For Besov: $1\le p\le \infty$, $1< q\le \infty$; For Triebel-Lizorkin:
$1<p<\infty$, $1\le q\le \infty$ or $p=q=\infty$).
It should be noted that the vector fields used in his counterexamples
are divergence-free.
In light of these considerations, a
well-known long standing open problem was the following

\begin{conj} \label{conj1}
The Euler equation \eqref{V_usual} is illposed for a class of
initial data in $H^{d/2+1}(\mathbb R^d)$.
\end{conj}

Of course one can state analogous versions of Conjecture
\ref{conj1} in similar Sobolev spaces $W^{d/p+1,p}$ or other Besov or Triebel-Lizorkin type
spaces with various boundary conditions. A rather delicate matter is
to give a precise (and satisfactory) formulation of the illposedness
statement in Conjecture \ref{conj1}. The formulation and the proof
of such a statement requires a deep understanding of how the
critical space topology changes under the Euler dynamics.

To begin, one can consider explicit solutions to \eqref{V_usual}. In \cite{DM87}, DiPerna and Majda
introduced the following shear flow (in their study of measure-valued solutions for 3D Euler):
\begin{align*}
u(t,x)=(f(x_2),0,g(x_1-tf(x_2))), \quad x=(x_1,x_2,x_3),
\end{align*}
where $f$ and $g$ are given single variable functions. This explicit flow (sometimes called "2+1/2"-dimensional
flow) solves \eqref{V_usual} with pressure $p=0$. DiPerna and Lions used the above flow (see e.g. p152 of
\cite{Lions_book}) to show that for every $1\le p<\infty$, $T>0$, $M>0$, there exists a smooth shear flow for which
$\| u(0)\|_{W^{1,p}(\mathbb T^3)}=1$ and $\| u(T)\|_{W^{1,p}(\mathbb T^3)}>M$.
 Recently Bardos and Titi \cite{BT10} revisited this
example and constructed a weak solution which initially lies in $C^{\alpha}$ but does not belong
to any $C^{\beta}$ for any $t>0$ and $1>\beta>\alpha^2$. By similar arguments one can also deduce illposedness in
$F^{1}_{\infty,2}$ and $B^1_{\infty,\infty}$ (see Remark 1 therein).
In \cite{MY12}, Misio{\l}ek and Yoneda considered
the logarithmic Lipschitz space $\operatorname{LL}_{\alpha}(\mathbb R^d)$ consisting of continuous functions
such that
\begin{align*}
\| f\|_{\operatorname{LL}_{\alpha}} = \| f \|_{\infty} + \sup_{0<|x-y|<\frac 12} \frac{|f(x)-f(y)|}{|x-y| |\log |x-y||^{\alpha}}<\infty.
\end{align*}
They used the above shear-flow example to generate illposedness of 3D Euler in $\operatorname{LL}_{\alpha}$ for any $0<\alpha\le 1$.
In connection with Conjecture \ref{conj1}, a related issue is the dependence of the solution operator on the
underlying topology. In \cite{Kato75}, to describe the sharpness of
the continuous dependence on initial data in his wellposedness
result, Kato showed that (see Example 5.2 therein) the solution
operator for the Burgers equation is not H\"older continuous in
$H^s(\mathbb R)$, $s\ge 2$ norm for any prescribed H\"older
exponent. In \cite{HM10} Himonas and Misio{\l}ek proved that for the
Euler equation the data-to-solution map is not uniform continuous in
$H^s(\Omega)$ topology where $s\in \mathbb R$ if $\Omega=\mathbb
T^d=\mathbb R^d/2\pi\mathbb Z^d$ and $s>0$ if $\Omega=\mathbb R^d$.
Very recently Inci \cite{In13} strengthened this result and showed for any $T>0$
that the solution map $u(0) \to u(T)$ is nowhere locally uniformly continuous for $H^s(\mathbb R^n)$, $s>n/2+1$.
In \cite{CS10}, Cheskidov and Shvydkoy proved illposedness of
$d$-dimensional Euler in Besov spaces $B^s_{r,\infty}(\mathbb T^d)$
where $s>0$ if $r>2$ and $s>d(\frac 2r-1)$ if $1\le r\le 2$.
However, as was pointed out by the aforementioned authors, the above
works 
do not address the borderline Sobolev space $H^{d/2+1}$ or
similar critical spaces which was an outstanding open problem.

The purpose of this work is to completely settle the borderline case $H^{ d/2+1}$ (Conjecture \ref{conj1})
and several other related open problems. Roughly speaking, we prove the following
\medskip

\noindent
\textbf{Theorem}. Let the dimension $d=2,3$. The Euler equation \eqref{V_usual} is illposed
in the Sobolev space $W^{d/p+1,p}$ for any $1<p<\infty$ or the Besov space $B^{d/p+1}_{p,q}$ for
any $1<p<\infty$, $1<q\le \infty$.
\medskip

As a matter of fact, we shall show that in the
borderline case, ill-posedness holds in the strongest sense. Namely
for \emph{any} given smooth initial data, we shall find special
perturbations which can be made arbitrarily small in the critical
Sobolev  norm, such that the corresponding perturbed solution is
unique (in other functional spaces) but loses borderline Sobolev
regularity instantaneously in time. Our analysis shows that in
some sense the illposedness happens in a very generic way. In
particular, it is ``dense'' in the $H^{ d/2+1}$ (and similarly for
other critical spaces) topology.

We now state more precisely the main results.
The first result is for 2D Euler with non-compactly supported data. A special feature is that our constructed solutions
are $C^{\infty}$-smooth which are classical solutions.

\begin{thm}[2D non-compact case]\label{thm1}
For any given $\omega^{(g)}_0\in C_c^{\infty}(\mathbb R^2) \cap \dot H^{-1}(\mathbb R^2)$ and any $\epsilon>0$,
we can find a $C^{\infty}$ perturbation $\omega^{(p)}_0:\mathbb R^2\to \mathbb R$ such that the following hold true:

\begin{enumerate}
 \item $\| \omega^{(p)}_0 \|_{\dot H^1(\mathbb R^2)} + \| \omega^{(p)}_0\|_{L^1(\mathbb R^2)} +
 \| \omega^{(p)}_0\|_{L^{\infty}(\mathbb R^2)}+\|\omega^{(p)}_0\|_{\dot H^{-1}(\mathbb R^2)}
 <\epsilon$.
\item  Let $\omega_0= \omega^{(g)}_0 + \omega^{(p)}_0$. The initial velocity $u_0 = \Delta^{-1} \nabla^{\perp} \omega_0$ has regularity
  $u_0 \in H^2(\mathbb R^2) \cap C^{\infty} (\mathbb R^2) \cap L^{\infty}(\mathbb R^2)$.
 \item There exists a unique classical solution $\omega = \omega(t)$ to the 2D Euler equation (in vorticity form)
 \begin{align*}
  \begin{cases}
   \partial_t \omega + (\Delta^{-1} \nabla^{\perp} \omega \cdot \nabla) \omega =0, \quad 0<t\le 1, \, x \in \mathbb R^2,\\
   \omega \Bigr|_{t=0} = \omega_0,
  \end{cases}
 \end{align*}
satisfying
\begin{align*}
 \max_{0\le t\le 1}\Bigl( \| \omega(t,\cdot) \|_{L^1}+\|\omega(t,\cdot)\|_{L^\infty} + \|\omega(t,\cdot)\|_{\dot H^{-1}} \Bigr)<\infty.
\end{align*}
Here $\omega(t) \in C^{\infty}$, $u(t)=\Delta^{-1} \nabla^{\perp} \omega (t)\in C^{\infty}
\cap L^2\cap L^{\infty}$ for each $0\le t \le 1$.

\item For any $0<t_0 \le 1$, we have
\begin{align} \label{thm1_refined_e1}
 \operatorname{ess-sup}_{0<t \le t_0} \| \omega(t,\cdot) \|_{\dot H^1} =+\infty.
\end{align}

\end{enumerate}

\end{thm}
\begin{rem}
 The $\dot H^{-1}$ assumption on the vorticity data $\omega^{(g)}_0$
 can be removed. We include it here simply to stress that the
 perturbed solution can inherit $\dot H^{-1}$ regularity which is natural since the
  corresponding velocity will be in $L^2$.
 Of course one can also state  similar results for $\omega^{(g)}_0 \in H^s$ with $s>1$ or
 some other subcritical functional spaces.
\end{rem}

\begin{rem}
In our construction, although the initial velocity $u_0$ is $C^{\infty}$-smooth, its
gradient turns out to be unbounded, i.e.
$\| \nabla u_0 \|_{L^{\infty}(\mathbb R^2)}=+\infty$.
\end{rem}

\begin{rem}
In  \cite{Kato75} Kato introduced the uniformly local Sobolev spaces
$L^p_{ul}(\mathbb R^d)$ (see \eqref{uniform_local_def1}) and
$H^s_{ul}(\mathbb R^d)$. These spaces contain $H^s(\mathbb R^d)$ and
the periodic space $H^s(\mathbb T^d)$. The statement
\eqref{thm1_refined_e1} in Theorem \ref{thm1} can be improved to
\begin{align} \notag
 \operatorname{ess-sup}_{0<t \le t_0} \| \nabla \omega(t,\cdot) \|_{L^2_{ul}(\mathbb R^2)} =+\infty.
\end{align}
Similar results also hold for Theorem \ref{thm2}--\ref{thm4} below. We shall not state them but leave it to interested readers.

\end{rem}

Our next result is for the compactly supported data for the 2D Euler equation. Note that this result carries over
(with simple changes) to the periodic case as well. For simplicity we shall consider vorticity functions having one-fold
symmetry. For example, we shall say $g=g(x_1,x_2):\, \mathbb R^2\to \mathbb R$ is odd in $x_1$ if
\begin{align*}
 g(-x_1,x_2)=-g(x_1,x_2),\quad \forall\, x=(x_1,x_2)\, \in \mathbb R^2.
\end{align*}
It is not difficult to check that the one-fold odd symmetry (of the vorticity function) is preserved by the Euler flow.

\begin{thm}[2D compact case]\label{thm2}
Let $\omega^{(g)}_0\in C_c^{\infty}(\mathbb R^2) \cap \dot
H^{-1}(\mathbb R^2)$ be any given vorticity function which is odd in
$x_2$.\footnote{Similar results also hold for vorticity functions
which are odd in $x_1$, or odd in both $x_1$ and $x_2$.} For any
such $\omega^{(g)}_0$ and any $\epsilon>0$, we can find a
perturbation $\omega^{(p)}_0:\mathbb R^2\to \mathbb R$ such that the
following hold true:

\begin{enumerate}
 \item $\omega^{(p)}_0$ is compactly supported (in a ball of radius $\le 1$), continuous and
 $$\| \omega^{(p)}_0 \|_{\dot H^1(\mathbb R^2)}  +
 \| \omega^{(p)}_0\|_{L^{\infty}(\mathbb R^2)}+\|\omega^{(p)}_0\|_{\dot H^{-1}(\mathbb R^2)}
 <\epsilon.$$

 \item Let $\omega_0= \omega^{(g)}_0 + \omega^{(p)}_0$. Corresponding to $\omega_0$
 there exists a unique time-global solution $\omega = \omega(t)$ to the Euler equation
satisfying $\omega(t) \in  L^{\infty} \cap \dot H^{-1}$. Furthermore $\omega \in C_t^0 C_x^0$ and\footnote{Actually it is
easy to show that $u$ is log-Lipschitz.}
$u=\Delta^{-1} \nabla^{\perp} \omega \in C_t^0 L_x^2 \cap C_t^0 C_x^{\alpha}$ for any $0<\alpha<1$.

\item $\omega(t)$ has additional local regularity in the following sense: there exists $x_* \in \mathbb R^2$ such
that for any $x\ne x_*$, there exists a neighborhood $N_x \ni x$, $t_x >0$ such that $w(t, \cdot) \in C^{\infty} (N_x)$ for any
$0\le t \le t_x$.

\item For any $0<t_0 \le 1$, we have
\begin{align*}
 \operatorname{ess-sup}_{0<t \le t_0} \| \omega(t,\cdot) \|_{\dot H^1} =+\infty.
\end{align*}
More precisely, there exist $0<t_n^1<t_n^2 <\frac 1n$, open precompact sets $\Omega_n$, $n=1,2,3,\cdots$ such that
$\omega(t) \in C^{\infty}(\Omega_n)$ for all $0\le t \le t_n^2$, and
\begin{align*}
 \| \nabla \omega(t,\cdot) \|_{L^2(\Omega_n)} >n, \quad \forall\, t\in[t_n^1,t_n^2].
\end{align*}

\end{enumerate}

\end{thm}

\begin{rem}
In \cite{Yu63} Yudovich proved the existence and uniqueness of weak solutions to 2D Euler in bounded domains
for $L^{\infty}$ vorticity data. The uniqueness result (for bounded domain in general dimensions
$d\ge 2$) was improved in \cite{Yu95} allowing vorticty $\omega \in \cap_{p_0 \le p<\infty} L^p$
and $\| \omega \|_p \le C \theta (p)$ with $\theta(p)$ growing relatively slowly in $p$ (such as
$\theta(p)=\log p$). Vishik \cite{Vishik99} proved the uniqueness of weak solutions to Euler in
$\mathbb R^d$, $d\ge 2$, under the following assumptions:
\begin{itemize}
\item  $\omega \in L^{p_0}$, $1<p_0<d$,
\item For some $a(k)>0$ with the property
\begin{align*}
\int_1^{\infty} \frac 1 {a(k)} d k =+\infty,
\end{align*}
it holds that
\begin{align*}
\Bigl|\sum_{j=2}^{k} \| P_{2^j} \omega \|_{\infty} \Bigr| \le \operatorname{const} \cdot a(k), \qquad \forall\, k\ge 4.
\end{align*}
\end{itemize}
In other words uniqueness is guaranteed as long as $\omega$ has a little bit integrability and
the partial sum of the Besov $\dot B^0_{\infty,1}$ norm of $\omega$ is allowed to diverge in a controlled fashion.
Since we have uniform in time
$L^{\infty}$ control of the vorticity $\omega$ in both 2D and 3D (see Theorem \ref{thm3}--\ref{thm4} below),
the uniqueness of the constructed solution is not an issue and we shall not discuss this point further
in this work.
\end{rem}

Our third result is for 3D Euler with non-compactly supported data. As is well-known the lifespan of solutions to
3D Euler emanating from smooth initial data is an outstanding open problem. Since we are perturbing smooth
initial data using functions with critical Sobolev regularity, we need to make sure the perturbed solution has
a positive lifespan in some suitable functional spaces.
In the non-compact data case, this issue turns out to be immaterial
since we can choose the patches sufficiently far away from each other and the lifespan of each patch is
well under control.

\begin{thm}[3D non-compact case] \label{thm3}
Consider the 3D incompressible Euler equation in vorticity form:
\begin{align} \label{thm3_1}
 \begin{cases}
  \partial_t \omega + (u\cdot \nabla) \omega = (\omega \cdot \nabla)u, \quad t>0, \; x= (x_1,x_2,z) \in \mathbb R^3;\\
  u=-\Delta^{-1} \nabla \times \omega,\\
  \omega \Bigr|_{t=0} =\omega_0.
 \end{cases}
\end{align}

For any given $\omega^{(g)}_0 \in C_c^{\infty}(\mathbb R^3)$ and any $\epsilon>0$,
we can find a $T_0=T_0(\omega^{(g)}_0)>0$ and $C^{\infty}$ perturbation $\omega^{(p)}_0:\mathbb R^3\to \mathbb R^3$
such that the following hold true:

\begin{enumerate}
 \item $\| \omega^{(p)}_0 \|_{\dot H^{\frac 32}(\mathbb R^3)} + \| \omega^{(p)}_0\|_{L^1(\mathbb R^3)} +
 \| \omega^{(p)}_0\|_{L^{\infty}(\mathbb R^3)} <\epsilon$.
\item Let $\omega_0= \omega^{(g)}_0+ \omega^{(p)}_0$. Let $u_0$ be the velocity corresponding to the initial vorticity $\omega_0$. We have
$u_0 \in H^{\frac 52}(\mathbb R^3) \cap C^{\infty} (\mathbb R^3)\cap L^{\infty} (\mathbb R^3)$.

\item Corresponding to $\omega_0$, there exists a unique  solution $\omega = \omega(t)$ to \eqref{thm3_1}
on the whole time interval $[0,T_0]$ such that
\begin{align*}
 \sup_{0\le t \le 1} (\| \omega(t,\cdot)\|_{L^1} + \| \omega(t,\cdot) \|_{L^\infty} )<\infty.
\end{align*}
Moreover $\omega \in C^{\infty}$ and $u \in C^{\infty}$ so that the solution is actually classical.

\item For any $0<t_0 \le T_0$, we have
\begin{align*}
 \operatorname{ess-sup}_{0<t \le t_0} \| \omega(t,\cdot) \|_{\dot H^{\frac 32}} =+\infty.
\end{align*}

\end{enumerate}

\end{thm}
\begin{rem}
If the vorticity $\omega_0^{(g)}$ is axisymmetric (see \eqref{jun12_add1} below), then we can choose
$T_0>0$ to be any positive number. This is due to the fact that for 3D Euler smooth axisymmetric flows without swirl
exist globally in time.
\end{rem}

The following theorem concerns the 3D Euler case with compactly supported initial vorticity. In this case
the situation is more complicated than that in Theorem \ref{thm3}. For convenience we will work with
a class of axisymmetric vorticity functions $\omega$ having the form:
\begin{align} \label{jun12_add1}
\omega(x)=\omega^{\theta}(r,z) e_{\theta}, \quad x=(x_1,x_2,z), \, r=\sqrt{x_1^2+x_2^2},
\end{align}
where $\omega^{\theta}$ is scalar-valued and $e_{\theta}=\frac 1r(-x_2,x_1,0)$. The corresponding velocity fields
are usually called axisymmetric without swirl flows. In this paper
we shall call such $\omega$ axisymmetric without swirl vorticity or simply axisymmetric vorticity when there is no
obvious confusion. The theory of axisymmetric flows on $\mathbb R^3$
and some recent developments are reviewed in the beginning of Section 7.

\begin{thm}[3D compact case]\label{thm4}
For any given axisymmetric vorticity $\omega^{(g)}_0\in C_c^{\infty}(\mathbb R^3)$ 
and any $\epsilon>0$,
we can find a perturbation $\omega^{(p)}_0:\mathbb R^3\to \mathbb R^3$ such that the following hold true:

\begin{enumerate}
 \item $\omega^{(p)}_0$ is compactly supported (in a ball of radius $\le 1$), continuous and $$\| \omega^{(p)}_0 \|_{\dot H^{\frac 32}(\mathbb R^3)}+
 \| \omega^{(p)}_0\|_{L^{\infty}(\mathbb R^3)}  <\epsilon.$$
 \item Let $\omega_0= \omega^{(g)}_0+ \omega^{(p)}_0$.
 Corresponding to $\omega_0$ there exists a unique solution $\omega = \omega(t,x)$ to the Euler equation \eqref{thm3_1} on the time
 interval $[0,1]$ satisfying
 \begin{align} \label{thm4_t1}
 &\sup_{0\le t\le 1} \|\omega(t,\cdot)\|_{L^\infty}<\infty, \notag \\
 &\operatorname{supp}(\omega(t,\cdot))\subset\{x, \: |x| < R\}, \quad \forall\, 0\le t\le 1,
 \end{align}
where $R>0$ is some constant. Furthermore $\omega \in C_t^0 C_x^0$ and $u \in C_t^0 L_x^2 \cap C_t^0 C_x^{\alpha}$ for any $\alpha<1$.

\item $\omega(t)$ has additional local regularity in the following sense: there exists $x_* \in \mathbb R^3$ such
that for any $x\ne x_*$, there exists a neighborhood $N_x \ni x$, $t_x >0$ such that $w(t) \in C^{\infty} (N_x)$ for any
$0\le t \le t_x$.

\item For any $0<t_0 \le 1$, we have
\begin{align} \label{thm4_t2}
 \operatorname{ess-sup}_{0<t \le t_0} \| \omega(t,\cdot) \|_{\dot H^{\frac 32}(\mathbb R^3) } =+\infty.
\end{align}
More precisely, there exist $0<t_n^1<t_n^2 <\frac 1n$, open precompact sets
$\Omega_n^1$, $\Omega_n^2$ with $\Omega_n^1\subset \overline{\Omega_n^1} \subset \Omega_n^2$, $n=1,2,3,\cdots$ such that
\begin{itemize}
\item $\omega(t) \in C^{\infty}(\Omega_n^2)$ for all $0\le t \le t_n^2$;
\item $\omega(t,x)\equiv 0$ for any $x\in \Omega_n^2 \setminus \Omega_n^1$, $0\le t\le t_n^2$.
\item Define $\omega_n(t,x)=\omega(t,x)$ for $x\in \Omega_n^1$, and $\omega_n(t,x)=0$ otherwise. Then
$\omega_n \in C^\infty_c(\mathbb R^3)$,
 \begin{align} \label{thm4_t3a}
 \| \omega_n(t,\cdot) \|_{\dot H^{\frac 32} (\mathbb R^3)} >n, \quad \forall \, t_n^1\le t\le t_n^2.
\end{align}
and
\begin{align} \label{thm4_t3b}
\| (|\nabla|^3 \omega_n)(t,\cdot)\|_{L^2(x\in \mathbb R^3 \setminus \Omega_n^2)} \le 1,
\quad\forall\, 0\le t\le t_n^2.
\end{align}
\end{itemize}

\end{enumerate}

\end{thm}

\begin{rem}
We stress that the situation here in Theorem
\ref{thm4} is much more complex than the 2D case in Theorem \ref{thm2}.
Due to the nonlocal character of the fractional differentiation operator $|\nabla|^{\frac 32}$, we
have to include the additional constraint \eqref{thm4_t3b} in our construction
in order to derive \eqref{thm4_t2} from \eqref{thm4_t3a}. We briefly sketch the argument as follows.
Suppose $\|\omega(\tau,\cdot)\|_{\dot H^{\frac 32}}<\infty$, for  some $\tau \in[t_n^1,t_n^2]$.
Then we write
\begin{align*}
\omega(\tau)= \omega_n(\tau)+ g_n(\tau),
\end{align*}
where $g_n(\tau)=\omega(\tau)-\omega_n(\tau)$ also has finite $\dot H^{\frac 32}$-norm. Clearly
\begin{align*}
\| \omega(\tau)\|_{\dot H^{\frac 32}}^2 &= \|\omega_n(\tau)\|_{\dot H^{\frac 32}}^2 +\|g_n(\tau)\|_{\dot H^{\frac 32}}^2
\\&\;\; +2\langle |\nabla|^{\frac 32}\omega_n(\tau), |\nabla|^{\frac 32} g_n(\tau) \rangle,
\end{align*}
where $\langle,\rangle$ denotes the usual $L^2$ inner product on $L^2(\mathbb R^3)$. Now observe that
$\operatorname{supp}(g_n(\tau)) \subset \mathbb R^3\setminus \Omega_n^2$ and $\|g_n(\tau)\|_{L^2} \lesssim \|\omega(\tau)\|_2$
($g_n(\tau)$ and $\omega_n(\tau)$ have disjoint supports),  therefore
\begin{align*}
& |\langle |\nabla|^{\frac 32}\omega_n(\tau), |\nabla|^{\frac 32} g_n(\tau) \rangle| \notag \\
= & |\langle |\nabla|^{3}\omega_n(\tau), g_n(\tau) \rangle| \notag \\
\le & \|(|\nabla|^3 \omega_n)(\tau) \|_{L^2(\mathbb R^3\setminus \Omega_n^2)} \| g_n(\tau)\|_{L^2} \notag \\
\lesssim & \| \omega(\tau)\|_2 \lesssim 1.
\end{align*}
Hence for $n$ sufficiently large, we have for any $\tau \in [t_n^1,t_n^2]$, either
\begin{align*}
\|\omega(\tau)\|_{\dot H^{\frac 32}} =+\infty
\end{align*}
or
\begin{align*}
\|\omega(\tau)\|_{\dot H^{\frac 32}} >\frac n2.
\end{align*}
This obviously implies \eqref{thm4_t2}.

\end{rem}

Theorem \ref{thm1}--\ref{thm4} can be sharpened significantly. We have the following Besov version which essentially includes
all previous theorems as special cases. In order not to
overburden with notations, we shall state an informal version.  The detailed (and more precise)
statements can be found in Section \ref{sec_Bp2} and Theorem \ref{thm1_Bp2}--\ref{thm4_Bp2} therein.

\begin{thm}[Besov case] \label{thm5}
Let $d=2,3$. For any smooth initial velocity $u_0^{(g)}$, any $\epsilon>0$,  and any $1<p<\infty$, $1<q\le \infty$,
there exist a nearby initial velocity $u_0 \in B^{\frac dp+1}_{p,q}$ such that $\|u_0-u_0^{(g)}\|_{B^{\frac dp}_{p,q}}<\epsilon$,
and  the corresponding solution satisfies
\begin{align*}
 \operatorname{ess-sup}_{0<t<t_0} \| u(t,\cdot)\|_{\dot B^{\frac dp+1}_{p,\infty}} = +\infty
\end{align*}
for any $t_0>0$.
\end{thm}

Our last result concerns the illposedness in the usual Sobolev $W^{s,p}$ spaces.

\begin{thm} \label{thm6}
Let $d=2,3$. For any smooth initial velocity $u_0^{(g)}$, any $\epsilon>0$, and any $1<p<\infty$, there exists
a nearby initial velocity $u_0 \in W^{d/p+1,p}$ such that $\| u_0 - u_0^{(g)}\|_{W^{d/p+1,p}}<\epsilon$, and
the solution corresponding to $u_0$ satisfies
\begin{align*}
\operatorname{ess-sup}_{0<t<t_0} \| u(t,\cdot)\|_{ \dot W^{d/p+1,p}} =+\infty
\end{align*}
for any $t_0>0$.
\end{thm}

The proof of Theorem \ref{thm6} will be omitted. It can be subsumed under a more general
argument dealing with Triebel-Lizorkin spaces which we will address in a forthcoming paper.

In the rest of this introduction, we give a brief overview of the proofs of Theorem \ref{thm1}--\ref{thm5}.
The overall scheme consists of three steps. The first two steps are devoted to local constructions.
The last step is a global patching argument. We shall explain the main arguments for the 2D Euler case with
$H^2$ (for vorticity the space is  $H^1$) as the working critical space. Some additional technical points needed to treat the 3D case will be
clarified along the way.


\begin{itemize}

\item[Step 1.] Creation of large Lagrangian deformation. Define the flow map associated to \eqref{V_usual}
as $\phi=\phi(t,x)$ which solves
\begin{align*}
\begin{cases}
\partial_t \phi(t,x) = u(t,\phi(t,x)),\\
\phi(0,x)=x.
\end{cases}
\end{align*}
For any $0<T\ll 1$, $B(x_0,\delta) \subset \mathbb R^2$ with $x_0 \in \mathbb R^2$ arbitrary and $\delta\ll 1$,
we choose initial (vorticity) data $\omega^{(0)}_a$ such that
\begin{align*}
\|\omega^{(0)}_a \|_{L^1} + \| \omega^{(0)}_a\|_{L^{\infty}} + \| \omega^{(0)}_a \|_{H^1} \ll 1,
\end{align*}
and
\begin{align*}
\sup_{0<t\le T} \| D \phi_a(t,\cdot) \|_{\infty} \gg 1.
\end{align*}
Here $\phi_a$ is the flow map associated with the velocity $u=u_a$ which solves \eqref{V_usual} with $\omega^{(0)}_a$
as vorticity initial data. By translation invariance of Euler it suffices to consider the case $x_0=0$. In our
construction we restrict to some special flows which have odd symmetry and admit the origin as a stagnation point.
We prove that the deformation matrix $D u$ remains essentially hyperbolic near the spatial origin in the short time
interval considered (cf. Proposition \ref{prop40} and Proposition \ref{prop_gener_1}).

\item[Step 2.] Local inflation of critical norm. As was already mentioned, the critical norm for the vorticity
is $H^1$. The solution constructed in Step 1 does not necessarily obey $\sup_{0<t\le T} \|\nabla \omega_a(t)\|_2
\gg 1$. We then perturb the initial data $\omega^{(0)}_a$ and take
\begin{align*}
\omega^{(0)}_b = \omega^{(0)}_a + \frac 1 k \sin (k f(x))g(x),
\end{align*}
where $k$ is a very large parameter. The function $g$ is smooth and has $o(1)$ $L^2$ norm.\footnote{In the actual
perturbation argument, we need to divide by a suitable power
of $\|D \phi\|_{\infty}$.} The function $f(x)$ and the support of $g$ will be chosen depending on the exact location
of the maximum of $\|D \phi_a(t,\cdot)\|_{\infty}$. Of course since the initial data is altered, the corresponding
characteristic line (flow map) is changed as well. For this we run a perturbation argument in $W^{1,4}$ so that
$\| D\phi_b(t,\cdot)-D \phi_a(t,\cdot)\|_{\infty}\ll 1 $. The same argument is used to show that in the main
order the $H^1$ norm of the solution
corresponding to $\omega^{(0)}_b$ is inflated through the Lagrangian deformation matrix $D \phi_a$.
The technical details are elaborated in Proposition \ref{prop20}.

\item[Step 3.] Gluing of patch solutions. The construction in previous two steps can be repeated in infinitely
many small patches which stay away from each other initially. To glue these solutions together we need to differentiate
 two situations. In the case of Theorem \ref{thm1}, we exploit the unboundedness nature of $\mathbb R^2$
and add each patches sequentially. Each time a new patch is added, we choose the distance between it and the old
patches so large such that their interaction is very small.  The key properties exploited here are the finite
transport speed of the Euler flow and spatial decay of the Riesz kernel. In the case of Theorem \ref{thm2}, we need to
deal with compactly supported data. This forces us to analyze in detail the interactions of the patches since the
patches can become infinitely close to each other. For each $n\ge 2$, define $\omega_{\le n-1}$ the existing patch
and $\omega_n$ the current (to be added) patch. It turns out that there exists a patch time $T_n$ such that
for $0\le t \le T_n$, the patch $\omega_n$ has disjoint support from $\omega_{\le n-1}$, and obeys the dynamics
\begin{align*}
\partial_t \omega_n + \Delta^{-1} \nabla^{\perp} \omega_n \cdot \nabla \omega_{\le n-1}
+ \Delta^{-1} \nabla^{\perp} \omega_n \cdot \nabla \omega_n =0.
\end{align*}
By a suitable re-definition of the patch center and change of variable, we find that
$\tilde \omega_n$ (which is $\omega_n$ expressed in the new variable) satisfies the equation
\begin{align}
&\partial_t \tilde \omega_n + \Delta^{-1} \nabla^{\perp} \tilde \omega_n \cdot \nabla \tilde \omega_n  \notag \\
& \quad + b(t) \begin{pmatrix} -y_1 \\ y_2 \end{pmatrix} \cdot \nabla \tilde \omega_n
+r(t,y) \cdot \nabla \tilde \omega_n=0, \notag
\end{align}
 where $b(t)=O(1)$ and $|r(t,y)| \lesssim |y|^2$.  We then choose initial data for $\omega_n$ such that
within patch time $0<t\le T_n$ the critical norm of $\omega_n$ inflates rapidly. As we take $n\to \infty$, the patch time
$T_n \to 0$ and $\omega_n$ becomes more and more localized. Note that the whole solution (consisting of
all patches $\omega_n$) is actually a time-global solution. During interaction time $T_n$
 the patch $\omega_n$ produces the desired norm inflation since it stays disjoint from all the other patches.
The details of the perturbation analysis can be found in Lemma \ref{lem57} (and some related lemmas in Section 6).

\item[\underline{The 3D case}.] Compared with the 2D case, the first difficulty in 3D is the lack of $L^p$ conservation of the
vorticity. It is deeply connected with the vorticity stretching term $(\omega\cdot \nabla) u$. To simplify the analysis
we take the axisymmetric flow without swirl as the basic building block for the whole construction. The vorticity
equation in the axisymmetric case (see the beginning of Section 7) takes the form
\begin{align*}
\partial_t \left( \frac{\omega} r \right) + (u\cdot \nabla) \left( \frac{\omega} r \right)=0, \quad
r=\sqrt{x_1^2+x_2^2}, \, x=(x_1,x_2,z).
\end{align*}
Owing to the denominator $r$, the solution
formula for $\omega$ then acquires an additional metric factor (compared with 2D)  which represents the vorticity stretching effect
in the axisymmetric setting. A lot of analysis (cf. Proposition \ref{x6}) goes into controlling the metric factor by
the large Lagrangian deformation matrix and producing the desired $H^{3/2}$ norm inflation. In our construction the
patch solutions which are made of asymmetric without swirl flows typically carry infinite $\| \omega/r \|_{L^{3,1}}$ norm
(when summing all the patches together). To glue these solutions together in the 3D compactly supported case,
we need to run a new perturbation argument (cf. Lemma \ref{y1}) which allows to add each new patch $\omega_n$ with sufficiently
small $\|\omega_n\|_{\infty}$ norm (over the whole lifespan)  such that the effect of
the large $\| \omega_n/r \|_{L^{3,1}}$ becomes negligible. All in all, the constructed patch solutions converge in
the $C^0$ metric after building several auxiliary lemmas (cf. Lemma \ref{y5}, Proposition \ref{y6}).

\end{itemize}

We have roughly described the whole strategy of the proof although some technical points could not be elucidated or
even mentioned
in this short introduction. In some sense our approach is a hybrid of the Lagrangian point of view and the Eulerian
one, using in an essential way several features of the Euler dynamics: finite speed propagation and weak interaction
between well-separated vortex patches. The rest of this paper is organized as follows. In Section 2 we set up some
basic notations and preliminaries. In Section 3 we describe in detail the first part of the local construction for the 2D case.
Section 4 is devoted to the perturbation argument needed for the 2D local construction step. In Section 5 and 6 we treat
the 2D noncompact case and compactly supported case separately. Section 7-8 are devoted to the construction in
the 3D case. Finally in Section 9 we give details for the Besov space case.

\subsection*{Acknowledgements}
J. Bourgain was supported in part by NSF No. DMS-0808042 and DMS-0835373.
D. Li was supported in part by NSF under agreement No. DMS-1128155. Any opinions, findings
and conclusions or recommendations expressed in this material are those of the authors and
do not necessarily reflect the views of the National Science Foundation.
D. Li was also supported by an Nserc discovery grant.

\section{Notation and Preliminaries}

 For any two quantities $X$ and $Y$, we denote $X \lesssim Y$ if
$X \le C Y$ for some harmless constant $C>0$. Similarly $X \gtrsim Y$ if $X
\ge CY$ for some $C>0$. We denote $X \sim Y$ if $X\lesssim Y$ and $Y
\lesssim X$. We shall write $X\lesssim_{Z_1,Z_2,\cdots, Z_k} Y$ if
$X \le CY$ and the constant $C$ depends on the quantities
$(Z_1,\cdots, Z_k)$. Similarly we define $\gtrsim_{Z_1,\cdots, Z_k}$
and $\sim_{Z_1,\cdots,Z_k}$.

 We shall denote by $X+$ any quantity of the form $X+\epsilon$ for any $\epsilon>0$.
For example we shall write
\begin{align}
Y \lesssim 2^{X+} \label{notation_plus}
\end{align}
if $Y\lesssim_{\epsilon} 2^{X+\epsilon}$ for any $\epsilon>0$. The notation $X-$ is similarly
defined.

For any center $x_0 \in \R^d$ and radius $R>0$, we use $B(x_0,R) := \{ x \in \R^d: |x-x_0| < R \}$ to
denote the open Euclidean ball. More generally for any set $A\subset \mathbb R^d$, we denote
\begin{align}
B(A,R):=\{y \in \R^d:\; |y-x|<R\text{ for some $x\in A$} \}. \label{def_ball_1}
\end{align}
For any two sets $A_1$, $A_2\subset \mathbb R^d$, we define
\begin{align}
 d(A_1,A_2) =\operatorname{dist}(A_1,A_2)= \inf\{|x-y|:\; x\in A_1,\, y \in A_2\}. \notag
\end{align}

For any $f$ on $\mathbb R^d$, we denote the Fourier transform of
$f$ has
\begin{align*}
(\mathcal F f)(\xi) = \hat f (\xi) = \int_{\mathbb R^d} f(x) e^{-i \xi \cdot x}\,dx.
\end{align*}
The inverse Fourier transform of any $g$  is given by
\begin{align*}
(\mathcal F^{-1} g)(x) = \frac 1 {(2\pi)^d} \int_{\mathbb R^d}  g(\xi) e^{i x \cdot \xi} \,d\xi.
\end{align*}


 For any $1\le p \le \infty$ we use $\|f\|_p$, $\|f \|_{L^p(\mathbb R^d)}$, or $\|f\|_{L^p_x(\mathbb R^d)}$ to  denote the
usual Lebesgue norm on $\mathbb R^d$. The Sobolev space $H^1(\mathbb
R^d)$ is defined in the usual way as the completion of
$C_c^{\infty}$ functions under the norm $\|f\|_{H^1} = \|f\|_2 + \|
\nabla f \|_2$. For any $s\in \mathbb R$, we define the homogeneous
Sobolev norm of a tempered distribution $f:\mathbb R^d\to \mathbb R$
as
\begin{align*}
\| f \|_{\dot H^s} = \Bigl(\int_{\mathbb R^d} |\xi|^{2s} |\hat f(\xi)|^2 d\xi \Bigr)^{\frac 12}.
\end{align*}
We use the Fourier transform to define the fractional differentiation operators $|\nabla|^s$ by the formula
\begin{align*}
 \widehat{|\nabla|^s f} (\xi) = |\xi|^s \hat f(\xi).
\end{align*}
For any integer $n\ge 0$ and any open set $U\subset \mathbb R^d$,
we use the notation $C^n(U)$ to denote  functions on $U$ whose $n^{th}$ derivatives
are all continuous.

For any $1\le p<\infty$, we denote by $L^p_{ul}(\mathbb R^d)$ the
Banach space endowed with the norm
\begin{align} \label{uniform_local_def1}
\| u \|_{L^p_{ul}(\mathbb R^d)} :=\sup_{x\in \mathbb R^d} \Bigl(
\int_{|y-x|<1} |u(y)|^p dy \Bigr)^{\frac 1p}.
\end{align}
Let $\phi\in C_c^{\infty}(\mathbb R^d)$ be not identically zero. The
condition $u\in L^p_{ul}$ is equivalent to
\begin{align*}
\sup_{x\in \mathbb R^d} \| \phi(\cdot-x) u(\cdot) \|_{L^p(\mathbb
R^d)} <\infty.
\end{align*}
 For any $s\in \mathbb R$ and any function $u \in H^s_{loc}(\mathbb R^d)$, one can define
 \begin{align*}
 \|u\|_{H^s_{ul}(\mathbb R^d)} = \sup_{x\in \mathbb R^d} \|\phi(\cdot-x)
 u(\cdot)\|_{H^s(\mathbb R^d)}.
 \end{align*}

In Section \ref{sec_3D_1} and later sections, we need to use Lorentz spaces. We recall the definitions here.
For a measurable function $f$, the nonincreasing rearrangement $f^*$ is defined by
\begin{align*}
f^*(t) = \inf \Bigl\{s:\, \operatorname{Leb}(x:\, |f(x)|>s) \le t \Bigr\}.
\end{align*}
For $1\le p,q<\infty$, the Lorentz space $L^{p,q}$ is the set of functions $f$ which satisfy
\begin{align*}
\| f \|_{L^{p,q}}:= \Bigl( \int_0^{\infty} (t^{\frac 1p} f^*(t) )^q \frac {dt} t\Bigr)^{\frac 1q}
<\infty.
\end{align*}
For $q=\infty$, $L^{p,\infty}$ is the set of functions such that
\begin{align*}
\| f \|_{L^{p,\infty}}=\sup_{t>0}t^{\frac 1p} f^*(t) <\infty.
\end{align*}
For $p=\infty$, we set $L^{\infty,q}=L^{\infty}$ for all $1\le q\le \infty$.
Note that $L^{p,p}=L^p$. For $1<p<\infty$, the space $L^{p,q}$ coincides with the real interpolation
from Lebesgue spaces.

 We will need to use the Littlewood--Paley frequency projection operators. Let $\varphi(\xi)$
be a smooth bump function supported in the ball $|\xi| \leq 2$ and
equal to one on the ball $|\xi| \leq 1$. For any real number $N>0$ and $f\in \mathcal S^{\prime} (\mathbb R^d)$, define the frequency
localized (LP) projection operators:
\begin{align*}
\widehat{P_{\le N }f}(\xi) &:=  \varphi(\xi/N )\hat f (\xi), \notag\\
\widehat{P_{> N }f}(\xi) &:=  [1-\varphi(\xi/N)]\hat f (\xi), \notag\\
\widehat{P_N f}(\xi) &:=  [\varphi(\xi/N) - \varphi (2 \xi /N )] \hat
f (\xi). 
\end{align*}
Similarly we can define $P_{<N}$, $P_{\geq N }$, and $P_{M < \cdot
\leq N} := P_{\leq N} - P_{\leq M}$, whenever $N>M>0$ are real numbers. We will usually
use these operators when $M$ and $N$ are dyadic numbers. The summation over $N$ or $M$ are
understood to be over dyadic numbers. Occasionally for convenience of notation we allow $M$ and $N$ not to be a power of
$2$.

We recall the following Bernstein estimates:  for any $1\le p\le q\le \infty$, $s\in \mathbb R$,
\begin{align}
& \| |\nabla|^s P_N f \|_{L_x^p(\mathbb R^d)} \sim N^s \| P_N f \|_{L_x^p(\mathbb R^d)}, \notag \\
&\| P_{\le N} f\|_{L_x^q (\mathbb R^d)} \lesssim_d N^{d(\frac 1p -\frac 1q)} \| P_{\le N} f\|_{L_x^p(\mathbb R^d)}, \notag \\
&\| P_{N} f\|_{L_x^q (\mathbb R^d)} \lesssim_d N^{d(\frac 1p -\frac 1q)} \| P_{N} f\|_{L_x^p(\mathbb R^d)}. \notag
\end{align}

For any $s\in \mathbb R$, $1\le p,q\le \infty$, we define the homogeneous Besov seminorm as
\begin{align*}
 \| f \|_{\dot B^{s}_{p,q}}
  := \begin{cases}
  \Bigl(\sum_{N>0}   N^{sq}\| P_N f \|_{L^p(\mathbb R^d)}^q \Bigr)^{\frac 1q}, \qquad \text{if $1\le q<\infty$,} \\
  \sup_{N>0} N^{s} \| P_N f \|_{L^p(\mathbb R^d)}, \qquad \text{if $q=\infty$.}
\end{cases}
\end{align*}
The inhomogeneous Besov norm $\|f \|_{B^{s}_{p,q}}$ of $f \in\mathcal S^{\prime}(\mathbb R^d)$ is
\begin{align*}
 \|f \|_{B^s_{p,q}} = \| f \|_p + \| f \|_{\dot B^s_{p,q}}.
\end{align*}

For any $s \in \mathbb R$, $1<p<\infty$, $1\le q\le \infty$, the homogeneous Triebel-Lizorkin seminorm
is defined by
\begin{align*}
\| f \|_{\dot F^s_{p,q}}:=
\begin{cases}
\Bigl\| (\sum_{N>0} N^{sq} |P_N f|^q)^{\frac 1q} \Bigr\|_{L^p},
\qquad \text{if $1\le q <\infty$,} \\
\Bigl\| \sup_{N>0} N^s |P_N f| \Bigr\|_{L^p}, \quad \text{if $q=\infty$.}
\end{cases}
\end{align*}
The inhomogeneous Triebel-Lizorkin norm is
\begin{align*}
\| f \|_{F^s_{p,q}} =\| f \|_p+ \| f \|_{\dot F^s_{p,q}}.
\end{align*}

\section{Local construction for 2D case}
We begin by describing the choice of initial data for the local
construction.

Let $\varphi_0\in C_c^{\infty}(\mathbb R^2)$ be a radial bump function
such that $\operatorname{supp}(\varphi_0) \subset B(0,1)$ and $0\le \varphi_0\le
1$. Define
\begin{align*}
 \eta_0(x_1,x_2) = \sum_{a_1,a_2=\pm 1} a_1 a_2 \cdot \varphi_0\Bigl(\frac{(x_1-a_1,x_2-a_2)} {2^{-10}} \Bigr).
\end{align*}
Clearly by definition $\eta_0$ is odd in $x_1$, $x_2$, i.e.
\begin{align*}
 \eta_0(x_1,x_2) = -\eta_0(-x_1,x_2) = -\eta_0(x_1,-x_2), \quad \forall\, x=(x_1,x_2) \in \mathbb R^2.
\end{align*}

Define for each integer $k\ge 1$,
\begin{align} \label{30_0a}
\eta_k(x) = \eta_0(2^k x).
\end{align}

Obviously,
\begin{align}
\operatorname{supp}(\eta_k)\; \subset \bigcup_{a_1,a_2=\pm 1} B\biggl(
(2^{-k}a_1,2^{-k}a_2),\, 2^{-(k+10)} \biggr), \label{30_0}
\end{align}
so that $\eta_k$ and $\eta_l$ have disjoint supports for $k\ne l$, and
\begin{align}
\| \partial^{(\alpha)} \eta_k \|_{\infty} \lesssim_{\alpha}
2^{k|\alpha|}. \label{30_1}
\end{align}

Take any $A\gg 1$ and define the following one parameter family of functions:
\begin{align}
h_A(x) = \frac{\haa} A \sum_{A\le k \le 2A} \eta_k(x). \label{30_2}
\end{align}

It is easy to check
\begin{align}
\| h_A\|_1 + \| h_A \|_{\infty} \lesssim \frac {\haa} A \notag
\end{align}
and
\begin{align}
\| h_A \|_{H^1} \lesssim \frac {\haa} {\sqrt A}. \notag
\end{align}
Note that in computing the $H^1$-norm above, we have a saving of $A^{\frac 12}$ due to the fact that
each composing piece $\eta_k$ has $O(1)$ $H^1$-norm and they have disjoint supports.

We begin with a simple interpolation lemma.

\begin{lem} \label{lem30p}
Let $\mathcal R = \mathcal R_{ij}$ be a Riesz transform on $\mathbb
R^2$, then
\begin{align}
\| \mathcal R f \|_{\infty} \lesssim \| f \|_2^{\frac 12} \| \nabla
f \|_{\infty}^{\frac 12}. \label{lem30p_1}
\end{align}
\end{lem}

\begin{proof}
By using the Littlewood-Paley decomposition, splitting into dyadic
frequencies and the Bernstein inequality, we have
\begin{align*}
\| \mathcal R f \|_{\infty} & \lesssim \sum_{N} \|P_N f \|_{\infty}
\notag \\
& \lesssim \sum_{N<N_0} N \| P_N f\|_2 + \sum_{N>N_0} N^{-1} \| P_N
\nabla f \|_{\infty} \notag \\
& \lesssim N_0 \| f\|_2 + N_0^{-1} \| \nabla f \|_{\infty}.
\end{align*}

Choosing $N_0\in 2^{\mathbb Z}$ such that $N_0 \sim \bigl( \frac{\|
\nabla f \|_{\infty}}{\| f \|_2} \bigr)^{\frac 12}$ then yields
\eqref{lem30p_1}.
\end{proof}

The following lemma gives the estimates of Riesz
transforms of compositions with Lipschitz maps on $\mathbb R^2$ for the functions $h_A$ defined
earlier.

\begin{lem} \label{lem30a}
Let $\phi: \; \mathbb R^2 \to \mathbb R^2$ be a bi-Lipschitz
function satisfying the following conditions:
\begin{itemize}
\item[(i)] $\phi(0)=0$.
\item[(ii)] $\phi=(\phi_1,\phi_2)$ commutes with the reflection map $\sigma_2(x_1,x_2)=(x_1,-x_2)$, i.e.
\begin{align*}
 &\phi_1(x_1,-x_2)= \phi_1(x_1,x_2), \notag \\
 &\phi_2(x_1,-x_2)=-\phi_2(x_1,x_2),\qquad \forall\, x=(x_1,x_2) \in \mathbb R^2.
\end{align*}

\item[(iii)] For some integer $n_0\ge 1$,
\begin{align}
\| D\phi \|_{\infty} \le 2^{n_0} \quad \text{and } \|
D (\phi^{-1})\|_{\infty} \le 2^{n_0}. \label{lem30a_1}
\end{align}
Here $\phi^{-1}$ denotes the inverse map of $\phi$. Note that equivalently we can write
\begin{align*}
 \| (D \phi)^{-1} \|_{\infty} \le 2^{n_0},
\end{align*}
where $(D\phi)^{-1}$ is the matrix inverse of $D \phi$.

\end{itemize}
Then with $w=h_A$ defined in \eqref{30_2}, we have
\begin{align}
\| \mathcal R_{11} (\omega \circ \phi)\|_{\infty} \le C \cdot
{3^{n_0}} \cdot \frac {\haa} A, \label{lem30a_2} \\
\| \mathcal R_{22} (\omega \circ \phi)\|_{\infty} \le C \cdot
{3^{n_0}} \cdot \frac {\haa} A. \label{lem30a_3}
\end{align}
Here $C>0$ is an absolute constant. $\mathcal R_{11}=\Delta^{-1}
\partial_{11}$ and $\mathcal R_{22}= \Delta^{-1} \partial_{22}$ are
the Riesz transforms.
\end{lem}
\begin{rem}
The same result holds if $\phi$ commutes with the map $\sigma_1(x_1,x_2)=(-x_1,x_2)$.
Note also that in the proof below,
we only used the oddness in $x_2$ of $h_A$ defined in \eqref{30_2}.
\end{rem}

\begin{proof}[Proof of Lemma \ref{lem30a}]
First, note that by assumption (ii) on the map $\phi$,
the function ${\eta_k} \circ \phi$ is still odd in $x_2$. Since $\mathcal
R_{11}$ is an even operator, it follows that $\mathcal R_{11}({\eta_k}
\circ \phi)(0)=0$. (More precisely one just recalls that $\mathcal
R_{11}$ is obtained by convolution with the even kernel $K(x)=
p.v.\left(\frac 1{2\pi} \cdot\frac{x_2^2-x_1^2}{(x_1^2+x_2^2)^2}\right) +\frac 12 \delta(x)$, and
$\mathcal R_{11}({\eta_k} \circ \phi)(0)= \langle {\eta_k} \circ \phi,
K\rangle =0$.)

Now let $x\in \mathbb R^2 \setminus \{0\}$, $|x| \sim 2^{-l}$. We
evaluate $\mathcal R_{11} ({\eta_k} \circ \phi)(x)$ by considering $3$
cases.

\texttt{Case 1}. $2^k\ll 2^{l-n_0}$. (see \eqref{lem30a_1} for the
definition of $n_0$.)

By definition
\begin{align}
| \mathcal R_{11} ({\eta_k} \circ \phi) (x) | = \left| \int_{\mathbb R^2} ({\eta_k} \circ
\phi)(x-y) K(y) dy \right|. \label{lem30a_4}
\end{align}

The integrand in \eqref{lem30a_4} vanishes unless $|\phi(x-y)| \sim
2^{-k}$ (see \eqref{30_0}). By \eqref{lem30a_1} and $\phi(0)=0$, we
have $$2^{-k+n_0} \gtrsim |x-y| \gtrsim 2^{-k-n_0} \gg 2^{-l}.$$
Therefore $2^{-k-n_0} \lesssim |y| \lesssim 2^{-k+n_0}$. Since
$(\eta_k\circ\phi)(y_1,y_2)$ is odd in the $y_2$ variable, obviously
\begin{align*}
\int_{2^{-k-n_0} \lesssim |y| \lesssim 2^{-k+n_0} } ({\eta_k} \circ \phi)(-y) K(y)dy
=0.
\end{align*}
We then insert the above into \eqref{lem30a_4} and compute
\begin{align*}
| \mathcal R_{11} ({\eta_k} \circ \phi)(x) |& \le \int_{2^{-k-n_0} \lesssim |y| \lesssim 2^{-k+n_0}} | ({\eta_k} \circ
\phi)(x-y) -({\eta_k} \circ \phi)(-y) | |K(y)| dy \notag \\
& \le |x| \cdot \| \nabla ({\eta_k} \circ \phi) \|_{\infty} \cdot
\int_{2^{-k-n_0} \lesssim |y| \lesssim 2^{-k+n_0}} |K(y)| dy \notag
\\
& \lesssim 2^{-l} \cdot 2^{n_0} \cdot 2^{k} \cdot n_0.
\end{align*}

\texttt{Case 2}. $2^k\gg 2^{l+n_0}$.

Again the integrand in \eqref{lem30a_4} vanishes unless $|\phi(x-y)|
\sim 2^{-k}$ which yields $2^{-k+n_0} \gtrsim |x-y| \gtrsim
2^{-k-n_0}$. Since $2^{-l}\gg 2^{-k+n_0}$ and $|x| \sim 2^{-l}$, we get $|y|
\sim 2^{-l}$. Therefore
\begin{align*}
|\mathcal R_{11} ({\eta_k} \circ \phi)(x) | & \le  \| K\|_{L^{\infty}(|y|
\sim 2^{-l})} \cdot \| {\eta_k} \circ \phi\|_1 \notag\\
& \lesssim 4^l \cdot 4^{-k} \cdot 4^{n_0} = 4^{-k+l +n_0}.
\end{align*}

\texttt{Case 3}. $2^{l-n_0} \lesssim 2^k \lesssim 2^{l+n_0}$.

In this case we use Lemma \ref{lem30p}. Then by \eqref{lem30p_1} and
\eqref{30_1},
\begin{align}
\| \mathcal R_{11} ({\eta_k} \circ \phi) \|_{\infty} & \lesssim \| {\eta_k}
\circ \phi \|_2^{\frac 12} \cdot \| \nabla ( {\eta_k} \circ \phi)
\|_{\infty}^{\frac 12} \notag \\
& \lesssim 2^{\frac{1}2 n_0} \cdot \| {\eta_k}\|_2^{\frac 12} \cdot
\| \nabla
{\eta_k}\|_{\infty}^{\frac 12} \cdot 2^{\frac 12 n_0} \notag \\
& \lesssim 2^{ n_0} \cdot 2^{-\frac k2} \cdot 2^{\frac k2} \lesssim
2^{ n_0}. \label{lem30a_5}
\end{align}

Collecting all the estimates, we then obtain
\begin{align*}
\sum_k | \mathcal R_{11} ({\eta_k} \circ \phi)(x) | & \lesssim 2^{
 n_0} \cdot n_0 + n_0 \notag \\
& \lesssim 3^{n_0}.
\end{align*}

The bound \eqref{lem30a_2} follows from this and the normalizing
factor in \eqref{30_2}. Similarly one can prove \eqref{lem30a_3} or
just use the identity $\mathcal R_{11} + \mathcal R_{22}=Id$.
\end{proof}

Since we will be dealing with symplectic maps later on, we now state
a symplectic variant of Lemma \ref{lem30a}.

\begin{lem} \label{lem30}
Let $\phi: \; \mathbb R^2 \to \mathbb R^2$ be a smooth symplectic
(i.e. ${\det}(D\phi)\equiv 1$) function satisfying the following
conditions:
\begin{itemize}
\item[(i)] $\phi(0)=0$.
\item[(ii)] $\phi$ commutes with $\sigma_2(x_1,x_2)=(x_1,-x_2)$.
\item[(iii)] For some integer $n_0\ge 1$,
\begin{align}
\| D\phi \|_{\infty} \le 2^{n_0}. \label{lem30_1}
\end{align}
\end{itemize}
Then with $w=h_j$ defined in \eqref{30_2}, we have
\begin{align}
\| \mathcal R_{11} (\omega \circ \phi)\|_{\infty} \le C \cdot
{2^{n_0}} \cdot \frac{\haa} A, \label{lem30_2} \\
\| \mathcal R_{22} (\omega \circ \phi)\|_{\infty} \le C \cdot
{2^{n_0}} \cdot \frac{\haa} A. \label{lem30_3}
\end{align}
Here $C>0$ is an absolute constant.
\end{lem}

\begin{proof}[Proof of Lemma \ref{lem30}]
This is essentially a repetition of the proof of Lemma \ref{lem30a}.
Note that by sympleticity, \eqref{lem30_1} implies that
$\|(D\phi)^{-1} \|_{\infty} \le 2^{n_0}$. Also there is a slight
improvement of constant in \eqref{lem30_2}--\eqref{lem30_3}. This is
because when bounding \eqref{lem30a_5} we no longer need to bound
the Jacobian since the map is volume-preserving.
\end{proof}

We are now ready to describe the details of the local construction:
namely the existence of large deformation for well-chosen initial
data.

To be more specific, we consider the Euler equation
\begin{align} \label{40_1}
\begin{cases}
\partial_t \omega + \left(\Delta^{-1} \nabla^{\perp} \omega \cdot \nabla\right)
\omega =0, \quad t>0, \\
\omega\Bigr|_{t=0}=h_A,
\end{cases}
\end{align}
where $h_A$ is defined in \eqref{30_2}. Easy to check that $\omega$
is odd in both $x_1$ and $x_2$. We suppress the dependence of the
solution $\omega$ on the parameter $A$ for simplicity of notation.

The equation for the (forward) characteristic lines takes the form
\begin{align} \label{40_2}
\begin{cases}
\partial_t \phi(t,x) = (\Delta^{-1} \nabla^{\perp} \omega)(t,\phi(t,x)),
\\
\phi(0,x)=x \in \mathbb R^2.
\end{cases}
\end{align}

It is easy to check that $\phi=\phi(t,x)$ is a symplectic map and
$\phi(t,0)\equiv 0$. Due to the special choice of the initial data
$h_A$, the flow associated with \eqref{40_1} and \eqref{40_2} is
hyperbolic near the origin with a large deformation gradient. The
following proposition quantifies this fact.

\begin{prop} \label{prop40}
With the notations in \eqref{40_1}--\eqref{40_2}, we have for $A$ sufficiently large,
\begin{align}
\max_{0\le t \le t_A} \| (D\phi)(t,\cdot) \|_{\infty}
> M_A, \label{40_3}
\end{align}
where $M_A=\log\log A$ and $t_A = 1/\log\log A$.
\end{prop}

\begin{proof}[Proof of Proposition \ref{prop40}]
We shall argue by contradiction. Assume that
\begin{align}
\max_{0\le t \le t_A} \| (D\phi)(t,\cdot) \|_{\infty} \le M_A.
\label{40_4}
\end{align}

By Lemma \ref{lem30}, we have
\begin{align}
&\max_{0\le t \le t_A} \| \mathcal R_{11} \omega \|_{\infty} \lesssim
M_A \frac{\haa} A, \notag \\
&\max_{0\le t \le t_A} \| \mathcal
R_{22} \omega \|_{\infty}  \lesssim M_A \frac{\haa} A.
\label{40_5}
\end{align}

Denote $D(t) = D(t,\cdot) =(D\phi)(t,\cdot)$. By \eqref{40_2} and
\eqref{40_5}, we have
\begin{align}
\partial_t D(t)& = \begin{pmatrix}
-\mathcal R_{12} \omega \quad & -\mathcal R_{22} \omega \\
\mathcal R_{11} \omega \quad & \mathcal R_{12} \omega
\end{pmatrix}
D(t)  \notag \\
&=: \begin{pmatrix} -\lambda(t) \quad  0 \\
0 \quad \lambda(t)
\end{pmatrix}
D(t) + P(t) D(t), \label{40_6}
\end{align}
where $\lambda(t,x) = (\mathcal R_{12} \omega)(t, \phi(t,x))$, and
\begin{align}
\max_{0\le t \le t_A} \| P(t) \|_{\infty} \lesssim
M_A \frac{\haa} A. \label{40_7}
\end{align}

Integrating \eqref{40_6} in time and noting that $D(0)=Id$, we get
\begin{align}
D(t) =\begin{pmatrix} e^{-\int_0^t \lambda }\quad 0 \\
0 \quad e^{\int_0^t \lambda}
\end{pmatrix}
+ \int_0^t
\begin{pmatrix}
e^{-\int_{\tau}^t \lambda }\quad 0 \\
0 \quad e^{\int_{\tau}^t \lambda}
\end{pmatrix}
P(\tau) D(\tau) d\tau. \label{40_8}
\end{align}

By \eqref{40_4}, \eqref{40_7} and \eqref{40_8}, we have for all
$0\le t \le t_A$,
\begin{align*}
e^{|\int_0^t \lambda|} & \le M_A + C_2 \cdot M_A^2 \cdot \frac{\haa} A\cdot \max_{0\le \tau \le t} \bigl(
e^{2|\int_0^{\tau} \lambda|} \bigr),
\end{align*}
where $C_2>0$ is some absolute constant.

By taking $A$ sufficiently large and a standard continuity argument, we get
\begin{align}
e^{|\int_0^t \lambda|} \le 2 M_A,\quad \forall\, 0\le t \le t_A. \label{40_10}
\end{align}

Now denote
\begin{align} \label{40_11}
D=
\begin{pmatrix}
e^{-\alpha} \quad 0 \\
0 \quad e^{\alpha}
\end{pmatrix}
+ \beta,
\end{align}
where $\alpha(t,x):= \int_0^t \lambda(\tau,x) d\tau$ and
\begin{align}
|\beta| \le C_2 \cdot M_A^2 \cdot \frac{\haa} A \cdot 4 M_A^2. \notag
\end{align}

From \eqref{40_11} we can get more information on the transport map
$\phi=\phi(t,x)$. Indeed for fixed $t$, using the fact that $\phi(t,0)\equiv 0$, we have
\begin{align*}
\phi(t,x) & = \phi(t,x) - \phi(t,0) \notag \\
& = \int_0^1 \frac d {ds} \bigl( \phi(t,sx) \bigr) ds \notag \\
& = \Bigl(\int_0^1 (D\phi)(t,sx) ds \Bigr)x \notag \\
& = \Bigl( (\int_0^1 e^{-\alpha(t,sx)} ds ) x_1,\;
(\int_0^1 e^{\alpha(t,sx)} ds) x_2 \Bigr) + \tilde \beta,
\end{align*}
where
\begin{align*}
|\tilde \beta| \lesssim M_A^4 \cdot \frac{\haa} A \cdot |x|.
\end{align*}

Note that by \eqref{40_10}, for any $0\le t \le t_A$,
\begin{align*}
\frac 1 {2M_A} \le \int_0^1 e^{ \alpha(t,sx)} ds \le 2M_A, \notag \\
\frac 1 {2M_A} \le \int_0^1 e^{ -\alpha(t,sx)} ds \le 2M_A.
\end{align*}

Since
\begin{align*}
M_A^4 \cdot \frac{\haa} A \ll \frac1 {M_A},
\end{align*}
we have if $x_1>0$, $x_2>0$, and
\begin{align*}
\frac 12 < \frac {x_1}{x_2} <2,
\end{align*}
then for $\phi(t,x)=(\phi_1(t,x),\phi_2(t,x))$, $0\le t\le t_A$,
\begin{align}
\frac 1 {10 M_A^2}
< \frac{\phi_1(t,x)} {\phi_2(t,x)} < 10 M_A^2. \label{40_13}
\end{align}

By \eqref{40_4}, we also have
\begin{align}
|\phi(t,x)| \le M_A |x|. \label{40_14}
\end{align}

These bounds will be needed later.

Now we analyze $\lambda(t,\cdot)$ at $x=0$ to get a contradiction.
We have (recall $\omega(0,x)=h_A(x))$
\begin{align}
\lambda(t,0) &= (\mathcal R_{12} \omega)(t,\phi(t,0)) = (\mathcal R_{12} \omega)(t,0) \notag\\
& =-\frac 1 {\pi} \int_{\mathbb R^2} \omega(t,x) \frac{x_1 x_2}{ (x_1^2+x_2^2)^2} dx \notag \\
& = - \frac 1 {\pi} \int_{\mathbb R^2} h_A(x) \cdot
\frac{\phi_1(t,x) \phi_2(t,x)}{ (\phi_1(t,x)^2+\phi_2(t,x)^2)^2} dx. \label{40_41}
\end{align}
In the last step above we have made a change of variable $x\to \phi(t,x)$ and used the fact $\omega(t,\phi(t,x))=\omega(0,x)=h_A(x)$.

To continue, let us observe that the maps $\phi_1$ and $\phi_2$ are sign-preserving, i.e.
if $x_1\ge 0$ (resp. $x_2\ge 0$) then $\phi_1\ge 0$ (resp $\phi_2\ge 0$). To check this, one can
use \eqref{40_2} and the fact that $\omega$ is odd in $x_1$ and $x_2$ to get
\begin{align*}
\partial_t \phi_1 & = (-\Delta^{-1} \partial_2 \omega)(t,\phi_1,\phi_2)
- (-\Delta^{-1} \partial_2 \omega)(t,0,\phi_2) \notag\\
& = F(t,\phi_1,\phi_2) \phi_1,
\end{align*}
which (by integrating in time) yields that $\text{sign}(\phi_1(t)) = \text{sign}(\phi_1(0)) = \text{sign}(x_1)$.

By using the sign property mentioned above and the parity of our solution, we conclude that the RHS integral
of \eqref{40_41} is always non-negative and can be restricted to the first quadrant. Hence by \eqref{40_41},
\eqref{40_13} and \eqref{40_14}, we have for all $0\le t \le t_A$,
\begin{align*}
-\frac{\pi} 4\lambda(t,0) & = \int_{x_1>0,x_2>0} h_A(x)
\cdot \frac{\phi_1(t,x) \phi_2(t,x)}{ (\phi_1^2(t,x) + \phi_2^2(t,x))^2}  dx \notag \\
&= \int_{x_1>0,x_2>0} h_A(x) \cdot \frac 1 {\frac {\phi_1(t,x)}{\phi_2(t,x)} + \frac{\phi_2(t,x)}{\phi_1(t,x)}} \cdot \frac 1
{\phi_1^2(t,x)+\phi_2^2(t,x)} dx \notag \\
& \ge \int_{x_1>0,x_2>0} h_A(x)
\cdot \frac 1 {20 M_A^2}
\cdot \frac 1{ M_A^2} \cdot \frac 1 {|x|^2} dx \notag \\
& \gtrsim \frac 1 {M_A^4}
\cdot \frac {\haa} A
\cdot \sum_{A\le k\le 2A } \int_{x_1>0,x_2>0}
\frac{\eta_k(x)}{|x|^2} dx \notag \\
& \gtrsim M_A^{-4} \cdot \haa.
\end{align*}

Therefore
\begin{align*}
\int_0^{t_A} \lambda(t,0) dt \gtrsim t_A \cdot M_A^{-4} \cdot \haa
\end{align*}
which obviously contradicts \eqref{40_10}.
\end{proof}

The special initial data $h_A$ in Proposition \ref{prop40} can be generalized to a slightly larger class
of functions. Also the proof of Proposition \ref{prop40} can  be simplified if we take
full advantage of the odd symmetry of the data. The main observation is that by parity $x=0$ is invariant under the flow
and $(Du)(t,0)$ is diagonal for all $t>0$. We now state a more general result taking into account all these considerations.
The argument below
bypasses Lemma \ref{lem30a} and is more streamlined and quantitative. In particular  the contradiction
argument is replaced by a more effective integral (in time) inequality.


Consider
\begin{align*}
 \begin{cases}
  \partial_t \omega + (\Delta^{-1} \nabla^{\perp} \omega \cdot \nabla)\omega =0, \quad t>0,\\
  \omega \Bigr|_{t=0}=g.
 \end{cases}
\end{align*}
Assume $g\in C_c^{\infty}(\mathbb R^2)$ satisfies
\begin{itemize}
 \item[(i)] $g$ is odd in $x_1$ and $x_2$, and
 \begin{align*}
  g(x_1,x_2) \ge 0, \qquad \text{if $x_1\ge 0$ and $x_2\ge 0$.}
 \end{align*}
\item[(ii)]
\begin{align*}
 \int_{\mathbb R^2} g(x) \frac{x_1 x_2} {|x|^4} dx = B>0.
\end{align*}
\end{itemize}


Denoting by $\phi=\phi(t,x)$ the (forward) characteristic lines, we have
\begin{prop}\label{prop_gener_1}
\begin{align} \notag
  \int_0^t \frac 1 { \| D \phi(s) \|_{\infty}^4} ds   \le \frac{\pi}{4B}\log\left(1+\frac{4B}{\pi}t\right), \quad\forall\, t\ge 0.
\end{align}
In particular,
\begin{align*}
 \max_{0\le s \le t} \| D\phi(s)\|_{\infty} \ge \left( \frac {4B} {\pi} \cdot \frac t {\log(1+\frac{4B}{\pi} t)} \right)^{\frac 14},
 \quad \forall\, t>0.
\end{align*}

\end{prop}
\begin{proof}[Proof of Proposition \ref{prop_gener_1}]
By parity, we have $\phi(t,0) \equiv 0$ and
\begin{align*}
 (Du)(t,0) = \begin{pmatrix} -\lambda(t) \quad 0 \\
              0 \quad \lambda(t)
             \end{pmatrix},
\end{align*}
where $\lambda(t) =(\mathcal R_{12} \omega)(t,0)$. The off-diagonal terms of $Du$ vanishes at $x=0$ since $\mathcal R_{11} \omega$
and $\mathcal R_{22}\omega$ are both odd functions of $x_1$, $x_2$. Integrating in time gives
\begin{align*}
 (D\phi)(t,0) = \begin{pmatrix} e^{-\int_0^t \lambda(\tau) d\tau} \quad 0 \\
                 0 \quad e^{\int_0^t \lambda(\tau) d\tau}
                \end{pmatrix}.
\end{align*}
Write $\phi=(\phi_1,\phi_2)$. By parity it is easy to check
$\phi_1(t,0,x_2)\equiv 0$, $\phi_2(t,x_1,0)\equiv 0$ for any $x_1$,
$x_2\in \mathbb R$. By this and sign preservation it follows that
for any $x_1\ge 0$, $x_2\ge 0$,
\begin{align*}
&\frac 1 {\|D\phi(t)\|_{\infty}} \phi_1(t,x_1,x_2) \le x_1 \le \phi_1(t,x_1,x_2) \cdot \|D\phi(t)\|_{\infty}, \notag \\
&\frac 1 {\|D\phi(t)\|_{\infty}} \phi_2(t,x_1,x_2) \le  x_2\le \phi_2(t,x_1,x_2) \cdot \|D\phi(t)\|_{\infty}.
\end{align*}
Therefore for any $x_1>0$, $x_2>0$,
\begin{align*}
  \frac{\phi_1 \phi_2} {(\phi_1^2+\phi_2^2)^2}
  =& \frac 1 {\frac{\phi_1}{\phi_2} + \frac{\phi_2} {\phi_1} } \cdot \frac 1 {\phi_1^2+\phi_2^2} \notag \\
  \ge & \frac 1 {\| D\phi\|_{\infty}^4} \cdot \frac 1 {\frac {x_1}{x_2}+\frac{x_2}{x_1} } \cdot \frac 1 {|x|^2} \notag \\
 = & \frac 1 {\| D\phi\|_{\infty}^4}   \cdot \frac {x_1 x_2} {|x|^4}.
\end{align*}

We  compute $\lambda(t)$ as
\begin{align*}
 - \pi \lambda(t)   & =
 \int_{\mathbb R^2} g(x) \frac{\phi_1(t,x) \phi_2(t,x)} {|\phi(t,x)|^4} dx \notag \\
 & \ge  4 \int_{x_1>0,x_2>0} g(x) \frac{\phi_1(t,x) \phi_2(t,x)} {|\phi(t,x)|^4} dx \notag \\
 & \ge  \frac 4 { \|D\phi(t) \|_{\infty}^4} \int_{x_1>0,x_2>0} g(x) \frac{x_1 x_2} {|x|^4} dx \notag \\
 & = \frac B {\|D\phi(t)\|_{\infty}^4}.
\end{align*}

Since
\begin{align}
\|D\phi(t,\cdot)\|_{\infty} \ge \| (D\phi)(t,0)\|_{\infty} \ge \exp\left(- \int_0^t \lambda(s) ds \right), \notag
\end{align}
we get
\begin{align}
 \| D\phi(t) \|_{\infty} \ge \exp\left( \frac{B}{\pi}  \int_0^t \frac 1 { \| D \phi(s) \|_{\infty}^4} ds \right). \notag
\end{align}
Equivalently,
\begin{align*}
 \frac d {dt} \left( \exp\left( \frac{4B}{\pi}  \int_0^t \frac 1 { \| D \phi(s) \|_{\infty}^4} ds \right) \right) \le \frac {4B}{\pi},
 \quad \forall\, t\ge 0.
\end{align*}
Integrating in time, we get
\begin{align} \notag
  \int_0^t \frac 1 { \| D \phi(s) \|_{\infty}^4} ds   \le \frac{\pi}{4B}\log\left(1+\frac{4B}{\pi}t\right), \quad\forall\, t\ge 0.
\end{align}

\end{proof}

\section{$\dot H^1$ norm inflation by large Lagrangian deformation}
We begin with a simple ODE perturbation lemma.

\begin{lem} \label{lem20}
Suppose $u=u(t,x):\; \mathbb R\times \mathbb R^2 \to \mathbb R$,
$v=v(t,x):\; \mathbb R \times \mathbb R^2 \to \mathbb R$ are given smooth vector fields.
Let $\phi_1$, $\phi_2$ solve respectively
\begin{align*}
\begin{cases}
\partial_t \phi_1(t,x) = u(t,\phi_1(t,x)), \\
\phi_1(0,x)=x \in \mathbb R^2,
\end{cases}
\end{align*}
and
\begin{align*}
\begin{cases}
\partial_t \phi_2(t,x) = u(t, \phi_2(t,x)) + v(t,\phi_2(t,x)), \\
\phi_2(0,x)=x \in \mathbb R^2.
\end{cases}
\end{align*}

Then for some constant $C=C(\max_{0\le t\le 1} \|D^2 u(t) \|_{\infty},\,
\max_{0\le t \le 1} \|Du(t) \|_\infty)>0$, we have
\begin{align*}
\max_{0\le t \le 1}
\Bigl( \| \phi_2(t,\cdot) -\phi_1(t,\cdot )\|_{\infty}
&+ \| (D \phi_2) (t) -
(D\phi_1) (t) \|_{\infty} \Bigr)
\le  \notag \\
&C \cdot \max_{0\le t \le 1}
( \| v(t) \|_{\infty} + \| Dv(t) \|_{\infty} ).
\end{align*}
\end{lem}

\begin{proof}[Proof of Lemma \ref{lem20}]
This is quite standard. We sketch the details for the sake of completeness.

Set $\eta(t,x) = \phi_2(t,x)-\phi_1(t,x)$. Then
\begin{align*}
\partial_t \eta & = u(t,\phi_2) -u(t,\phi_1) + v(t,\phi_2) \notag \\
& = \int_0^1 (Du)(t,  \phi_1+(\phi_2-\phi_1) \theta ) d\theta \,
\eta + v(t,\phi_2).
\end{align*}

A Gronwall in time argument then yields
\begin{align*}
\max_{0\le t \le 1} \| \eta(t) \|_{\infty}
\le C \max_{0\le t \le 1} \| v(t) \|_{\infty},
\end{align*}
where the constant $C=C(\max_{0\le t \le 1} \| Du(t) \|_{\infty} )$.

Now for $\partial_x \eta$ note that
\begin{align*}
\partial_t (D\eta) &= (D u)(t,\phi_2)
D\phi_2 - (Du)(t,\phi_1) D \phi_1
+(Dv)(t,\phi_2)  D \phi_2 & \notag \\
& = ( (Du)(t,\phi_2)-(Du)(t,\phi_1) ) D \phi_2
+(Du)(t,\phi_1) D \eta +(Dv) D \phi_2  \notag \\
& = O( \| D^2 u \|_{\infty} \cdot \| \eta \|_{\infty}
\cdot \| D \phi_2 \|_{\infty} )
+ O(\| Du \|_{\infty} ) D \eta
+ O(\| D v \|_{\infty} \cdot \| D \phi_2  \|_{\infty}).
\end{align*}

It is easy to estimate
\begin{align*}
\max_{0\le t \le 1} \| (D \phi_2)(t,\cdot) \|_{\infty}
\le \exp\left( \operatorname{const} \cdot \bigl(\max_{0\le t \le 1}  (\| Du(t)\|_{\infty} + \| Dv(t) \|_{\infty}) \bigr) \right).
\end{align*}

Hence the desired bound follows from Gronwall.
\end{proof}

The following key proposition shows that large deformation of the transportation map
can produce large $\dot H^1$ norm, provided we perturb the initial data judiciously.

\begin{prop}[Large deformation induces $\dot H^1$ inflation] \label{prop20}
Suppose $\omega$ is a smooth solution to the Euler equation
\begin{align*}
\begin{cases}
\partial_t \omega + \Delta^{-1} \nabla^{\perp} \omega \cdot \nabla \omega =0, \quad
0<t\le 1, \\
\omega \Bigr|_{t=0} = \omega_0
\end{cases}
\end{align*}
satisfying the following conditions:

\begin{itemize}
\item $\| \omega_0\|_{L^1} + \| \omega_0 \|_{L^{\infty}} + \| \omega_0 \|_{\dot H^{-1}} <\infty$.
\item For some $z_0\in \mathbb R^2$, $R_0>0$, we have
\begin{align*}
\operatorname{supp}(\omega(t,\cdot) ) \subset B(z_0, \frac 12 R_0),\quad \forall\, 0\le t\le 1.
\end{align*}

\item For some $0<t_0\le 1$ and some $M\gg 1$ ($M\ge 10^{7}$ will suffice),
we have
\begin{align} \label{prop20_2}
\| (D \phi)(t_0,\cdot) \|_{\infty} >M,
\end{align}
where $\phi=\phi(t,x)$ is the (forward) characteristics:
\begin{align*}
\begin{cases}
\partial_t \phi(t,x) = (\Delta^{-1} \nabla^{\perp}\omega)(t,\phi(t,x)), \\
\phi(0,x)=x.
\end{cases}
\end{align*}
\end{itemize}

Then we can find a smooth solution $\tilde \omega$ also solving the Euler equation
\begin{align*}
\begin{cases}
\partial_t \tilde \omega + \Delta^{-1} \nabla^{\perp} \tilde \omega \cdot \nabla \tilde \omega =0,
\quad 0<t\le 1 \\
\tilde \omega \Bigr|_{t=0} = \tilde \omega_0
\end{cases}
\end{align*}
such that the following hold

\begin{enumerate}
\item $\tilde \omega_0$ is a small perturbation of $\omega_0$:
\begin{align}
\| \tilde \omega_0 \|_{L^1} & \le 2 \| \omega_0 \|_{L^1}, \label{prop20_3} \\
\| \tilde \omega_0\|_{L^{\infty}} & \le 2 \| \omega_0\|_{L^{\infty}},
\label{prop20_4} \\
\| \tilde \omega_0 \|_{\dot H^{-1}} & \le 2 \| \omega_0 \|_{\dot H^{-1}}, \\
\| \tilde \omega_0\|_{\dot H^1} &\le \| \omega_0 \|_{\dot H^1} + M^{-\frac 12}. \label{prop20_5}
\end{align}

\item For the same $t_0$ as in \eqref{prop20_2}, we have
\begin{align}
\| \tilde \omega(t_0,\cdot ) \|_{\dot H^1} > M^{\frac 13}. \label{prop20_6}
\end{align}

\item $\tilde \omega$ is also compactly supported:
\begin{align}
\operatorname{supp}(\tilde \omega(t) ) \subset B(z_0,\, R_0),
\quad \forall\, 0 \le t\le1. \label{prop20_7}
\end{align}
\end{enumerate}
\end{prop}

\begin{proof}[Proof of Proposition \ref{prop20}]
To simplify the later computation, we begin with a general derivation.
Let $W=W(t,x)$ be a smooth solution to the Euler equation
\begin{align*}
\begin{cases}
\partial_t W + \Delta^{-1} \nabla^{\perp} W \cdot \nabla W=0, \\
W\Bigr|_{t=0} =f.
\end{cases}
\end{align*}
Denote the associated (forward) characteristics as $\Phi=\Phi(t,x)$ which solves
\begin{align*}
\begin{cases}
\partial_t  \Phi(t,x) = (\Delta^{-1} \nabla^{\perp} W)(t,\Phi(t,x)),\\
\Phi(0,x)=x.
\end{cases}
\end{align*}

Let $\tilde \Phi(t,x)$ be the inverse map of $\Phi(t,x)$. Then
\begin{align*}
\tilde \Phi (t,\Phi(t,x)) = x.
\end{align*}

Differentiating the above gives us
\begin{align*}
(D\tilde \Phi)(t,\Phi(t,x)) ( D \Phi)(t,x)= Id
\end{align*}
or
\begin{align}
(D \tilde \Phi)(t,\Phi(t,x)) = (D\Phi(t,x))^{-1}, \label{prop20_8}
\end{align}
where $(D\Phi(t,x))^{-1}$ is the usual matrix inverse.

Since $\Phi(t)$ is a smooth symplectic map with $\Phi(0,x)=x$, we have $\det(D\Phi)=1$.
Denote $\Phi(t,x)=(\Phi_1(t,x),\Phi_2(t,x))$ and recall
\begin{align*}
D\Phi= \begin{pmatrix}
\frac{\partial \Phi_1}{\partial x_1}
\quad \frac{\partial \Phi_1}{\partial x_2} \\
\frac{\partial \Phi_2}{\partial x_1} \quad \frac{\partial \Phi_2}{\partial x_2}
\end{pmatrix}.
\end{align*}
Then
\begin{align} \label{prop20_9}
(D\Phi)^{-1} =
\begin{pmatrix}
\frac{\partial \Phi_2}{\partial x_2}
\quad -\frac{\partial \Phi_1}{\partial x_2} \\
- \frac{\partial \Phi_2}{\partial x_1}
\quad \frac{\partial \Phi_1}{\partial x_1}
\end{pmatrix}.
\end{align}

Since $W(t,x)=f(\tilde \Phi(t,x))$, we get
\begin{align}
\int_{\mathbb R^2} |(DW)(t,x)|^2 dx & =
\int_{\mathbb R^2} | (Df)(\tilde \Phi(t,x)) (D \tilde \Phi)(t,x)|^2 dx \notag \\
& = \int_{\mathbb R^2} |(Df)(x) (D\Phi(t,x))^{-1} |^2 dx, \label{prop20_10}
\end{align}
where we have performed a measure-preserving change of variables $x\to \Phi(t,x)$ and
used \eqref{prop20_8}.

By \eqref{prop20_9}, we can then write \eqref{prop20_10} as
\begin{align}
\| W(t,\cdot)\|_{\dot H^1}^2 &
= \int_{\mathbb R^2} |(\nabla f)(x) \cdot (\nabla^{\perp} \Phi_2)(t,x) |^2 dx \notag \\
& \qquad + \int_{\mathbb R^2} | (\nabla f )(x) \cdot (\nabla^{\perp} \Phi_1)(t,x)|^2 dx.
\label{prop20_11}
\end{align}
We shall need this formula below.

Now discuss two cases.

\texttt{Case 1}: $\| \omega(t_0,\cdot)\|_{\dot H^1} >M^{\frac 13}$. In this case
we just set $\tilde \omega = \omega$ and no work is needed.

\texttt{Case 2}: $ \| \omega(t_0, \cdot) \|_{\dot H^1} \le M^{\frac 13}$. It is this
case which requires a nontrivial analysis. We shall use a perturbation argument.

By \eqref{prop20_2}, we can find $x_*$ such that
\begin{align*}
\|({D \phi})(t_0,x_*)\|_{\infty} > M.
\end{align*}
Here for a matrix $A=(a_{ij})$, $\| A\|_{\infty} := \max
{|a_{ij}|}$.

Denote $\phi(t_0,x)=(\phi_1(t_0,x),\phi_2(t_0,x))$. Without loss of generality, we may assume
one of the entries of $(D\phi)(t_0,x_*)$ is at least $M$, namely
\begin{align*}
\left| \frac {\partial \phi_2}{\partial x_2}(t_0,x_*) \right| > M.
\end{align*}

By continuity we can find $\delta>0$ sufficiently small such that $\{ x:\,
|x-x_*|\le 2\delta\} \subset B(z_0,R_0)$ and
\begin{align}
\left| \frac{\partial \phi_2}{\partial x_2} (t_0, x) \right|> M,
\quad \forall\, |x-x_*|\le 2\delta. \label{prop20_12}
\end{align}

Now let $\Phi_0 \in C_c^{\infty} (\mathbb R^2)$ be a  radial bump function such
that $0\le \Phi_0(x) \le 1$ for all $x\in \mathbb R^2$, $\Phi_0(x)=1$ for $|x| \le 1$ and $\Phi_0(x)=0$ for $|x| \ge 2$.
Obviously
\begin{align}
\sqrt {\pi} \le \| \Phi_0\|_2 \le 2 \sqrt{\pi}. \label{prop20_13}
\end{align}

Depending on the location of $x_*$, we need to shrink $\delta>0$ slightly further if necessary and define an even function
$b \in C_c^{\infty} (\mathbb R^2)$ as follows. If $x_*=(0,0)$, we just define
\begin{align*}
 b(x) = \frac 1 {\delta} \Phi_0\Bigl(\frac {x}{\delta} \Bigr).
\end{align*}
If $x_*=(a_*,0)$ for some $a_*\ne 0$, then we shrink $\delta>0$ such that $\delta\ll |a_*|$ and define
\begin{align*}
 b(x) = \frac 1 {\delta} \left( \Phi_0\Bigl( \frac{x-x_*} {\delta} \Bigr) + \Phi_0\Bigl( \frac{x+x_*} {\delta} \Bigr) \right).
\end{align*}
The case $x_*=(0,a_*)$ for some $a_*\ne 0$ is similar. Now if $x_*=(a_*,c_*)$ for some $a_*\ne 0$ and $c_*\ne 0$, then
we take $\delta\ll \min\{|a_*|,\,|c_*|\}$ and define
\begin{align*}
 b(x) = \frac 1 {\delta} \sum_{\epsilon_1,\epsilon_2=\pm 1} \Phi_0\Bigl( \frac{x -(\epsilon_1 a_*,\epsilon_2 c_*)} {\delta} \Bigr).
\end{align*}
Easy to check that in all cases the function $b(x)$ defined above is even in $x_1$, $x_2$, i.e.
\begin{align*}
 b(x_1,x_2) = b(-x_1,x_2) = b(x_1,-x_2),\qquad\forall\, x=(x_1,x_2) \in \mathbb R^2.
\end{align*}

Now introduce the perturbation
\begin{align}
\beta(x) =\frac 1 {10k} \sin(kx_1) \cdot b(x) \cdot
\frac 1 {M^{\frac 12}} , \label{prop20_15}
\end{align}
and define
\begin{align}
\tilde \omega_0(x) =\omega_0(x) +\beta(x). \label{prop20_16}
\end{align}

We now show that if the parameter $k>0$ is taken sufficiently large then the corresponding solution
$\tilde \omega$ will satisfy all the requirements. In the rest of this proof, to simplify the presentation,
we shall use the notation $X=O(\frac 1 {k^{\alpha}})$ if the quantity $X$ obeys the bound
$X\le C_1 \cdot \frac 1 {k^{\alpha}}$ and the constant $C_1$ can depend on $(\omega, M,\Phi_0,\delta, \phi, R_0)$.

We first check \eqref{prop20_3}--\eqref{prop20_5}.

Obviously by \eqref{prop20_15}, if $k$ is sufficiently large, then
\begin{align*}
\| \beta\|_{L^1} \le \frac 1 k \cdot \frac 1 {\sqrt M} \| b\|_{L^1} \le \| \omega_0\|_{L^1},
\end{align*}
Similarly we can take $k$ large such that
\begin{align*}
\| \beta \|_{L^{\infty}} \le \| \omega_0\|_{L^{\infty}}.
\end{align*}
For the $\dot H^{-1}$-norm, note that $\beta$ is an odd function and $\hat \beta(0) =0$. Thus
\begin{align*}
 \| |\nabla|^{-1} \beta \|_2 &\lesssim \| \widehat{x \beta} \|_2 + \| \beta\|_2 \notag \\
 & =O(k^{-1})\le \|  \omega_0 \|_{\dot H^{-1}}
\end{align*}
if $k$ is taken sufficiently large.

For the $\dot H^1$-norm, by \eqref{prop20_13} we have
\begin{align*}
\| \nabla \beta\|_{L^2}^2 & \le O\left(\frac 1 {k^2} \right) + \frac 1 M \cdot 10^{-2}\int
b^2(x) \cos^2 k x_1 dx \notag \\
&  \le O\left(\frac 1 {k^2} \right) + \frac 1 {2M} \cdot 10^{-2}\int
b^2(x) dx \notag \\
& \le O \left(\frac 1 {k^2} \right) + \frac 1 {2M} \cdot 10^{-2} \cdot 4  \cdot 4\pi < \frac 1 M,
\end{align*}
where we again take $k$ sufficiently large. Consequently the bound \eqref{prop20_5} follows. It is also
not difficult to check that \eqref{prop20_7} can be fulfilled by taking $k$ large.

It remains to show \eqref{prop20_6}. We shall proceed in several steps.

First we shall show
\begin{align}
\max_{0\le t \le 1} \| \nabla \tilde \omega(t) \|_{L^4} \lesssim 1. \label{prop20_17}
\end{align}
Here the implied constant is independent of $k$ (but is allowed to depend on other parameters).

By a standard energy estimate, we have
\begin{align*}
\frac d {dt} \Bigl( \| \nabla \tilde \omega (t) \|_4^4 \Bigr)
& \lesssim \| \mathcal R_{ij} \tilde \omega(t) \|_{\infty}
\cdot \| \nabla \tilde \omega (t) \|_4^4 \notag \\
& \lesssim \log(10+\| \tilde \omega\|_2^2 + \| \nabla \tilde \omega \|_4^4 )
\cdot \| \nabla \tilde \omega\|_4^4.
\end{align*}

A Gronwall in time argument then yields \eqref{prop20_17} (by \eqref{prop20_15}, it is easy
to check that the initial data $\tilde \omega_0$ satisfies \eqref{prop20_17}).

Set $\eta=\omega-\tilde \omega$. Then
\begin{align*}
\partial_t \eta + \Delta^{-1} \nabla^{\perp} \omega \cdot \nabla \eta
+ \Delta^{-1} \nabla^{\perp} \eta \cdot \nabla \tilde \omega =0.
\end{align*}

Therefore noting that $\operatorname{supp} (\eta(t) ) \subset B(z_0,R_0)$ for any $0\le t\le 1$, we have
\begin{align*}
\frac d {dt} ( \| \eta\|_2^2)
& \lesssim \| \Delta^{-1} \nabla^{\perp} \eta \|_4 \cdot \| \nabla \tilde \omega \|_4
\cdot \| \eta\|_2 \notag \\
& \lesssim \| \eta\|_2^2 \cdot \| \nabla \tilde \omega \|_4.
\end{align*}

Integrating in time then gives
\begin{align}
\max_{0\le t \le 1} \| \eta(t) \|_2 = O\left(k^{-1} \right). \label{prop20_18}
\end{align}

Interpolating the bound \eqref{prop20_18} with \eqref{prop20_17} (note that $\omega$ also satisfies the same
bound \eqref{prop20_17}), we obtain
\begin{align}
\max_{0\le t \le 1}
\| D\Delta^{-1} \nabla^{\perp} (\tilde \omega(t) - \omega(t) ) \|_{\infty}
+ \max_{0\le t \le 1}
\| \Delta^{-1} \nabla^{\perp} (\tilde \omega(t) -\omega(t)) \|_{\infty}
=O(\frac 1 {k^{\alpha}}), \label{prop20_20}
\end{align}
where $\alpha>0$ is some absolute constant.

Denote the forward characteristic lines  associated with $\tilde \omega$ as $\tilde \phi(t,x)$
which solves
\begin{align*}
\begin{cases}
\partial_t \tilde \phi(t,x) = (\Delta^{-1} \nabla^{\perp}\tilde \omega )(t, \tilde \phi(t,x) ), \\
\tilde \phi(0,x)=x.
\end{cases}
\end{align*}

By Lemma \ref{lem20} and \eqref{prop20_20}, we have
\begin{align*}
\max_{0 \le t \le 1}
\Bigl( \|\tilde \phi(t,\cdot) -\phi(t,\cdot)\|_{\infty}
+ \| (D\phi) (t,\cdot) -
(D\tilde \phi)(t,\cdot) \|_{\infty} \Bigr)
=O(\frac 1 {k^{\alpha}}).
\end{align*}

Write $\tilde \phi(t,x)=(\tilde \phi_1(t,x),\tilde \phi_2(t,x))$. By
\eqref{prop20_11}, we get
\begin{align}
\| \tilde \omega(t_0,\cdot)\|_{\dot H^1}^2
& \ge \int | (\nabla \tilde \omega_0)(x) \cdot \nabla^{\perp}
\tilde \phi_2(t_0,x)|^2 dx \notag \\
& \ge \int | \nabla \tilde \omega_0 (x) \cdot \nabla^{\perp} \phi_2(t_0,x)|^2 dx
-O( \frac 1 {k^{\alpha}} ) \notag \\
& \ge \frac 1 2 \int |\nabla \beta(x) \cdot \nabla^{\perp}
\phi_2(t_0,x)|^2 dx - \int |\nabla \omega_0(x) \cdot \nabla^{\perp}
\phi_2(t_0,x)|^2 dx -O( \frac 1 {k^{\alpha}} ), \label{prop20_22}
\end{align}
where in the last step we used the simple inequality
\begin{align*}
 |a+b|^2 \ge \frac 12 |a|^2 - |b|^2, \qquad \forall\, a,b\in \mathbb R^d.
\end{align*}

Since we are in Case 2, we have $\| \omega(t_0, \cdot )\|_{\dot H^1} \le M^{\frac 13}$.
By \eqref{prop20_11}, we get
\begin{align}
\int|\nabla \omega_0(x) \cdot \nabla^{\perp} \phi_2(t_0,x)|^2 dx
\le \| \omega(t_0,\cdot)\|_{\dot H^1}^2 \le M^{\frac 23}. \label{prop20_24}
\end{align}

By our choice of the function $\beta$  and
\eqref{prop20_12}, we have
\begin{align}
& \frac 12 \int_{\mathbb R^2} |\nabla \beta(x) \cdot \nabla^{\perp} \phi_2(t_0,x)|^2 dx \notag \\
\ge & \frac 12 \int_{\mathbb R^2} \left| \frac{\cos (kx_1) b(x)}{10\sqrt M}
\cdot \frac{\partial \phi_2}{\partial x_2} (t_0,x) \right|^2 dx - O\left(k^{-2} \right) \notag \\
\ge & \frac 12 10^{-2} \cdot M \cdot
\int b^2(x) \cos^2(k x_1) dx  - O\left(k^{-2} \right) \notag \\
\ge & \frac {\pi} 4 \cdot 10^{-2} \cdot M  - O\left(k^{-2} \right). \label{prop20_26}
\end{align}

Plugging \eqref{prop20_24} and \eqref{prop20_26} into \eqref{prop20_22}, we get
\begin{align*}
\| \tilde \omega(t_0,\cdot) \|_{\dot H^1}^2
& \ge \frac{\pi }4 10^{-2} M -M^{\frac 23}- O\left(k^{-2}\right) - O\left(k^{-\alpha} \right) \notag \\
& \ge 0.7 \cdot 10^{-2} M-M^{\frac 23},
\end{align*}
if $k$ is taken sufficiently large. Clearly \eqref{prop20_6} follows.
\end{proof}

\section{Local to global: gluing the patches}

In this section we prove a general proposition which allows us to glue the local solutions into
a global one. We begin with some auxiliary lemmas.

To state the next lemma, we need to fix a sufficiently large constant $A_1>1$ such that
\begin{align}
\| \Delta^{-1} \nabla^{\perp} f \|_{\infty}
\le A_1 \cdot ( \| f\|_1 + \|f \|_{\infty}),
\quad\forall\, f \in L^1(\mathbb R^2) \cap L^{\infty}(\mathbb R^2). \label{lem10_1}
\end{align}
Note that $A_1$ is an absolute constant which does not depend on any parameters.

\begin{lem} \label{lem10}
Consider the Euler equation on $\mathbb R^2$:
\begin{align}
\begin{cases}
\partial_t \omega + \Delta^{-1} \nabla^{\perp} \omega \cdot \nabla \omega =0,
\quad 0<t\le 1, \\
\omega\Bigr|_{t=0} =\omega_0 =f+g.
\end{cases}
\label{lem10_2}
\end{align}
Assume $f\in H^k \cap L^1$ for some $k\ge 2$, $g\in H^2 \cap L^1$ and
\begin{align}
& \| \omega_0 \|_{L^1} + \| \omega_0\|_{L^{\infty}} \le C_1<\infty, \label{lem10_3} \\
& d(\operatorname{supp}(f), \, \operatorname{supp}(g) ) \ge 100 A_1 C_1>0, \label{lem10_4}
\end{align}
where $A_1$ is the same constant as in \eqref{lem10_1}.

Then for any $0\le t\le 1$, the following hold true:
\begin{enumerate}
\item The solution $\omega(t)$ to \eqref{lem10_2} can be decomposed as
\begin{align}
\omega(t) = \omega_f (t) + \omega_g(t), \label{lem10_5}
\end{align}
where $\omega_f(0)=f$, $\omega_g(0)=g$, and (see \eqref{def_ball_1})
\begin{align*}
&\operatorname{supp}(\omega_{f} (t) ) \subset B( \operatorname{supp}(f), \, 2A_1C_1), \\
&\operatorname{supp} (\omega_{g}(t) ) \subset B(\operatorname{supp}(g), 2A_1 C_1).
\end{align*}

\item The Sobolev norm of $\omega_f(t)$ can be bounded in terms of $\|f\|_{H^k}$ and $C_1$ only:
\begin{align}
\max_{0\le t \le 1}
\| \omega_f (t) \|_{H^k} \le C(\| f\|_{H^k}, C_1)<\infty. \label{lem10_6}
\end{align}

\end{enumerate}

\end{lem}

\begin{proof} [Proof of Lemma \ref{lem10}]
By \eqref{lem10_3} and \eqref{lem10_1}, we have
\begin{align*}
\max_{0\le t \le 1} \| \Delta^{-1} \nabla^{\perp} \omega(t) \|_{\infty} \le A_1 C_1.
\end{align*}

By the transport nature of the equation, the support of the solution $\omega(t)$ is enlarged at most
a distance $A_1 C_1$ from its original support in unit time. The decomposition \eqref{lem10_5} follows
easily from this observation and \eqref{lem10_4}. More precisely, $\omega_f$ and $\omega_g$ are solutions
to the following \emph{linear} equations:
\begin{align*}
 \begin{cases}
  \partial_t \omega_f + (u(t)\cdot \nabla) \omega_f =0, \\
  \omega_f \Bigr|_{t=0} =f;
 \end{cases}
\end{align*}
\begin{align*}
 \begin{cases}
  \partial_t \omega_g + (u(t)\cdot \nabla) \omega_g =0, \\
  \omega_f \Bigr|_{t=0} =g.
 \end{cases}
\end{align*}
Here $u=\Delta^{-1} \nabla^{\perp} \omega$. Note that $\omega_f(t)$
and $\omega_g(t)$ stay well separated for all $0\le t\le 1$:
\begin{align}
d(\operatorname{supp}(\omega_f (t)), \,  \operatorname{supp}(\omega_g(t)) ) \ge 90 A_1C_1>0. \label{lem10_7}
\end{align}

To show \eqref{lem10_6}, we note that the equation for $\omega_f (t)$ can be rewritten as
\begin{align}
\partial_t \omega_f + \Delta^{-1} \nabla^{\perp} \omega_f \cdot \nabla \omega_f
+ \Delta^{-1} \nabla^{\perp} \omega_g \cdot \nabla \omega_f =0.
\label{lem10_7a}
\end{align}

Note that for any multi-index $\alpha$, we have
\begin{align}
(\Delta^{-1} \nabla^{\perp} \partial^{\alpha} \omega_g )(x) =
\int_{\mathbb R^2} K(x-y)  (\partial^{\alpha} \omega_g )(y) dy,  \label{lem10_8}
\end{align}
where $K(\cdot)$ is the kernel function corresponding to the operator $\Delta^{-1} \nabla^{\perp}$.

By \eqref{lem10_7}, for any $x\in \operatorname{supp}(\omega_f(t))$, $y\in \operatorname{supp}(\omega_g(t))$, we have
$|x-y|\ge 90A_1 C_1$. Therefore we can introduce a smooth cut-off function $\chi$ on the kernel $K(\cdot)$ and rewrite
\eqref{lem10_8} as
\begin{align}
(\Delta^{-1} \nabla^{\perp} \partial^{\alpha} \omega_g)(x)
& = \int_{\mathbb R^2} K(x-y) \chi_{|x-y|\ge 80A_1C_1} (\partial^{\alpha} \omega_g)(y) dy \notag \\
& = \int_{\mathbb R^2} (-1)^{|\alpha|} \partial_y^{\alpha} \Bigl(
K(x-y) \chi_{|x-y|\ge 80 A_1C_1} \Bigr) \omega_g(y)dy \notag \\
& = \int_{\mathbb R^2} \tilde K_{\alpha} (x-y) \omega_g(y) dy, \label{lem10_9}
\end{align}
where the modified kernel $\tilde K_{\alpha}$ satisfies
\begin{align}
|\tilde K_{\alpha}(z)| \lesssim_{C_1,\alpha} (1+ |z|^2)^{-\frac 12},\quad \forall\, z \in \mathbb R^2. \label{lem10_10}
\end{align}

By using $L^1$ and $L^{\infty}$ conservation, we have
\begin{align}
\| \omega_f (t) \|_{L^1} + \| \omega_f(t) \|_{L^{\infty}} + \| \omega_g(t) \|_{L^1} +
\| \omega_g(t) \|_{L^{\infty}} \le C_1. \label{lem10_11}
\end{align}

Therefore by \eqref{lem10_9}, \eqref{lem10_10}, \eqref{lem10_11} and the Cauchy-Schwartz
inequality, we have
\begin{align}
\max_{0\le t \le 1} \max_{x \in \operatorname{supp}(\omega_f(t))}
| (\Delta^{-1} \nabla^{\perp} \partial^{\alpha} \omega_g )(t,x) |
\lesssim_{C_1,\alpha} 1. \label{lem10_12}
\end{align}

The estimate \eqref{lem10_12} shows that the drift term $\Delta^{-1} \nabla^{\perp} \omega_g$
in \eqref{lem10_7a} is arbitrarily smooth on the support of $\omega_f$. Therefore the estimate
\eqref{lem10_6} follows from the standard energy estimate. For the sake of completeness we sketch the
detail here for $k=2$. By \eqref{lem10_7a}, we have
\begin{align*}
&\frac d {dt} \Bigl( \| \Delta \omega_f (t) \|_2^2 \Bigr) \\
& \le \left |\int_{\mathbb R^2} \Delta (\Delta^{-1} \nabla^{\perp} \omega_f
\cdot \nabla \omega_f ) \Delta \omega_f dx \right|
+ \left| \int_{\mathbb R^2} \Delta(\Delta^{-1} \nabla^{\perp} \omega_g
\cdot \nabla \omega_f) \Delta \omega_f dx \right| \\
& \le \int_{\mathbb R^2}|\Delta^{-1} \nabla^{\perp} \partial \omega_f|
\cdot |\partial^2 \omega_f| \cdot |\Delta \omega_f|dx
+ \max_{\substack{x\in\operatorname{supp}(\omega_f(t))\\|\alpha|\le 2}}
|\Delta^{-1} \nabla^{\perp} \partial^{\alpha} \omega_g(x)|
\cdot \| \omega_f (t)\|_{H^2}^2 \notag \\
& \lesssim_{C_1} (1+\|  \mathcal R_{ij} \omega_f(t) \|_{\infty} )
\cdot \| \omega_f (t) \|_{H^2}^2,
\end{align*}
where $\mathcal R_{ij}$ denotes the Riesz transform. By the usual $\log$ interpolation inequality
and \eqref{lem10_11}, we have
\begin{align*}
\| \mathcal R_{ij} \omega_f \|_{\infty}
\lesssim_{C_1} \log(10+ \| \omega_f (t) \|_{H^2}^2 ).
\end{align*}
Therefore
\begin{align*}
\frac d {dt} \Bigl( \| \omega_f(t) \|_{H^2}^2 \Bigr)
& \lesssim_{C_1}
\log(10+\|\omega_f (t) \|_{H^2}^2) \cdot \| \omega_f(t) \|_{H^2}^2.
\end{align*}
A $\log$ Gronwall in time argument then yields \eqref{lem10_6}.

\end{proof}

\begin{lem} \label{lem11}
Let $\omega$ and $\tilde \omega$ be solutions to the Euler equations
\begin{align*}
\begin{cases}
\partial_t \omega + \Delta^{-1} \nabla^{\perp} \omega \cdot \nabla \omega=0, \quad 0<t\le 1,\\
\omega\Bigr|_{t=0} =\omega_0=f+g.
\end{cases}
\end{align*}
and
\begin{align*}
 \begin{cases}
  \partial_t \tilde \omega +\Delta^{-1} \nabla^{\perp} \tilde \omega \cdot \nabla \tilde \omega =0, \quad 0<t\le 1 \\
  \tilde \omega \Bigr|_{t=0}=f.
 \end{cases}
\end{align*}
Assume $f\in H^{3}\cap L^1$, $g\in H^2 \cap L^1$ and
\begin{align*}
 \| \omega_0\|_{L^1} +\| \omega_0\|_{L^{\infty}} \le C_1 <\infty.
\end{align*}
Assume also $f$ is compactly supported such that
\begin{align} \label{lem11_0}
 \text{Leb}(\operatorname{supp}(f) ) \le C_2<\infty.
\end{align}
Then for any $\epsilon>0$, there exists $R_{\epsilon} = R_{\epsilon} (\epsilon,\|f\|_{H^3}, C_1,C_2)>0$ such that
if
\begin{align*}
 d(\operatorname{supp}(f), \operatorname{supp}(g) ) \ge R_{\epsilon}>0,
\end{align*}
then for any $0<t\le 1$, the following hold true:
\begin{enumerate}
 \item $\omega(t)$ has the decomposition
 \begin{align}\label{lem11_1}
  \omega(t)=\omega_f(t) + \omega_g(t),
 \end{align}
where
\begin{align}
 &\operatorname{supp}(\omega_f (t) ) \subset B( \operatorname{supp}(f),2A_1C_1); \label{lem11_1a}\\
 &\operatorname{supp}(\omega_g(t) ) \subset B( \operatorname{supp}(g), 2A_1C_1) ; \notag \\
 & d(\operatorname{supp}(\omega_f(t), \, \operatorname{supp}(\omega_g(t) ) )\ge 100 A_1C_1. \notag
\end{align}
Here $A_1$ is the same constant in \eqref{lem10_1}.

\item The support of $\tilde \omega(t)$ also satisfies
\begin{align}
 \operatorname{supp} (\tilde \omega (t) ) \subset B(\operatorname{supp}(f), 2 A_1 C_1). \label{lem11_2}
\end{align}

\item $\omega_f(t)$ and $\tilde \omega(t)$ are close:
\begin{align}
 \max_{0\le t \le 1} \| \omega_f(t) -\tilde \omega(t) \|_{H^2} <\epsilon. \label{lem11_3}
\end{align}
\end{enumerate}

\end{lem}

\begin{proof}[Proof of Lemma \ref{lem11}]
 Note that \eqref{lem11_1} and \eqref{lem11_2} follows directly from Lemma \ref{lem10}: we just need to take
 $R_{\epsilon} \ge 100 A_1 C_1$. By Lemma \ref{lem10}, we have
 \begin{align}
  &\max_{0\le t \le 1} \| \omega_f(t) - \tilde \omega (t) \|_{H^3} \notag \\
  \le &\; \max_{0\le t \le 1} \| \omega_f(t) \|_{H^3} + \max_{0\le t \le 1} \| \tilde \omega (t)\|_{H^3} \notag \\
  \le & C_3 = C_3(\|f\|_{H^3}, C_1). \label{lem11_4}
 \end{align}

Set $\eta(t) = \omega_f(t) -\tilde \omega (t)$. Then by \eqref{lem10_7a}, we have
\begin{align*}
 \begin{cases}
  \partial_t \eta + \Delta^{-1} \nabla^{\perp} \eta \cdot \nabla \omega_f
  + \Delta^{-1} \nabla^{\perp} \tilde \omega \cdot \nabla \eta + \Delta^{-1} \nabla^{\perp} \omega_g \cdot \nabla \omega_f =0,
  \quad 0<t\le 1, \\
  \eta(0)=0.
 \end{cases}
\end{align*}
For $x \in \operatorname{supp} (\omega_f(t))$, we have
\begin{align*}
 | (\Delta^{-1} \nabla^{\perp} \omega_g) (t,x)|
 & \lesssim \int_{|x-y| \ge \frac 12 R_{\epsilon}} \frac 1 {|x-y|} |\omega_g (t,y) |dy \notag \\
 & \lesssim R_{\epsilon}^{-\frac 12} \cdot \| \omega_g \|_{\frac 43} \lesssim R_{\epsilon}^{-\frac 12} \cdot C_1.
\end{align*}

Therefore
\begin{align}
 & \frac{d }{dt} \Bigl( \| \eta (t) \|_2^2 \Bigr) \notag \\
 \lesssim & \; \| \Delta^{-1} \nabla^{\perp} \eta \|_3 \cdot \| \eta\|_2 \cdot \| \nabla \omega_f \|_6
 + R_{\epsilon}^{-\frac 12} \cdot C_1 \cdot \| \nabla \omega_f \|_2 \cdot \| \eta\|_2. \label{lem11_5}
\end{align}

By Sobolev embedding, \eqref{lem11_0}, \eqref{lem11_1a}, \eqref{lem11_2} and H\"older, we have
\begin{align*}
 \| \Delta^{-1} \nabla^{\perp} \eta \|_3 & \lesssim \| \eta \|_{\frac 65} \notag \\
 & \lesssim_{C_1,C_2} \, \| \eta \|_2.
\end{align*}

By \eqref{lem11_4} and Sobolev embedding we have
\begin{align*}
 \| \nabla \omega_f\|_6 \lesssim C_3.
\end{align*}

Therefore integrating \eqref{lem11_5} in time, we obtain for some $C_4=C_4(C_1,C_2,C_3)>0$ that
\begin{align}
 \max_{0\le t \le 1} \| \eta(t) \|_2 \le R_{\epsilon}^{-\frac 12} \cdot C_4. \label{lem11_6}
\end{align}

The desired estimate \eqref{lem11_3} follows easily from interpolating \eqref{lem11_6}, \eqref{lem11_4} and
taking $R_{\epsilon}$ sufficiently large.
\end{proof}

\begin{prop}[Almost non-interacting patches] \label{prop10}
  Let $\{ \omega_j \}_{j=1}^{\infty}$ be a sequence of functions in $C_c^{\infty} (B(0,1))$ and satisfy the following condition:
  \begin{align} \label{prop10_1}
   \sum_{j=1}^{\infty} \| \omega_j \|_{H^1}^2 + \sum_{j=1}^{\infty} \| \omega_j \|_{L^1} + \sup_{j} \| \omega_j \|_{L^{\infty}}
   \le C_1 <\infty.
  \end{align}

Here we may assume $C_1>1$.

Then there exist centers $x_j \in \mathbb R^2$ whose mutual distance are sufficiently large (i.e. $|x_j-x_k|\gg 1$ if $j\ne k$) such
that the following hold:

\begin{enumerate}
 \item Take the initial data
 \begin{align*}
  \omega_0(x) = \sum_{j=1}^{\infty} \omega_j(x-x_j),
 \end{align*}
then $\omega_0 \in L^1 \cap L^{\infty} \cap H^1 \cap C^{\infty}$. Furthermore for any $j\ne k$
\begin{align} \label{prop10_2}
 B(x_j, 100A_1C_1) \cap B(x_k, 100 A_1C_1) = \varnothing.
\end{align}
Here $A_1$ is the same absolute constant as in \eqref{lem10_1}.

\item With $\omega_0$ as initial data, there exists a unique solution $\omega$ to the Euler equation
\begin{align*}
 \partial_t \omega + \Delta^{-1} \nabla^{\perp} \omega \cdot \nabla \omega =0
\end{align*}
on the time interval $[0,1]$ satisfying $\omega \in L^1\cap L^{\infty} \cap C^{\infty}$, $u=\Delta^{-1} \nabla^{\perp} \omega \in C^{\infty}$.
Moreover for any $0\le t \le1$,
\begin{align}
 \operatorname{supp} ( \omega(t,\cdot) ) \subset \bigcup_{j=1}^{\infty} B(x_j, 3A_1C_1). \label{prop10_3}
\end{align}

\item For any $\epsilon>0$, there exists an integer $J_{\epsilon}$ sufficiently large such that if $j\ge J_{\epsilon}$, then
\begin{align}
 \max_{0\le t \le 1} \| \omega(t,\cdot) - \tilde \omega_j(t,\cdot)\|_{H^2(B(x_j,3A_1C_1))} <\epsilon. \label{prop10_4}
\end{align}
Here $\tilde \omega_j$ is the solution solving the equation
\begin{align*}
 \begin{cases}
  \partial_t \tilde \omega_j + \Delta^{-1} \nabla^{\perp} \tilde \omega_j \cdot \nabla \tilde \omega_j =0,
  \quad 0<t\le 1,\, x \in \mathbb R^2; \\
  \tilde \omega_j(t=0,x)= \omega_j(x-x_j), \quad x \in \mathbb R^2.
 \end{cases}
\end{align*}
\end{enumerate}

\end{prop}

\begin{proof} [Proof of Proposition \ref{prop10}]

$$ \cdot $$

 \texttt{Step 1}. Choice of the centers $x_j$.

 For each $\omega_j$, $j\ge 1$,  we choose $R_j=R_j(\| \omega_j\|_{H^3}, C_1) >0$ corresponding to $f=\omega_j$ and $\epsilon =2^{-j}$
 in Lemma \ref{lem11} ($C_1$ is the same constant as in \eqref{prop10_1}). More precisely, if we take
 \begin{align*}
  \begin{cases}
   \partial_t \omega + \Delta^{-1} \nabla^{\perp} \omega \cdot \nabla \omega =0, \quad 0<t\le 1, \\
   \omega \Bigr|_{t=0} = \omega_j +g,
  \end{cases}
 \end{align*}
and
\begin{align*}
 \begin{cases}
  \partial_t \tilde \omega + \Delta^{-1} \nabla^{\perp} \tilde \omega \cdot \nabla \tilde \omega =0, \\
  \tilde \omega \Bigr|_{t=0}= \omega_j,
 \end{cases}
\end{align*}
with
\begin{align*}
 \| \omega_j +g \|_{L^1} + \| \omega_j + g\|_{L^{\infty}} \le C_1<\infty,
\end{align*}
and
\begin{align}
 d(\operatorname{supp} (\omega_j), \operatorname{supp}(g) ) \ge R_j, \label{prop10_5}
\end{align}
then $\omega(t) = \omega_f(t) + \omega_g(t)$ with

\begin{align*}
 \operatorname{supp} (\omega_f(t) ) \subset B(0, 1+2A_1C_1)
\end{align*}
and
\begin{align}
 \max_{0\le t \le 1} \| \omega_f(t) -\tilde \omega(t) \|_{H^2} <2^{-j}. \label{prop10_6}
\end{align}

With the numbers  $R_j$ properly defined, we now describe how to choose the centers $x_j$ inductively. First set $x_1=0$.
For $j\ge 2$, assume $x_1$,$\cdots$,$x_{j-1}$ have already been chosen. Let
\begin{align*}
 f_{j-1} (x) = \sum_{l=1}^{j-1} \omega_l(x-x_l)
\end{align*}
and consider the problems
\begin{align*}
 \begin{cases}
 \partial_t \omega + \Delta^{-1} \nabla^{\perp} \omega \cdot \nabla \omega =0, \quad 0<t\le 1, \\
 \omega \Bigr|_{t=0} = f_{j-1} +g,
\end{cases}
\end{align*}
and
\begin{align*}
 \begin{cases}
 \partial_t \tilde \omega + \Delta^{-1} \nabla^{\perp} \tilde \omega \cdot \nabla \tilde \omega =0, \\
 \tilde \omega \Bigr|_{t=0} =f_{j-1}
\end{cases}
\end{align*}
with
\begin{align*}
 \| f_{j-1} +g \|_{L^1} + \| f_{j-1} +g \|_{L^{\infty}} \le C_1 <\infty.
\end{align*}

By Lemma \ref{lem11}, we can find $\tilde R_j = \tilde R_j ( \| f_{j-1} \|_{H^3}, C_1)>0$ such that if
\begin{align}
 d( \operatorname{supp} (f_{j-1}), \, \operatorname{supp} (g) )>\tilde R_j, \label{prop10_7}
\end{align}
then
\begin{align}
 \max_{0\le t \le 1} \| \omega_{f_{j-1}} (t)  -\tilde \omega (t) \|_{H^2} <2^{-j}. \label{prop10_8}
\end{align}

We now choose $x_j$ such that
\begin{align}
 d(\operatorname{supp}(f_{j-1}),  x_j)>2\tilde R_j + 2 \sum_{l=1}^j R_l +1000A_1 C_1 +10^j. \label{prop10_8a}
\end{align}

By induction it is easy to verify that \eqref{prop10_2} holds.

\texttt{Step 2.} Construction of the solution $\omega(t)$ by patching.

Since $\omega_0 \in L^{1} \cap L^{\infty}$, the usual Yudovich theory already gives existence and uniqueness of a weak
solution in $L^1 \cap L^{\infty}$. Here thanks to the special type of initial data we shall give a more direct construction
which also yields the regularity of the solution at one stroke.

To this end, denote for each $m\ge 2$
\begin{align*}
 \omega_0^{(m)} (x)  = \sum_{j=1}^m \omega_j (x-x_j)
\end{align*}
and let $\omega^{(m)} (t,x)$ be the corresponding solution to the Euler equation. Obviously for $0\le t\le 1$,
\begin{align*}
 \operatorname{supp} (\omega^{(m)} (t) ) \subset \bigcup_{j=1}^m B(x_j, 3A_1 C_1).
\end{align*}

Now we define $\omega(t,x)$ as follows

\begin{align*}
 \omega(t,x)=
 \begin{cases}
  \lim_{m\to \infty} \omega^{(m)}(t,x), \qquad \text{if $x \in \bigcup_{j=1}^{\infty}B(x_j,3A_1 C_1)$}, \\
  0,\qquad \text{otherwise.}
 \end{cases}
\end{align*}

We now justify that $\omega(t,x)$ is well-defined and is the desired solution.

Fix $j_0\ge 1$ and consider the ball $B(x_{j_0}, 3A_1C_1)$. By \eqref{prop10_8} (setting $\omega= \omega^{(m)}$
and $\tilde \omega=\omega^{(m-1)}$), we have
\begin{align*}
 \max_{0\le t \le 1} \| \omega^{(m)} (t) - \omega^{(m-1)} (t) \|_{H^2(B(x_{j_0}, 3A_1 C_1) )} \le 2^{-m},
 \quad \text{if } m\ge j_0+1.
\end{align*}

By Lemma \ref{lem10}, we also have for any $k\ge 3$,
\begin{align*}
 \max_{0\le t \le 1} \| \omega^{(m)} (t) \|_{H^k(B(x_{j_0}, 3 A_1 C_1) )} \le C_k=C_k(\|\omega_{j_0}\|_{H^k}, C_1),
 \quad \text{if }m\ge j_0+1.
\end{align*}

Thus $(\omega^{(m)})$ forms a Cauchy sequence in $H^k( B(x_{j_0}, 3A_1 C_1 ))$ for any $k\ge 2$ and hence converge to a unique
limit $\omega(t,x) \in C^{\infty} (B(x_{j_0}, 3A_1 C_1))$. Clearly \eqref{prop10_3} holds. Easy to check $\omega \in L^\infty$.

By using the Lebesgue Dominated Convergence Theorem, we have
\begin{align*}
 \| \omega(t)\|_{L^1( B(x_{j_0}, 3A_1 C_1)  )} & \le \lim_{m\to \infty}
\| \omega^{(m)} (t)\|_{L^1(B(x_{j_0}, 3A_1 C_1 ))} = \| \omega_{j_0}\|_{L^1}.
\end{align*}
Summing in $j_0$ then gives us $\omega \in L^1$.

We now show that $\Delta^{-1} \nabla^{\perp} \omega^{(m)}$ converges locally uniformly to $\Delta^{-1} \nabla^{\perp} \omega$
on $\bigcup_{j=1}^{\infty} B(x_j, 3A_1C_1)$.
By construction we can decompose
\begin{align*}
\omega^{(m)} (t,x) = \sum_{j=1}^m \omega_j^{(m)}(t,x),
\end{align*}
where
\begin{align*}
 \operatorname{supp} (\omega_j^{(m)}) \subset B(x_j, 3A_1 C_1).
\end{align*}

Also we have
\begin{align}
 \omega(t,x) =\sum_{j=1}^{\infty} \omega_j^{(\infty)} (t,x), \; \operatorname{supp} (\omega_j^{(\infty)}) \subset B(x_j,3A_1 C_1).
 \label{prop10_10a}
\end{align}

The summation above is actually a finite sum since for each $x$ there exists at most one $j$ such that $\omega_j^{(\infty)}(t,x) \ne 0$.

Now fix $j_0\ge 1$. Then for $x \in B(x_{j_0}, 2A_1 C_1)$ and $m\ge j_0+1$, we have
\begin{align}
 \Bigl| (\dpp \omega^{(m)} )(x) -
(\dpp \omega )(x) \Bigr|
 & \le \Bigl|  ( \dpp (\omega_{j_0}^{(m)}  - \omega_{j_0}^{(\infty)})(x) \Bigr| \label{prop10_10} \\
 & \quad + \sum_{\substack{j=1\\j\ne j_0}}^m \Bigl|
 ( \dpp ( \omega_j^{(m)} -\omega_j^{(\infty)}   ) )(x) \Bigr| \label{prop10_11} \\
 & \quad + \sum_{j=m+1}^{\infty} \Bigl|  (\dpp \omega_j^{(\infty)})(x) \Bigr|. \label{prop10_12}
\end{align}

For \eqref{prop10_10}, we use the inequality \eqref{lem10_1} to get
\begin{align*}
 \Bigl\|   ( \dpp (\omega_{j_0}^{(m)}  - \omega_{j_0}^{(\infty)})(x) \Bigr\|_{\infty}
 & \le A_1 \Bigl(  \| \omega_{j_0}^{(m)} -\omega_{j_0}^{(\infty)} \|_1 + \|\omega_{j_0}^{(m)} - \omega_{j_0}^{(\infty)} \|_{\infty}  \Bigr)
 \notag \\
 & \lesssim_{C_1} \| \omega_{j_0}^{(m)} - \omega_{j_0}^{(\infty)} \|_{\infty} \notag \\
 & \lesssim_{C_1} \| \omega^{(m)} - \omega \|_{L^{\infty} (B(x_{j_0}, 3 A_1 C_1 ) } \notag \\
 & \to 0, \quad \text{as } m \to \infty,
\end{align*}
since $\omega^{(m)}$ converges uniformly to $\omega$ on the ball $B(x_{j_0}, 3 A_1 C_1)$.

For \eqref{prop10_11}, note that for $j\ne j_0$ (see \eqref{prop10_8a})
\begin{align*}
 d\bigl( \operatorname{supp} (\omega_j^{(m)} - \omega_j^{(\infty)}), \,
 B(x_{j_0}, 3A_1 C_1 ) \bigr) \ge 2^j.
\end{align*}

Therefore by using an estimate similar to \eqref{lem10_9}, we have
\begin{align*}
 \eqref{prop10_11} & \lesssim \sum_{\substack{j=1\\j\ne j_0} }^{\infty} 2^{-j}
 \| \omega_j^{(m)} - \omega_{j}^{(\infty) } \|_{L^1 \cap L^{\infty}} \notag \\
 & \lesssim_{C_1} \sum_{j=1}^{\infty} 2^{-j} \| \omega^{(m)} -\omega \|_{L^{\infty} (B(x_j,3A_1C_1)} \notag \\
 & \to 0, \quad \text{as $m\to \infty$.}
\end{align*}

Similarly
\begin{align*}
 \eqref{prop10_12} \lesssim_{C_1} \sum_{j=m+1}^{\infty} 2^{-j} \to 0, \quad \text{as } m\to \infty.
\end{align*}

Hence we have shown that $\dpp \omega^{(m)} \to \dpp \omega$ locally uniformly on compact sets (and also
uniformly in $t$) as $m$ tends to infinity. By writing
\begin{align*}
 \omega^{(m)} (t) = \omega^{(m)} (0) + \int_0^t (\dpp \omega^{(m)}  \cdot \nabla \omega^{(m)})(\tau) d\tau,
\end{align*}
and sending $m$ to infinity, we conclude that $\omega$ is the desired solution on the time interval $[0,1]$.

Finally \eqref{prop10_4} is a simple consequence of Lemma \ref{lem11} and our choice of the centers
$x_j$ (see \eqref{prop10_6}).
\end{proof}

We are now ready to complete the
\begin{proof} [Proof of Theorem \ref{thm1}]
For each $j\ge 2$, we choose (by a slight abuse of notation)
$h_j=h_{A_j}$ according to \eqref{30_2} with the parameter $A_j$ to be taken sufficiently large. Consider the  Euler equation
 \begin{align*}
  \begin{cases}
   \partial_t \omega + \Delta^{-1} \nabla^{\perp} \omega \cdot \nabla \omega =0, \quad 0<t\le 1, \\
   \omega \Bigr|_{t=0} =h_j.
  \end{cases}
 \end{align*}

By Proposition \ref{prop40}, we obtain for some $t_j \in (0, \frac 1 {\log\log A_j})$,
\begin{align*}
  \| (D\phi) (t_j,\cdot) \|_{\infty} > \log\log A_j,
\end{align*}
where $\phi$ is defined in \eqref{40_2}.

We then use Proposition \ref{prop20} to find $\tilde \omega_j^{(0)} \in C_c^{\infty} (B(0,1))$,
$\tilde \omega_j^{(0)}$ odd in both $x_1$ and $x_2$, such that
\begin{align}
& \|\tilde \omega_j^{(0)} \|_{L^1} \le 2 \| h_j \|_{L^1}, \notag \\
 & \| \tilde \omega_j^{(0)} \|_{L^{\infty}} \le 2 \| h_j \|_{L^{\infty}}, \notag \\
& \| \tilde \omega_j^{(0)} \|_{\dot H^1} \le \| h_j \|_{\dot H^1} + 2^{-j}, \notag \\
&\|  \tilde \omega_j^{(0)} \|_{\dot H^{-1}} \le 2 \|h_j \|_{\dot H^{-1}}, \notag \\
 &  \| \tilde \omega_j(t_j,\cdot) \|_{\dot H^1} >j, \label{thm1_100}
\end{align}
where $\tilde \omega_j(t)$ is the solution to the Euler equation
\begin{align*}
 \begin{cases}
  \partial_t \tilde \omega_j + \Delta^{-1} \nabla^{\perp} \tilde \omega_j \cdot \nabla \tilde \omega_j =0, \quad 0<t\le 1, \\
  \tilde \omega_j \Bigr|_{t=0} = \tilde \omega_j^{(0)}.
 \end{cases}
\end{align*}

We then apply Proposition \ref{prop10} to $\omega_1=\omega_0^{(p)}$, $\omega_j = \tilde \omega_j^{(0)}$ for $j\ge 2$ and find the centers $x_j$.
Obviously by \eqref{thm1_100} and \eqref{prop10_4}, we have
\begin{align*}
 \operatorname{ess-sup}_{0<t\le t_0} \| \omega(t,\cdot) \|_{\dot H^1} =+\infty, \quad \forall\, 0<t_0\le 1.
\end{align*}

It is not difficult to check the $\dot H^{-1}$ regularity of the constructed solution since on each patch the $L^2$ norm of the
velocity field is preserved. The theorem is proved.
\end{proof}

\section{The 2D compactly supported case}

\begin{lem}[Control of the support] \label{lem50}
 Suppose $\omega = \omega(t,x)$ is a smooth solution to the following equation:
 \begin{align*}
  \begin{cases}
   \partial_t \omega + \Delta^{-1} \nabla^{\perp} \omega \cdot \nabla \omega + (b_1+b_2) \cdot \nabla \omega =0, \\
   \omega \Bigr|_{t=0} =f,
  \end{cases}
 \end{align*}
where $b_1=b_1(t,x)$, $b_2=b_2(t,x)$, $f=f(x)$ are smooth functions satisfying the following conditions:
\begin{itemize}
 \item $\|f\|_{\infty} \le C_f$ for some constant $C_f>0$, and
 \begin{align}
\operatorname{supp}(f) \subset B(0,R), \quad R>0. \notag 
\end{align}
 \item $b_1$, $b_2$ are incompressible, i.e. $\nabla \cdot b_1 = \nabla \cdot b_2=0$.
 \item For some $B_1>0$,
 \begin{align*}
  |b_1(t,x)| \le B_1|x|, \quad \forall\, x \in \mathbb R^2.
 \end{align*}
\item For some $B_2>0$,
\begin{align*}
|b_2(t,x)| \le B_2|x|^2, \quad \forall\, x \in \mathbb R^2.
\end{align*}

\end{itemize}
Then there exists $R_0=R_0(C_f, B_1,B_2)>0$, $t_0=t_0(C_f,B_1,B_2)>0$, such that if $0<R\le R_0$, then
\begin{align*}
\operatorname{supp}(\omega(t,\cdot) ) \subset B(0,2R), \quad \forall\, 0\le t \le t_0.
\end{align*}
\end{lem}

\begin{proof}[Proof of Lemma \ref{lem50}]
 Define the forward characteristic lines $\phi=\phi(t,x)$ which solves the ODE
 \begin{align*}
  \begin{cases}
   \partial_t \phi(t,x)= (\Delta^{-1} \nabla^{\perp} \omega + b_1 +b_2) (t,\phi(t,x) ), \\
   \phi(t=0,x)=x, \quad x \in \mathbb R^2.
  \end{cases}
 \end{align*}

By using the assumptions, we compute
\begin{align}
 &\frac d {dt} \Bigl( |\phi(t,x)|^2 \Bigr)  \notag \\
 \le & \; \| (\Delta^{-1} \nabla^{\perp} \omega)(t,\cdot) \|_{\infty} |\phi(t,x)|
 +B_1 |\phi(t,x)|^2 + B_2  |\phi(t,x)|^3. \label{lem50_2}
\end{align}

Since both $b_1$ and $b_2$ are  incompressible, we have
\begin{align}
 \| \omega(t,\cdot) \|_{L^1} = \| \omega(t=0,\cdot) \|_{L^1} = \| f \|_{L^1} \le C_f \cdot \pi R^2.  \label{lem50_3}
\end{align}

Then by interpolation and $L^{\infty}$ conservation, we get
\begin{align}
 \| (\Delta^{-1} \nabla^{\perp} \omega)(t,\cdot) \|_{\infty} & \lesssim
 \| \omega(t,\cdot)\|_{L^1}^{\frac 12} \| \omega(t,\cdot) \|_{L^{\infty}}^{\frac 12}
 \notag \\
 & \lesssim \| f \|_{L^1}^{\frac 12} \| f\|_{L^{\infty}}^{\frac 12} \notag \\
 & \lesssim C_f R, \label{lem50_4}
\end{align}
where in the last inequality we have used \eqref{lem50_3} and all the implied constants are absolute constants.

Plugging \eqref{lem50_4} into \eqref{lem50_2}, we obtain
\begin{align*}
 \frac d {dt} \Bigl( |\phi(t,x)| \Bigr)
 \lesssim C_f R + B_1 |\phi(t,x)| + B_2 |\phi(t,x)|^2.
\end{align*}
The desired result then follows from time integration and choosing $R_0$, $t_0$ sufficiently small.
\end{proof}

For the compactly supported case, we need to use a slight variant of the function $h_A$ defined in \eqref{30_2}.
We now take any $A\gg 1$ and
\begin{align} \label{50_1}
 g_A(x) = \frac{1 }{\log\log\log\log A} \cdot \frac 1 {\sqrt{\log A}}\sum_{A\le k \le A+\log A} \eta_k(x),
\end{align}
where $\eta_k$ was defined in \eqref{30_0a}.

It is easy to check that
\begin{align}
 &\operatorname{supp} (g_A) \subset B(0,R_A), \quad \text{with }R_A \sim 2^{-A} , \notag \\
 & \| g_A \|_{H^1} \lesssim \frac{1} {\log\log\log\log A}, \notag \\
 & \| g_A \|_{L^{\infty}} \lesssim \frac  1 {\sqrt{\log A}}, \notag \\
 & \| D^2 g_A \|_{L^{\infty}} \lesssim 2^{2(A+\log A)}. \notag
\end{align}

The main difference between $g_A$ and $h_A$ is that the former has weaker dependence on $A$ in terms of the bounds on higher
derivatives. This fact will be used in the perturbation theory later (see Lemma \ref{lem55}).

The following is a variant of Proposition \ref{prop40}. Note that the additional drift term has a special
form which makes the class of odd flows invariant.

\begin{lem} \label{lem53}
 Let $\omega =\omega(t,x)$ be the smooth solution to the equation
 \begin{align*}
  \begin{cases}
   \partial_t \omega + \Delta^{-1} \nabla^{\perp} \omega \cdot \nabla \omega + b \cdot \nabla \omega =0, \\
   \omega \Bigr|_{t=0} =g_A,
  \end{cases}
 \end{align*}
where $g_A$ is defined in \eqref{50_1}, $b=b(t,x)$ takes the form
\begin{align}
 b(t,x) =b_0(t) \begin{pmatrix} -x_1 \\ x_2 \end{pmatrix}, \quad x \in \mathbb R^2; \label{lem53_3}
\end{align}
and $b_0(t)$ is a smooth function satisfying
\begin{align}
 \| b_0\|_{\infty} \le B_0 <\infty. \label{lem53_4}
\end{align}

Let $\phi=\phi(t,x)$ be the associated forward characteristic line which solves
\begin{align*}
 \begin{cases}
  \partial_t \phi(t,x) = (\Delta^{-1} \nabla^{\perp} \omega + b) (t,\phi(t,x)), \\
  \phi(t=0,x)=x, \quad x \in \mathbb R^2.
 \end{cases}
\end{align*}

Then there exists $A_0=A_0(B_0)>0$ such that if $A>A_0$, then
\begin{align}
 \max_{0\le t \le \frac 1 {\log \log A}} \| (D\phi)(t,\cdot) \|_{\infty} > \log \log A. \label{lem53_4a}
\end{align}
\end{lem}

\begin{proof} [Proof of Lemma \ref{lem53}]
Thanks to the special assumption \eqref{lem53_3}, it is easy to check
 that $\omega(t,x)$ is still an odd function in $x_1$ and $x_2$ for any $t$.
 We can then repeat the proof of Proposition \ref{prop40} or use the simplified version
 as in the proof of Proposition \ref{prop_gener_1}. We omit the details.
\end{proof}

The next lemma shows that the patch dynamics can still be controlled under a suitable perturbation in the drift term. This will
play an important role in our later constructions. Since we no longer have odd symmetry at our disposal, we need to carry out
a perturbative analysis.

\begin{lem} \label{lem55}
 Let $W =W(t,x)$ be a smooth solution to the  equation
 \begin{align*}
  \begin{cases}
   \partial_t W + \Delta^{-1} \nabla^{\perp} W \cdot \nabla W + (b(t,x) +r(t,x)) \cdot \nabla
   W =0, \\
   W \Bigr|_{t=0} =W_0=g_A,
  \end{cases}
 \end{align*}
where the functions $g_A$, $b$, $r$ satisfies the following conditions:
\begin{itemize}
 \item  $g_A$ is the same as defined in \eqref{50_1};
 \item $b(t,x)=b_0(t) \begin{pmatrix} -x_1 \\ x_2 \end{pmatrix}, \quad \| b_0\|_{\infty} \le B_0<\infty;$
 \item $r$ is incompressible and
 \begin{align}
  &| r(t,x)| \le B_1 \cdot |x|^2, \notag \\
  & |(Dr)(t,x)| \le B_1 \cdot |x|, \notag \\
  & |(D^2 r)(t,x)| \le B_1, \quad \forall\, x \in \mathbb R^2, \, 0\le t\le 1.  \label{50_1tmp_aa}
 \end{align}
Here $B_1>0$ is a constant.
\end{itemize}

Let $\Phi =\Phi (t,x)$ be the characteristic line which solves the ODE
\begin{align} \label{50_1tmp_aa_1}
 \begin{cases}
  \partial_t \Phi(t,x)  = (\Delta^{-1} \nabla^{\perp } W + b+r) (t,\Phi(t,x) ), \\
  \Phi (t=0,x)=x,\quad x \in \mathbb R^2.
 \end{cases}
\end{align}

Then there exists $A_0=A_0(B_0,B_1)>0$ such that if $A>A_0$, then
\begin{align} \label{lem55_2}
 \max_{0\le t \le \frac 1 {\log \log A}}  \| (D\Phi )(t,\cdot) \|_{\infty} > \log\log\log A.
\end{align}
\end{lem}

\begin{proof} [Proof of Lemma \ref{lem55}]
 We shall argue by contradiction. Assume \eqref{lem55_2} is not true, then
 \begin{align} \label{lem55_2a}
 \max_{0\le t \le \frac 1 {\log \log A}}  \| (D\Phi )(t,\cdot) \|_{\infty}  \le  \log\log\log A.
\end{align}
By the definition of the characteristic line $\Phi$, we have
$W(t,x)=W_0(\tilde \Phi(t,x))$ where $\tilde \Phi$ is the inverse map of $\Phi$. By
\eqref{lem55_2a} and using a computation similar to \eqref{prop20_11}, we get
\begin{align}
\max_{0\le t\le \frac 1 {\log\log A}} \| DW(t,\cdot) \|_2
&\lesssim \| DW_0\|_2 \cdot \max_{0\le t\le \frac 1{\log\log A}} \|D\Phi(t,\cdot)\|_{\infty} \notag \\
& \lesssim \frac 1 {\log\log\log\log A} \cdot \log\log\log A \notag \\
& \lesssim \log\log\log A. \label{lem55_2b}
\end{align}
We shall need this estimate later.

The main idea is to compare $W$
 with the other solution $\omega$ which solves the ``unperturbed'' equation
 \begin{align*}
  \begin{cases}
   \partial_t \omega + \Delta^{-1} \nabla^{\perp} \omega \cdot \nabla \omega + b(t,x) \cdot \nabla \omega =0, \\
   \omega \Bigr|_{t=0} =g_A.
  \end{cases}
 \end{align*}

The perturbation theory requires a bit of work so we shall proceed in several steps.

\texttt{Step 1}. Set $\eta= W -\omega$. We first show that
\begin{align}
 \| \eta (t,\cdot) \|_{B^{0}_{\infty,1}} \lesssim 2^{-\frac 2 3 A+}, \qquad
 \forall\, 0\le t \le \frac 1 {\log\log A}. \label{lem55_3}
\end{align}
Here and below we use the notation $X+$ as in \eqref{notation_plus}. Also to simplify notations we shall write $\lesssim_{B_1}$ as $\lesssim$
(i.e. we suppress the notational dependence on $B_1$) since $A$ will be taken sufficiently large.

The equation for $\eta$ takes the form
\begin{align*}
 \begin{cases}
  \partial_t \eta + \Delta^{-1} \nabla^{\perp} \eta \cdot \nabla W + \Delta^{-1} \nabla^{\perp} \omega
  \cdot \nabla \eta + b \cdot \nabla \eta + r \cdot \nabla W =0, \\
  \eta(0)=0.
 \end{cases}
\end{align*}

By Lemma \ref{lem50} and \eqref{50_1tmp_aa}, we have
\begin{align}
  &\| r(t,\cdot)\|_{L^{\infty} (\operatorname{supp} (W(t,\cdot)))} \lesssim  4^{-A}, \notag \\
  & \| (Dr) (t,\cdot) \|_{L^{\infty} (\operatorname{supp} (W(t,\cdot)))} \lesssim  2^{-A}, \notag \\
  & \| (D^2 r)(t,\cdot) \|_{L^{\infty}( \operatorname{supp} (W(t,\cdot)) )}  \lesssim 1,
  \quad \forall\, 0\le t\le \frac 1 {\log\log A}. \label{lem55_3add}
 \end{align}

Let $1<p<2$. By Sobolev embedding and \eqref{lem55_3add}, we compute
\begin{align*}
 \frac d {dt} \Bigl( \| \eta \|_p^p \Bigr)
 & \lesssim \| \Delta^{-1} \nabla^{\perp} \eta \|_{(\frac 1p -\frac 12)^{-1}} \| \nabla W \|_2
 \cdot \| \eta \|_p^{p-1} + 4^{-A} \| \nabla W \|_2 \cdot \| \eta \|_p^{p-1} \notag \\
 & \lesssim  \| \nabla W\|_2 \cdot \| \eta\|_p^p + 4^{-A} \| \nabla W \|_2 \cdot \| \eta \|_p^{p-1}.
\end{align*}

Therefore for $0\le t \le \frac 1 {\log \log A}$, by using \eqref{lem55_2b}, we get
\begin{align}
 \| \eta(t,\cdot ) \|_p & \lesssim 4^{-A} \int_0^t e^{(t-s) \log\log\log A} ds \cdot \log\log\log A \notag \\
 & \lesssim 4^{-A} \cdot \frac {\log\log \log A} {\log\log A} \lesssim 4^{-A}. \label{lem55_4}
\end{align}
This estimate is particularly good for $p=2-$.

On the other hand, for any $2\le q<\infty$, a standard energy estimate gives for any $0\le t \le 1$,
\begin{align*}
 \| \eta(t,\cdot) \|_{W^{2,q}} & \lesssim \| W(t,\cdot) \|_{W^{2,q}} + \| \omega(t,\cdot) \|_{W^{2,q}} \notag \\
 & \lesssim \| g_A \|_{W^{2,q}} \notag \\
 & \lesssim \frac 1 {\sqrt{\log A}} 2^{2(A+\log A) (1-\frac 1q)} \notag \\
 & \lesssim 4^A.
\end{align*}

Interpolating the above with \eqref{lem55_4} then yields \eqref{lem55_3} (note that
$\|\eta\|_{B^0_{\infty,1}(\mathbb R^2)} \lesssim \| \eta\|_{L^2(\mathbb R^2)}^{\frac 23} \| \Delta \eta \|_{L^\infty(\mathbb R^2)}^{\frac 13}.$)

\texttt{Step 2.} Let $\phi$ be the characteristic line associated with the equation for $\omega$, i.e.
\begin{align*}
 \begin{cases}
  \partial_t \phi(t,x) = (\Delta^{-1} \nabla^{\perp} \omega + b)(t,\phi(t,x) ), \\
  \phi(t=0,x)=x, \quad x \in \mathbb R^2.
 \end{cases}
\end{align*}
We show that
\begin{align}
 \max_{0\le t \le \frac 1 {\log\log A}} \| \phi(t,\cdot) -\Phi (t,\cdot) \|_{\infty} \lesssim 2^{-\frac 43 A+}.
 \label{lem55_6}
\end{align}

Set $Y(t,x) = \Phi (t,x)-\phi(t,x)$.
By Lemma \ref{lem50}, we only need to consider $|x| \lesssim 2^{-A}$ sine $\phi(t,x)=\Phi (t,x)=x$ for $|x| \gg 2^{-A}$.

Then for $|x| \lesssim 2^{-A}$,
\begin{align*}
 \frac d {dt} Y & = (\Delta^{-1} \nabla^{\perp} W)(\Phi ) - (\Delta^{-1} \nabla^{\perp}\omega)(\phi)  \notag\\
 & \qquad + b_0(t) \begin{pmatrix} -Y_1 \\ Y_2 \end{pmatrix} + r(t,\Phi ) \notag \\
 & = (\Delta^{-1}\nabla^{\perp} W)(\Phi ) - (\Delta^{-1} \nabla^{\perp} W) (\phi)
 + (\Delta^{-1} \nabla^{\perp}(W -\omega) )(\phi) \notag \\
 & \qquad + b_0(t) \begin{pmatrix} -Y_1 \\ Y_2 \end{pmatrix} + r(t,\Phi ).
\end{align*}

For $|x| \lesssim 2^{-A}$ and $0\le t \le \frac 1 {\log\log A}$, we have $|\Phi (t,x)| \lesssim 2^{-A}$.
By \eqref{50_1tmp_aa}, we get
\begin{align*}
\max_{|x| \lesssim 2^{-A}} |r(t,\Phi (t,x))| \lesssim 4^{-A}, \qquad \forall\, 0\le t\le \frac 1 {\log\log A}.
\end{align*}

Therefore
\begin{align}
 \frac d {dt} \Bigl( |Y(t,x)| \Bigr)
 & \lesssim (B_0+\| \mathcal R W \|_{\infty} )\cdot |Y(t,x)| + \| \Delta^{-1} \nabla^{\perp} (W -\omega)\|_{\infty}
 +4^{-A}. \label{lem55_7}
\end{align}

By using the usual log-interpolation inequality, we have
\begin{align}
 \| \mathcal R W \|_{\infty} & \lesssim \| W\|_2+ \| W \|_{\infty} \log(10+\|W \|_{H^2}) \notag \\
 & \lesssim A. \label{lem55_8}
\end{align}

On the other hand, by \eqref{lem55_3}, we have
\begin{align*}
 \| \Delta^{-1} \nabla^{\perp} (W -\omega) \|_{\infty} & \lesssim \| W -\omega \|_1^{\frac 12}
 \| W -\omega \|_{\infty}^{\frac 12} \notag \\
 & \lesssim 2^{-A} \cdot 2^{-\frac 1 3 A+} \notag \\
 & \lesssim 2^{-\frac 4 3 A+}.
\end{align*}

Plugging these estimates into \eqref{lem55_7} and integrating in time, we obtain
\begin{align*}
 |Y(t,x)| & \lesssim \int_0^t e^{(t-s) C\cdot A} (4^{-A} + 2^{-\frac 4 3 A+}) ds \notag \\
 & \lesssim \int_0^{\frac 1 {\log \log A}} e^{\frac {CA}{\log\log A}} (4^{-A} +2^{-\frac 43 A+}) ds \notag \\
 & \lesssim 2^{-\frac 43A+}.
\end{align*}

\texttt{Step 3.} Set $\tilde J(t) = (D\Phi ) (t,x)$, $J(t) = (D\phi)(t,x)$, then obviously
\begin{align*}
 &\partial_t \tilde J = (\mathcal R W )(\Phi ) \tilde J + b_0(t)
 \begin{pmatrix} -1 \quad 0 \\0 \quad 1 \end{pmatrix} \tilde J + (Dr)(\Phi ) \tilde J, \\
 & \partial_t J = (\mathcal R \omega)(\phi) J + b_0(t)\begin{pmatrix} -1 \quad 0 \\0 \quad 1 \end{pmatrix}  J.
 \end{align*}

Let $q=\tilde J-J$. Then  $q$ satisfies the equation
\begin{align*}
 \partial_t q & = \Bigl( (\mathcal R W)(\Phi ) - (\mathcal R W)(\phi) \Bigr) \tilde J
 + \Bigl( (\mathcal R W)(\phi) - (\mathcal R \omega) (\phi) \Bigr) \tilde J \notag \\
 & \qquad + (\mathcal R \omega) (\phi) q + b_0(t) \begin{pmatrix} -1 \quad 0 \\0 \quad 1 \end{pmatrix} q
 + (Dr)(\Phi ) \tilde J.
\end{align*}

By \eqref{lem55_3}, \eqref{lem55_6}, \eqref{lem55_8}, we obtain
\begin{align*}
 \partial_t |q| & \lesssim \| D \mathcal R W \|_{\infty} |\Phi -\phi| \cdot \| \tilde J \|_{\infty}
 + 2^{-\frac 2 3 A+} \| \tilde J \|_{\infty} \notag \\
 & \qquad + A|q| + 2^{-A} \| \tilde J \|_{\infty} \notag \\
 & \lesssim 2^{A+} 2^{-\frac 43A+} \log\log\log A + 2^{-\frac 23 A+} \log\log \log A + A|q| \notag \\
 & \qquad + 2^{-A} \log\log\log A.
\end{align*}

Integrating in time and noting $t\le \frac 1 {\log\log A}$, we obtain
\begin{align*}
 \| q\|_{\infty} \lesssim 1.
\end{align*}

But this obviously contradicts \eqref{lem53_4a}.
\end{proof}

The next lemma is the main building block for our construction in the compactly supported data case.

\begin{lem} \label{lem57}
 Suppose $f_{-1} \in C_c^{\infty} (\mathbb R^2)$ is a given real-valued function such that
 \begin{itemize}
  \item for some $R_0>0$,
  \begin{align*}
   \operatorname{supp}(f_{-1}) \subset\{x=(x_1,x_2):\;\; x_1 \le -2R_0 \};
  \end{align*}
\item $f_{-1}$ is an odd function of $x_2$, i.e.
\begin{align*}
 f_{-1}(x_1,x_2)= - f_{-1}(x_1,-x_2),\quad\forall\,x \in \mathbb R^2.
\end{align*}
 \end{itemize}

Then for any $0<\epsilon < \frac{R_0}{100}$, one can find $\delta_0=\delta_0(f_{-1}, \epsilon, R_0)>0$,
$0<t_0=t_0(f_{-1},\epsilon, R_0)<\epsilon$, and $f_0 \in C_c^{\infty} (B(0,\epsilon))$ ($f_0$ depends only on
$(f_{-1}, \epsilon, R_0)$) with the properties:
\begin{itemize}
 \item $f_0$ is an odd function of $x_2$;
 \item \begin{align} \label{lem57_1a}
 \|f_0\|_{L^1} + \| f_0\|_{L^{\infty}} + \| f_0\|_{H^1} + \| f_0\|_{\dot H^{-1}} \le \epsilon,
 \end{align}
\end{itemize}
such that for any $f_1 \in C_c^{\infty} (\mathbb R^2)$ with
\begin{itemize}
 \item $\operatorname{supp} (f_1) \subset \{ x=(x_1,x_2):\quad x_1\ge R_0\}$;
 \item $\|f_1\|_{L^1} +\|f_1\|_{L^{\infty}} \le \delta_0$,
\end{itemize}
the following hold true:

Consider the Euler equation
\begin{align*}
 \begin{cases}
  \partial_t \omega + \Delta^{-1} \nabla^{\perp} \omega \cdot \nabla \omega =0, \\
  \omega \Bigr|_{t=0} = f_{-1} +f_0 +f_1,
 \end{cases}
\end{align*}
then the smooth solution $\omega=\omega(t,x)$ satisfies the following properties:
\begin{enumerate}
 \item for any $0\le t \le t_0$, we have the decomposition
 \begin{align}
  \omega(t,x) = \omega_{-1} (t,x) + \omega_0(t,x)+ \omega_1(t,x), \label{lem57_1}
 \end{align}
where
\begin{align*}
 &\operatorname{supp} (\omega_{-1} (t)) \subset B( \operatorname{supp} (f_{-1}), \, \frac 18 R_0);\\
 &\operatorname{supp} (\omega_0(t) ) \subset B(0, \epsilon+\frac 1 8 R_0);\\
 & \operatorname{supp} (\omega_1(t)) \subset B(\operatorname{supp}(f_1, \frac 18 R_0)).
\end{align*}

\item $\|\omega_0(t=0,\cdot)\|_{H^1} = \|f_0\|_{H^1} \le \epsilon$, but
\begin{align}
  \| \omega_0(t_0,\cdot)\|_{\dot H^1} >\frac 1 {\epsilon}. \label{lem57_2}
\end{align}

\end{enumerate}

\end{lem}

\begin{proof}[Proof of Lemma \ref{lem57}]
 The decomposition \eqref{lem57_1} is a simple consequence of finite transportation speed. Therefore we only
 need to show how to choose $f_0$ to achieve \eqref{lem57_2} and the other conditions.

 Consider first the equation
 \begin{align} \label{lem57_3}
  \begin{cases}
   \partial_t \omega^{(1)} + \Delta^{-1} \nabla^{\perp} \omega^{(1)} \cdot \nabla \omega^{(1)} =0, \\
   \omega^{(1)} \Bigr|_{t=0} =f_{-1} +g_A,
  \end{cases}
 \end{align}
where $g_A$ was defined in \eqref{50_1} and we shall choose $A$ sufficiently large.

For $0\le t \le \frac 1 {\log \log A}$,we decompose the solution $\omega^{(1)} (t)$ to \eqref{lem57_3} as
\begin{align*}
 \omega^{(1)} (t,x)= \omega^{(1)}_{-1} (t,x) + \omega_0^{(1)} (t,x),
\end{align*}
with
\begin{align}
 &\operatorname{supp} (\omega^{(1)}_{-1} (t,\cdot) ) \subset B(\operatorname{supp} (f_{-1}), \, \frac 1 {10} R_0), \notag \\
 & \operatorname{supp} (\omega^{(1)}_0) (t,\cdot) ) \subset B(\operatorname{supp} (g_A), \frac 1 {10} R_0). \label{lem57_3a}
\end{align}

Obviously $\omega_0^{(1)}(t)$ satisfies the equation
\begin{align} \label{lem57_4}
 \begin{cases}
  \partial_t \omega_0^{(1)} + \Delta^{-1} \nabla^{\perp} \omega_0^{(1)} \cdot \nabla \omega_0^{(1)} +
  \Delta^{-1} \nabla^{\perp} \omega_{-1}^{(1)} (t) \cdot \nabla \omega_0^{(1)} =0, \\
  \omega_0^{(1)} \Bigr|_{t=0} =g_A.
 \end{cases}
\end{align}

Since by assumption $f_{-1}$ is odd in $x_2$ and $g_A$ is also odd in $x_2$, it is easy to check that both
$\omega_{-1}^{(1)}(t)$ and $\omega_0^{(1)}(t)$ are odd functions of $x_2$. Therefore we have
\begin{align}
 &\Bigl(\Delta^{-1} \partial_{22} \omega_{-1}^{(1)} \Bigr)(t,x_1,0) = 0, \notag \\
 &\Bigl( \Delta^{-1} \partial_{11} \omega_{-1}^{(1)} \Bigr)(t,x_1,0) = 0, \notag \\
 &\Bigl( \Delta^{-1} \partial_1 \omega_{-1}^{(1)}  \Bigr)(t,x_1,0) = 0, \quad\forall\, 0\le t\le \frac 1 {\log\log A},\, x_1\in \mathbb R.
 \label{lem57_5}
\end{align}

Now let $\xi(t)$ solve the ODE
\begin{align} \label{lem57_6}
 \begin{cases}
  \frac d {dt} \xi(t) =  (\Delta^{-1} \partial_2  \omega_{-1}^{(1)}) (t,\xi(t),0), \\
  \xi(0)=0.
 \end{cases}
\end{align}

Since for $0\le t \le \frac 1 {\log\log A}$ and $A$ is sufficiently large, the function $\omega_{-1}^{(1)}$ is supported
away from the origin (see \eqref{lem57_3a}), it is easy to check that the function  $(\Delta^{-1} \partial_2 \omega_{-1}^{(1)})(t,\cdot)$
is smooth and has uniform (independent of $A$) Sobolev bounds in a small neighborhood of the origin.
Thus $\xi(t)$ is well-defined and remains close to the origin for $t\le \frac 1 {\log\log A}$.

By \eqref{lem57_5}, we have
\begin{align*}
 \Bigl(-\Delta^{-1} \partial_2 \omega_{-1}^{(1)} \Bigr) (t,\xi(t)+y_1,y_2) & =
 -\Bigl(\Delta^{-1} \partial_2 \omega_{-1}^{(1)} \Bigr)(t,\xi(t),0) -
 \Bigl(\Delta^{-1} \partial_{12} \omega_{-1}^{(1)} \Bigr)(t,\xi(t),0)y_1 \notag \\
&\qquad   + r_1^{(1)}(t,y), \notag \\
\Bigl(\Delta^{-1} \partial_1 \omega_{-1}^{(1)} \Bigr)(t,\xi(t)+y_1,y_2) &=
\Bigl(\Delta^{-1} \partial_{12} \omega_{-1}^{(1)} \Bigr)(t,\xi(t),0)y_2
  + r_2^{(1)}(t,y), \notag \\
\end{align*}
where for $0\le t\le \frac 1{\log\log A}$,
\begin{align*}
 &|r_1^{(1)}(t,y)| + |r_2^{(1)}(t,y)| \lesssim_{f_{-1},R_0} |y|^2, \notag \\
 &|(D r_1^{(1)})(t,y)| + |(D r_2^{(1)})(t,y)| \lesssim_{f_{-1},R_0} |y|, \notag \\
 & |(D^2 r_1^{(1)})(t,y)| + |(D^2 r_2^{(1)})(t,y)| \lesssim_{f_{-1},R_0} 1. \notag
\end{align*}

By \eqref{lem57_6}, we may write the above more compactly as
\begin{align}
& (\Delta^{-1} \nabla^{\perp} \omega_{-1}^{(1)}) (t,\xi(t)+y_1,y_2) \notag \\
=&\; \begin{pmatrix}  -\frac d {dt}{\xi}(t) \\ 0 \end{pmatrix} + b_0(t)
\begin{pmatrix} -y_1\\ y_2 \end{pmatrix} +r(t,y), \label{lem57_8}
\end{align}
where
\begin{align}
& |b_0(t)| \lesssim_{f_{-1},R_0} 1, \notag \\
& |r(t,y) | \lesssim_{f_{-1},R_0} |y|^2, \notag \\
& |(Dr)(t,y)| \lesssim_{f_{-1},R_0} |y|, \notag \\
& |(D^2r)(t,y)| \lesssim_{f_{-1},R_0} 1. \label{lem57_9}
\end{align}

Now we make a change of variable and set
\begin{align}
& x = (\xi(t)+y_1,y_2), \notag \\
& \omega_0^{(1)}(t,x) = \omega_0^{(1)} (t,\xi(t)+y_1,y_2) =: W_0^{(1)} (t,y_1,y_2). \label{lem57_11}
\end{align}

By using \eqref{lem57_8} and the above expressions, we can write \eqref{lem57_4} as
\begin{align*}
 &\partial_t W_0^{(1)} (t,y) + (\Delta^{-1} \nabla^{\perp} W_0^{(1)} \cdot \nabla W_0^{(1)})(t,y) \notag \\
 &\quad + b_0(t) \begin{pmatrix} -y_1\\ y_2 \end{pmatrix} \cdot \nabla W_0^{(1)} (t,y)
 +r(t,y) \cdot \nabla W_0^{(1)}(t,y)=0,
\end{align*}
where $b_0(t)$, $r(t,y)$ satisfies \eqref{lem57_9}.

By Lemma \ref{lem50}, we have for $0\le t \le \frac 1 {\log\log A}$,
\begin{align*}
 \operatorname{supp} (W_0^{(1)} (t,\cdot) ) \subset B(0,\tilde R),\quad \text{with } \tilde R\sim 2^{-A}.
\end{align*}

Therefore by \eqref{lem57_9} and Lemma \ref{lem55}, we have for $A$ sufficiently large,
\begin{align} \label{57_4_new_add1}
 \max_{0\le t \le \frac 1 {\log\log A}} \| (D \Phi ) (t,\cdot ) \|_{\infty} > \log\log\log A,
\end{align}
were $\Phi $ is the forward characteristic line associated with $W_0^{(1)}$ (see \eqref{50_1tmp_aa_1}).

Let $\phi$ be the characteristic line solving the ODE
\begin{align*}
 \begin{cases}
  \partial_t \phi(t,x) = (\Delta^{-1} \nabla^{\perp} \omega_0^{(1)} + \Delta^{-1} \nabla^{\perp} \omega_{-1}^{(1)})(t,\phi(t,x)), \\
  \phi(t=0,x)=x.
 \end{cases}
\end{align*}

Denote by $\tilde \Phi$, $\tilde \phi$ the inverse maps of $\Phi$ and $\phi$ respectively. By \eqref{lem57_11}, it is easy to check
that
\begin{align}
 \tilde \Phi(t,y) = \tilde \phi(t,\xi(t)+y), \qquad  \text{ for any $t\ge 0$ and $y\in \mathbb R^2$.}
\end{align}
 Therefore by \eqref{57_4_new_add1},
\begin{align} \label{lem57_c2}
 \max_{0\le t \le \frac 1 {\log\log A}} \| (D \phi) (t,\cdot ) \|_{\infty} > \log\log\log A.
\end{align}
Now  we just need to modify slightly the proof of Proposition \ref{prop20}.
Note that one can always choose the perturbation $\beta(x)$ (see \eqref{prop20_15}) to be odd in $x_1$ and $x_2$, for example,
\begin{align*}
 \beta(x)= \frac 1k \sin(kx_1) \sin(x_2) b(x) \frac 1 {\sqrt {M}}.
\end{align*}
Denote by $\tilde g_A$ the perturbed initial data and let $\tilde \omega^{(1)}$ be the solution
to
\begin{align*}
 \begin{cases}
  \partial_t \tilde \omega^{(1)} + \Delta^{-1} \nabla^{\perp} \tilde \omega^{(1)} \cdot \nabla \tilde \omega^{(1)}=0, \\
  \tilde \omega^{(1)} \Bigr|_{t=0} = f_{-1} +\tilde g_A.
 \end{cases}
\end{align*}
Similar to $\omega^{(1)}$ (see \eqref{lem57_3}), we also have the decomposition similar to that in \eqref{lem57_3a}:
\begin{align*}
\tilde \omega^{(1)} = \tilde \omega^{(1)}_{-1} + \tilde \omega_0^{(1)}.
\end{align*}
By our choice of perturbation (and taking $A$ sufficiently large), we have
\begin{align}  \notag
 \max_{0\le t\le \frac 1 {\log\log A}} \| \tilde \omega_0^{(1)} (t,\cdot)\|_{\dot H^1} > {(\log\log\log A)}^{\frac 13}.
\end{align}

Let $f_0=\tilde g_A$. We then compare $\tilde \omega^{(1)}$ with $\omega$ which
solves
\begin{align*}
 \begin{cases}
  \partial_t \omega + \Delta^{-1} \nabla^{\perp} \omega \cdot \nabla \omega =0, \\
  \omega \Bigr|_{t=0} = f_{-1} + f_0 +f_1,
 \end{cases}
\end{align*}
with $f_1 \in C_c^{\infty}(\mathbb R^2)$ satisfying
\begin{itemize}
 \item $\operatorname{supp} (f_1) \subset \{x=(x_1,x_2):\quad x_1\ge \frac 12 R_0\}$;
 \item $\|f_1\|_{L^1} + \| f_1\|_{L^{\infty}} \le \delta_0$
\end{itemize}
and $\delta_0$ is to be taken sufficiently small.

By an argument similar to the proof of \eqref{lem11_3}, we then have (see \eqref{lem57_1} for the definition of
$\omega_0(t)$)
\begin{align*}
 \max_{0\le t \le \frac 1 {\log\log A}} \| \omega_0(t) - \tilde \omega_0^{(1)} (t) \|_{H^2} \lesssim_{\epsilon,f_{-1},R_0} \delta_0.
\end{align*}

Therefore \eqref{lem57_2} follows by choosing $\delta_0$
sufficiently small.

\end{proof}

We are now state a weaker version of Theorem \ref{thm2} whose proof
is simpler (but elucidates the main ideas). The main difference is
that the constructed solution $\omega$ is only in $L^{\infty}$
rather than being continuous. This already gives the desired
illposedness in $H^1$ for compactly supported initial data.  We
defer the proof of Theorem \ref{thm2} slightly later since it needs
an additional perturbation argument in $C^0$.

\begin{thm} \label{thm2_weak}
Let $\omega^{(g)}_0\in C_c^{\infty}(\mathbb R^2) \cap \dot
H^{-1}(\mathbb R^2)$ be any given vorticity function which is odd in
$x_2$. For any
such $\omega^{(g)}_0$ and any $\epsilon>0$, we can find a
perturbation $\omega^{(p)}_0:\mathbb R^2\to \mathbb R$ such that the
following hold true:

\begin{enumerate}
 \item $\omega^{(p)}_0$ is compactly supported (in a ball of radius $\le 1$), continuous and
 \begin{align}
 \| \omega^{(p)}_0 \|_{\dot H^1(\mathbb R^2)}  +
 \| \omega^{(p)}_0\|_{L^{\infty}(\mathbb R^2)}+\|\omega^{(p)}_0\|_{\dot H^{-1}(\mathbb R^2)}
 <\epsilon. \notag
 \end{align}

 \item Let $\omega_0= \omega^{(g)}_0 + \omega^{(p)}_0$. Corresponding to $\omega_0$
 there exists a unique time-global solution $\omega = \omega(t)$ to the Euler equation
satisfying
\begin{align*}
\sup_{0<t<\infty} (\| \omega(t,\cdot)\|_{\infty} + \| \omega(t,\cdot)\|_{\dot H^{-1}} ) <\infty.
\end{align*}

\item $\omega(t)$ has additional local regularity in the following sense: there exists $x_* \in \mathbb R^2$ such
that for any $x\ne x_*$, there exists a neighborhood $N_x \ni x$, $t_x >0$ such that $w(t, \cdot) \in C^{\infty} (N_x)$ for any
$0\le t \le t_x$.

\item For any $0<t_0 \le 1$, we have
\begin{align*}
 \operatorname{ess-sup}_{0<t \le t_0} \| \omega(t,\cdot) \|_{\dot H^1} =+\infty.
\end{align*}
More precisely, there exist $0<t_n^1<t_n^2 <\frac 1n$, open precompact sets $\Omega_n$, $n=1,2,3,\cdots$ such that
$\omega(t) \in C^{\infty}(\Omega_n)$ for all $0\le t \le t_n^2$, and
\begin{align*}
 \| \nabla \omega(t,\cdot) \|_{L^2(\Omega_n)} >n, \quad \forall\, t\in[t_n^1,t_n^2].
\end{align*}

\end{enumerate}

\end{thm}

\subsection{Proof of Theorem \ref{thm2_weak}}
We begin by noting that the support condition in Statement (1) (``compactly supported in a ball of radius $\le 1$'') is rather
easy to achieve: one only needs to change the parameters of the distances between the patch solutions in our construction below.
Similar comment also applies to the condition ``$\| \omega^{(p)}_0 \|_{\dot H^1(\mathbb R^2)}  +
 \| \omega^{(p)}_0\|_{L^{\infty}(\mathbb R^2)}+\|\omega^{(p)}_0\|_{\dot H^{-1}(\mathbb R^2)}
 <\epsilon$''.  Therefore we shall ignore all these conditions below. In particular
 to simplify notation we will construct
 $\omega_0^{(p)}$ of order $1$.
 Also without loss of generality we may assume $\omega_0^{(g)}$ is supported (say) in a ball of radius $\le \frac 1{1000}$.

Define $z_0=(-2,0)$, $z_1=(0,0)$. For each integer $j\ge 2$, define
\begin{align}
 z_j = (z_j^{(1)}, 0)=\Bigl( \sum_{k=1}^{j-1} \frac {100} {2^k}, \, 0\Bigr) \label{def_zj1}
\end{align}

Obviously for any $j\ge 2$, we have
\begin{align*}
 & |z_j -z_{j-1}| = \frac {100} {2^{j-1}}, \notag \\
 & |z_{j+1}-z_j|= \frac {100} {2^j}.
\end{align*}

We shall choose $z_j$, $j\ge 0$ to be the center of the $j^{th}$ patch.

Now  define $h_0(x)=\omega_0^{(g)} (x-z_0)$ and $\delta_0=1$. By Lemma \ref{lem57} with $f_{-1}=h_0$, $R_0=\frac 14$,
$\epsilon = 1/800$, we can find $\delta_1>0$, $0<t_1<\frac 12$, and $h_1 \in C_c^{\infty} (B(0,\frac 1{800}))$ with the properties
\begin{itemize}
 \item $h_1$ is an odd function of $x_2$;
 \item $\| h_1 \|_{L^1} +\|h_1\|_{L^{\infty}} + \| h_1 \|_{H^1} + \| h_1 \|_{\dot H^{-1}} \le \frac 18$;
\end{itemize}
such that for any $\tilde f \in C_c^{\infty} (\mathbb R^2)$ with
\begin{itemize}
 \item $\operatorname{supp}(\tilde f ) \subset \{ x =(x_1,x_2):\quad x_1\ge \frac 14\}$;
 \item $\| \tilde f \|_{L^1} + \| \tilde f\|_{L^{\infty}} \le \delta_1$,
\end{itemize}
the following hold true:

For the Euler equation
\begin{align*}
 \begin{cases}
  \partial_t \omega +\Delta^{-1} \nabla^{\perp } \omega \cdot \nabla \omega =0, \\
  \omega \Bigr|_{t=0} =h_0+ h_1 + \tilde f,
 \end{cases}
\end{align*}
the smooth solution $\omega=\omega(t)$ satisfies:
\begin{enumerate}
 \item For any $0\le t \le t_1$, $\omega(t)$ can be decomposed as
 \begin{align*}
  \omega(t,x) = \omega_{h_0}(t,x)+\omega_{h_1} (t,x) +\omega_{\tilde f}(t,x),
 \end{align*}
where
\begin{align*}
 & \operatorname{supp} (\omega_{h_0} (t,\cdot) ) \subset B(0,-2+\frac 1{32}), \\
 & \operatorname{supp} (\omega_{h_1} (t,\cdot) ) \subset B(0,\frac 18+\frac 1{32}), \\
 & \operatorname{supp} (\omega_{\tilde f} (t,\cdot )) \subset \{x=(x_1,x_2):\quad x_1\ge \frac 14 +\frac 1{32} \};
\end{align*}

\item \begin{align*}
        \| \omega_{h_1} (t_1,\cdot) \|_{\dot H^1} >8.
      \end{align*}
\end{enumerate}

We now inductively assume that for $1\le i\le j$, we have chosen $h_i \in C_c^{\infty} (B(z_i, \frac 1 {2^{i+9}} ))$ which is odd in $x_2$,
$0<t_i<\frac 1 {2^i}$, $\delta_i>0$, with
\begin{align}
 & \| h_i \|_{L^1} + \| h_i \|_{L^{\infty}} +\|h_i\|_{H^1} + \| h_i \|_{\dot H^{-1}} \notag \\
 & \quad \le \frac 1 {2^i} \min_{0\le k <i} \delta_k, \label{pf2_10}
\end{align}
such that for any $\tilde f \in C_c^{\infty} (\mathbb R^2)$ with
\begin{itemize}
 \item $\operatorname{supp}(\tilde f) \subset \{ x=(x_1,x_2):\quad x_1 \ge z_{i+1}^{(1)} -\frac 1 {2^i}\}$ (see \eqref{def_zj1}) for the definition of
 $z_j^{(1)}$);
 \item $\| \tilde f \|_{L^1} +\| \tilde f \|_{L^{\infty}} \le \delta_i$,
\end{itemize}
the solution $\omega(t)$ to the equation
\begin{align*}
 \begin{cases}
  \partial_t \omega + \Delta^{-1} \nabla^{\perp} \omega \cdot \nabla \omega =0, \\
  \omega\Bigr|_{t=0} = \sum_{l=1}^{i-1} h_l +h_i +\tilde f,
 \end{cases}
\end{align*}
satisfies the properties:

\begin{enumerate}
 \item for any $0\le t \le t_i$, we have the decomposition
 \begin{align} \label{pf2_20}
  \omega(t,x) = \omega_{\le i-1} (t,x) + \omega_i(t,x) + \omega_{\tilde f }(t,x),
 \end{align}
where
\begin{align*}
 & \operatorname{supp} (\omega_{\le i-1} (t,\cdot )) \subset \{ x=(x_1,x_2):\quad x_1 \le z_{i-1}^{(1)}+\frac 1{2^{i}} \}; \\
 & \operatorname{supp} ( \omega_{i} (t,\cdot ) ) \subset \{ x=(x_1,x_2): \quad |x-z_i| \le \frac 1 {2^i} \}; \\
 & \operatorname{supp} (\omega_{\tilde f} (t,\cdot) ) \subset\{ x=(x_1,x_2): \quad x_1 \ge z_{i+1}^{(1)}-\frac 1{2^i} \};
\end{align*}

\item $\|\omega_i(t_i,\cdot ) \|_{\dot H^1} >2^i$.
\end{enumerate}

Then for $i=j+1$, by shifting the coordinate axis to $z_{j+1}$ if necessary, we can apply Lemma \ref{lem57}
 with $f_{-1} = \sum_{i=0}^j h_i$, $\epsilon \ll \frac 1 {2^{i+1}} \min_{0\le k\le i} \delta_k$, and choose
 $h_{j+1} \in C_c^{\infty} (B(z_{j+1}, \frac 1 {2^{j+9}} )$ to satisfy all the needed properties similar to
 the $i^{th}$ step. This way we have completely specified the profiles of all $h_j$, $j=0,1,2,\cdots$.

 Now we define the initial data
 $$\omega_0 =\sum_{j=0}^{\infty} h_j.$$

 It is easy to check that $\omega_0$ is compactly supported and $\omega_0 \in  L^{\infty} \cap \dot H^{-1} \cap H^1$.
 We can then appeal to Yudovich theory to construct the corresponding solution in $L^{\infty}$. However here for the sake
 of completeness we shall carry out a direct construction which also yields additional local regularity of the solution.
 For simplicity we shall be content with constructing a local solution on some time interval $[0,T_0]$ with $T_0>0$. The breakdown
 of $\dot H^1$ regularity will happen arbitrarily close to time $t=0$.

 To this end, denote the approximating initial data
 \begin{align*}
  \omega_0^{(J)} = \sum_{j=0}^J h_j
 \end{align*}
and let $\omega^{(J)}$ be the solution to the Euler equation
\begin{align} \label{pf2_99_a}
 \begin{cases}
  \partial_t \omega^{(J)} + \Delta^{-1} \nabla^{\perp} \omega^{(J)} \cdot \nabla \omega^{(J)} =0, \\
  \omega^{(J)} \Bigr|_{t=0} = \omega_0^{(J)}.
 \end{cases}
\end{align}

By using $L^p$, $1\le p\le \infty$ conservation of vorticity, $L^2$ conservation of velocity, it is easy to check that
\begin{align}
\sup_J \sup_{0\le t <\infty } \Bigl( \| \omega^{(J)}(t,\cdot) \|_{L^{1}} + \| \omega^{(J)}(t,\cdot) \|_{L^{\infty}} +
\| \omega^{(J)} (t,\cdot) \|_{\dot H^{-1}} \Bigr)\lesssim 1. \label{pf2_99}
\end{align}

We now verify that $(\omega^{(J)} (t) )$ forms a Cauchy sequence in the Banach space $C_t^0 \dot H^{-1} ([0,T_0])$ for some $T_0>0$
sufficiently small.

Denote $\eta^{(J+1)} = \omega^{(J+1)} - \omega^{(J)}$. Then
\begin{align}
 \label{pf2_100}
 \begin{cases}
  \partial_t \eta^{(J+1)} + \Delta^{-1} \nabla^{\perp} \eta^{(J+1)} \cdot \nabla \omega^{(J+1)} + \Delta^{-1} \nabla^{\perp} \omega^{(J)}
  \cdot \nabla \eta^{(J+1)} =0, \\
  \eta^{(J+1)} \Bigr|_{t=0}=h_{J+1}.
 \end{cases}
\end{align}

Multiplying both sides of \eqref{pf2_100} by $(-\Delta)^{-1}
\eta^{(J+1)}$, integrating by parts and using \eqref{pf2_99}, we
have

\begin{align}
 &\frac d {dt} \Bigl( \| |\nabla|^{-1} \eta^{(J+1)}  \|_2^2 \Bigr) \notag \\
 \le & \| |\nabla|^{-1} \eta^{(J+1) } \|_2^2 \cdot \| \omega^{(J+1)} \|_{\infty}
 +\left| \int_{\R^2} (\Delta^{-1} \nabla^{\perp} \omega^{(J)} \cdot \nabla \eta^{(J+1)}) \Delta^{-1} \eta^{(J+1)} dx \right| \notag \\
 & \lesssim \| |\nabla|^{-1} \eta^{(J+1)} \|_2^2
   +\left| \int_{\R^2 } (\Delta^{-1} \nabla^{\perp} \omega^{(J)} \cdot \nabla \eta^{(J+1)}) \Delta^{-1} \eta^{(J+1)} dx \right|.
 \label{pf2_101}
\end{align}

Here note that since $\eta^{(J+1)}$ is an odd function of $x_2$, we have $\widehat{\eta^{(J+1)}}(\xi=0)=0$. From this it is
easy to check that $\Delta^{-1} \eta^{(J+1)}$ is a smooth function. By successive integration by parts and H\"older, we have
\begin{align}
  & \left|\int_{\R^2} (\Delta^{-1} \nabla^{\perp} \omega^{(J)} \cdot \nabla \eta^{(J+1)}) \Delta^{-1} \eta^{(J+1)} dx \right| \notag \\
  = & \left|\int_{\R^2} (\Delta^{-1} \nabla^{\perp} \omega^{(J)} \cdot \nabla (\Delta \Delta^{-1}\eta^{(J+1)}) )\Delta^{-1} \eta^{(J+1)} dx \right| \notag \\
  = & \left| \sum_{k=1}^2
  \int_{\R^2} (\Delta^{-1}  \nabla^{\perp} \partial_k \omega^{(J)} \cdot \nabla (\Delta^{-1}\eta^{(J+1)} ) )\partial_k \Delta^{-1} \eta^{(J+1)} dx \right|
  \notag \\
  \lesssim & \| \mathcal R_{ij} \omega^{(J) }\|_p \cdot \| |\nabla|^{-1} \eta^{(J+1)} \|^2_{\frac {2p}{p-1}}, \label{pf2_102}
\end{align}
where $\mathcal R_{ij}$ is the usual Riesz transform and $1<p<\infty$.

Now recall that for $2\le p<\infty$, we have
\begin{align*}
 \| \mathcal R_{ij} \omega^{(J)} \|_p \le C_1 \cdot p \|\omega^{(J)}\|_p,
\end{align*}
where $C_1>0$ is an absolute constant.

By H\"older and \eqref{pf2_99}, we have
\begin{align*}
 \| |\nabla|^{-1} \eta^{(J+1)} \|_{\frac{2p}{p-1}} & \le
 \| |\nabla|^{-1} \eta^{(J+1)} \|_2^{\frac{p-1} p} \cdot \| |\nabla|^{-1} \eta^{(J+1)} \|_{\infty}^{\frac 1p} \notag \\
 & \lesssim \| |\nabla |^{-1} \eta^{(J+1)} \|_2^{\frac{p-1} p}.
\end{align*}

Plugging the last two estimates and \eqref{pf2_102} into \eqref{pf2_101}, we obtain
\begin{align}
 \frac d {dt} \Bigl( \| |\nabla|^{-1} \eta^{(J+1)} \|_2^2 \Bigr)
 \lesssim \| |\nabla|^{-1} \eta^{(J+1)} \|_2^2 + p \cdot \| |\nabla|^{-1} \eta^{(J+1)} \|_2^{\frac{2(p-1)} p}.
\notag
\end{align}

Now denote
$$a(t):= \| |\nabla|^{-1} \eta^{(J+1)}(t)\|_2^2.$$

We then obtain for some absolute constant $C>0$,
\begin{align*}
 a^{\prime} (t) \le C \cdot (a(t) + p a(t)^{\frac{p-1} p}).
\end{align*}
By \eqref{pf2_99}, we have $a(t)=a(t)^{\frac {p-1}p} a(t)^{\frac 1p} \lesssim a(t)^{\frac {p-1}p}$. Hence
\begin{align*}
 a^{\prime} (t) \le  C^{\prime} p a(t)^{\frac {p-1}p}, \quad\forall\, t\ge 0,
\end{align*}
where $C^{\prime}>0$ is another absolute constant.
Integrating in time gives the inequality,
\begin{align*}
 a(t)^{\frac 1p} \le a(0)^{\frac 1p} +C^{\prime } \cdot t, \quad\forall\, t\ge 0.
\end{align*}
 We now choose $T_0>0$ such that $C^{\prime }T_0\le \frac 14$. Then
for all $0\le t \le T_0$,
\begin{align*}
 a(t) \le (a(0)^{\frac 1p} +\frac 14)^p.
\end{align*}

Now since $a(0) = \| |\nabla|^{-1} \eta^{(J+1)} (0)\|_2^2 = \| |\nabla|^{-1} h_{J+1} \|_2^2 \le  {4^{-(J+1)}}$ (see \eqref{pf2_10}),
taking $p=J$ gives us
\begin{align*}
 a(t) & \le (4^{-\frac{J+1} p} +\frac 14 )^p \notag \\
 & \le (\frac 12)^p=2^{-J},\quad \forall \, 0\le t\le T_0.
\end{align*}

Therefore
\begin{align*}
 \max_{0\le t \le T_0} \| |\nabla|^{-1} \eta^{(J+1)} (t,\cdot) \|_2^2 \le 2^{-J}
\end{align*}
and consequently $(\omega^{(J)})$ is Cauchy in $C_t^0 \dot H^{-1} ([0,T_0])$. Taking the limit $J\to \infty$ gives us
the desired solution in $L^1 \cap L^{\infty} \cap \dot H^{-1}$. By a simple interpolation argument we also have
$(\omega^{(J)})$ is Cauchy in $L^p$ for any $1<p<\infty$.

Set $x_*= \lim_{j\to \infty} z_j = (100,0)$. We now prove statement
(3) and (4) in Theorem \ref{thm2_weak}. Fix any integer $n\ge 2$ and
we choose $t_n<\frac 1{2^n}$ in the same way as specified in
\eqref{pf2_10}. By our way of construction, the fact that
$(\omega^{(J)})$ is Cauchy in $L^p$ for any
$1<p<\infty$\footnote{This is used to show the contraction of all
higher Sobolev norms in patches away from the limiting point $x_*$
(and in a very small time interval) which will yield the
$C^{\infty}$ regularity of $\omega_{<n}$ and $\omega_n$ in
\eqref{pf2_104} below.} and (a version of) Lemma \ref{lem10}, we
have that the limit solution $\omega$ obeys the following
decomposition similar to that in \eqref{pf2_20}:

More precisely define $t_n^2=t_n$, then for any $0\le t \le t_n^2$, we have
\begin{align} \label{pf2_104}
 \omega(t,x) = \omega_{<n} (t,x) + \omega_n(t,x) + \omega_{>n}(t,x),
\end{align}
where $\omega_{<n}(t,\cdot) \in C^{\infty}_c(\Omega_{<n})$, $\omega_n(t,\cdot) \in C_c^{\infty}(\Omega_n)$, and
\begin{align*}
&\Omega_{<n}:= \{ x=(x_1,x_2):\quad |x|<1000 \text{ and } x_1 < z_{n-1}^{(1)}+\frac 2{2^{n}} \}; \\
 & \Omega_{n} := \{ x=(x_1,x_2): \quad |x-z_n| < \frac 2 {2^n} \}; \\
 & \operatorname{supp} (\omega_{>n} (t,\cdot) ) \subset\{ x=(x_1,x_2): \quad x_1 \ge z_{n+1}^{(1)}-\frac 1{2^n} \};
\end{align*}
Furthermore we can choose $t_n^1<t_n^2$ ($t_n^1$ is sufficiently close to $t_n^2$) such that
 \begin{align*}
 \|\omega_n(t) \|_{\dot H^1} >n, \quad \forall\,t \in[t_n^1,t_n^2].
 \end{align*}

Therefore statement (4) in Theorem \ref{thm2_weak} is proved. Now for statement (3) we discuss two cases.
If $x=(x_1,x_2) \ne x_*=(100,0)$ and $x_1\ge 100$, then by using finite transportation speed we can find
a neighborhood $N_x$ of $x$ and $t_x>0$ sufficiently small such that $\omega(t,x)=0$ for any $0\le t\le t_x$ and $x \in N_x$.
Similarly we can treat the case $x=(x_1,x_2)$, $x_1<100$ and $|x|>500$.
On the other hand if $x=(x_1,x_2)$ and $x_1<100$ with $|x|\le 500$, then we can find $n$ sufficiently large such that $x \in \Omega_{<n}$.
Obviously we just need to define $N_x=\Omega_{<n}$ and $t_x=t_n^2$ so that $\omega(t,\cdot) \in C^{\infty} (N_x)$ for all
$0\le t\le t_x$.

This concludes the proof of Theorem \ref{thm2_weak}.
\medskip
$$\;$$

We now state the $C^0$-perturbation lemma needed for the proof of
Theorem \ref{thm2}.

\begin{lem} \label{lem_C0thm2}
Let $R_0>0$ and $f\in C_c^{\infty}(B(0,R_0))$, $g\in C_c^{\infty}(B(0,R_0))$.
Let $\omega^a$ and $\omega$ be smooth solutions to the following
2D Euler equations:
\begin{align} \label{lem_C0thm2_t1}
\begin{cases}
\partial_t {\omega^a} + (u^a \cdot \nabla) {\omega^a}
 =0, \quad 0<t\le 1, \, x\in\mathbb R^2, \\
u^a=\Delta^{-1} \nabla^{\perp} \omega^a, \\
\omega^a \Bigr|_{t=0} =f.
\end{cases}
\end{align}
\begin{align} \label{lem_C0thm2_t2}
\begin{cases}
\partial_t \omega + (u \cdot \nabla)\omega
 =0,\quad 0<t\le 1,\, x\in \mathbb R^2, \\
u=\Delta^{-1} \nabla^{\perp} \omega, \\
\omega \Bigr|_{t=0} =f+g.
\end{cases}
\end{align}

For any $\epsilon>0$, there exists $\delta=\delta(\epsilon,R_0,f)>0$
sufficiently small such that if
\begin{align*}
\| g\|_{\infty}<\delta
\end{align*}
then
\begin{align} \label{lem_C0thm2_t3}
\max_{0\le t\le 1}\| \omega^a(t,\cdot) -\omega(t,\cdot)\|_{\infty}
<\epsilon.
\end{align}

\end{lem}
\begin{proof}[Proof of Lemma \ref{lem_C0thm2}]
By first taking $\|g\|_{\infty}\lesssim 1$, we have $\|f+g\|_{\infty} \lesssim_{f,R_0} 1$.
Since $\operatorname{supp}(f) \subset B(0,R_0)$ and $\operatorname{supp}(g) \subset B(0,R_0)$, we get
\begin{align*}
&\operatorname{supp}(\omega(t,\cdot))\subset B(0,R_1), \\
&\operatorname{supp}(\omega^a(t,\cdot))\subset B(0,R_1), \quad\forall\, 0\le t\le 1,
\end{align*}
where $R_1>0$ is some constant depending on $R_0$ and $\|f\|_{\infty}$ only.

Set $\eta=\omega^a-\omega$. Then $\eta$ satisfies the equation
\begin{align} \label{lem5.6_pf1}
\begin{cases}
\partial_t \eta + (\Delta^{-1} \nabla^{\perp} \eta)\cdot \nabla \omega^a + u\cdot \nabla \eta =0, \\
\eta(0)=g.
\end{cases}
\end{align}

By a simple energy estimate, we have
\begin{align*}
\max_{0\le t\le 1} \| \nabla \omega^a(t,\cdot)\|_{\infty} \lesssim_{f} 1.
\end{align*}

On the other hand, since $\operatorname{supp}(\eta(t,\cdot) ) \subset B(0,R_1)$ for any $0\le t\le 1$, we have
\begin{align*}
 \| \Delta^{-1} \nabla^{\perp} \eta(t,\cdot)\|_{\infty} & \lesssim \| \eta(t,\cdot)\|_1 + \| \eta(t,\cdot)\|_{\infty} \notag \\
 & \lesssim_{R_1} \| \eta(t,\cdot)\|_{\infty}.
\end{align*}

By \eqref{lem5.6_pf1}, we then get for any $0\le t\le 1$,
\begin{align*}
\| \eta(t,\cdot)\|_{\infty} \lesssim_{R_1,f} \| g\|_{\infty} + \int_0^t \| \eta(s,\cdot)\|_{\infty} ds.
\end{align*}

A Gronwall argument then yields
\begin{align*}
\max_{0\le t\le 1} \| \eta(t,\cdot)\|_{\infty} \lesssim_{R_1,f} \| g\|_{\infty}.
\end{align*}

Therefore \eqref{lem_C0thm2_t3} follows by choosing $\| g\|_{\infty}$ sufficiently small.
\end{proof}

We now sketch the

\begin{proof}[Proof of Theorem \ref{thm2}]
We still choose the centers $z_j$ as in \eqref{def_zj1}. The main modification
is the choice of the initial profiles $h_j$ (at $z_j$). For this we need to modify
Lemma \ref{lem57} and combine Lemma \ref{lem_C0thm2} so that each when a new profile
$h_J$ ,$J\ge 2$ is added, besides satisfying the existing constraints (as specified in Lemma \ref{lem57}),
the following also hold:

let $\omega^{(J)}$ be defined the same way as in \eqref{pf2_99_a}, then $\|h_{J}\|_{\infty}$ is
sufficiently small (as specified in Lemma \ref{lem_C0thm2}) such that
\begin{align*}
\max_{0\le t\le 1}\| \omega^{(J)}(t,\cdot) -\omega^{(J-1)}(t,\cdot) \|_{\infty}\le 2^{-J}.
\end{align*}

Note that by \eqref{pf2_99} we can always guarantee for some constant $R>0$ that
\begin{align*}
\operatorname{supp}( \omega^{(J)}(t,\cdot)) \subset B(0,R), \quad \forall\, 0\le t\le 1,\, J\ge 1.
\end{align*}

We then view $(\omega^{(J)})_{J\ge 1}$ as a Cauchy sequence in the Banach space
$C_t^0 C_x^0 ([0,1]\times \overline{B(0,R)})$ and extracts the limit solution $\omega$ in the
same space. By interpolation and Sobolev embedding, easy to check that $u^{(J)}=\Delta^{-1} \nabla^{\perp} \omega^{(J)}$
also forms a Cauchy sequence in $C_t^0 L_x^2 \cap C_t^0 C_x^{\alpha} ([0,1]\times \mathbb R^2)$ for any $0<\alpha<1$.
Therefore $u^{(J)}$  converges to the limit $u=\Delta^{-1} \nabla^{\perp} \omega$ and $\omega$ is the desired solution.
The other properties of $\omega$ can be proved in the same way as in Theorem \ref{thm2_weak}. We omit
the repetitive details.

\end{proof}

\section{3D Euler with non-compactly supported data} \label{sec_3D_1}
The main building block in our construction for 3D Euler is the axisymmetric flow without swirl on $\mathbb R^3$.
To unify the notation we first recall the definition and review some useful properties.
We shall review some relevant literature along the way.

By an axisymmetric scalar function on $\mathbb R^3$, we mean a function of the form $f=f(r,z)$, with
$x=(x_1,x_2,z)$, $r=\sqrt{x_1^2+x_2^2}$. In other words the function is invariant under any rotation
about the vertical axis $OZ=\{(0,0,z):\, z\in \mathbb R\}$. Analogously one can define an axisymmetric
vector field on $\mathbb R^3$. The velocity field $u$ representing an axisymmetric flow on
$\mathbb R^3$ generally takes the form
\begin{align*}
u(x_1,x_2,z)= u^r(r,z) e_r + u^{\theta}(r,z) e_{\theta} + u^z(r,z) e_z,
\end{align*}
where $(e_r,e_{\theta},e_{z})$ is the standard orthogonal basis for the cylindrical coordinate system:
\begin{align*}
&e_r = \frac 1 r (x_1,x_2,0), \quad e_{\theta}= \frac 1r (-x_2,x_1,0), \\
&e_z=(0,0,1).
\end{align*}
Here $u^r$, $u^{\theta}$, $u^{z}$ are called the radial, angular/swirl and axial velocity respectively.
By an axisymmetric flow without swirl, we mean the angular (swirl) component $u^{\theta}\equiv 0$, that is
\begin{align*}
u(x_1,x_2,z)=u^r(r,z) e_r + u^z(r,z) e_z.
\end{align*}
The corresponding vorticity $\omega=\nabla\times u$ then reduces to
\begin{align*}
\omega(x_1,x_2,z)= \omega^{\theta}(r,z) e_{\theta} =(\partial_z u^r - \partial_r u^z) e_{\theta}.
\end{align*}
It follows easily that
\begin{align*}
(\omega \cdot \nabla)u &= (\omega^{\theta} e_{\theta} \cdot \nabla ) u \notag \\
&= \omega^{\theta} \frac 1 r u^r e_{\theta} = \frac{u^r} r \omega.
\end{align*}
Therefore the vorticity equation reads as
\begin{align*}
\partial_t \omega + (u^r \partial_r + u^z \partial_z ) \omega = \frac 1 r u^r \omega,
\end{align*}
or simply,
\begin{align*}
\partial_t \left( \frac{\omega} r \right) + (u\cdot \nabla) \left( \frac{\omega} r \right)=0.
\end{align*}
In this way we obtain a transport equation for the quantity $\omega/r$. Since the velocity
$u$ is divergence-free, all $L^p$, $1\le p\le \infty$ and similar Lorentz space norms of $\omega/r$
are preserved in time. These important conservation laws are the key to obtaining global solutions.
Indeed Ukhovskii and Yudovich \cite{UY68}, and independently Ladyzhenskaya \cite{La68} first proved
global wellposedness for initial velocity $u_0 \in H^1$ with initial vorticity satisfying
$\omega_0 \in L^2 \cap L^{\infty}$, $\frac{\omega_0}r \in L^2 \cap L^{\infty}$. In terms of Sobolev
regularity, one need the initial velocity $u_0 \in H^s$, $s>7/2$ to have $\frac 1r \omega_0 \in L^{\infty}$.
In \cite{SY94}, Shirota and Yanagisawa weakened the regularity on velocity to $u_0 \in H^s$, $s>5/2$ which
is the borderline in view of the $H^s$ local wellposedness theory in 3D. Danchin \cite{Danchin07} showed
global existence and uniqueness of solutions for initial vorticity $\omega_0 \in L^{\infty}\cap L^{3,1}$
with $\frac{\omega}r \in L^{3,1}$ in bounded domains with $C^{2,\alpha}$ boundary
or the whole space $\mathbb R^3$. In \cite{AHK10} Abidi, Hmidi and Keraani\footnote{The work of Danchin \cite{Danchin07}
does not address the propagation of critical Besov regularity $B^{\frac 3p+1}_{p,1}$ due to the lack of
Beale-Kato-Majda criteria in borderline Besov spaces.} proved global wellposedness in the space
$C_t^0 B^{1+\frac 3p}_{p,1}(\mathbb R_+\times \mathbb R^3)$ for initial velocity $u_0 \in B_p^{\frac 3p+1}(\mathbb R^3)$,
$1<p\le \infty$ with the additional mild assumption that $\frac{\omega_0} r \in L^{3,1}(\mathbb R^3)$ and
$\omega_0 \in \widetilde B^0_{\infty,1}$. Here $\widetilde B^0_{\infty,1}$ is a subspace of $B^0_{\infty,1}$
containing
\begin{align*}
\Bigl\{ u \in \mathcal S^{\prime}(\mathbb R^3):\quad \sum_{\substack{N\ge 2\\ \text{$N$ is dyadic}}}
(\log N) \| P_N u\|_{\infty} <\infty\Bigr\}.
\end{align*}

We now state a few basic lemmas needed for our construction later. Some of these are well-known
facts which are already extensively used in the aforementioned works.

\begin{lem}[$L^{p,q}$-preservation] \label{lem_Lpq_1}
Let $1\le p,q\le\infty$. Suppose $u$ is a given smooth divergence-free vector field on $\mathbb R^d$, $d\ge 2$.
Let $h$ be the smooth solution to the transport equation
\begin{align*}
\begin{cases}
\partial_t h + (u \cdot \nabla) h =f, \\
h\Bigr|_{t=0}=h_0.
\end{cases}
\end{align*}
Then for any $t>0$, we have
\begin{align*}
\| h(t) \|_{L^{p,q}(\mathbb R^d)} \le \|h_0 \|_{L^{p,q}(\mathbb R^d)} + \int_0^t \| f(\tau)\|_{L^{p,q}(\mathbb R^d)}
d \tau.
\end{align*}
If $f\equiv 0$, then
\begin{align*}
\| h(t) \|_{L^{p,q}(\mathbb R^d)} = \|h_0\|_{L^{p,q}(\mathbb R^d)}.
\end{align*}
\end{lem}

\begin{proof}[Proof of Lemma \ref{lem_Lpq_1}]
See for example Proposition 2 on p484 of Danchin \cite{Danchin07} or Prop 2.2 of Abidi-Hmidi-Keraani \cite{AHK10}.
In the homogeneous case $f\equiv 0$, one just observes that $h(t)=h_0 \circ \phi(t)$ where the flow map $\phi(t)$
is measure-preserving.  This obviously implies $\|h_0 \circ \phi(t)\|_{L^{p,q}}= \| h_0\|_{L^{p,q}}$ by using
the definition of Lorentz spaces. Alternatively one can use $L^p$ conservation and interpolation, as done in
\cite{AHK10}. The general case of nonzero $f$ follows from Duhamel.
\end{proof}

We shall use Lemma \ref{lem_Lpq_1} often without explicit mentioning. The most useful case for us is the
$L^{3,1}$ conservation of vorticity.

The next useful lemma is a Biot-Savart law estimate in the axisymmetric setting. It is the key to the proof of
global wellposedness for the axisymmetric (without swirl) flow.

\begin{lem} \label{lem_ur_r}
There exists an absolute constant $C>0$ such that
\begin{align*}
\left\| \frac{u^r} r \right\|_{L^{\infty}(\mathbb R^3)} \le C \left\| \frac{\omega^{\theta}}r
\right\|_{L^{3,1}(\mathbb R^3)},
\end{align*}
where $u=u^r e_r+u^z e_z$, $\omega=\nabla \times u= \omega^{\theta} e_{\theta}$.
\end{lem}
\begin{proof}[Proof of Lemma \ref{lem_ur_r}]
See, for example, Proposition 3.1 of \cite{AHK10}. The key is to use Lemma 1 from \cite{SY94} which gives
the kernel estimate:
\begin{align*}
|u^r (x) | \lesssim \int_{|y-x|\le r } \frac{|\omega(y)|}{|x-y|^2} dy
+ r \int_{|y-x| \ge r } \frac{|\omega(y)|}{|x-y|^3} dy.
\end{align*}
\end{proof}

Occasionally we need the following lemma.

\begin{lem} \label{lem_ur_r_2}
Let $1<p<3$ and $u\in B^{\frac 3p+1}_{p,q}(\mathbb R^3)$ be an axisymmetric divergence-free vector field.
Let $\omega=\nabla \times u$ be its vorticity. Then
\begin{align*}
\| \frac {\omega} r \|_{L^{3,1}(\mathbb R^3)} \lesssim \| u \|_{B^{\frac 3p+1}_{p,1}(\mathbb R^3)}.
\end{align*}
\end{lem}
\begin{proof}[Proof of Lemma \ref{lem_ur_r_2}]
This is Proposition 2.1 from \cite{AHK10}. the key is to prove $\| \frac {\omega} r \|_{L^{3,1}} \lesssim
\| \nabla \omega\|_{L^{3,1}}$ which is obtained by using $\omega(0,0,z)\equiv 0$ and the Fundamental
Theorem of calculus. The embedding $B^{\frac3p-1}_{p,1} \hookrightarrow L^{3,1}$ comes from  interpolation
\begin{align*}
&(L^p,L^r)_{(\theta,1)} = L^{3,1} \notag \\
&(B^0_{p,1},B^{\frac 3p-\frac 3r}_{p,1})_{(\theta,1)}=B^{\theta(\frac 3p-\frac 3r)}_{p,1}=B^{\frac 3p-1}_{p,1},
\end{align*}
where $1<p<3<r<\infty$ and $\frac 13= \frac{1-\theta}p+\frac{\theta}r$,
followed by the embedding of Besov into Lebesgue.
\end{proof}

We now state a lemma which sets up some basic properties of the axisymmetric flow and the associated characteristic
lines.

\begin{lem} \label{x-1}
 Consider the Euler equation (in vorticity form)
 \begin{align*}
  \begin{cases}
   \partial_t \omega + u \cdot \nabla \omega =(\omega \cdot \nabla) u, \quad (t,x) \in[0,\infty)\times \mathbb R^3,\\
   u=-\Delta^{-1}\nabla \times \omega,\\
   \omega \Bigr|_{t=0} =\omega_0,
  \end{cases}
 \end{align*}
where the initial vorticity $\omega_0 = \nabla \times u_0$ and $u_0$ is a smooth axisymmetric velocity field without swirl.

Write
\begin{align}
 u= u^r e_r + u^z e_z, \label{x-1_0}
\end{align}
and consider the forward characteristic lines expressed in cylindrical coordinates, i.e.
\begin{align}\label{x-1_00}
 \begin{cases}
  \partial_t \phi^r(t,r,z) = u^r (t, \phi^r(t,r,z), \phi^z(t,r,z)), \\
  \partial_t \phi^z(t,r,z) =u^z (t, \phi^r(t,r,z), \phi^z(t,r,z)),\\
  (\phi^r, \phi^z) \Bigr|_{t=0} = (r,z).
 \end{cases}
\end{align}

Denote by $\tilde \phi(t,r,z)= (\tilde \phi^r (t,r,z), \tilde \phi^z(t,r,z) )$ the inverse map of $\phi=(\phi^r,\phi^z)$. Then
the following hold:

\begin{itemize}
 \item For any $z \in \mathbb R$, $t\ge 0$, we have
 \begin{align}
 & u^r(t,0,z)=0, \notag \\
  &\phi^r(t,0,z)=0=\tilde \phi^r(t,0,z). \label{x-1_1}
 \end{align}
Also
\begin{align}
\phi^r(t,r,z)>0, \, \tilde \phi^r(t,r,z)>0,\quad\forall\, r>0, z\in
\mathbb R,\, t\ge 0. \label{x-1_1a}
\end{align}
\item For any $t\ge 0$, $r>0$, $z\in \mathbb R$, we have
\begin{align}
 &\det \Bigl( \frac{\partial (\phi^r(t), \phi^z(t) ) }
 {\partial(r,z)} \Bigr)\notag \\
  =& \frac r {\phi^r(t,r,z)} \notag \\
=& \exp\Bigl( - \int_0^t \frac 1 {\phi^r(s,r,z)} u^r (s,\phi^r(s,r,z), \phi^z(s,r,z) ) ds \Bigr).  \label{x-1_2}
\end{align}
For $t\ge 0$, $r=0$, $z\in \mathbb R$, we have
\begin{align}
 &\det \Bigl( \frac{\partial (\phi^r(t), \phi^z(t) ) }
 {\partial(r,z)}
\Bigr) \Bigr|_{(r,z)
=(0,z)} \notag \\
=& \lim_{r\to 0}\frac r {\phi^r(t,r,z)} \notag \\
=& \exp\Bigl( - \int_0^t \lim_{r\to 0}\frac 1 {\phi^r(s,r,z)} u^r (s,\phi^r(s,r,z), \phi^z(s,r,z) ) ds \Bigr).  \label{x-1_3}
\end{align}
Here both limits exist and are finite.

\item Similarly for any $t\ge 0$, $r\ge 0$, $z\in \mathbb R$, we have
\begin{align} \label{x-1_4}
 \det \Bigl( \frac{\partial (\tilde \phi^r(t), \tilde \phi^z(t) ) }
 {\partial(r,z)}
\Bigr) & = \frac r {\tilde \phi^r(t,r,z)}.
\end{align}

\end{itemize}

\end{lem}

\begin{proof}[Proof of Lemma \ref{x-1}]
 We first show \eqref{x-1_1}. Obviously by \eqref{x-1_0}, we must have
 \begin{align*}
  u^r(t,0,z)=0, \quad\forall\, t\ge 0,\, z \in \mathbb R,
 \end{align*}
since otherwise $u$ would not be smooth at $x=(0,0,z)$. By a similar
consideration, we can Taylor expand $(u_1,u_2)$ around the point
$(0,0,z)$ to get
\begin{align*}
\begin{pmatrix} u_1(t,x_1,x_2,z),\; u_2(t,x_1,x_2,z) \end{pmatrix}
& = c(t,z) \begin{pmatrix} x_1, \; x_2 \end{pmatrix} +O(r^2) \notag \\
& = c(t,z) r e_r +O(r^2) \\
&= u^r(t,r,z) e_r.
\end{align*}
Here the coefficient $c(t,z)$ is scalar valued and
\begin{align*}
 c(t,z)= (\partial_1 u_1)(t,0,0,z) = (\partial_2 u_2)(t,0,0,z).
\end{align*}

Therefore
\begin{align*}
 u^r(t,r,z) = c(t,z) r +O(r^2)
\end{align*}
and
\begin{align}
 & \lim_{r\to 0} \frac{u^r(t,r,z)} r = c(t,z), \notag \\
 & \lim_{r\to 0} (\partial_r u^r) (t,r,z) =c(t,z). \label{x-1_5a}
\end{align}

By \eqref{x-1_00}, we have
\begin{align*}
 \frac d {dt} \phi^r & = u^r (t,\phi^r,\phi^z) - u^r(t,0,\phi^z) \notag \\
 & = \Bigl( \int_0^1 (\partial_r u^r ) (t, \theta \phi^r, \phi^z) d\theta \Bigr) \phi^r.
\end{align*}

Integrating in time then gives
\begin{align}
 \phi^r (t,r,z) & = r \exp \Bigl( \int_0^t \int_0^1 (\partial_r u^r) (s, \theta \phi^r(s,r,z), \phi^z(s,r,z)
 )d\theta ds \Bigr). \label{x-1_5}
\end{align}
Clearly \eqref{x-1_1a} follows. Also
\begin{align*}
 \phi^r(t,0,z)=0, \quad \text{for any }t\ge 0, z \in \mathbb R.
\end{align*}

Next we show \eqref{x-1_2}. We shall calculate $\det \Bigl( \frac{\partial(\phi^r,\phi^z)} {\partial(r,z)}   \Bigr)$ in
two ways which in turn would yield the first and second identity in \eqref{x-1_2}.

Introduce new variables
\begin{align}
 &R := \frac1 2 \phi^r(t,r,z)^2, \notag \\
 &Z:= \phi^z(t,r,z). \label{x-1_6}
\end{align}

Then
\begin{align} \notag
 \begin{cases}
  \frac d {dt} R =V^r(t,R,Z) \\
  \frac d {dt} Z = V^z(t,R,Z) \\
  (R,Z) \Bigr|_{t=0} =(R_0,Z_0).
 \end{cases}
\end{align}
Here
\begin{align*}
 V^r(t,R,Z) &:= \sqrt{2R} u^r(t,\sqrt{2R}, Z), \\
 V^z(t,R,Z) &: = u^z(t,\sqrt{2R}, Z).
\end{align*}

By the incompressibility condition $\nabla \cdot u=0$, we have
\begin{align*}
 \frac 1 r \partial_r (r u^r(t,r,z) ) + \partial_z u^z(t,r,z)=0.
\end{align*}

Therefore easy to check that
\begin{align*}
 \partial_R V^r(t,R,Z) + \partial_Z V^z(t,R,Z)=0.
\end{align*}

It follows that
\begin{align*}
 \det\Bigl( \frac{\partial(R(t), Z(t) )} {\partial(R_0,Z_0)}   \Bigr)=1, \quad \forall\, t\ge 0,
\end{align*}
or
\begin{align*}
 \frac{\partial R} {\partial R_0} \frac{\partial Z} {\partial Z_0} - \frac{\partial R}{\partial Z_0} \frac{\partial Z} {\partial R_0}=1.
\end{align*}

Letting $R_0=\frac 1 2 r^2$, $Z_0=z$ and using \eqref{x-1_6} then gives
\begin{align*}
 \det \Bigl( \frac{\partial(\phi^r,\phi^z)} {\partial(r,z)}\Bigr) = \frac r {\phi^r}.
\end{align*}

For the second way of calculating $\det \Bigl( \frac{\partial(\phi^r,\phi^z)} {\partial(r,z)}   \Bigr)$, we need to use the following
elementary fact: if $A(t) \in \mathbb R^{d\times d}$, $B(t) \in \mathbb R^{d\times d}$ are smooth matrix-valued functions solving the ODE
\begin{align*}
 \frac d {dt} A(t) = B(t) A(t)
\end{align*}
with $A(0)=I_d$ (the identity matrix on $\mathbb R^{d\times d}$), then
\begin{align*}
 \det (A(t)) = \exp\Bigl( \int_0^t \text{tr} (B(s)) ds \Bigr),
\end{align*}
where $\text{tr}(\cdot)$ denotes the usual matrix trace.

By \eqref{x-1_00}, we have
\begin{align*}
 \frac d {dt}
 \begin{pmatrix}
  \frac{\partial \phi^r} {\partial r} \quad \frac{\partial \phi^r} {\partial z} \\
  \frac{\partial \phi^z} {\partial r } \quad \frac{\partial \phi^z} {\partial z}
 \end{pmatrix}
=\begin{pmatrix}
  \partial_r u^r \quad \partial_z u^r \\
  \partial_r u^z \quad \partial_z u^z
 \end{pmatrix}
\begin{pmatrix}
  \frac{\partial \phi^r} {\partial r} \quad \frac{\partial \phi^r} {\partial z} \\
  \frac{\partial \phi^z} {\partial r } \quad \frac{\partial \phi^z} {\partial z}
 \end{pmatrix}.
\end{align*}

Since $\partial_r u^r + \partial_z u^z = -\frac1 r u^r$, we get
\begin{align*}
 &(\partial_r u^r)(t,\phi^r,\phi^z) + (\partial_z u^z)(t,\phi^r,\phi^z) \notag \\
 = & -\frac1{\phi^r} u^r(t,\phi^r,\phi^z).
\end{align*}

Therefore the second identity in \eqref{x-1_2} follows. It is not difficult to check that
this also coincides with \eqref{x-1_5} derived earlier.

The existence of the limits \eqref{x-1_3} is a simple consequence of \eqref{x-1_5a} and \eqref{x-1_5}.

Finally \eqref{x-1_4} follows easily from \eqref{x-1_2} and the identity
\begin{align*}
 D\tilde \phi = \Bigl( (D\phi) (\tilde \phi) \Bigr)^{-1}.
\end{align*}

\end{proof}

We now take a parameter $A\gg 1$ and define (by a slight abuse of notation) a class of axisymmetric functions
\begin{align}
 g_A(x_1,x_2,z) & = g_A (r,z) \notag \\
 & =   \frac {\sqrt{\log A}} {{ A}}
 \sum_{A\le k \le 2 A} \eta_k (r,z), \label{x2_1}
\end{align}
where $r=\sqrt{x_1^2+x_2^2}$ and
\begin{align}
 \eta_k(r,z) =
  \sum_{\epsilon_1 = \pm 1} \epsilon_1 \eta_0 \Bigl( \frac{(r-2^{-k}, z-\epsilon_1 2^{-k})} {2^{-(k+10)}} \Bigr).
\label{x2_1a}
\end{align}

Here we choose $\eta_0 \in C_c^{\infty} ( B(0,1) )$ to be a radial compactly supported function such that $0\le \eta_0\le 1$.
Note that by construction $\eta_k$ is an odd function of $z$ and so is $g_A$. Also it is easy to see that
\begin{align}
 \eta_k(2^{-k}r, 2^{-k}z) & = \sum_{\epsilon_1 = \pm 1} \epsilon_1 \eta_0 \Bigl(  \frac{(r-1,z-\epsilon_1)} {2^{-10}}\Bigr)
 \notag \\
 &=: \tilde \eta_0(r,z). \label{x2_2}
\end{align}

Recall $e_{\theta}=\frac 1r (-x_2,x_1,0)$. We first check that
\begin{align}
& \| g_A e_{\theta} \|_{L^{\infty}(\mathbb R^3)}  \lesssim  \frac {\sqrt{\log A}} { A}. \label{x2_3b}\\
& \| g_A e_{\theta} \|_{\dot B^{\frac 32}_{2,1}(\mathbb R^3)}  \lesssim  {\sqrt{\log A}}. \label{x2_3c}\\
& \| \frac {g_A}r e_{\theta} \|_{L^{3,1}(\mathbb R^3)}  \lesssim  {\sqrt{\log A}}. \label{x2_3d}\\
& \| g_A \|_{\dot H^{\frac 32} (\mathbb R^3) } \lesssim \frac {\sqrt{\log A}} {\sqrt{ A}}, \label{x2_3}\\
& \| g_A e_{\theta} \|_{\dot H^{\frac 32} (\mathbb R^3)} \lesssim \frac {\sqrt{\log A}} {\sqrt{ A}}. \label{x2_3a}
\end{align}

Note that \eqref{x2_3b} is obvious. The property \eqref{x2_3c} follows from the triangle inequality
and the fact that each $\eta_k e_{\theta} $ has the same $\dot B^{\frac 32}_{2,1}$ norm.
The inequality \eqref{x2_3d} follows from Lemma \ref{lem_ur_r_2}.

For \eqref{x2_3}, albeit standard, we explain how to check it.
By definition and direct expansion, we have
\begin{align*}
 &\| g_A \|_{\dot H^{\frac 32}}^2  \notag \\
 = &  \frac  {\log A}
  { A^2}
 \sum_{A\le k \le 2 A}
 \sum_{A\le l \le 2 A}
 \int_{\mathbb R^3} |\nabla|^{\frac 32} \eta_k \cdot |\nabla|^{\frac 32} \eta_l dx.
\end{align*}

Therefore it suffices to show for each $k$,
\begin{align}
 \sum_{A\le l \le 2A}
 \left| \int_{\mathbb R^3} |\nabla|^{\frac 32} \eta_k \cdot |\nabla|^{\frac 32} \eta_l dx
 \right| \lesssim 1. \label{x2_4}
\end{align}

By scaling $(r,z) \to (2^{-k} r, 2^{-k}z)$ and \eqref{x2_2}, we have
\begin{align*}
 \text{LHS of \eqref{x2_4} } \lesssim \sum_{l\in \mathbb Z}
 \left| \int_{\mathbb R^3} |\nabla|^{\frac 32} \tilde \eta_0 \cdot |\nabla|^{\frac 32} \eta_l dx \right|,
\end{align*}
where $\tilde \eta_0$ was defined in \eqref{x2_2}. Note that $\tilde \eta_0 \in C_c^{\infty}(\mathbb R^3)$
since it is supported on $r\sim 1$, $|z|\sim 1$.

Now discuss two cases. If $l\ge 0$, then
\begin{align*}
 & \left| \int_{\mathbb R^3} |\nabla|^{\frac 32} \tilde \eta_0 \cdot |\nabla|^{\frac 32} \eta_l dx \right| \notag \\
 = & \left| \int_{\mathbb R^3} |\nabla|^3 \tilde \eta_0 \cdot \eta_l dx \right| \notag \\
 \lesssim & \| |\nabla|^3 \tilde \eta_0 \|_{L^\infty} \cdot \|\eta_l \|_{L^1} \lesssim 2^{-3l},
\end{align*}
which is summable for $l\ge 0$.

On the other hand if $l<0$, then
\begin{align*}
  & \left| \int_{\mathbb R^3} |\nabla|^{\frac 32} \tilde \eta_0 \cdot |\nabla|^{\frac 32} \eta_l dx \right| \notag \\
  = & \left| \int_{\mathbb R^3} |\nabla| \tilde \eta_0 \cdot \Delta \eta_l dx \right| \notag \\
  \lesssim & \| |\nabla| \tilde \eta_0 \|_{L^2} \cdot \| \Delta \eta_l \|_{L^2} \notag \\
  \lesssim & 2^{\frac 12 l}
\end{align*}
which is also summable for $l<0$.

Therefore \eqref{x2_4} follows and \eqref{x2_3} is proved.

For \eqref{x2_3a}, we note that by \eqref{x2_2},
\begin{align*}
 & \eta_k(r,z) e_{\theta} \notag \\
 =\; & \tilde \eta_0 (2^k r,2^k z) \frac 1 {2^k r} \begin{pmatrix} -2^k x_2,\; 2^k x_1,\; 0 \end{pmatrix} \notag \\
 =\; & {\eta_{\text{vec} }} (2^k x),
\end{align*}
where
\begin{align*}
 { \eta_{\text{vec} }}(x)=\eta_{\text{vec} }(x_1,x_2,z) = \tilde \eta_0 (r,z) \frac 1 r
 \begin{pmatrix}
 -x_2,\;x_1,\; 0
 \end{pmatrix}.
\end{align*}

Since by definition $\tilde \eta_0$ is supported on $r\sim 1$, $\eta_{\text{vec} }$ is a smooth function.
Therefore \eqref{x2_3a} can be proved in the same way as \eqref{x2_3} (note that only scaling is used in the argument).

\begin{lem} \label{x3}
For any smooth axisymmetric $f$ on $\mathbb R^3$ ( i.e. $f=f(x)=f(r,z)$, $x=(x_1,x_2,z)$, $r=\sqrt{x_1^2+x_2^2}$),
we have
\begin{align}
 & \Bigl( \partial_z (\Delta-\frac 1 {r^2} )^{-1} \partial_z f \Bigr)(r,z) \notag \\
 = & \Bigl( \partial_{zz} (\Delta- \frac 1 {r^2})^{-1} f \Bigr)(r,z) \notag \\
 = &\, C \cdot \int_{\mathbb R^5} K(x-y) \cdot \frac{|x^{\prime}|} {|y^{\prime}|} f(y)dy, \quad
 x=(r,0,0,0,z),
 \label{x3_1}
 \end{align}
where $C>0$ is an absolute constant, $x^{\prime}=(x_1,\cdots,x_4)$, $y^{\prime}=(y_1,\cdots,y_4)$, and
\begin{align} \label{x3_1a}
K(\tilde x)= \frac{|\tilde x^{\prime}|^2-4\tilde x_5^2} {|\tilde x|^7} + \frac 1{5C} \delta(\tilde x).
\end{align}
Here $\delta(\cdot)$ is the Dirac delta function on $\mathbb R^5$. On the RHS of \eqref{x3_1} we naturally regard $f$ as an axisymmetric function on $\mathbb R^5$ with
the identification $f(y)=f(|y^{\prime}|,y_5)$. Also by axisymmetry the RHS of \eqref{x3_1} depends
only on $(|x^{\prime}|,x_5)$ so that we can actually choose any $x$ with $|x^{\prime}|=r$, $x_5=z$.

Similarly we have
\begin{align}
 & \Bigl( \partial_r (\Delta-\frac 1 {r^2})^{-1} \partial_z f \Bigr)(r,z) \notag \\
 = & C \cdot \int_{\mathbb R^5} \Bigl( \frac 1 {|x-y|^5}
 -5 r \cdot \frac {r-y_1} {|x-y|^7} \Bigr) \cdot \frac {x_5-y_5} {|y^{\prime}|} f(y) dy,
 \quad x=(r,0,0,0,z).  \label{x3_1a_01}
\end{align}

\end{lem}

\begin{proof}[Proof of Lemma \ref{x3}]
We first compute the kernel $(\Delta-\frac 1 {r^2})^{-1}$. To derive the explicit representation, consider
the equation
\begin{align*}
(\Delta-\frac 1 {r^2}) u = f
\end{align*}
or
\begin{align*}
(\partial_{rr}+\frac 1 r \partial_r -\frac 1 {r^2} + \partial_{zz}) u =f.
\end{align*}
Set $u=rv$. Then
\begin{align*}
\Delta(rv) &= r \Delta v + 2 (\partial_r r) \partial_r v + v \Delta(r) \notag \\
&=r \Delta v + 2 \partial_r v + \frac 1 r v \notag \\
&=r \Delta_5 v + \frac 1 r v,
\end{align*}
where
\begin{align*}
\Delta_5= \partial_{rr} + \frac 3 r \partial_r + \partial_{zz}
\end{align*}
is the five dimensional Laplacian in cylindrical coordinates.

Therefore
\begin{align*}
(\Delta-\frac 1 {r^2}) u & = (\Delta-\frac 1{r^2}) (rv) = r \Delta_5 v
\end{align*}
and we only need to solve
\begin{align*}
\Delta_5 v= \frac 1 r f.
\end{align*}
Inverting the Laplacian then gives
\begin{align*}
v= C \cdot \int_{\mathbb R^5} \frac {-1} {|x-y|^3} \cdot \frac 1 {|y^{\prime}|} f(y)dy
\end{align*}
and obviously
\begin{align*}
u(x) = C \cdot \int_{\mathbb R^5} \frac {-1} {|x-y|^3} \cdot \frac{|x^{\prime}|}{|y^{\prime}|} f(y)dy.
\end{align*}
Note that in the above expression we have pure convolution in the $y_5$ variable. Therefore the first equality
in \eqref{x3_1} hold. Differentiating in $z=x_5$ variable twice then gives \eqref{x3_1}. In a similar way we can
derive \eqref{x3_1a_01}.
\end{proof}

\begin{lem} \label{x4}
Let $\phi=\phi(r,z)=(\phi^r(r,z),\phi^z(r,z))$ be a bi-Lipschitz map on $[0,\infty)\times \mathbb R$ such that the following hold:
\begin{itemize}
\item for any $r\ge 0$, $z\in \mathbb R$,
\begin{align} \label{x4_0}
\phi^r(0,z)=0 \quad \text{and } \phi^z(r,0)=0;
\end{align}
\item for some integer $n_0\ge 1$,
\begin{align} \label{x4_1}
\| D\phi \|_{\infty} + \| D \tilde \phi\|_{\infty} \le 2^{n_0},
\end{align}
where $\tilde \phi$ is the inverse map of $\phi$.
\end{itemize}
Define
\begin{align}
w(r,z)=(T\omega_0)(r,z)= \frac{\omega_0(\phi(r,z))}{\phi^r(r,z)} r, \label{x4_2aa}
\end{align}
where $\omega_0=g_A$ and $g_A$ is defined in \eqref{x2_1}.

Then
\begin{align}
\left\| \partial_{zz} (\Delta-\frac 1 {r^2})^{-1} \omega \right\|_{\infty} \le C \cdot  {2^{2n_0}}
\cdot \frac {\sqrt{\log A}} {{ A}}, \label{x4_2}
\end{align}
where $C>0$ is an absolute constant.

If in addition the map $\phi$ preserves the measure $rdr dz$, i.e.
\begin{align*}
 \frac {\phi^r} r\det\Bigl( \frac {\partial(\phi^r, \phi^z)} {\partial(r,z)} \Bigr) \equiv 1;
\end{align*}
and $\phi$ is odd in the $z$-variable, i.e.
\begin{align*}
 &\phi^r(r,-z)= \phi^r(r,z), \quad \forall\, r\ge0,\, z \in \mathbb R;\notag \\
 &\phi^z(r,-z)=-\phi^z(r,z),\quad \forall\, r\ge 0, \, z \in \mathbb R;
\end{align*}

then for some absolute constant $C_1>0$,
\begin{align}
  & -\Bigl( \partial_r (\Delta-\frac 1 {r^2})^{-1} \partial_z w \Bigr)(0,0) \notag \\
  \ge & C_1 \cdot \sqrt{\log A} \cdot 2^{-8n_0}. \label{x4_2_new}
\end{align}

\end{lem}
\begin{rem}
In the proof of \eqref{x4_2} below, we do not use the odd symmetry in the $z$-variable of the function $g_A$.
By itself the axisymmetry is enough to control the singular operator $\partial_{zz} (\Delta-\frac 1{r^2})^{-1}$.
\end{rem}
\begin{rem}
For our application later, $\omega$ actually correspond to $\omega^{\theta}$, and the expression
$-(\Delta-\frac 1 {r^2})^{-1} \partial_z \omega$ in \eqref{x4_2_new} will correspond to the radial
velocity $u^r$, see \eqref{x5_6}.
\end{rem}

\begin{proof}[Proof of Lemma \ref{x4}]
We shall adopt the same notations as in Lemma \ref{x3}. By \eqref{x4_0} and \eqref{x4_1},
it is not difficult to check that
if $r\sim 2^{-m}$ and $|z|\sim 2^{-m}$ for some integer $m$, then
\begin{align*}
2^{-m-n_0} \lesssim \phi^r(r,z) \lesssim 2^{-m+n_0}, \\
2^{-m-n_0} \lesssim |\phi^z(r,z)| \lesssim 2^{-m+n_0}.
\end{align*}
Similarly if $\phi^r(r,z) \sim 2^{-m}$, $|\phi^z(r,z)| \sim 2^{-m}$, then
\begin{align*}
2^{-m-n_0} \lesssim r \lesssim 2^{-m+n_0}, \\
2^{-m-n_0} \lesssim |z| \lesssim 2^{-m+n_0}.
\end{align*}
These facts will be used below.

By \eqref{x4_2aa} and \eqref{x2_1}, we have
\begin{align*}
&\partial_{zz} (\Delta-\frac 1 {r^2})^{-1} \omega  \notag \\
= & \frac  {\sqrt{\log A}}  {{ A}}
\sum_{A\le k\le 2A} \Bigl( \partial_{zz} (\Delta-\frac 1 {r^2})^{-1}  \Bigr)(T \eta_k).
\end{align*}

We shall estimate each piece $\partial_{zz}(\Delta-\frac 1 {r^2})^{-1}(T \eta_k)$.

By \eqref{x3_1} and \eqref{x3_1a}, we write
\begin{align}
&|(\partial_{zz}(\Delta-\frac 1 {r^2})^{-1}(T \eta_k))(x) | \notag \\
\lesssim & \left| \int_{\mathbb R^5}  K_1(x-y) {|x^{\prime}|}
\frac{\eta_k(\phi(|y^{\prime}|,y_5))} {\phi^r(|y^{\prime}|,y_5)} dy\right| \label{x4_6} \\
& \quad + r \frac{\eta_k(\phi(r,z))}{\phi^r(r,z)}, \label{x4_6a}
\end{align}

where
\begin{align} \notag
K_1(\tilde x)= \frac{|\tilde x^{\prime}|^2-4\tilde x_5^2} {|\tilde x|^7}.
\end{align}

The contribution of \eqref{x4_6a} is of no problem for us. Indeed on the support of $\eta_k (\phi(r,z))$, we have
\begin{align*}
 \phi^r(r,z) \sim 2^{-k}, \quad \text{and } |\phi^z(r,z)| \sim 2^{-k}.
\end{align*}
Therefore $2^{-k-n_0} \lesssim r \lesssim 2^{-k+n_0}$ and
\begin{align*}
 \frac r {\phi^r(r,z)} \lesssim 2^{n_0}
\end{align*}
Since the supports of the $\eta_k$ functions are mutually disjoint, it follows that
\begin{align*}
 \sum_{A\le k \le 2A}r \frac{\eta_k(\phi(r,z))}{\phi^r(r,z)} \lesssim 2^{n_0}
\end{align*}

Hence we only need to consider the contribution of \eqref{x4_6} to \eqref{x4_2}. By the same consideration as before we have
in \eqref{x4_6} the integration variable $y$ is localized to the regime:
\begin{align*}
 & 2^{-k-n_0} \lesssim |y^{\prime}| \lesssim 2^{-k+n_0}, \notag \\
 & 2^{-k-n_0} \lesssim |y_5| \lesssim 2^{-k+n_0}, \notag \\
 & |\phi(|y^{\prime}|, y_5)|\sim 2^{-k}, \quad\text{and } \phi^r(|y^{\prime}|,y_5) \sim 2^{-k}.
\end{align*}

Obviously if $x=0$, then due to the factor $|x^{\prime}|=0$, the integral \eqref{x4_6} also vanishes. Therefore we only need
to consider the case $x \ne 0$.

Assume $|x| \sim 2^{-l}$. We discuss two cases.

\texttt{Case 1}: $2^k\gg 2^{l+n_0} $. In this case $|x| \gg |y|$ (recall $|y| \lesssim 2^{-k+n_0}$) and $|x-y| \sim |x|$. Therefore
\begin{align*}
 \eqref{x4_6} & \lesssim \| K_1(y) \|_{L^{\infty} (|y| \sim 2^{-l})} \cdot 2^{-l}
 \cdot \frac{ \text{Leb} \{y \in \mathbb R^5:\; 2^{-k-n_0} \lesssim |y| \lesssim 2^{-k+n_0}\} } {2^{-k}} \notag \\
 & \lesssim 2^{5l} \cdot 2^{-l} \cdot \frac{(2^{-k+n_0})^5} {2^{-k}} \notag \\
 & \lesssim 2^{-4k+4l+5n_0}.
\end{align*}

\texttt{Case 2}: $2^k\lesssim 2^{l+n_0}$. In this case we note that the kernel $K_1$ in \eqref{x4_6} corresponds to a Riesz-type
operator on $\mathbb R^5$. By using the interpolation inequality
\begin{align*}
 \| \mathcal R_{ij} f \|_{L^{\infty} (\mathbb R^5)} \lesssim \| f \|_{L^5(\mathbb R^5)}^{\frac 12}
 \cdot \| \nabla f \|_{L^{\infty} (\mathbb R^5)}^{\frac 12},
\end{align*}
we can bound \eqref{x4_6} as
\begin{align*}
 \eqref{x4_6} & \lesssim 2^{-l}
 \cdot \left | \int_{\mathbb R^5} K_1(x-y) \cdot \frac{\eta_k ( \phi(|y^{\prime}|, y_5))} {\phi^r(|y^{\prime}|,y_5 )} dy \right | \notag \\
 & \lesssim 2^{-l} \cdot \| \frac 1 {2^{-k}} \|^{\frac 12}_{L^5(y\in \mathbb R^5:\, 2^{-k-n_0} \lesssim |y| \lesssim 2^{-k+n_0})} \notag \\
 & \qquad \cdot \Bigl( \| \partial_r ( \frac{\eta_k(\phi(r,z))} {\phi^r(r,z)}) \|_{\infty}+
 \| \partial_z ( \frac{\eta_k(\phi(r,z))} {\phi^r(r,z)}) \|_{\infty} \Bigr)^{\frac 12} \notag \\
 & \lesssim 2^{-l} \cdot 2^{\frac k2} \cdot (2^{-k+n_0})^{\frac 12} \cdot (2^{n_0} \cdot 2^{2k})^{\frac 12} \notag \\
 & \lesssim 2^{n_0-l+k}.
\end{align*}

Collecting the estimates, we have
\begin{align*}
 \sum_{A\le k \le 2A} \eqref{x4_6}  & \lesssim \sum_{2^k\gg 2^{l+n_0} } 2^{-4k+4l+5n_0} + \sum_{2^k\lesssim 2^{l+n_0} } 2^{n_0-l+k} \notag \\
 & \lesssim 2^{2n_0}.
\end{align*}

Hence the estimate \eqref{x4_2} follows.

By \eqref{x3_1a_01} and parity of $\phi$ and $\eta_k$, we have
\begin{align}
  & -\Bigl( \partial_r (\Delta-\frac 1 {r^2})^{-1} \partial_z \omega) (0,0) \notag \\
 = & C \int_{\mathbb R^5} \frac {y_5} {|y|^5} \cdot \frac 1 {|y^{\prime}|} \omega(y) dy \notag \\
 = & C \int_{\mathbb R^5 } \frac {y_5} {|y|^5} \cdot \frac{\omega_0(\phi(|y^{\prime}|, y_5))} {\phi^r(|y^{\prime}|, y_5)} dy \notag \\
 = &  C \frac {\sqrt{\log A}} { {A}} \cdot \sum_{A \le k \le 2 A}
 \int_{\mathbb R^5} \frac {y_5} {|y|^5} \cdot \frac{\eta_k(\phi(|y^{\prime}|, y_5))} {\phi^r(|y^{\prime}|, y_5)} dy \notag \\
 = &  {2C} \frac {\sqrt{\log A}} { {A}} \cdot \sum_{A \le k \le 2 A}
 \int_{y\in \mathbb R^5:\, y_5>0}  \frac {y_5} {|y|^5} \cdot \frac{\eta_k(\phi(|y^{\prime}|, y_5))} {\phi^r(|y^{\prime}|, y_5)} dy \notag \\
 \gtrsim & \frac {\sqrt{\log A}} { {A}} \sum_{A \le k \le 2 A}
 \frac {2^{-k-n_0}} {(2^{-k+n_0})^5} \cdot \frac 1 {2^{-k}} \cdot \int \eta_k(\phi(r,z)) r^3 dr dz. \label{x4_8}
\end{align}

Since the map $\phi$ preserves the measure $rdrdz$, the inverse map $\tilde \phi$ also preserves the same measure. By the change of
variable $(r,z) \to \tilde \phi(r,z)$, we have
\begin{align*}
 & \int \eta_k(\phi(r,z)) r^3 dr dz \notag \\
= & \int \eta_k(r,z) \tilde \phi^r(r,z)^2 r drdz \notag \\
\gtrsim & (2^{-k-n_0})^2 \| \eta_k \|_{L^1(\mathbb R^3)} \notag \\
\gtrsim & 2^{-5k-2n_0}.
\end{align*}

Plugging this estimate into \eqref{x4_8}, we obtain
\begin{align*}
 & -\Bigl( \partial_r (\Delta-\frac 1 {r^2})^{-1} \partial_z \omega) (0,0) \notag \\
 \gtrsim & \frac {\sqrt{\log A}} { {A}} \sum_{A \le k \le 2A} 2^{-8n_0} \notag \\
 \gtrsim &  {\sqrt{\log A}} \cdot  2^{-8n_0}.
\end{align*}

Hence \eqref{x4_2_new} is proved.
\end{proof}

We now prove the existence of large deformation in the 3D axisymmetric case. To fix the notation, consider the
3D axisymmetric Euler equation without swirl in simplified form
\begin{align*}
 \begin{cases}
  \partial_t \left( \frac{\omega^{\theta} }r \right) + (u\cdot \nabla) \left( \frac{\omega^{\theta} } r\right) =0, \quad (r,z) \in(0,\infty)\times \mathbb R, \,
  0<t\le 1, \\
  \omega^{\theta} \Bigr|_{t=0}=g_A,
 \end{cases}
\end{align*}
where $u=u^r e_r +u^z e_z$ and $g_A$ is defined in \eqref{x2_1}. Note that here $\omega^{\theta}$ is a scalar-valued function which is related to
$u$ by the relation $\operatorname{curl}(u) = \omega^{\theta} e_{\theta}$.

Let $\phi=\phi(t,r,z)$ be the forward characteristic line as defined in \eqref{x-1_00} and $\tilde \phi=\tilde \phi(t,r,z)$ be the inverse map.
Then

\begin{prop} \label{x5}
 For all $A$ sufficiently large, we have
 \begin{align} \notag
  \max_{0\le t \le \frac 1 {\log \log A}}
  ( \| D\phi(t)\|_{\infty} + \| D \tilde \phi (t)\|_{\infty} ) \ge   \log\log A.
 \end{align}
\end{prop}

\begin{proof}[Proof of Proposition \ref{x5}]
 We argue by contradiction. Assume
 \begin{align}
  \max_{0\le t \le \frac 1 {B}} \Bigl( \| D\phi(t)\|_{\infty} + \| D \tilde \phi(t)\|_{\infty} \Bigr)
 \le B, \label{x5_2}
 \end{align}
where for simplicity of the notation we denote $B=\log\log A$.

Denote $D(t)=D(t,r,z)=(D\phi)(t,r,z)$. By \eqref{x-1_00}, we have
\begin{align}
 \partial_t D(t) & = \begin{pmatrix} \partial_r u^r  \quad \partial_z u^r \\
                      \partial_r u^z \quad \partial_z u^z
                      \end{pmatrix}
D(t)  \notag \\
& = \begin{pmatrix} \partial_r u^r  \quad 0 \\
                      0 \quad \partial_z u^z
                      \end{pmatrix}
D(t) +P(t) D(t), \label{x5_3}
\end{align}
where we denote
\begin{align*}
  P(t)=P(t,r,z) = \begin{pmatrix}
  0\quad (\partial_z u^r)(t,\phi(t,r,z)) \\
  (\partial_r u^z)(t,\phi(t,r,z)) \quad 0
 \end{pmatrix}.
\end{align*}

Now since
\begin{align*}
 \frac {\omega^{\theta}(t, \phi(t,r,z))} {\phi^r(t,r,z)} = \frac{\omega^{\theta}_0(t,r,z)} r,
\end{align*}
we have
\begin{align*}
 \omega^{\theta}(t,r,z) = \frac{\omega^{\theta}_0(\tilde \phi(t,r,z))} {\tilde \phi^r(t,r,z)} r.
\end{align*}

By \eqref{x-1_1}, we have
\begin{align*}
r &= \phi^r(t, \tilde \phi^r(t,r,z), \tilde \phi^z(t,r,z)) - \phi^r(t,0, \tilde \phi^z(t,r,z)) \notag \\
  & \le \| \partial_r \phi^r(t,\cdot)\|_{\infty} \tilde \phi^r(t,r,z) \notag \\
  & \le B \tilde \phi^r(t,r,z).
  \end{align*}
  Therefore
  \begin{align} \label{x5_4na}
  \max_{0\le t \le \frac 1 B} \|\omega (t,\cdot )\|_{\infty} \le \frac {\sqrt{\log A}} { {A}} B.
  \end{align}

Since
\begin{align}
 & \omega^{\theta}= \partial_r u^z - \partial_z u^r, \label{x5_4} \\
 & \frac 1 r \partial_r (ru^r) +\partial_z u^z=0, \label{x5_5}
\end{align}
it is not difficult to check that
\begin{align}
 u^r = - (\Delta-\frac 1 {r^2})^{-1} \partial_z \omega^{\theta}, \label{x5_6}
\end{align}
and
\begin{align*}
 \partial_z u^r = - \partial_{zz} (\Delta-\frac 1 {r^2})^{-1} \omega^{\theta}.
\end{align*}

By \eqref{x4_2} and \eqref{x5_2}, we get
\begin{align*}
 \| (\partial_z u^r)(t) \|_{\infty} \lesssim {B^2} \frac {\sqrt{\log A}} { {A}}.
\end{align*}

By \eqref{x5_4} and \eqref{x5_4na}, we also get
\begin{align*}
 \| \partial_r u^z (t)\|_{\infty} & \lesssim \| \omega^{\theta}(t) \|_{\infty} + \| (\partial_z u^r)(t) \|_{\infty} \notag \\
 & \lesssim {B^2} \frac {\sqrt{\log A}} { {A}}.
\end{align*}

Hence
\begin{align*}
  \| P(t)\|_{\infty} \lesssim  {B^2} \frac {\sqrt{\log A}} { {A}}.
\end{align*}

Denote
\begin{align*}
 \lambda (t,r,z) = - (\partial_r u^r)(t,\phi(t,r,z)).
\end{align*}

By \eqref{x5_6} and \eqref{x4_2_new}, we have
\begin{align} \label{x5_9}
 -\lambda(t,0,0) \gtrsim  {\sqrt{\log A}} \cdot B^{-8},
 \quad \forall\, 0\le t\le \frac 1 {B}.
\end{align}

By \eqref{x5_5}, we can write
\begin{align*}
 (\partial_z u^z)(t,\phi(t,r,z)) & = - (\partial_r u^r)(t,\phi) - \frac 1 {\phi^r} u^r(t,\phi) \notag \\
 & = \lambda(t,r,z) - \frac 1 {\phi^r(t,r,z)} u^r(t,\phi(t,r,z)).
\end{align*}

Using the above computation and integrating \eqref{x5_3} in time, we get
\begin{align*}
 D(t) & = \begin{pmatrix}
           e^{-\int_0^t \lambda } \quad &0 \\
           0 \quad &e^{\int_0^t \lambda - \int_0^t \frac 1 {\phi^r} u^r (s,\phi) ds}
          \end{pmatrix}
 \notag \\
 & \quad + \int_0^t
\begin{pmatrix}
           e^{-\int_{\tau}^t \lambda } \quad &0 \\
           0 \quad &e^{\int_{\tau}^t \lambda - \int_{\tau}^t \frac 1 {\phi^r} u^r (s,\phi) ds}
          \end{pmatrix}
          P(\tau) D(\tau) d\tau.
\end{align*}

By \eqref{x-1_2}, we have
\begin{align*}
 \max_{0\le t \le \frac 1 {B}}
 e^{|\int_0^t \frac {u^r(s,\phi)} {\phi^r} |}
 & \lesssim B.
\end{align*}

Therefore we get
\begin{align*}
 \frac 1 {B} e^{|\int_0^t \lambda|}
\lesssim B + \max_{0\le \tau \le t}
\left( e^{2|\int_0^{\tau} \lambda |} \right) \cdot {B^{10}}\cdot\frac {\sqrt{\log A}} { {A}}
\end{align*}
or
\begin{align*}
  e^{|\int_0^t \lambda|}
\lesssim B^2 + \max_{0\le \tau \le t}
\left( e^{2|\int_0^{\tau} \lambda |} \right) \cdot {B^{11}} \cdot \frac {\sqrt{\log A}} { {A}}.
\end{align*}

Since $B^{12} \ll { A}/{\sqrt{\log A}}$, a standard continuity argument then gives
\begin{align*}
 e^{|\int_0^t \lambda(s,r,z) ds |} \lesssim B^2, \quad \forall\, 0\le t \le \frac 1 {B},
 \, r\ge 0, \, z \in \mathbb R.
\end{align*}

But this obviously contradicts \eqref{x5_9}.
\end{proof}

\begin{lem}[Vanishing near $r=0$] \label{lemx5}
Let $U=U^r e_r +U^{z} e_z$ be a (possibly time-dependent) given smooth axi-symmetric without swirl velocity field such that
$\nabla \cdot U=0$ and
\begin{align} \label{lemx5_1}
 \max_{0\le t \le 1} (\|D^2 U(t)\|_{4}+ \| DU (t)\|_{\infty} +\| U(t)\|_{\infty} ) \le C_1 <\infty.
\end{align}

Suppose $\omega$ is a smooth solution to the \emph{linear} system
\begin{align}
 \begin{cases}
  \partial_t \left( \frac {\omega }r \right) + (U\cdot \nabla ) \left( \frac {\omega} r  \right) =0,
  \quad \, x=(x_1,x_2,z),\, r=\sqrt{x_1^2+x_2^2}, \\
  \omega \Bigr|_{t=0} =\omega_0.
 \end{cases}
 \end{align}
Here the initial data $\omega_0 \in C_c^{\infty} (\mathbb R^3)$  and has the form
 \begin{align*}
  \omega_0 = \omega_0^{\theta} e_{\theta},
 \end{align*}
where $\omega_0^{\theta} = \omega_0^{\theta} (r,z)$ is scalar-valued and
\begin{align*}
 e_{\theta}= \frac 1 r \begin{pmatrix} -x_2,\,x_1,\,0 \end{pmatrix}.
\end{align*}

Assume that $\omega_0$ vanishes near $r=0$, i.e. there exists a constant $r_0>0$ such that
\begin{align*}
 \operatorname{supp}(\omega_0(r,z)) \subset\{(r,z):\, r_0 \le r \le  \frac 1 {r_0} \}.
\end{align*}

Then there exists a constant $R_0=R_0(r_0,C_1)>0$ such that
\begin{align} \label{lemx5_2}
 \operatorname{supp} (\omega(t,r,z) )  \subset\{ (r,z):\, {R_0} \le r \le \frac 1 {R_0} \}, \quad \forall\, 0\le t\le 1.
\end{align}

Furthermore we have the estimate
\begin{align} \label{lemx5_3}
 \max_{0\le t \le 1} \| \omega (t) \|_{H^2} \le C_2,
\end{align}
where the constant $C_2$ only depends on $(\|\omega_0\|_{H^2},r_0, C_1)$.

\end{lem}

\begin{proof}[Proof of Lemma \ref{lemx5}]
The property \eqref{lemx5_2} follows easily from finite-speed propagation of the transport equation. For example by
\eqref{x-1_5} (with $u^r$ replaced by $U^r$) and \eqref{lemx5_1}, we have for some $C_3=C_3(C_1)>0$
\begin{align*}
 \frac 1 {C_3} \le \phi^r(t,r,z)/r  \le C_3, \quad \forall\, 0\le t \le 1.
\end{align*}
This shows that the boundary of the support is supported away from $r=0$. Clearly \eqref{lemx5_2} follows.

Denote $g= \frac {\omega} r$ and $g_0= \frac{\omega_0} r$.
Since $\omega$ is supported away from $r=0$,  obviously we have
\begin{align*}
 \| g_0 \|_{H^2} \lesssim_{r_0} \|\omega_0\|_{H^2}.
\end{align*}

Since
\begin{align*}
 \partial_t g + (U\cdot \nabla) g=0,
\end{align*}
a simple $H^2$ energy estimate then gives
\begin{align*}
 \partial_t( \|g \|_{H^2}^2)  & \lesssim \|D^2 U\|_4 \| \nabla g \|_4 \| g\|_{H^2} + \|DU\|_{\infty} \| g\|_{H^2}^2 \notag \\
 & \lesssim_{C_1} \| g \|_{H^2}^2.
\end{align*}
Therefore
\begin{align*}
 \max_{0\le t \le 1} \|g(t)\|_{H^2} \lesssim_{C_1,r_0} \|\omega_0 \|_{H^2}.
\end{align*}

Since $\omega=rg$ and $\omega$ is supported on $r\sim_{r_0,C_1} 1$, obviously \eqref{lemx5_3} follows.

\end{proof}

\begin{prop} \label{x6}
 Suppose $\omega$ is a smooth solution to the axisymmetric (without swirl) Euler equation in the form
\begin{align} \label{x6_00}
 \begin{cases}
  \partial_t \left( \frac {\omega }r \right) + (u\cdot \nabla ) \left( \frac {\omega} r  \right) =0,
  \quad 0<t\le 1,\, x=(x_1,x_2,z),\, r=\sqrt{x_1^2+x_2^2}, \\
  u=-\Delta^{-1}\nabla \times \omega, \\
  \omega \Bigr|_{t=0} =\omega_0
 \end{cases}
 \end{align}
and satisfy the following conditions:
\begin{itemize}
 \item $\omega_0 \in C_c^{\infty} (\mathbb R^3)$  and has the form
 \begin{align} \label{x6_00a}
  \omega_0 = \omega_0^{\theta} e_{\theta},
 \end{align}
where $\omega_0^{\theta} = \omega_0^{\theta} (r,z)$ is scalar-valued and
$\omega_0$ vanishes near $r=0$, i.e. there exists a constant $r_0>0$ such that
\begin{align*}
 \operatorname{supp}(\omega_0(r,z)) \subset\{(r,z):\, r>r_0 \}.
\end{align*}

\item Let $\phi=(\phi^r,\phi^z)$ be the characteristic lines defined in \eqref{x-1_00} and $\tilde \phi$ be the inverse.
For some $0<t_0\le 1$, $\tilde r_*\ge 0$, $\tilde z_*\in \mathbb R$, we have
\begin{align}
 & \left\| (D\tilde \phi) (t_0, \tilde r_*, \tilde z_*) \right\|_{\infty} =
 \sup_{r\ge 0,z\in \mathbb R}\left\| (D\tilde \phi) (t_0, r, z) \right\|_{\infty} =:M \gg 1. \label{x6_2}
\end{align}
Here in \eqref{x6_2}, $\| \cdot \|_{\infty}$ denotes the matrix max norm defined by $\| A\|_{\infty} = \max (|a_{ij}|)$ ($A=(a_{ij})$).
The notation $M\gg 1$ means that $M$ is sufficiently large than an absolute constant.
\end{itemize}

Then we can find a smooth solution also solving the axisymmetric (without swirl) Euler equation

\begin{align*}
  \begin{cases}
   \partial_t \left( \frac {\tilde \omega } r\right) + (\tilde u\cdot \nabla ) \left( \frac {\tilde \omega} r  \right) =0,
   \quad 0<t\le 1, \\
    \tilde u = -\Delta^{-1} \nabla \times \tilde \omega, \\
   \tilde \omega \Bigr|_{t=0} =\tilde \omega_0
  \end{cases}
 \end{align*}

such that the following hold:

\begin{enumerate}
 \item $\tilde \omega_0  \in C_c^{\infty} (\mathbb R^3)$ and has the form
 \begin{align*}
  \tilde \omega_0 = \tilde \omega_0^{\theta} e_{\theta},
 \end{align*}
 where $\tilde \omega_0^{\theta} = \tilde \omega_0^{\theta} (r,z)$. The function
 $\tilde \omega_0$ also vanishes near $r=0$, i.e. there exists a constant $\tilde r_0>0$ such that
\begin{align*}
 \operatorname{supp}(\tilde \omega_0(r,z)) \subset\{(r,z):\, r>\tilde r_0 \}.
\end{align*}
Furthermore
\begin{align}
  & \| \frac {\tilde \omega_0}r \|_{L^1(\mathbb R^3)} \le 2 \| \frac{\omega_0}r \|_{L^1(\mathbb R^3)},  \notag \\
  &\| \frac{\tilde \omega_0}r \|_{L^{\infty}(\mathbb R^3)} \le 2 \| \frac{\omega_0}r \|_{L^{\infty}(\mathbb R^3)}, \notag \\
  &\| \frac{\tilde \omega_0}r \|_{L^{3,1}(\mathbb R^3)} \le 2 \| \frac{\omega_0}r \|_{L^{3,1}(\mathbb R^3)}. \label{x6_5a}
 \end{align}

 \item $\tilde \omega_0$ is a small perturbation of $\omega_0$:
 \begin{align}
  & \| \tilde \omega_0\|_{L^1(\mathbb R^3)} \le 2 \| \omega_0 \|_{L^1(\mathbb R^3)},  \notag \\
  &\| \tilde \omega_0 \|_{L^{\infty}(\mathbb R^3)} \le 2 \| \omega_0 \|_{L^{\infty}(\mathbb R^3)}, \notag \\
  &\| \tilde \omega_0 \|_{\dot H^{\frac 32} (\mathbb R^3)} \le \| \omega_0 \|_{\dot H^{\frac 32}} +
  \frac {\tilde C}{M^{\frac 16} }. \label{x6_5}
 \end{align}
Here $\tilde C>0$ is an absolute constant.

\item For the same $t_0$ as in \eqref{x6_2}, we have
\begin{align} \label{x6_7}
 \| \tilde \omega (t_0 ) \|_{\dot H^{\frac 32}} > M^{\frac 16}.
\end{align}

\end{enumerate}

\end{prop}

\begin{proof}[Proof of Proposition \ref{x6}]
We begin with a general derivation. Let $W$ be a smooth solution to the
system
\begin{align} \label{x6_p1}
 \begin{cases}
  \partial_t ( \frac {W} r) + (U\cdot \nabla) ( \frac W r) =0, \\
  U=-\Delta^{-1}\nabla \times W, \\
  W \Bigr|_{t=0} =W_0=f e_{\theta}.
 \end{cases}
\end{align}

Here $f=f(r,z)$ is scalar-valued.
Define the corresponding forward characteristic lines $\Phi=(\Phi^r, \Phi^z)$ in the same way as
\eqref{x-1_00} and let $\tilde \Phi$ be the inverse map. Then we have
\begin{align*}
 W(t,x)=  W^{\theta}{(t,r,z)} e_{\theta},
\end{align*}
where $W^{\theta}$ is scalar-valued and
\begin{align*}
 \frac{W^{\theta}(t, \Phi(t,r,z))} { \Phi^r (t,r,z)} = \frac {f(r,z)} r.
\end{align*}
Therefore
\begin{align}
 W^{\theta}{(t,r,z)} = \frac{ f(\tilde \Phi(t,r,z))} {\tilde \Phi^r (t,r,z)}  r
\end{align}
and
\begin{align}
 W(t,x) = \frac {f(\tilde \Phi(t,r,z))} {\tilde \Phi^r (t,r,z)} r e_\theta, \quad x=(x_1,x_2,z),\, r=\sqrt{x_1^2+x_2^2}.
 \label{x6_p2}
\end{align}

Now we discuss two cases:

\texttt{Case 1}: $\| \omega(t_0,\cdot )\|_{\dot H^{\frac 32}}
>M^{\frac 16}$. In this case we just set $\tilde \omega=\omega$ and
no work is needed.

\texttt{Case 2}:
\begin{align} \label{x6_case2}
 \|\omega(t_0,\cdot)\|_{\dot H^{\frac 32}} \le M^{\frac 16}.
\end{align}

In this case in order not to confuse with some notations later on we shall denote
$\tilde \omega_0$ as $W_0$ and $\tilde \omega$ as $W$.
We take the initial data $W_0$  in \eqref{x6_p1} as
\begin{align} \label{x6_p4}
 W_0 = \omega_0 + k^{-\frac 32} G_0,
\end{align}
where $\omega_0$ is the same as in \eqref{x6_00}. The function $G_0$ has the form
\begin{align} \label{x6_p5}
 G_0(x)= g_0(r,z) e_{\theta}
\end{align}
where $g_0$ is scalar-valued. The detailed form of $g_0$ will be specified later in the course of the proof.

We shall take the parameter $k$ sufficiently large.  In the rest of this proof, to simplify
the presentation, we shall use the notation $X=O(k^{\alpha})$ ($\alpha$ is a real number) if the quantity
$X$ obeys the bound $X \le C_1 k^{\alpha}$ and the positive constant $C_1$ can depend on all other parameters except $k$.

Now we assume $G_0$ in \eqref{x6_p5} is a smooth compactly-supported function which obeys the following bounds:
\begin{align} \label{x6_p6}
 &\|G_0 \|_{L^1(\mathbb R^3)} + \|G_0\|_{L^{\infty}(\mathbb R^3)}
 +\|\frac {G_0} r \|_{L^1(\mathbb R^3)} +\| \frac{G_0} r \|_{L^{\infty} (\mathbb R^3)} = O(1), \notag \\
 & \|D G_0 \|_{L^1(\mathbb R^3)} + \| D G_0\|_{L^{\infty}(\mathbb R^3)} = O(k), \notag \\
 &  \|D^2 G_0 \|_{L^1(\mathbb R^3)} + \|D^2 G_0\|_{L^{\infty}(\mathbb R^3)} = O(k^2).
\end{align}

By \eqref{x6_p2}, \eqref{x6_p4}, \eqref{x6_p5} and \eqref{x6_00a}, we have
\begin{align}
W(t,x) &= \frac {\omega_0^{\theta}(\tilde \Phi(t,r,z))} {\tilde \Phi^r (t,r,z)} r e_\theta +
k^{-\frac 32 }
\frac {g_0(\tilde \Phi(t,r,z))} {\tilde \Phi^r (t,r,z)} r e_\theta \notag \\
& = \frac {\omega_0^{\theta}(\tilde \phi(t,r,z))} {\tilde \phi^r (t,r,z)} r e_\theta +
k^{-\frac 32 }
\frac {g_0(\tilde \phi(t,r,z))} {\tilde \phi^r (t,r,z)} r e_\theta \notag \\
& \quad + E_1+E_2, \label{x6_p7}
\end{align}
where  $\tilde \phi$ is the same as in \eqref{x6_2} and
\begin{align*}
 E_1 & = \frac {\omega_0^{\theta}(\tilde \Phi(t,r,z))} {\tilde \Phi^r (t,r,z)} r e_\theta-
 \frac {\omega_0^{\theta}(\tilde \phi(t,r,z))} {\tilde \phi^r (t,r,z)} r e_\theta, \notag \\
 E_2 & = k^{-\frac 32} \left( \frac {g_0(\tilde \Phi(t,r,z))} {\tilde \Phi^r (t,r,z)} r e_\theta
-\frac {g_0(\tilde \phi(t,r,z))} {\tilde \phi^r (t,r,z)} r e_\theta \right).
\end{align*}

We now show that the terms $E_1$, $E_2$ in \eqref{x6_p7} are negligible in the computation
of $H^{\frac 32}$-norm of $W$. More precisely, we shall show
for some $\alpha>0$,
\begin{align}
\max_{0\le t \le 1} \| E_1 (t)\|_{H^{\frac 32}(\mathbb R^3)} + \max_{0\le t \le 1 }\| E_2 (t)\|_{H^{\frac 32}(\mathbb R^3)} = O(k^{-\alpha}). \label{x6_p8}
\end{align}

To show \eqref{x6_p8}, let us introduce $\omega_2$, $W_1$, $W_2$ which solve the following
\emph{linear} systems:
\begin{align} \label{x6_p9}
 \begin{cases}
  \partial_t ( \frac{\omega_2} r ) + (U\cdot \nabla ) ( \frac {\omega_2} r )=0, \\
  \omega_2 \Bigr|_{t=0} =\omega_0;
 \end{cases}
\end{align}
\begin{align} \label{x6_p10}
 \begin{cases}
  \partial_t ( \frac{W_1} r ) + (u\cdot \nabla ) ( \frac {W_1} r )=0, \\
  W_1 \Bigr|_{t=0} =k^{-\frac 32} G_0;
 \end{cases}
\end{align}
\begin{align} \label{x6_p11}
 \begin{cases}
  \partial_t ( \frac{W_2} r ) + (U\cdot \nabla ) ( \frac {W_2} r )=0, \\
  W_2 \Bigr|_{t=0} =k^{-\frac 32} G_0;
 \end{cases}
\end{align}
Here the drift terms $u$, $U$ and the data $\omega_0$, $G_0$ are the same as in the nonlinear systems
\eqref{x6_00} and \eqref{x6_p1}.

It is not difficult to check that
\begin{align*}
 &W= \omega +W_1+E_1+E_2, \notag \\
 &E_1=\omega_2-\omega, \notag \\
 &E_2=W_2-W_1.
\end{align*}

Therefore we only need to run perturbation arguments between the nonlinear systems \eqref{x6_00}, \eqref{x6_p1} and
the linear systems \eqref{x6_p9}--\eqref{x6_p11}.

We first control the drift difference $u-U$.

By \eqref{x6_p4} and \eqref{x6_p6}, we have $\|W_0\|_{W^{1,p}} =O(1)$ for any $1\le p\le \infty$.

Thanks to the axisymmetry without swirl, we may write the system \eqref{x6_p1} as either
\begin{align}
 \partial_t W + (U\cdot \nabla) W = (W\cdot \nabla) U, \label{x6_p12}
\end{align}
or
\begin{align}
 \partial_t W + (U\cdot \nabla) W = \frac {U^r} r W. \label{x6_p13}
\end{align}

Take any $3<p<\infty$. A standard energy estimate on \eqref{x6_p12} in $W^{1,p}$ gives
\begin{align} \label{x6_p14}
 \frac d {dt}\Bigl( \|W(t)\|_{W^{1,p}}^p \Bigr)
 \lesssim ( \| Du(t)\|_{\infty} + \| W(t)\|_{\infty} )
 \|W(t)\|_{W^{1,p}}^p.
\end{align}

Note that by \eqref{x6_p6} and Lemma \ref{lem_ur_r},
\begin{align*}
\max_{0\le t\le 1} \| \frac{U^r(t)} r \|_{\infty} \lesssim \| \frac{W_0} r \|_{L^{3,1}} =O(1).
\end{align*}

Therefore by \eqref{x6_p13}, we have
\begin{align*}
 \max_{0\le t\le 1} (\|W(t)\|_2 + \| W(t)\|_{\infty}) = O(1).
\end{align*}

By the usual log-interpolation inequality, we have
\begin{align*}
 \| D U(t) \|_{\infty} & \lesssim \| W(t)\|_2 + \log ( 10 + \|W(t)\|^p_{W^{1,p}} ) \|W(t)\|_{\infty} \notag \\
 & \lesssim O(1) \cdot \log ( 10+ \| W(t) \|_{W^{1,p}}^p ).
\end{align*}

Plugging the last estimate into \eqref{x6_p14}, we obtain
\begin{align*}
 \frac d {dt} ( \|W(t) \|_{W^{1,p}}^p ) \lesssim O(1)
 \cdot \log ( 10 + \|W(t)\|^p_{W^{1,p}} )  \|W(t) \|_{W^{1,p}}^p.
\end{align*}

Integrating in time then gives
\begin{align}
 \max_{0\le t \le 1} \| W(t)\|_{W^{1,p}} =O(1),\quad\forall\, 3<p<\infty. \label{x6_p15}
\end{align}

By Sobolev embedding, we get
\begin{align} \label{x6_p16}
 \max_{0\le t\le 1} (\|D^2U(t)\|_p+\| DU(t)\|_{\infty}) =O(1), \quad \forall\,  3<p<\infty.
\end{align}

Similarly using \eqref{x6_p6} we can derive
\begin{align} \label{x6_p16a}
 \max_{0\le t \le 1} \| W(t)\|_{H^2} =O(k^{\frac 12}).
\end{align}

Note that the system \eqref{x6_00} is independent of the parameter $k$, therefore we have
\begin{align}
 \max_{0\le t \le 1} \| u(t)\|_{W^{20,p}} = O(1), \quad \forall\, 2\le p<\infty. \label{x6_p17}
\end{align}

Now to control the difference, we recall
\begin{align*}
 \begin{cases}
  \partial_t \omega + (u\cdot \nabla) \omega = (\omega \cdot \nabla) u, \\
  \partial_t W + (U\cdot \nabla )W= (W\cdot \nabla) U, \\
  (W-\omega) \Bigr|_{t=0} = k^{-\frac 32} G_0.
 \end{cases}
\end{align*}

Obviously
\begin{align*}
 &\partial_t (W-\omega) + (U\cdot \nabla) (W-\omega) + \Bigl( (U-u) \cdot \nabla \Bigr) \omega \notag \\
 & \qquad = (W\cdot \nabla) (U-u) + \Bigl( (W-\omega) \cdot \nabla \Bigr) u.
\end{align*}

By \eqref{x6_p15}, \eqref{x6_p16}, \eqref{x6_p17} and Sobolev embedding, we then obtain
\begin{align*}
 \partial_t \Bigl( \| W-\omega \|_2^2 \Bigr) & \lesssim \| U-u \|_6 \cdot \|W-\omega\|_2 \cdot \| \nabla \omega\|_3 \notag \\
 & \qquad + \|W \|_{\infty} \cdot \| D(U-u)\|_2 \cdot \|W-\omega\|_2 + \| D u\|_{\infty} \cdot \|W-\omega\|_2^2 \notag \\
 & \lesssim O(1)\cdot \|W-\omega\|_2^2.
\end{align*}

Therefore
\begin{align*}
 \max_{0\le t\le 1} \| W(t)-\omega(t)\|_2 =O(k^{-\frac 32}).
\end{align*}

In a similar way, we can derive
\begin{align}
 & \max_{0\le t \le 1} \| W(t)-\omega(t)\|_p = O(k^{-\frac 32}), \quad \forall\, 1<p<\infty, \notag \\
 & \max_{0\le t \le 1} \Bigl( \|U(t)-u(t)\|_p + \| \nabla (U(t)-u(t)) \|_p \Bigr)
 = O(k^{-\frac 32}), \quad \forall\, 2\le p<\infty. \label{x6_p18}
\end{align}

We are now ready to control $E_1= \omega_2 - \omega$. By \eqref{x6_p9}, \eqref{x6_p16} and Lemma \ref{lemx5}, we have
\begin{align*}
\max_{0\le t \le 1} \|\omega_2 (t)\|_{H^2} =O(1).
\end{align*}
By \eqref{x6_p17}, we get
\begin{align*}
\max_{0\le t \le 1} \|\omega_2 (t)- \omega(t)\|_{H^2} =O(1).
\end{align*}
On the other hand, using \eqref{x6_p18}, it is not difficult to check  that
\begin{align*}
 \max_{0\le t \le 1} \| \omega_2 (t) -\omega(t)\|_2 =O(k^{-\frac 32}).
\end{align*}
Interpolating the above two bounds then gives
\begin{align*}
 \max_{0\le t \le 1} \|\omega_2 (t)- \omega(t)\|_{H^{\frac 32}} =O(k^{-\frac3 8}).
\end{align*}
Therefore $E_1$ is OK for us.

To control $E_2$, we note that by \eqref{x6_p10}--\eqref{x6_p11}, we have
\begin{align}
 &\partial_t W_1 + (u\cdot \nabla ) W_1= (W_1\cdot \nabla )u, \label{x6_p19a} \\
 &\partial_t W_2 + (U\cdot \nabla ) W_2= (W_2\cdot \nabla )U, \label{x6_p19b} \\
 & \partial_t (W_1-W_2) + ((u-U)\cdot \nabla) W_1 + (U\cdot \nabla )(W_1-W_2) \notag \\
 & \qquad =( (W_1-W_2) \cdot \nabla ) u + (W_2 \cdot \nabla) (u-U).  \label{x6_p19c}
\end{align}

For \eqref{x6_p19a}, a simple energy estimate using \eqref{x6_p6} and \eqref{x6_p17} gives
\begin{align}
&\max_{0\le t \le 1} \| W_1\|_2 = O(k^{-\frac 32}), \notag \\
&\max_{0\le t \le 1} \| \nabla W_1 \|_4 = O(k^{-\frac 12}), \notag \\
&\max_{0\le t \le 1} \| W_1 \|_{H^2} = O(k^{\frac 12}). \label{x6_p19d}
\end{align}

Similarly for \eqref{x6_p19b}, we use \eqref{x6_p6}, \eqref{x6_p16} and \eqref{x6_p16a} to get
\begin{align}
 &\max_{0\le t \le 1} \| W_2\|_4 = O(k^{-\frac 32}), \notag \\
&\max_{0\le t \le 1} \| W_2 \|_{H^2} = O(k^{\frac 12}). \label{x6_p19e}
\end{align}

For \eqref{x6_p19c}, a simple $L^2$ estimate using \eqref{x6_p18}, \eqref{x6_p19d} and \eqref{x6_p19e} gives
\begin{align*}
 \partial_t (\|W_1-W_2\|_2^2) & \lesssim \| W_1-W_2\|_2 \cdot \|u-U\|_4 \cdot \| \nabla W_1\|_4 \notag \\
 &\quad + \|W_1-W_2\|_2^2 \cdot \|\nabla u \|_{\infty} + \|W_2\|_4 \cdot \| \nabla(u-U)\|_4\cdot \|W_1-W_2\|_2 \notag\\
 & \lesssim O(k^{-2} )\cdot \|W_1-W_2\|_2 + O(1) \cdot \|W_1-W_2\|_2^2 \notag \\
 & \qquad + O(k^{-3}) \cdot \|W_1-W_2\|_2.
\end{align*}
Gronwall in time then gives
\begin{align*}
 \max_{0\le t\le 1} \|W_1(t)-W_2(t)\|_2 =O(k^{-2}).
\end{align*}
Interpolating this  with the trivial estimate
\begin{align*}
 \max_{0\le t\le 1} \|W_1(t)-W_2(t)\|_{H^2} =O(k^{\frac 12})
\end{align*}
then yields
\begin{align*}
  \max_{0\le t\le 1} \|W_1(t)-W_2(t)\|_{H^{\frac 32}} =O(k^{-\frac 18}).
\end{align*}
This shows that $\|E_2\|_{H^{\frac 32}}=O(k^{-\frac 18})$ and we have finished the proof of \eqref{x6_p8}.

We now specify the choice of $g_0$ in \eqref{x6_p5}.

By \eqref{x6_2}, we have
\begin{align*}
 & \max\{ |(\partial_r \tilde \phi^r)(t_0,\tilde r_*,\tilde z_*)|, \;\;  |(\partial_z \tilde \phi^r)(t_0,\tilde r_*,\tilde z_*)|, \\
&\quad |(\partial_r \tilde \phi^z)(t_0,\tilde r_*,\tilde z_*)|, \;\;  |(\partial_z \tilde \phi^z)(t_0,\tilde r_*,\tilde z_*)|\}=M.
\end{align*}

WLOG we assume
\begin{align}
 |(\partial_r \tilde \phi^r)(t_0,\tilde r_*,\tilde z_*)|=M. \label{x6_p20}
\end{align}
The other cases are similarly treated.

Let $(r_*,z_*)$ be the pre-image of $(\tilde r_*,\tilde z_*)$, i.e. $\tilde r_*= \phi^r (t_0,r_*,z_*)$,
$\tilde z_*=\phi^z(t_0,r_*,z_*)$.

By \eqref{x-1_4}, we have
\begin{align}
 \left|\det\Bigl( (D\tilde \phi)(t_0, \phi(t_0,r_*,z_*)) \Bigr)\right| = \frac{\phi^r(t_0,r_*,z_*)} {r_*} =:N_*>0. \label{x6_p20a}
\end{align}

By the Fundamental Theorem of Calculus and \eqref{x-1_1}, we have
\begin{align*}
 r_* & = \tilde \phi^r(t_0,\phi^r(t_0,r_*,z_*),\phi^z(t_0,r_*,z_*)) - \tilde \phi^r(t_0,0,\phi^z(t_0,r_*,z_*)) \notag \\
 & \le \| \partial_r \tilde \phi^r \|_{\infty} \cdot \phi^r (t_0,r_*,z_*) \notag \\
 & \le M \cdot \phi^r (t_0,r_*,z_*).
\end{align*}

Therefore
\begin{align} \label{x6_p20aa}
 N_* M\ge 1.
\end{align}
This relation will be used later.

By \eqref{x6_p20}, \eqref{x6_p20a} and continuity, we can find a nonempty open set $\Omega_0$ around
the point $(r_*,z_*)$ such that
\begin{align}
 & \frac M2 <| (\partial_r \tilde \phi^r)(t_0, \phi(t_0,r,z)) | <2M, \notag \\
 & \frac {N_*}2 <\frac{\phi^r(t_0,r,z)} r = | \det ( (D\tilde \phi)(t_0,\phi(t_0,r,z)) ) | <2 N_*,
 \quad \forall\, (r,z) \in \Omega_0. \label{x6_p21}
\end{align}

Furthermore we may shrink $\Omega_0$ slightly if necessary such that for some $\delta_1>0$,
\begin{align*}
 \Omega_0 \cap \{ (r,z):\, 0\le r \le \delta_1\} =\emptyset
\end{align*}

In yet other words, if $(r,z) \in \Omega_0$, then we must have $r>\delta_1$.

Now choose $b\in C_c^{\infty} (\Omega_0)$ such that
\begin{align}
 \int |b(r,z)|^2 r dr dz=1. \label{x6_p22}
\end{align}

Since by our choice $\Omega_0$ stays away from the axis $r=0$, the function $b$ can be
naturally regarded as a smooth function on $\mathbb R^3$.

We now let
\begin{align}
 g_0(r,z) = \frac 1 {M^{\frac 16}} \cos (kr) b(r,z) \label{x6_p24}
\end{align}
and recall from \eqref{x6_p5}
\begin{align*}
 G_0(x) & =g_0(r,z) e_{\theta} \notag \\
 & = \frac 1 {M^{\frac 16}}  \cos(kr) b(r,z) e_{\theta}.
\end{align*}

By \eqref{x6_p22}, it is not difficult to check that \eqref{x6_p6} is satisfied.

Since
\begin{align*}
 W_0=\omega_0 + k^{-\frac 32} G_0,
\end{align*}
by taking $k$ sufficiently large, obviously we can have
\begin{align*}
 & \| W_0 \|_{L^1(\mathbb R^3)}\le 2 \|\omega_0 \|_{L^1(\mathbb R^3)}, \notag \\
 & \| W_0 \|_{L^{\infty}(\mathbb R^3)} \le 2 \|\omega_0 \|_{L^\infty(\mathbb R^3)}, \notag \\
  & \| \frac {W_0}r \|_{L^1(\mathbb R^3)} \le 2 \| \frac{\omega_0}r \|_{L^1(\mathbb R^3)},  \notag \\
  &\| \frac{W_0}r \|_{L^{\infty}(\mathbb R^3)} \le 2 \| \frac{\omega_0}r \|_{L^{\infty}(\mathbb R^3)}, \notag \\
  &\| \frac{W_0}r \|_{L^{3,1}(\mathbb R^3)} \le 2 \| \frac{\omega_0}r \|_{L^{3,1}(\mathbb R^3)}.
\end{align*}
Therefore \eqref{x6_5a} and the first two conditions in \eqref{x6_5} are easily satisfied.  To check the third condition
therein, we note that by \eqref{x6_p22} and for $k$ sufficiently large,
\begin{align*}
 & \|G_0\|_{L^2(\mathbb R^3)} \lesssim \frac 1 {M^{\frac 16} }, \notag \\
 & \|G_0 \|_{H^2(\mathbb R^3)} \lesssim \frac 1 {M^{\frac 16} } \cdot k^2.
\end{align*}
Here the implied constants are absolute constants. Interpolation then gives
\begin{align*}
 \| G_0 \|_{H^{\frac 32} (\mathbb R^3)} \lesssim \frac 1 {M^{\frac 16}} k^{\frac 32}.
\end{align*}
Thus all conditions in \eqref{x6_5a} and \eqref{x6_5} are satisfied.

It remains to show \eqref{x6_7}.

By \eqref{x6_case2}, \eqref{x6_p7} and \eqref{x6_p8}, we have
\begin{align}
 \| W(t_0,\cdot)\|_{\dot H^{\frac 32}} & \ge \| k^{-\frac 32} \cdot \frac{g_0(\tilde \phi(t_0))}{\tilde \phi^r(t_0)} r e_{\theta}
 \|_{\dot H^{\frac 32}} \notag \\
 & \qquad - \| \frac{\omega_0^{\theta}(\tilde \phi(t_0))}{\tilde \phi^r(t_0)} r e_{\theta}
 \|_{\dot H^{\frac 32}} \notag \\
 & \qquad - \|E_1\|_{\dot H^{\frac 32}} - \|E_2\|_{\dot H^{\frac 32}} \notag \\
 & \ge  \| k^{-\frac 32} \cdot \frac{g_0(\tilde \phi(t_0))}{\tilde \phi^r(t_0)} r e_{\theta}
 \|_{\dot H^{\frac 32}} - M^{\frac 16} - O(k^{-\alpha}). \notag
\end{align}
Therefore \eqref{x6_7} will be established once we prove the stronger estimate
\begin{align} \label{x6_p25}
\| k^{-\frac 32} \cdot \frac{g_0(\tilde \phi(t_0))}{\tilde \phi^r(t_0)} r e_{\theta}
 \|_{\dot H^{\frac 32}} \gtrsim M^{\frac 13}.
 \end{align}

 We shall prove this via interpolation and inflation of $H^1$ norm.

 By \eqref{x6_p24}, we have
 \begin{align}
 \| k^{-\frac 32} \cdot \frac{g_0(\tilde \phi(t_0))}{\tilde \phi^r(t_0)} r e_{\theta}\|_{L^2(\mathbb R^3)}
 & \lesssim k^{-\frac 32} \left( \int \Bigl| \frac{g_0(\tilde \phi(t_0))} {\tilde \phi^r(t_0)} r \Bigr|^2 rdrdz \right)^{\frac 12}
 \notag \\
 & \lesssim k^{-\frac 32} \left( \int
\Bigl| \frac{g_0(r,z) \phi^r(t_0,r,z)} r \Bigr|^2 r drdz
 \right)^{\frac 12} \notag \\
 & \lesssim \frac {k^{-\frac 32}} {M^{\frac 16} }
 \left(\int \frac{\cos^2(kr) b^2(r,z) (\phi^r(t_0,r,z))^2} {r^2} r dr dz \right)^{\frac 12} \notag \\
 & \lesssim \frac{k^{-\frac 32} }{M^{\frac 16} } \| \frac{b \phi^r(t_0)} r \|_{L^2(\mathbb R^3)}
 \lesssim \frac{k^{-\frac 32}} {M^{\frac 16} } N_*, \label{x6_p26}
 \end{align}
where in the last inequality we have used \eqref{x6_p21} and \eqref{x6_p22}.

Now introduce
\begin{align*}
 g_1(r,z) = \sin(k \tilde \phi^r(t_0,r,z)) \frac{b(\tilde \phi(t_0,r,z))}{\tilde \phi^r(t_0,r,z)}
 (\partial_r \tilde \phi^r)(t_0,r,z) r e_{\theta}.
\end{align*}

By \eqref{x6_p21} and a similar calculation as in \eqref{x6_p26}, we have for $k$ sufficiently large,
\begin{align}
 \|g_1\|_{L^2(rdrdz)} & \ge
 \left( \int  \frac{ \sin^2(kr) b^2(r,z)  ((\partial_r \tilde \phi^r)(t_0, \phi(t_0,r,z)))^2} {r^2}
 (\phi^r(t_0,r,z))^2 rdrdz
 \right)^{\frac 12} \notag \\
 & \ge M \| \frac{ b \phi^r(t_0)} {r} \|_{L^2(\mathbb R^3)} -O(k^{-\alpha}) \notag \\
 & \ge \frac 23 M \| \frac{b \phi^r(t_0)} r \|_{L^2(\mathbb R^3)} \gtrsim M\cdot N_*. \label{x6_p27}
\end{align}

Now for the $\dot H^1$-norm, by using \eqref{x6_p27}, we have
\begin{align}
 \| k^{-\frac 32} \cdot \frac{g_0(\tilde \phi(t_0))}{\tilde \phi^r(t_0)} r e_{\theta}\|_{\dot H^1(\mathbb R^3)}
 & \ge k^{-\frac 32} \| \partial_r \Bigl(  \frac{g_0(\tilde \phi(t_0))}{\tilde \phi^r(t_0)} r e_{\theta}  \Bigr)
 \|_{L^2(rdrdz)} \notag \\
 & \ge \frac{k^{-\frac 32}} {M^{\frac 16}} \cdot ( k \|g_1\|_{L^2(rdrdz)} +O(1) ) \notag \\
 & \ge \frac 12 k^{-\frac 12} M^{\frac 56} N_*, \label{x6_p28}
\end{align}
where again we need to take $k$ sufficiently large.

We are now ready to prove \eqref{x6_p25}.

By the usual interpolation inequality
\begin{align*}
 \| f \|_{\dot H^1(\mathbb R^3)} \lesssim \| f \|_{L^2(\mathbb R^3)}^{\frac 13} \cdot \| f \|^{\frac 23}_{\dot H^{\frac 32}(\mathbb R^3)}
\end{align*}
and \eqref{x6_p26}, \eqref{x6_p28}, we have
\begin{align*}
 & k^{-\frac 12} M^{\frac 56} N_* \notag \\
 & \quad \lesssim \left( \frac{k^{-\frac 32} N_*} {M^{\frac 16}} \right)^{\frac 13}
 \cdot \| k^{-\frac 32} \cdot \frac{g_0(\tilde \phi(t_0))}{\tilde \phi^r(t_0)} r e_{\theta}\|^{\frac 23}_{\dot H^{\frac 32}(\mathbb R^3)}.
\end{align*}

By \eqref{x6_p20aa}, we then have
\begin{align*}
 \| k^{-\frac 32} \cdot \frac{g_0(\tilde \phi(t_0))}{\tilde \phi^r(t_0)}
 r e_{\theta}\|^{\frac 23}_{\dot H^{\frac 32}(\mathbb R^3)}
 & \gtrsim M^{\frac 89} N_*^{\frac 23} \notag \\
 & \gtrsim M^{\frac 29}.
\end{align*}

Hence
\begin{align*}
 \| k^{-\frac 32} \cdot \frac{g_0(\tilde \phi(t_0))}{\tilde \phi^r(t_0)} r e_{\theta}\|_{\dot H^{\frac 32}(\mathbb R^3)}
\gtrsim M^{\frac 13} \gg M^{\frac 16}.
\end{align*}

This ends the estimate of \eqref{x6_p25}.
\end{proof}

\begin{prop} \label{x7}
 For any $A\gg 1$, there exist $\delta_0=\delta_0(A) \to 0$, $t_0 = t_0(A) \to 0$, $M_0=M_0(A) \to \infty$ (as $A\to \infty$),
 and a smooth solution $\omega$ to the axisymmetric (without swirl) Euler equation
\begin{align} \notag
 \begin{cases}
  \partial_t \left( \frac {\omega }r \right) + (u\cdot \nabla ) \left( \frac {\omega} r  \right) =0,
  \quad 0<t\le 1,\, x=(x_1,x_2,z),\, r=\sqrt{x_1^2+x_2^2}, \\
  u=-\Delta^{-1} \nabla \times \omega, \\
  \omega \Bigr|_{t=0} =\omega_0
 \end{cases}
 \end{align}
such that the following conditions are satisfied:
\begin{enumerate}
 \item $\omega_0 \in C_c^{\infty} (\mathbb R^3)$, $\omega_0 =\omega_0^{\theta} (r,z) e_{\theta}$ and for some
 $r_0>0$,
\begin{align} \label{x7_1}
 \operatorname{supp}(\omega_0^{\theta}(r,z)) \subset\{(r,z):\, r>r_0\}.
\end{align}

\item The  $L^{\infty}$ norm of $\omega$ is uniformly small on the interval $[0,1]$:
\begin{align}
 \max_{0\le t \le 1} \| \omega(t)\|_{L^{\infty}} \le \delta_0(A). \label{x7_2}
\end{align}

\item The support of $\omega(t)$ remains close to the origin:
\begin{align} \label{x7_3}
 \operatorname{supp} ( \omega(t,\cdot) ) \subset \{ x:\quad |x| < \delta_0(A)\}, \quad\forall\, 0\le t\le 1.
\end{align}

\item The $\dot H^{\frac 32}$-norm of $\omega$ is inflated rapidly from $t=0$ to $t=t_0$:
\begin{align}
 &\| \omega_0\|_{\dot H^{\frac 32}} <\delta_0(A), \notag \\
 &\| \omega(t_0)\|_{\dot H^{\frac 32}} >M_0(A). \label{x7_4}
\end{align}
\end{enumerate}
\end{prop}

\begin{proof}[Proof of Proposition \ref{x7}]
We first note that it suffices to construct the solution $\omega$
satisfying all other conditions except \eqref{x7_3}. Indeed if
$\omega$ is such a solution, then for any $\lambda>0$,
\begin{align*}
&\omega_{\lambda}(t,x) := \omega(t,\lambda x) \notag
\end{align*}
is also a solution to the Euler equation. By finite speed
propagation, we have
\begin{align*}
\operatorname{supp}(\omega(t)) \subset K, \quad \forall\, 0\le t\le 1,
\end{align*}
where $K$ is a fixed compact set. On the other hand
\begin{align*}
\operatorname{supp}(\omega_{\lambda}(t)) \subset \frac 1 {\lambda}K=\{ \frac
1 {\lambda} x:\; x \in K\}, \quad\forall\, 0\le t\le 1.
\end{align*}
Obviously by taking $\lambda$ sufficiently large we can satisfy
\eqref{x7_3}. Note that \eqref{x7_2}, \eqref{x7_4} is invariant
under the scaling transformation $x\to \lambda x$.  Therefore in the
rest of this proof we shall ignore \eqref{x7_3}.

For $A\gg 1$, we choose $g_A$ as in \eqref{x2_1} and denote by $W$
the corresponding smooth solution to the Euler equation:
\begin{align} \notag
\begin{cases}
\partial_t (\frac {W}r) + (U\cdot \nabla)(\frac W r) =0, \quad
-2\le t\le 2, \\
\nabla \cdot U=0,\\
W\Bigr|_{t=0} = W_0 = g_A e_{\theta}.
\end{cases}
\end{align}

By \eqref{x2_3d} we have (recall $U=U^r e_r +U^z e_{z}$)
\begin{align*}
\| \frac{U^r(t)} r \|_{\infty} &\le C \| \frac {W(t)} r \|_{L^{3,1}}
\notag \\
& \le C \sqrt{\log A}, \quad\forall\, t\in \mathbb R,
\end{align*}
where $C>0$ is an absolute constant and we have used the
$L^{3,1}$-preservation of $W/r$:
\begin{align*}
\| \frac {W(t)} r \|_{L^{3,1}} = \| \frac {W_0} r \|_{L^{3,1}},
\quad \forall\, t\in \mathbb R.
\end{align*}

 Since
\begin{align*}
\partial_t W + (U\cdot \nabla) W = \frac{U^r}r W,
\end{align*}
we get
\begin{align}
\max_{-2\le t \le 2} \| W(t)\|_{\infty} &\le \|W_0\|_{\infty}
e^{\max_{-2\le t \le 2} \| \frac{U^r} r \|_{\infty}} \notag \\
& \le \frac{\sqrt{\log A}} A e^{C \sqrt{\log A}} <A^{-\frac 12},
\label{x7_7}
\end{align}
for $A$ sufficiently large.


By definition of $g_A$, the condition \eqref{x7_1} is trivially
satisfied. It remains to check \eqref{x7_4}. By \eqref{x2_3}, we
have
\begin{align*}
\| W_0\|_{\dot H^{\frac 32}} \lesssim \frac {\sqrt {\log A}} {\sqrt
A}.
\end{align*}

Let $\Phi=(\Phi^r,\Phi^z)$ be the forward characteristic lines as in
\eqref{x-1_00} and let $\tilde \Phi$ be the inverse. By Proposition
\ref{x5}, we have for some $0<t_1 \le \frac 1 {\log\log A}$,
\begin{align} \label{x7_700a}
  \|D \Phi(t_1) \|_{\infty} + \|D \tilde \Phi(t_1) \|_{\infty} \ge
\log\log A.
\end{align}

By differentiating the identity $\Phi\circ \tilde \Phi=id$ and using \eqref{x-1_4}, we have
\begin{align*}
 (D \Phi)(\tilde \Phi(r,z)) &= \Bigl( (D\tilde \Phi)(r,z) \Bigr)^{-1} \notag \\
 & = \frac 1 {\det(D\tilde \Phi(r,z))} \text{adj}((D\tilde \Phi)(r,z)) \notag \\
 & =  \frac{\tilde \Phi^r(r,z)} r \text{adj}((D\tilde \Phi)(r,z)),
\end{align*}
where $\text{adj}((D\tilde \Phi)(r,z))$ is the adjugate matrix of $D\tilde \phi(r,z)$.  Recall that for any $2\times 2$ matrix
\begin{align*}
 B=\begin{pmatrix} a \quad b \\ c \quad d \end{pmatrix},
\end{align*}
we have
\begin{align*}
 \text{adj}(B)=\begin{pmatrix} d \quad -b\\-c\quad a \end{pmatrix}
\end{align*}
and obviously
\begin{align*}
 \| B \|_{\infty}  & = \max\{ |a|, |b|, |c|,|d| \} \\
 &=\| \text{adj}(B)\|_{\infty}.
\end{align*}

Therefore
\begin{align*}
 \| D\Phi(t_1)\|_{\infty} & \le \| \frac{\tilde \Phi^r} r \|_{\infty} \| \text{adj} (D\tilde \Phi) \|_{\infty} \notag \\
 & \le \| D\tilde \Phi(t_1)\|_{\infty}^2
\end{align*}

Consequently we have
\begin{align} \label{x7_700b}
 \| D\tilde \Phi(t_1)\|_{\infty} \gtrsim \sqrt{\log\log A}.
\end{align}

We can then apply Proposition \ref{x6} (with $W$ as the input solution $\omega$) and obtain $\tilde \omega$ as the desired solution
(note that $\| \frac{\tilde \omega_0}r \|_{L^{3,1}} \lesssim
\sqrt{\log A}$, $\|\tilde \omega_0\|_{\infty} \lesssim
\frac{\sqrt{\log A}} {A}$ so that we can repeat the computation of
\eqref{x7_7} and still have $\max_{0\le t \le 1} \| \tilde \omega(t)
\|_{\infty} \lesssim A^{-\frac 12}$.)

\end{proof}

\begin{lem} \label{x8}
 Suppose $\omega^1$, $\omega^2$ are given smooth solutions to the 3D Euler equations (in vorticity form):
 \begin{align} \notag
  \begin{cases}
   \partial_t \omega^j+ (u^j \cdot \nabla )\omega^j = (\omega^j \cdot \nabla) u^j, \qquad 0<t\le 1, \\
    u^j =-\Delta^{-1} \nabla \times \omega^j, \\
   \omega^j \Bigr|_{t=0}=\omega^j_0 \in C_c^{\infty}(\mathbb R^3), \quad j=1,2.
  \end{cases}
 \end{align}
Here we assume the lifespan of each $\omega^j$ is at least $[0,1]$.

Define
\begin{align} \label{x8_1}
 r_0= \max_{j=1,2} \max_{0\le t\le 1} \| u^j(t)\|_{\infty}.
\end{align}

Consider the problem
\begin{align} \label{x8_2}
 \begin{cases}
  \partial_t W + (U\cdot\nabla )W = (W\cdot \nabla)U, \\
  U= -\Delta^{-1} \nabla \times W,\\
  W\Bigr|_{t=0} =W_0,
 \end{cases}
\end{align}
where
\begin{align*}
 W_0(x) =\omega_0^1(x) + \omega_0^2(x-x_W),
\end{align*}
and $x_W\in \mathbb R^3$ is a vector which controls the mutual distance between $\omega_0^1$ and $\omega_0^2$.

For any $\epsilon>0$, there exists $R_{\epsilon}
=R_{\epsilon}(\epsilon, \max_{j=1,2}\max_{0\le t \le 1} \|u^j
(t)\|_{H^4})> 100r_0$ sufficiently large, such that if $|x_{W}|\ge
R_{\epsilon}$, then the following hold:

\begin{enumerate}
 \item There exists a unique smooth solution $W$ to \eqref{x8_2} on the time interval $[0,1]$. Furthermore for any $0\le t \le1$ it
 has the decomposition
 \begin{align}
  W(t)=W^{1}(t) + W^2 (t), \label{x8_3}
 \end{align}
where
\begin{align*}
 &\operatorname{supp}(W^{1} (t)) \subset \Omega_1^{\epsilon}, \\
 &\Omega_1^{\epsilon} := \{ x\in \mathbb R^3: \, d(x,\operatorname{supp}(\omega_1^0) )< r_0+\epsilon\}, \\
 &\operatorname{supp}(W^2(t)) \subset \Omega_W^{\epsilon},\\
 &\Omega_W^{\epsilon} :=\{ y=x+x_W:\; d(x, \operatorname{supp}(\omega_2^0))< r_0+\epsilon\}.
\end{align*}

\item The flow $W$ is uniformly close to $\omega^1(\cdot) +\omega^2(\cdot-x_W)$:
\begin{align} \label{x8_4}
 & \max_{0\le t\le 1}\| W^1(t,\cdot) -\omega^1(t,\cdot)\|_{H^2} <\epsilon, \notag \\
 & \max_{0\le t\le 1}\| W^2(t,\cdot) -\omega^2(t,\cdot-x_W)\|_{H^2} <\epsilon
\end{align}

\item All higher Sobolev norms of $W^1$ and $W^2$ can be controlled in terms of $\omega_0^1$ and $\omega_0^2$
respectively:  Let
\begin{align*}
 N= \max_{0\le t\le 1} (\| \omega^1(t,\cdot) \|_{\infty} + \|\omega^2(t,\cdot)\|_{\infty}) + \|u_0^1\|_2+\|u_0^2\|_2.
\end{align*}
Here $u_0^1$, $u_0^2$ the velocity fields corresponding to the vorticity $\omega_0^1$, $\omega_0^2$
respectively. Then for any $k\ge 3$,
\begin{align} \label{x8_5}
 &\max_{0\le t \le 1 } \| W^1(t,\cdot) \|_{H^k} \le C(k, \| \omega_0^1\|_{H^k}, N)<\infty, \notag \\
 &\max_{0\le t \le 1 } \| W^2(t,\cdot) \|_{H^k} \le C(k, \| \omega_0^2\|_{H^k}, N)<\infty.
\end{align}

\end{enumerate}

\end{lem}

\begin{proof}[Proof of Lemma \ref{x8}]
Let $R=|x_W|$ and denote
\begin{align} \label{x8_p1}
 &M_0 = 100 ( \|u_0^1\|_{H^4} + \| u_0^2\|_{H^4}),\notag \\
 &M_1=\max_{0\le t\le 1} ( \| \omega^1(t)\|_{\infty} + \| \omega^2(t)\|_{\infty}+1).
\end{align}

Consider the 3D Euler equation (in velocity formulation)
\begin{align*}
 \begin{cases}
  \partial_t u + (u\cdot \nabla) u =-\nabla p, \\
  \nabla \cdot u =0, \\
  u \Bigr|_{t=0} =u_0.
 \end{cases}
\end{align*}
Suppose
\begin{align*}
 \|u_0\|_{H^4} \le M_0,
\end{align*}
and on some time interval $[0,\tau]$, $\tau\le 1$,
\begin{align*}
 \max_{0\le t\le \tau} \| \omega(t)\|_{\infty} \le M_1,
\end{align*}
where $\omega=\operatorname{curl}(u)$. Then by a simple energy estimate, we have
\begin{align} \label{x8_p3}
 \max_{0\le t\le \tau} \|u(t)\|_{H^4} \le M_2,
\end{align}
where $M_2=M_2(M_0,M_1)>0$ can be taken as a constant which is uniform for all $\tau\le 1$.  We shall need this constant below.
Also by standard local wellposedness theory, if for some $t_0$ we have $\|u(t_0)\|_{H^4} \le M_2$,
then there exists $\tau_0=\tau_0(M_2)>0$ such that the corresponding local solution has lifespan at least $[t_0,t_0+\tau_0]$ and
\begin{align} \notag
 \max_{t_0\le t\le t_0+\tau_0} \|u(t)\|_{H^4} \le 2M_2.
\end{align}
This fact will also be used below.

Now let $0=t_0<t_1<\cdots <t_{L-1} <t_L=1$ be a partition of the time interval $[0,1]$ such that
\begin{align*}
 \max_{0\le i \le L-1} (t_{i+1}-t_i) <\tau_0.
\end{align*}

We now inductively check the following

\texttt{Claim}: For each $i=0,1,\cdots, L$, there exists $R_i>0$ sufficiently large such that
if $R>R_i$, then the following hold:

\begin{enumerate}
 \item $W(t)$ has the decomposition \eqref{x8_3} for all $0\le t\le t_{i}$.

 \item \begin{align}
        &\max_{0\le t \le t_{i}} \| W^1(t,\cdot) -\omega^1(t,\cdot)\|_2 < R^{-\frac 14}, \label{x8_p4a}\\
        & \max_{0\le t\le t_{i}} \| W^2(t,\cdot)-\omega^2(t,\cdot-x_{W})\|_2 <R^{-\frac 14}. \label{x8_p4b}
       \end{align}

 \item \begin{align}
        &\max_{0\le t\le t_i} \|U(t) \|_{H^4} \le M_2. \label{x8_p5a}
       \end{align}

\end{enumerate}

Indeed the claim holds trivially for $i=0$. Now assume the claim holds for all $i\le l-1$ ($l\ge 1$), and we need to prove
the claim for $i=l$. Since $\| U(t_{l-1})\|_{H^4} \le M_2$, by our choice of $\tau_0$, $U(t)$ can be extended to
$[t_{l-1}, t_l]$, and
\begin{align}
 \max_{0 \le t\le t_l} \| U(t)\|_{H^4} &\le  \max\{ \max_{0\le t \le t_{l-1}} \|U(t)\|_{H^4}, \max_{t_{l-1} \le t \le t_l}
 \|U(t) \|_{H^4} \} \notag \\
 &\le 2M_2.  \label{x8_p7}
\end{align}
By the inductive assumption, we have
\begin{align*}
 W(t_{l-1}) =W^1(t_{l-1}) +W^2(t_{l-1})
\end{align*}
and for $R$ sufficiently large,
\begin{align*}
 \operatorname{dist}\Bigl(\operatorname{supp} (W^1(t_{l-1})),  \operatorname{supp} (W^2(t_{l-1})) \Bigr) >\frac 23 R.
\end{align*}
By \eqref{x8_p7} and finite speed propagation, we then have for $R$ sufficiently large,
\begin{align} \label{x8_p9}
 \operatorname{dist}\Bigl(\operatorname{supp} (W^1(t)),  \operatorname{supp} (W^2(t)) \Bigr) >\frac 13 R, \qquad \forall\, 0\le t \le t_l.
\end{align}
Denote by $U^1$, $U^2$ the velocity fields corresponding to the vorticity $W^1$ and $W^2$ respectively. By \eqref{x8_p9} and
an argument similar to \eqref{lem10_9}, it is not difficult to check that
\begin{align} \label{x8_p11}
 \max_{0\le t \le t_l} \max_{|\alpha|\le 3} \|D^{\alpha }U^2 (t,\cdot)\|_{L^{\infty}(x\in \operatorname{supp}(W^1(t))) } \le R^{-\frac 13},
\end{align}
where again we need to take $R$ sufficiently large (to kill some prefactors).

Now observe that
\begin{align*}
 \begin{cases}
  \partial_t W^1 + (U^1 \cdot \nabla) W^1 = (W^1 \cdot \nabla) U^1 - (U^2 \cdot \nabla) W^1 + (W^1 \cdot \nabla) U^2,
  \\
  \partial_t \omega^1 + (u^1 \cdot \nabla) \omega^1 = (\omega^1 \cdot \nabla) u^1, \\
  W^1 \Bigr|_{t=0} =\omega^1 \Bigr|_{t=0} =\omega_0^1.
 \end{cases}
\end{align*}
Set $\eta=W^1-\omega^1$, $v=U^1-u^1$. Then clearly
\begin{align*}
 \partial_t \eta + &(v\cdot \nabla ) W^1 + (u^1 \cdot \nabla )\eta \notag \\
 &= (\eta \cdot \nabla ) U^1 + (\omega^1 \cdot \nabla) v -(U^2\cdot \nabla )W^1 + (W^1 \cdot \nabla) U^2.
\end{align*}
A simple $L^2$ estimate using \eqref{x8_p7} and \eqref{x8_p11} then gives for $0<t\le t_l$:
\begin{align*}
 \partial_t (\|\eta\|_2^2) & \lesssim \|\eta\|_2 \cdot \|v\|_6 \cdot \| \nabla W^1 \|_3 +
 \| \eta\|_2^2 \cdot \| \nabla U^1 \|_{\infty} \notag \\
& \qquad + \| \nabla v \|_2 \cdot \|\eta\|_2 \cdot \|\omega^1 \|_{\infty}
+ R^{-\frac 13} \| \nabla W^1 \|_2 \cdot \| \eta \|_2 \notag \\
& \qquad + R^{-\frac 13} \| W^1 \|_2 \cdot \|\eta\|_2 \notag \\
& \lesssim_{M_2, \omega^1} \| \eta\|_2^2 + R^{-\frac 13} \| \eta\|_2.
\end{align*}
Integrating in time up to $t_{l}$ and taking $R$ sufficiently large then gives
\begin{align*}
&\max_{0\le t \le t_{l}} \| \eta(t)\|_2 < R^{-\frac 14}.
\end{align*}
This settles \eqref{x8_p4a} for $i=l$. The inequality \eqref{x8_p4b} is proved similarly.
Interpolating \eqref{x8_p4a}, \eqref{x8_p4b} with \eqref{x8_p7} then easily yields that (see \eqref{x8_p1})
\begin{align*}
 \max_{0\le t \le t_l} \|W(t) \|_{\infty} \le M_1.
\end{align*}
Therefore by \eqref{x8_p3}, we can upgrade the rough estimate \eqref{x8_p7} to \eqref{x8_p5a} for $i=l$.
By \eqref{x8_p4a}--\eqref{x8_p4b}, interpolation and taking $R$ sufficiently large, we can easily have (see \eqref{x8_1})
\begin{align*}
 \max_{0\le t \le t_l} \|U(t,\cdot)\|_{\infty} \le r_0 +\epsilon/2.
\end{align*}
Hence the decomposition \eqref{x8_3} follows. We have completely proved the claim.

By using the claim and a simple interpolation argument, it is not difficult to check that \eqref{x8_4} holds.
Finally \eqref{x8_5} follows from a simple energy estimate using the disjointness of the support of $W^1$ and $W^2$ and
an estimate similar to \eqref{x8_p11}.
The lemma is proved.

\end{proof}

\begin{prop} \label{x9}
Assume $\{\omega^j\}_{j=1}^{\infty}$ is a sequence of smooth functions each of which solves the 3D incompressible Euler equation
(in vorticity form)
\begin{align*}
 \begin{cases}
  \partial_t \omega^j + (u^j \cdot \nabla ) \omega^j = (\omega^j \cdot \nabla )u^j, \quad 0<t\le 1,\\
  u^j=-\Delta^{-1}\nabla \times \omega^j, \\
  \omega^j \Bigr|_{t=0} = \omega^j_0 \in C_c^{\infty}(\mathbb R^3),
 \end{cases}
\end{align*}
and satisfies the following condition:

For each $j\ge 1$,  $\operatorname{supp}(\omega^j(t)) \subset B(0,2^{-10j})$
for any $0\le t\le 1$ and
\begin{align} \label{x9_1}
 \|u^j_0\|_{H^{\frac 52}}+\max_{0\le t\le 1} (\| \omega^j(t) \|_{\infty}
+\|u^j(t)\|_{\infty})\le 2^{-10j}.
\end{align}
Here $u^j_0$ is the velocity corresponding to the vorticity
$\omega^j_0$.

Then there exist centers $x_j \in \mathbb R^3$ whose mutual distance
are sufficiently large (i.e. $|x_j-x_k|\gg 1$ if $j\ne k$) such that
the following hold:

\begin{enumerate}
 \item Take the initial data (vorticity)
 \begin{align*}
  W_0(x) = \sum_{j=1}^{\infty} \omega^j_0(x-x_j),
 \end{align*}
then $W_0 \in L^1 \cap L^{\infty} \cap H^{\frac 32} \cap
C^{\infty}$. The corresponding initial velocity $U_0 \in H^{\frac
52} \cap C^{\infty}$. Furthermore for any $j\ne k$
\begin{align} \notag 
 B(x_j, 100) \cap B(x_k, 100 ) = \varnothing.
\end{align}

\item With $W_0$ as initial data, there exists a unique smooth solution $W$ to the Euler
equation (in vorticity form)
\begin{align*}
\begin{cases}
 \partial_t W + (U\cdot \nabla)W=(W\cdot \nabla)U,\\
 U=-\Delta^{-1} \nabla \times W,\\
 W\Bigr|_{t=0} =W_0.
 \end{cases}
\end{align*}
on the time interval $[0,1]$ satisfying $W \in L_t^{\infty}L_x^1\cap L_t^{\infty}L_x^{\infty}
\cap C^{\infty}$, $U \in C^{\infty} \cap L_t^{\infty} L_x^2$. Moreover for any
$0\le t \le1$,
\begin{align}
 \operatorname{supp} (W(t,\cdot) ) \subset \bigcup_{j=1}^{\infty} B(x_j, 1). \label{x9_3}
\end{align}

\item For any $\epsilon>0$, there exists an integer $J_{\epsilon}$ sufficiently large such that
if $j\ge J_{\epsilon}$, then
\begin{align}
 \max_{0\le t \le 1} \| W(t,\cdot) - \omega^j(t,\cdot-x_j)\|_{H^2(B(x_j,1))} <\epsilon. \label{x9_4}
\end{align}

\end{enumerate}

\end{prop}

\begin{proof}[Proof of Proposition \ref{x9}]
 Define $x_1=0$. By recursively applying Lemma \ref{x8}, we can choose centers $x_j$ whose
 mutual distance is sufficiently large such that for each $l\ge 2$, we can find a unique smooth
 solution $W^l$ solving the system
 \begin{align*}
  \begin{cases}
   \partial_t W^l +(U^l \cdot \nabla) W^l = (W^l \cdot \nabla) U^l, \quad 0<t\le 1,\\
   U^l =-\Delta^{-1} \nabla \times W^l,\\
   W^l \Bigr|_{t=0} =W_0^l,
  \end{cases}
 \end{align*}
where
\begin{align*}
 W_0^l = \sum_{j=1}^l \omega_0^j (x-x_j).
\end{align*}
Furthermore $W^l$ satisfies
\begin{itemize}
 \item $\operatorname{supp} (W^l(t)) \subset \bigcup_{j=1}^l B(x_j,\frac 12)$, for all $0\le t\le 1$.
 \item $\max_{0\le t\le 1} \| W^l (t,\cdot) -\omega^j(t,\cdot-x_j) \|_{H^2(B(x_j,1))} <2^{-j}$, for any $1\le j\le l$.
 \item $\max_{0\le t\le 1} \| W^{l+1} (t,\cdot) - W^{l}(t,\cdot) \|_{H^2( \bigcup_{j=1}^l B(x_j,1) )} <2^{-l}$.
 \item $\max_{0\le t \le 1} \| W^{l+1}(t,\cdot) - W^{l} (t,\cdot) \|_{L^2}<2^{-l}$.
 \item $\max_{0\le t \le 1} \| W^l(t,\cdot)\|_{H^k(B(x_j,1))} \le C_k =C_k (k,\| \omega_0^j\|_{H^k})<\infty$, for any $1\le j\le l$.
\end{itemize}
Note that in the last inequality above we have no dependence on other constants thanks to the strong assumption \eqref{x9_1}.

Now define
\begin{align*}
 W(t,x)=
 \begin{cases}
  \lim_{l\to \infty} W^l(t,x), \quad \text{if $x \in \bigcup_{j=1}^\infty B(x_j, 1)$,}\\
 0, \quad \text{otherwise.}
 \end{cases}
\end{align*}
Fix any $j_0\ge 1$. By using the properties of $W^l$ listed above, we have
\begin{align*}
 \max_{0\le t \le 1} \| W^{l+1}(t,\cdot) - W^l(t,\cdot) \|_{H^2(B(x_{j_0},1))} \le 2^{-l}, \quad \text{if $l\ge j_0+1$.}
\end{align*}
Also for any $k\ge 3$,
\begin{align*}
 \max_{0\le t \le 1} \| W^{l}(t,\cdot)\|_{H^k(B(x_{j_0},1))} \le C_k, \quad \forall\, l\ge j_0+1.
\end{align*}
Therefore $(W^l)$ is Cauchy in $H^k (B(x_{j_0},1))$ for any $k\ge 2$. Hence $W^l$ converges uniformly to $W \in C^{\infty}((B(x_{j_0},1))$.
Since $j_0$ is arbitrary, we obtain $W\in C^{\infty} (\mathbb R^3)$.  Similarly fix any $j_0\ge 1$. By Sobolev embedding, we have
\begin{align*}
 &\max_{0\le t\le 1} \| W^l(t,\cdot) -\omega^{j_0} (t,\cdot) \|_{L^{\infty} (B(x_{j_0},1))} \notag \\
 \lesssim & \max_{0\le t\le 1} \|W^l(t,\cdot) - \omega^{j_0}(t,\cdot) \|_{H^2(B(x_{j_0},1))} \lesssim 2^{-j_0}, \quad\forall\, l\ge j_0+1.
\end{align*}
By \eqref{x9_1} and sending $l\to \infty$, we obtain $\max_{0\le t\le 1}\|W(t,\cdot)\|_{L^{\infty}} \lesssim 1$. Similarly it is also
easy to check that $W\in L_t^{\infty} L_x^1$.  Since $W^l$ is Cauchy in $L^2$, by Sobolev embedding we have $U^l$ is Cauchy in $L^{6}$ and
converges to the limit $U$. It is not difficult to check that $U$ is smooth and $W$ is the desired solution. The estimate \eqref{x9_4} follows
obviously from the property of $W^l$ and passing $l$ to the limit. The proposition is proved.

\end{proof}

We are now ready to complete the

\begin{proof}[Proof of Theorem \ref{thm3}]
It suffices for us to prove the case $\omega_0^{(g)}\equiv 0$. The case for nonzero $\omega_0^{(g)}$ is a simple
modification of the proof below.

 For each $j\ge 1$, by using Proposition \ref{x7}, we can find a smooth solution $\omega^j$ solving the system
 \begin{align*}
  \begin{cases}
   \partial_t \omega^j + (u^j \cdot \nabla) \omega^j = (\omega^j \cdot \nabla) u^j, \quad 0<t\le 1, \\
   u^j=-\Delta^{-1} \nabla \times \omega^j, \\
   \omega^j \Bigr|_{t=0} =\omega_0^j,
  \end{cases}
 \end{align*}
such that the following hold:
\begin{itemize}
 \item $\operatorname{supp} (\omega^j(t,\cdot)) \subset\{x,\; |x|<2^{-100j} \}$, for any $0\le t\le 1$.
 \item $\max_{0\le t \le 1} (\| \omega^j(t)\|_{L^{\infty}} + \| u^j(t)\|_{L^{\infty}}) \le 2^{-100j}$.
 \item Let $u_0^j$ be the velocity corresponding to the vorticity $\omega_0^j$, then
 \begin{align*}
  \| u_0^j \|_{H^{\frac 52}} < 2^{-100j}.
 \end{align*}
\item For some $0<t_j^0 <\frac 1j$, we have
\begin{align*}
 \| \omega^j(t_j^0,\cdot) \|_{\dot H^{\frac 32}} >2^j.
\end{align*}

\end{itemize}
By continuity and the last inequality above, we can find $0<t_j^1<t_j^2<\frac 1j$ such that
\begin{align} \label{sec6_6.104}
\| \omega^j(t,\cdot)\|_{\dot H^{\frac 32}} >2^j, \quad\forall\, t_j^1\le t\le t_j^2.
\end{align}

By Proposition \ref{x9}, we can then find centers $x_j$ and build a smooth solution $W$ having initial data
\begin{align*}
 W(0,x) = \sum_{j=1}^{\infty} \omega_0^j(x-x_j).
\end{align*}

The regularity properties of $W$ are simple consequences of Proposition \ref{x9}.

By \eqref{x9_3}, we can write
\begin{align*}
W(t,x)= \sum_{j=1}^{\infty} W^j(t,x),
\end{align*}
where $W^j\in C_c^{\infty}(B(x_j, 1))$.

Now we make the following

\textbf{Claim}: there exists an integer $J_1>0$ and constants $C_1>0$, $C_2>0$ such that the following hold:
for any $0\le \tau_0\le 1$, if $\|W(\tau_0,\cdot)\|_{\dot H^{\frac 32}(\mathbb R^3)} <\infty$,
then
\begin{align} \label{sec6_claim_i1}
\| W(\tau_0,\cdot)\|_{\dot H^{\frac 32}} \ge C_1 \|\omega^j(\tau_0,\cdot)\|_{\dot H^{\frac 32}} -C_2, \quad \forall\, j\ge J_1.
\end{align}
Here the constant $C_1>0$ is actually an absolute constant. The constant $C_2$ depends on
$\max_{0\le t\le 1}\|W(t,\cdot)\|_2$.

To prove the claim, fix a smooth cut-off function
$\phi \in C_c^{\infty}(\mathbb R^3)$ such that $\phi(x)= 1$ for $|x|\le 1$ and $\phi(x)=0$ for $|x| \ge 2$.
Since $|x_j-x_k|\gg 1$ for $j\ne k$, by \eqref{x9_3}, we have for any
$j\ge 1$, we have
\begin{align*}
W^j(\tau_0,x)=W(\tau_0,x) \phi(x-x_j)=W(\tau_0,x)\phi_j(x), \quad \text{here $\phi_j(x):=\phi(x-x_j)$}.
\end{align*}
Fourier transform and the triangle inequality then give
\begin{align*}
|\xi|^{\frac 32} |\widehat{W^j}(\tau_0,\xi)| &\lesssim |\xi|^{\frac 32}\int_{\mathbb R^3}
|\hat{W}(\tau_0,\xi-\eta)| |\hat{\phi_j}(\eta)| d\eta \notag \\
& \lesssim \int_{\mathbb R^3} |\xi-\eta|^{\frac 32} |\hat{W}(\tau_0,\xi-\eta)| |\hat{\phi_j}(\eta)| d\eta \notag \\
& \qquad + \int_{\mathbb R^3} |\hat{W}(\tau_0,\xi-\eta)| |\eta|^{\frac 32} |\hat{\phi_j}(\eta)| d\eta.
\end{align*}
Young's inequality then gives for any $j\ge 1$,
\begin{align*}
\| W^j(\tau_0,\cdot)\|_{\dot H^{\frac 32}} \lesssim \| W(\tau_0,\cdot)\|_{\dot H^{\frac 32}} +\| W(\tau_0,\cdot)\|_{L^2}.
\end{align*}
Easy to check that the implied constants in the above inequalities are only absolute constants (they depend only
on the cut-off function $\phi$). By \eqref{x9_4} and choosing $\epsilon=1$, we get for any $j\ge J_1$,
\begin{align*}
\|\omega^j(\tau_0,\cdot)\|_{\dot H^{\frac 32}} \le \tilde C_1 \| W(\tau_0,\cdot)\|_{\dot H^{\frac 32}} +\tilde C_2,
\end{align*}
where $\tilde C_1>0$ is an absolute constant and $\tilde C_2$ depends only on
$\max_{0\le t\le 1}\|W(t,\cdot)\|_{L^2}$. The claim is proved.

With \eqref{sec6_claim_i1} in hand, we now argue by contradiction to finish the proof of the theorem.
Assume for some $t_0<1$, we have
\begin{align} \label{sec6_thm_pf_L0}
 L_0:=\operatorname{ess-sup}_{0\le t\le t_0} \| W(t,\cdot)\|_{\dot H^{\frac 32}}<\infty.
\end{align}

By \eqref{sec6_6.104}, we choose $j\gg 1$ sufficiently large such that
\begin{align*}
&C_1 2^j-C_2>2L_0, \\
&t_j^2<t_0.
\end{align*}
By \eqref{sec6_claim_i1}, for any $t_j^1\le t\le t_j^2$, we must have
\begin{align*}
2L_0\le \|W(t,\cdot)\|_{\dot H^{\frac 32}}<\infty, \quad\text{or $\|W(t,\cdot)\|_{\dot H^{\frac 32}}=+\infty$}.
\end{align*}
This obviously contradicts \eqref{sec6_thm_pf_L0}. The theorem is proved.

\end{proof}

\section{3D compactly supported case}

\begin{lem} \label{y1}
Let $f\in C_c^{\infty}(B(0,100))$, $g\in C_c^{\infty}(B(0,100))$ be
axisymmetric functions on $\mathbb R^3$ having the form:
\begin{align*}
f(x)=f^{\theta}(r,z) e_{\theta},\quad g(x)=g^{\theta}(r,z)
e_{\theta},\quad x=(x_1,x_2,z),\, r=\sqrt{x_1^2+x_2^2},
\end{align*}
where $f^{\theta}$ and $g^{\theta}$ are scalar-valued and vanish
near $r=0$, i.e. for some $r_0>0$,
\begin{align*}
\operatorname{supp}(f^{\theta}) \subset\{(r,z):\, r>r_0\},\\
\operatorname{supp}(g^{\theta}) \subset\{(r,z):\, r>r_0\}.
\end{align*}

Let $\omega^a$ and $\omega$ be smooth solutions to the following
axisymmetric (without swirl) Euler equations:
\begin{align} \label{y1_t1}
\begin{cases}
\partial_t (\frac{\omega^a}r) + (u^a \cdot \nabla) ( \frac{\omega^a}
r) =0,  \\
u^a=-\Delta^{-1} \nabla \times \omega^a, \\
\omega^a \Bigr|_{t=0} =f.
\end{cases}
\end{align}
\begin{align} \label{y1_t2}
\begin{cases}
\partial_t (\frac{\omega}r) + (u \cdot \nabla) ( \frac{\omega}
r) =0, \\
u=-\Delta^{-1} \nabla \times \omega, \\
\omega \Bigr|_{t=0} =f+g.
\end{cases}
\end{align}

For any $\epsilon>0$, there exists $\delta=\delta(\epsilon, f)>0$
sufficiently small such that if
\begin{align} \label{y1_t3}
\|g\|_{\infty} \exp\left( C \cdot \| \frac g r \|_{L^{3,1}}
\right) <\delta,
\end{align}
then
\begin{align*}
\max_{0\le t\le 1}\| \omega^a(t,\cdot) -\omega(t,\cdot)\|_{\infty}
<\epsilon.
\end{align*}

Here in \eqref{y1_t3}, $C>0$ is the same absolute constant as in the inequality
\begin{align*}
 \| \frac{u^r} r \|_{\infty} \le C \| \frac{\omega} r \|_{L^{3,1}}.
\end{align*}

\end{lem}

\begin{proof}[Proof of Lemma \ref{y1}]
In this proof our main ``smallness'' parameter
is $\|g\|_{\infty}$. To simplify the notations,  we shall denote
$X=O(\delta)$ if the quantity $X$ can be made arbitrarily small depending on
$\|g\|_{\infty}$.  For example we shall write $X=O(\delta)$ if $X$ satisfies the inequality of the following sort:
\begin{align*}
|X| \lesssim_{f} \| g\|_{\infty} \exp( C \| \frac g r \|_{L^{3,1}} ).
\end{align*}
Here $C>0$ is some absolute constant. Other inequalities similar to the above will all be denoted by the
same notation $O(\delta)$ whenever there is no confusion.  We shall denote $X=O(1)$ if
\begin{align*}
 X \lesssim_{f} 1.
\end{align*}

We first decompose the solution to \eqref{y1_t2} as
\begin{align*}
 \omega=\omega^1+\omega^2,
\end{align*}
where $\omega^1$, $\omega^2$ solve the \emph{linear} systems
\begin{align} \label{y1_1a}
 \begin{cases}
 \partial_t (\frac{\omega^1} r) + (u\cdot \nabla) (\frac{\omega^1} r) =0, \\
 \omega^1 \Bigr|_{t=0} = f,
\end{cases}
\end{align}

\begin{align} \label{y1_1b}
 \begin{cases}
 \partial_t (\frac{\omega^2} r) + (u\cdot \nabla) (\frac{\omega^2} r) =0,  \\
 \omega^2 \Bigr|_{t=0} = g.
\end{cases}
\end{align}

Consider first \eqref{y1_1b}. Since in \eqref{y1_t2}, we have
\begin{align*}
 \| \frac {\omega(t)} r \|_{L^{3,1}} &= \|\frac {\omega(0)} r \|_{L^{3,1}} \notag \\
& \le \| \frac f r \|_{L^{3,1}} + \| \frac g r \|_{L^{3,1}}, \quad \forall\, t\ge 0.
\end{align*}

Recalling $u=u^r e_r +u^{z} e_z$, we get
\begin{align*}
 \| \frac {u^r(t)} r \|_{\infty} &\lesssim \| \frac{\omega(t)} r \|_{L^{3,1}} \notag \\
& \lesssim \| \frac f r \|_{L^{3,1}} + \| \frac g r \|_{L^{3,1}}, \quad \forall\, t\ge 0.
\end{align*}

Rewrite \eqref{y1_1b} as
\begin{align*}
 \partial_t \omega^2 + (u\cdot \nabla) \omega^2 = \frac{u^r} r \omega^2.
\end{align*}

We obtain for some absolute constant $C>0$,
\begin{align}
\max_{0\le t\le 1} \| \omega^2(t)\|_{\infty} &\le \| g\|_{\infty} \exp\left( C \max_{0\le t\le 1} \| \frac{u^r(t)} r \|_{\infty} \right)
\notag \\
& \le \| g\|_{\infty} \exp\left( C (\| \frac f r \|_{L^{3,1}} + \| \frac g r \|_{L^{3,1}} ) \right) \notag \\
& = O(\delta).  \label{y1_1c}
\end{align}
Thus we only need to control $\|\omega^1-\omega^a\|_{\infty}$.

Set $\eta= \omega^a-\omega^1$.Denote by $u^1$, $u^2$ the velocity fields corresponding to $\omega^1$, $\omega^2$ respectively.
We first show that
\begin{align} \notag
 \max_{0\le t\le 1}\| \eta(t,\cdot)\|_2 = O(\delta).
\end{align}

 Rewrite \eqref{y1_1a} as
\begin{align} \label{y1_3}
 \partial_t \omega^1 + (u^1 \cdot \nabla) \omega^1 = (\omega^1 \cdot \nabla) u^1 -(u^2 \cdot \nabla )\omega^1 +
 (\omega^1\cdot \nabla) u^2.
\end{align}
By \eqref{y1_t1}, we have
\begin{align*}
 \partial_t \omega^a + (u^a \cdot \nabla) \omega^a = (\omega^a \cdot \nabla )u^a.
\end{align*}
Therefore the equation for $\eta$ takes the form
\begin{align} \label{y1_3a}
 &\partial_t \eta  + ( (u^a-u^1) \cdot \nabla) \omega^a + (u^1 \cdot \nabla) \eta \notag \\
  =& (\eta\cdot \nabla) u^a + (\omega^1 \cdot \nabla) (u^a-u^1) \notag \\
 & \qquad + (u^2\cdot\nabla)\omega^a-(u^2\cdot \nabla) \eta -(\omega^1 \cdot \nabla ) u^2.
\end{align}
Computing the $L^2$ norm then gives
\begin{align}
 \partial_t ( \|\eta(t)\|_2^2 )
 & \lesssim \| u^a-u^1 \|_6 \cdot \| D \omega^a \|_3 \cdot \|\eta\|_2 +
 \| \eta\|_2^2 \cdot \| D u^a \|_{\infty}   \notag \\
&\qquad  + \| \omega^1 \|_{\infty } \cdot \| D(u^a-u^1) \|_2 \cdot \| \eta\|_2  \notag \\
 & \qquad + \| u^2 \|_6 \cdot \| D \omega^a \|_3 \cdot \| \eta\|_2
 + \| \omega^1 \|_{\infty} \cdot \| D u^2 \|_2 \cdot \| \eta\|_2 \notag \\
 & \lesssim ( \|D \omega^a \|_3 + \| D u^a \|_{\infty} + \| \omega^1 \|_{\infty}) \| \eta\|_2^2 \notag \\
 & \qquad + (\| u^2\|_6 \| D\omega^a \|_3 + \| \omega^1 \|_{\infty} \| Du^2\|_2) \| \eta\|_2, \notag \\
 & \lesssim (O(1) + \|\omega^1 \|_{\infty} ) \|\eta\|_2^2 \notag \\
 &\qquad + (O(1) \cdot \|Du^2\|_2 + \|\omega^1 \|_{\infty} \|Du^2\|_2) \|\eta\|_2. \label{y1_4}
\end{align}

By an estimate similar to \eqref{y1_1c}, we have
\begin{align} \label{y1_4a}
\max_{0\le t\le 1}\|Du^2(t)\|_2 \lesssim  \max_{0\le t\le 1} \| \omega^2(t)\|_{2}=O(\delta).
\end{align}

Similarly by Sobolev embedding,
\begin{align} \label{y1_4aa}
\max_{0\le t\le 1}\|u^2(t)\|_{2} \lesssim \max_{0\le t \le 1} \| \omega^2(t)\|_{\frac 65} =O(\delta).
\end{align}
This together with \eqref{y1_1c} gives
\begin{align} \label{y1_4b}
\max_{0\le t\le 1} \| u^2(t)\|_{\infty}  &\lesssim \max_{0\le t\le 1} \| u^2(t)\|_2 + \max_{0\le t\le 1} \| \omega^2(t)\|_{\infty}
\notag \\
&=O(\delta).
\end{align}

By \eqref{y1_1a}, we have
\begin{align}
 & \| \frac{\omega^1(t)} r \|_{L^{3,1}} = \| \frac f r \|_{L^{3,1}} = O(1), \quad \forall\, t\ge 0, \notag \\
 & \max_{0\le t\le 1} \| \frac {(u^1(t))^r} r \|_{\infty} \lesssim \| \frac{\omega^1(t)} r \|_{L^{3,1}}=O(1), \notag \\
 & \max_{0\le t\le 1} \| \frac{\omega^1(t)} r \|_{\infty} \le \| \frac f r \|_{\infty} =O(1). \label{y1_5}
\end{align}
Here we write $u^1=(u^1)^r e_r + (u^1)^z e_z$.

 Rewrite \eqref{y1_3} as
\begin{align} \notag
 \partial_t \omega^1 + (u \cdot \nabla) \omega^1 = \frac{(u^1)^r}r \omega^1 +
 {(u^2)}^r  \frac{\omega^1} r.
\end{align}
Using \eqref{y1_4b} and \eqref{y1_5}, we get

\begin{align} \label{y1_7}
 \max_{0\le t\le 1} \|\omega^1(t)\|_{\infty} =O(1)
\end{align}

Plugging \eqref{y1_4a} and \eqref{y1_7} into \eqref{y1_4}, we obtain
\begin{align} \label{y1_9}
 \max_{0\le t\le 1} \| \eta(t,\cdot)\|_2 = O(\delta).
\end{align}

By \eqref{y1_7}--\eqref{y1_9} and H\"older, we get
\begin{align} \label{y1_10}
\max_{0\le t\le 1} \| \omega^a(t)-\omega^1(t)\|_4 &\lesssim \max_{0\le t\le 1 } ( \|\eta(t)\|_2^{\frac 12}  \| \eta(t)\|_{\infty}^{\frac 12} )\notag \\
& = O(\delta)
\end{align}

By $L^2$-conservation of velocity for \eqref{y1_t2}, we have
\begin{align*}
 \|u(t,\cdot)\|_2 =\| u(0)\|_2  &\lesssim \| f \|_{\dot H^{-1}} +\|g \|_{\dot H^{-1}} \notag \\
& \lesssim \| f \|_1 + \| f\|_{\infty} + \| g\|_1 + \| g\|_{\infty} \notag \\
& \lesssim \| f \|_{\infty} + \| g\|_{\infty} = O(1), \qquad \forall\, t\ge 0.
\end{align*}
Therefore by \eqref{y1_4aa},
\begin{align} \label{y1_11}
\max_{0\le t\le 1} \| u^a(t) -u^1(t)\|_2 &\lesssim  \max_{0\le t\le 1} (\|u^a(t)\|_2 + \|u^2(t)\|_2 + \|u(t)\|_2) \notag \\
& =O(1).
\end{align}

By \eqref{y1_10}--\eqref{y1_11} and interpolation, we get
\begin{align} \label{y1_13}
 \max_{0\le t\le 1} \| u^a(t) -u^1(t)\|_{\infty} & \lesssim \max_{0\le t\le 1} \| u^a(t)-u^1(t)\|_2^{\frac 17} \cdot
 \| \omega^a(t) -\omega^1(t)\|_4^{\frac 67} \notag \\
 & =O(\delta).
\end{align}

Now using \eqref{y1_4b}, \eqref{y1_5} and \eqref{y1_13}, we can rewrite \eqref{y1_3a} as
\begin{align*}
 &\partial_t \eta + (u\cdot \nabla) \eta \notag \\
 = & \, - ( (u^a-u^1) \cdot \nabla) \omega^a + (\eta \cdot \nabla) u^a \notag \\
 & \quad + (u^a-u^1)^r \frac{\omega^1} r + (u^2 \cdot \nabla) \omega^a - (u^2)^r \frac {\omega^1}r \notag \\
= & O(\delta) + O(1)\eta + O(\delta) \cdot O(1).
\end{align*}
Obviously then
\begin{align*}
 \max_{0\le t\le 1} \|\eta(t)\|_{\infty} =O(\delta).
\end{align*}

The lemma is proved.
\end{proof}

We now state a proposition which gives the solvability of the 3D axisymmetric without swirl Euler equation
for a special class of initial data (vorticity). In particular we allow initial vorticity (denote it by $\omega_0$) to
carry infinite $\| \frac {\omega_0} r\|_{L^{3,1}}$ norm which is not covered by standard theory. The trade off here
is that we need a precise control of $L^{\infty}$-norm in the sense of Lemma \ref{y1}.

\begin{prop} \label{y2}
 Suppose $\{g_i\}_{i=1}^{\infty}$ is a sequence of axisymmetric functions on $\mathbb R^3$ satisfying the following conditions:
 \begin{itemize}
  \item For each $i\ge 1$, $g_i(x)=g_i^{\theta}(r,z) e_{\theta}$, where $g_i^{\theta}$ is scalar-valued and vanishes near $r=0$:
  \begin{align*}
   \operatorname{supp}(g_i^{\theta}) \subset \{(r,z): \, r>r_i\},\quad \text{for some $r_i>0$}.
  \end{align*}
 \item $g_i\in C_c^{\infty} (B(0,100))$ and $\| g_i\|_{\infty} <2^{-i}$.
 \item For each $i\ge 2$, denote $f_i =\sum_{j=1}^{i-1} g_j$, then
 \begin{align*}
  \| g_i\|_{\infty} \exp\left( C \| \frac {g_i} r \|_{L^{3,1}} \right) <\delta_{i},
 \end{align*}
where $\delta_i=\delta(2^{-i}, f_i)$ as defined in \eqref{y1_t3}.

 \end{itemize}
Let
\begin{align*}
 g= \sum_{i=1}^{\infty} g_i
\end{align*}
and consider the system
\begin{align} \label{y2_t3}
 \begin{cases}
  \partial_t \omega + (u\cdot \nabla)\omega=(\omega\cdot \nabla)u, \quad 0<t\le 1; \\
  u=-\Delta^{-1} \nabla \times \omega, \\
  \omega \Bigr|_{t=0}=g.
 \end{cases}
\end{align}

Then there exists a unique solution $\omega$ to \eqref{y2_t3} with the following properties:
\begin{enumerate}
 \item $\omega$ is compactly supported:
 \begin{align*}
  \operatorname{supp}(\omega(t,\cdot)) \subset B(0,R_0), \quad \forall\, 0<t\le 1.
 \end{align*}
Here $R_0>0$ is an absolute constant.

\item $\omega \in C_t^0 C_x^0([0,1] \times \overline{B(0,R_0)})$,
$u \in C_t^0 L_x^2 \cap L_t^{\infty} L_x^{\infty} ([0,1]\times \mathbb R^3)$. In fact $u\in C_t^0 C_x^{\alpha}([0,1]\times \mathbb R^3)$
for any $0<\alpha<1$.

\end{enumerate}

\end{prop}

\begin{proof}[Proof of Proposition \ref{y2}]
For each $l\ge 1$, let $\omega^l$ be the solution to the system
\begin{align*}
 \begin{cases}
  \partial_t (\frac{\omega^l} r ) + (u^l \cdot \nabla) ( \frac{\omega^l} r ) =0, \quad 0<t\le 1, \\
   u^l =-\Delta^{-1} \nabla \times \omega^l, \\
  \omega^l \Bigr|_{t=0} = \sum_{i=1}^l g_i.
 \end{cases}
\end{align*}
By Lemma \ref{y1} and the assumptions on $g_i$, we have
\begin{align} \label{y2_1}
 \max_{0\le t\le 1} \| \omega^{l+1}(t) -\omega^l(t) \|_{\infty} <2^{-l}.
\end{align}

Noting that
\begin{align*}
 \max_{0\le t\le 1} \| \omega^1 (t)\|_{\infty} & \lesssim \| g_1\|_{\infty} \exp \left( Const \cdot \| \frac{g_1} r \|_{L^{3,1}} \right) \notag \\
 & \lesssim 1,
\end{align*}
we obtain
\begin{align*}
 \sup_{l\ge 1} \max_{0\le t\le 1} \| \omega^l(t)\|_{\infty} \lesssim 1.
\end{align*}

By energy conservation, we have
\begin{align*}
 \| u^l(t) \|_2 &= \|u^l(0)\|_2 \lesssim \| \omega^l(0)\|_1 + \| \omega^l(0)\|_{\infty} \notag \\
 & \lesssim 1, \qquad \forall\, t\ge 0, \, l\ge 1.
\end{align*}

Therefore
\begin{align*}
 \sup_{l\ge 1}\max_{0\le t\le 1}\| u^l(t)\|_{\infty} \lesssim 1.
\end{align*}

This shows that for some absolute constant $R_0>0$, we have
\begin{align*}
 \omega^l (t) \in C_c^{\infty} (B(0,R_0)), \quad \forall\, 0<t\le 1,\, l\ge 1.
\end{align*}

By \eqref{y2_1}, the sequence $\omega^l$ is Cauchy in the Banach space $C_t^0C_x^0([0,1] \times \overline{B(0,R_0)})$ and
hence converges to the limit point $\omega$ in the same space. By interpolation and Sobolev embedding it is not difficult to check that
$u^l$ converges to $u\in C_t^0 L_x^2$. By Sobolev embedding we get $u \in L_{t}^{\infty} L_x^{\infty} \cap C_t^0 C_x^{\alpha}$ for any
$\alpha<1$. It is not difficult to check that $\omega$ is the desired solution.
The proposition is proved.

\end{proof}

We now take a parameter $A\gg 1$ and define
\begin{align}
\tilde g_A(x_1,x_2,z)  & = \tilde g_A (r,z) \notag \\
& = \frac{\sqrt{\log A}} {\sqrt A } \sum_{A\le k \le A+\sqrt A} \eta_k(r,z), \label{e_p63_1}
\end{align}
where $\eta_k$ is the same as in \eqref{x2_1a}. Note the slight difference between
$\tilde g_A$ and $g_A$ defined in \eqref{x2_1}. The main reason of choosing $\tilde g_A$ is that in the perturbation
theory later we need
better control of higher Sobolev norms of the solution, i.e. estimates like $\| \tilde g_A \|_{W^{1,q}}
\lesssim 2^{A+}$, for all $3<q\le \infty$. In comparison $\|g_A\|_{W^{1,q}} \sim 2^{2A}$ since there we are
summing $\eta_k$ over $k\le 2A$. This is why the modification is needed.

By a derivation similar to \eqref{x2_3b}--\eqref{x2_3a}, easy to check
\begin{align}
&\| \tilde g_A e_{\theta} \|_{\dot B^{\frac 32}_{2,1}(\mathbb R^3)}
+ \| \frac{\tilde g_A}r e_{\theta} \|_{L^{3,1}(\mathbb R^3)} \lesssim \sqrt{\log A}, \notag \\
&\|\tilde g_A \|_{\dot H^{\frac 32}(\mathbb R^3)} + \| \tilde g_A e_{\theta} \|_{\dot H^{\frac 32}(\mathbb R^3)}
\lesssim \frac{\sqrt{\log A}}{A^{\frac 14}}, \notag \\
& \| \tilde g_A \|_{L^p(\mathbb R^3)} + \| \tilde g_A e_{\theta} \|_{L^p(\mathbb R^3)}
\lesssim \frac{\sqrt{\log A}}{\sqrt A} \cdot 2^{-\frac{3A}p}, \quad \forall\, 1\le p\le \infty, \notag \\
& \| D( \tilde g_A e_{\theta} ) \|_{L^p(\mathbb R^3)} \lesssim \sqrt{\log A}
\cdot 2^{(1-\frac 3p)(A+\sqrt A)}, \quad \forall\, 3\le p\le \infty. \label{e_p63_2}
\end{align}
These estimates will be needed later.

\begin{lem} \label{y3}
Let $\omega$ be a smooth solution to the following system (written in axisymmetric vorticity form)
\begin{align*}
 \begin{cases}
  \partial_t ( \frac{\omega} r) + (u + u^{\text{ex}} )\cdot \nabla ( \frac{\omega} r) =0, \quad 0<t\le 1, r=\sqrt{x_1^2+x_2^2}, x=(x_1,x_2,z), \\
  u=-\Delta^{-1} \nabla \times \omega, \\
  \omega \Bigr|_{t=0} =\tilde g_A e_{\theta},
 \end{cases}
\end{align*}
where $\tilde g_A$ was the same as in \eqref{e_p63_1} and $u^{\text{ex}}$ is a given axisymmetric velocity field having the form
(note that it is incompressible)
\begin{align} \label{y3_t0}
 u^{\text{ex}} (t,r,z) = a(t) r e_r - 2a(t) z e_z.
\end{align}
Assume for some constant $B_0>0$,
\begin{align} \label{y3_t1}
 \sup_{0\le t\le 1} |a(t)| \le B_0 <\infty.
\end{align}

Let $\phi=(\phi^r,\phi^z)$ be the forward characteristic lines associated with the velocity $u+u^{\text{ex}}$ (see \eqref{x-1_00})
and let $\tilde \phi$ be the corresponding inverse map. Then there exists $A_0=A_0(B_0)>0$ such that if $A>A_0$, then
\begin{align} \label{y3_t2}
 \max_{0\le t \le \frac{1} {\log\log A}} \| (D\tilde \phi)(t,\cdot) \|_{\infty} > \log\log A.
\end{align}
\end{lem}

\begin{proof}[Proof of Lemma \ref{y3}]
 We shall give a slightly simpler proof than that given in Proposition \ref{x5}. The idea is to take full advantage of the symmetry assumption
 and the fact that the off-diagonal terms of $Du$ vanishes completely at $(r,z)=(0,0)$.  Assume \eqref{y3_t2} does not hold. By a derivation similar to \eqref{x7_700a}--\eqref{x7_700b}, we have
 \begin{align} \label{y3_1}
  \max_{0\le t \le \frac 1{\log\log A} } ( \| (D\tilde \phi)(t,\cdot) \|_{\infty} + \| (D\phi)(t,\cdot)\|_{\infty} )
  \lesssim (\log\log A)^2.
 \end{align}

By \eqref{e_p63_1} and \eqref{x2_2}, observe that $\tilde g_A$ is an odd function of $z$. Denote $v(t,r,z)= u+u^{ex}$ and
\begin{align*}
 &u^{\text{ex}} = (u^{\text{ex}})^r e_r + (u^{\text{ex}})^z e_z, \\
 &v=v^r e_r + v^z e_z.
\end{align*}
Easy to check that for any $t\ge 0$, $\omega(t)$ remains an odd function of $z$, and also
\begin{align*}
 & v^r (t,0,z) =v^z(t,r,0)=0, \notag \\
 & \phi^r (t,0,z)= \phi^z(t,r,0)=0, \quad \forall\, r\ge 0, \, z\in \mathbb R, \, t\ge 0.
\end{align*}
Clearly then
\begin{align*}
& (\partial_r v^z)(t,0,0) \equiv 0 \equiv (\partial_z v^r)(t,0,0), \notag \\
& (\partial_r \phi^z)(t,0,0) \equiv 0 \equiv (\partial_z \phi^r)(t,0,0), \quad \forall\, t\ge 0.
\end{align*}

Since $v$ is divergence-free (see \eqref{y3_t0}), we have
\begin{align*}
 2(\partial_r v^r)(t,0,0) + (\partial_z v^z)(t,0,0) \equiv 0, \quad \forall\, t\ge 0.
\end{align*}

From the above identities, we then easily obtain
\begin{align}
 & (\partial_r \phi^r)(t,0,0) = e^{\int_0^t (\partial_r v^r)(s,0,0) ds}, \notag \\
& (\partial_z \phi^z)(t,0,0) = e^{\int_0^t (\partial_z v^z)(s,0,0) ds} = e^{-2 \int_0^t (\partial_r v^r)(s,0,0) ds}.
\label{y3_3}
\end{align}

By \eqref{x4_2_new} (easy to check that same estimate holds with $g_A$ replaced
by $\tilde g_A$), \eqref{y3_1} and \eqref{y3_t1}, we get
\begin{align*}
 (\partial_r v^r)(t,0,0) & =  (\partial_r u^r)(t,0,0) + ( \partial_r (u^{ex})^r )(t,0,0) \notag \\
 & \gtrsim \sqrt{\log A} (\log\log A)^{-16} -B_0.
\end{align*}

Plugging this into \eqref{y3_3} then gives us
\begin{align*}
 (\partial_r \phi^r )(t,0,0) \gtrsim e^{t \frac{\sqrt {\log A} } { (\log\log A)^{16}} -2B_0}.
\end{align*}

This obviously contradicts \eqref{y3_1} for $t=\frac 1{\log\log A}$ and $A$ sufficiently large.

\end{proof}

\begin{lem}[Control of the support] \label{y4}
 Let $\omega$ be a smooth solution to the following system (written in axisymmetric vorticity form)
 \begin{align*}
 \begin{cases}
  \partial_t ( \frac{\omega} r) + (u + u_1+u_2 )\cdot \nabla ( \frac{\omega} r) =0, \quad 0<t\le 1, r=\sqrt{x_1^2+x_2^2}, x=(x_1,x_2,z), \\
  u=-\Delta^{-1} \nabla \times \omega, \\
  \omega \Bigr|_{t=0} =\tilde g_A e_{\theta},
 \end{cases}
\end{align*}

where the following conditions hold:
\begin{itemize}
 \item $\tilde g_A$ is the same as in \eqref{e_p63_1};
 \item $u_1$ and $u_2$ are given smooth incompressible axisymmetric vector fields having the form
 \begin{align*}
  &u_1=a(t) r e_r - 2a(t) z e_z, \\
  &u_2=u_2^r e_r + u_2^z e_z,
 \end{align*}
and for some constant $B>0$,
\begin{align}
& \sup_{0\le t\le 1}|a(t)| \le B, \notag \\
& \sup_{0\le t \le 1} \| \frac{u_2^r(t)} r \|_{\infty} \le B, \notag \\
& |u_2(t,x)| \le B|x|^2, \quad \forall\, x \in \mathbb R^2, \quad 0\le t\le 1. \label{y4_t1}
\end{align}
\end{itemize}

Then there exists a constant $A_0=A_0(B)>0$ sufficiently large such that if $A>A_0$, then
for any $0\le t \le 1$,
\begin{align} \label{y4_t2a}
 \operatorname{supp} (\omega(t,\cdot) ) \subset B(0,R), \qquad \text{with $R\le C_1\cdot 2^{-A}$},
\end{align}
where $C_1>0$ is a constant depending on $B$.

Also for any $0\le t\le \frac 1 {\log\log A}$, we have
\begin{align} \label{y4_t2}
 \operatorname{supp} (\omega(t,\cdot) ) \subset B(0,R), \qquad \text{with $R\sim 2^{-A}$},
\end{align}
where the implied constants (in $R\sim 2^{-A}$) are absolute constants.
\end{lem}

\begin{proof}[Proof of Lemma \ref{y4}]
 Denote $v=u+u_1+u_2$ and write
 \begin{align*}
  &v=v^r e_r + v^z e_z, \\
  &u=u^r e_r +u^z e_z.
 \end{align*}

By $L^{3,1}$ conservation of $\frac{\omega} r$ and \eqref{e_p63_2}, we have
\begin{align*}
 \sup_{0\le t\le 1} \| \frac{u^r(t)} r \|_{\infty} \lesssim \| \frac{\omega(t=0,\cdot)} r \|_{L^{3,1}} \lesssim \sqrt{\log A}.
\end{align*}

By \eqref{y4_t1}, we get
\begin{align}\label{y4_1}
 \sup_{0\le t\le 1} \| \frac{v^r(t)} r \|_{\infty} \lesssim B + \sqrt{\log A}.
\end{align}

Rewrite the equation for $\omega$ as
\begin{align*}
 \partial_t \omega + (v\cdot \nabla) \omega= \frac{v^r} r \omega.
\end{align*}

By \eqref{y4_1}, a simple $L^p$-estimate (note that $v$ is incompressible) then gives
\begin{align}
 &\sup_{0\le t \le 1} \| \omega(t)\|_2 \lesssim e^{ C  (B+\sqrt{\log A})} \| g_A \|_2, \notag \\
 &\sup_{0\le t \le 1 } \| \omega(t)\|_4 \lesssim e^{C  (B+\sqrt{\log A})} \| g_A \|_4, \label{y4_2}
\end{align}
where $C>0$ is an absolute constant.

Note that by \eqref{e_p63_2}, we have
\begin{align*}
 &\|\tilde g_A \|_2 \lesssim \frac{\sqrt{\log A}} {\sqrt A} 2^{-\frac 32 A},\\
 &\|\tilde g_A \|_4 \lesssim \frac{\sqrt{\log A}} {\sqrt A} 2^{-\frac 34 A}.
\end{align*}

By \eqref{y4_2} and interpolation, we then get
\begin{align}
 \sup_{0\le t \le 1} \| u(t)\|_{\infty}
 & \lesssim \sup_{0\le t \le 1} \| \omega(t)\|_2^{\frac 13} \cdot \| \omega(t)\|_4^{\frac 23} \notag \\
 & \lesssim e^{C (B+\sqrt{\log A})} \cdot \frac{\sqrt{\log A}} A \cdot 2^{-A} \notag \\
 & <  {A^{-\frac 13}} 2^{-A}, \label{y4_3}
\end{align}
where in the last inequality we need to take $A$ sufficiently large.

Denote $\phi$ as the (usual Euclidean) characteristic line associated with the velocity $v$. Then by \eqref{y4_t1} and \eqref{y4_3}, we get
\begin{align*}
 \frac d {dt} ( |\phi(t)|) \lesssim A^{-\frac 13} 2^{-A} + B |\phi(t)| + B |\phi(t)|^2.
\end{align*}
Since $|\phi(0)| \lesssim 2^{-A}$, obviously \eqref{y4_t2a} and
\eqref{y4_t2} follows (in the latter case since $t\le \frac 1 {\log
\log A}$ we can take $A$ sufficiently large to kill pre-factors).
\end{proof}

\begin{lem} \label{y5}
 Let $\omega$ be a smooth solution to the following system (written in axisymmetric vorticity form)
 \begin{align*}
 \begin{cases}
  \partial_t ( \frac{W} r) + (U+ u_1+u_2 )\cdot \nabla ( \frac{W} r) =0, \quad 0<t\le 1, r=\sqrt{x_1^2+x_2^2}, x=(x_1,x_2,z), \\
  U=-\Delta^{-1} \nabla \times W, \\
  W \Bigr|_{t=0} =\tilde g_A e_{\theta},
 \end{cases}
\end{align*}

where the following conditions hold:
\begin{itemize}
 \item $\tilde g_A$ is the same as in \eqref{e_p63_1};
 \item $u_1$ and $u_2$ are given smooth incompressible axisymmetric vector fields having the form
 \begin{align*}
  &u_1=a(t) r e_r - 2a(t) z e_z, \\
  &u_2=u_2^r e_r + u_2^z e_z,
 \end{align*}
and for some constant $B>0$,
\begin{align}
& \sup_{0\le t\le 1}|a(t)| \le B, \notag \\
  &| u_2(t,x)| \le B \cdot |x|^2, \notag \\
  & |(Du_2)(t,x)| \le B\cdot |x|, \notag \\
  & |(D^2 u_2)(t,x)| \le B, \quad \forall\, x \in \mathbb R^3, \, 0\le t\le 1.  \label{y5_t1_aa}
 \end{align}
\end{itemize}

Let $\Phi=(\Phi^r,\Phi^z)$ be the characteristic line associated with the velocity field $U+u_1+u_2$ (see \eqref{x-1_00}) and
let $\tilde \Phi$ be the corresponding inverse map.

There exists a constant $A_0=A_0(B)>0$ sufficiently large such that if $A>A_0$, then
either
\begin{align} \label{y5_t2a}
 \max_{0\le t\le \frac 1 {\log\log A}} \| W(t,\cdot) \|_{\dot H^{\frac 32}} >{\log \log\log A},
\end{align}
or
\begin{align} \label{y5_t2b}
 \max_{0\le t \le \frac 1 {\log\log A}} \| (D\tilde \Phi) (t,\cdot)\|_{\infty} > {\log \log\log A}.
\end{align}

\end{lem}

\begin{proof}[Proof of Lemma \ref{y5}]
By Lemma \ref{y4} and \eqref{y5_t1_aa}, we have $\operatorname{supp}(W(t,\cdot)) \subset \{x:\, |x| \lesssim 2^{-A}\}$ and
\begin{align}
& \| u_2(t,\cdot) \|_{L^{\infty}(\operatorname{supp} (W(t,\cdot) ))} \lesssim  4^{-A}, \quad \forall\, 0\le t\le 1, \notag \\
& \| (Du_2)(t,\cdot) \|_{L^{\infty}(\operatorname{supp} (W(t,\cdot)))} \lesssim  2^{-A}, \quad \forall\, 0\le t\le 1, \notag \\
& \| (D^2 u_2)(t,\cdot) \|_{L^{\infty}(\operatorname{supp} (W(t,\cdot)))} \lesssim 1, \quad \forall\, 0\le t\le 1.
\label{y5_t1}
\end{align}
Throughout this proof we suppress the dependence of the implied constants on $B$ since $A$ will be taken sufficiently large.

Assume that both \eqref{y5_t2a} and \eqref{y5_t2b} do not hold, i.e.
 \begin{align} \label{y5_0}
 \max_{0\le t\le \frac 1 {\log\log A}} \| W(t,\cdot) \|_{\dot H^{\frac 32}} +
 \max_{0\le t \le \frac 1 {\log\log A}} \| (D\tilde \Phi) (t,\cdot)\|_{\infty} \lesssim {\log \log\log A}.
\end{align}
Easy to check that
\begin{align} \label{y5_0a}
\max_{0\le t \le \frac 1 {\log\log A}} \| (D \Phi) (t,\cdot)\|_{\infty} \lesssim {(\log \log\log A)^2}.
\end{align}
 We shall derive a contradiction. The idea is to
 compare $W$ with the other solution $\omega$ to the following ``unperturbed'' system
 \begin{align*}
 \begin{cases}
  \partial_t ( \frac{\omega} r) + (u + u_1 )\cdot \nabla ( \frac{\omega} r) =0, \quad 0<t\le 1, r=\sqrt{x_1^2+x_2^2}, x=(x_1,x_2,z), \\
  u=-\Delta^{-1} \nabla \times \omega, \\
  \omega \Bigr|_{t=0} =\tilde g_A e_{\theta}.
 \end{cases}
\end{align*}
By using the conservation of $\|\frac W r \|_{L^{3,1}}$ and $\| \frac{\omega} r \|_{L^{3,1}}$ respectively, it is not difficult
to check that
\begin{align}
 & \sup_{0\le t\le 1} \| \frac{U^r(t)} r \|_{\infty} \lesssim \sqrt{\log A}, \notag \\
 & \sup_{0\le t \le 1} \| \frac{u^r(t)} r \|_{\infty} \lesssim \sqrt{\log A}, \notag \\
 & \sup_{0\le t \le 1} \| W(t)\|_q  \lesssim \frac {\sqrt{\log A}}  {\sqrt A} 2^{-\frac{3A} q}, \quad \forall\, 1<q\le \infty, \notag \\
 & \sup_{0\le t \le 1} \| \omega(t) \|_q \lesssim \frac {\sqrt{\log A}}
 {\sqrt A} 2^{-\frac{3A} q}, \quad \forall\, 1<q\le \infty,
 \label{y5_1}
\end{align}
where in the last two inequalities we have used \eqref{y5_t1}.

We carry out the perturbation argument in several steps.

\texttt{Step 1}. Set $\eta=\omega-W$. We first show that
\begin{align} \label{y5_3}
 \| \eta(t,\cdot) \|_{B^{0}_{\infty,1}} \lesssim 2^{-\frac A2+}, \quad \forall\, 0\le t\le \frac 1 {\log\log A}.
\end{align}
Here and below we use the notation $X+$ as in \eqref{notation_plus}.

Rewrite the equations for $\omega$ and $W$ as
\begin{align*}
 &\partial_t \omega + (u+u_1) \cdot \nabla \omega = (\omega \cdot \nabla) (u+u_1), \notag \\
 & \partial_t W + (U+u_1 +u_2) \cdot \nabla W = (W \cdot \nabla) (U+u_1+u_2).
\end{align*}
Taking the difference, we have
\begin{align*}
 \partial_t \eta &+ (u+u_1) \cdot \nabla \eta + (u- U-u_2) \cdot \nabla W \notag \\
 &= (\eta \cdot \nabla) (u+u_1) + (W\cdot \nabla) (u-U-u_2).
\end{align*}
Let $1<p<3$. By \eqref{y5_0} and \eqref{y5_1}, we have
\begin{align*}
 \partial_t ( \| \eta \|_p ) & \lesssim \| u - U \|_{(\frac 1p -\frac 13)^{-1} } \| D W\|_3 + \| DW\|_p \cdot 4^{-A} \cdot B \notag \\
 & \qquad + \| \eta \|_p \cdot ( \| \frac{u^r} r \|_{\infty} +B )  \notag \\
 & \qquad + \| D(u-U)\|_p \cdot \| W \|_{\infty} + \| W\|_p \cdot 2^{-A} \notag \\
 & \lesssim \| \eta \|_p \cdot {\log \log\log A} + {\log \log\log A} \cdot 4^{-A} \cdot B \notag \\
 & \qquad + \| \eta \|_p \cdot (\sqrt{\log A} + B) + \| \eta\|_p \cdot \frac {\sqrt{\log A}} {\sqrt A} + 4^{-A}.
\end{align*}

Set $\eta(0)=0$, integrating in $t\le \frac 1 {\log\log A}$ then gives
\begin{align} \label{y5_4}
 \max_{0\le t \le \frac 1 {\log\log A}} \| \eta(t,\cdot)\|_p \lesssim 4^{-A+}, \quad \forall\, 1<p<3.
\end{align}

This estimate is particularly good for $p=3-$.

Now for any $1<q<\infty$, a standard energy estimate using \eqref{e_p63_2}, \eqref{y5_t1} and \eqref{y5_1} (using $\|W\|_{\infty}$ and
$\|\omega\|_{\infty}$) gives for any $0\le t\le 1$,
\begin{align*}
 &\| W(t,\cdot)\|_{W^{1,q}} \lesssim 2^{A-}, \notag \\
 &\| \omega(t,\cdot)\|_{W^{1,q}} \lesssim 2^{A-},
\end{align*}
and obviously
\begin{align*}
 \max_{0\le t\le 1} \| \eta(t,\cdot)\|_{W^{1,q}} \lesssim 2^{A-}, \qquad \forall\, 1<q<\infty.
\end{align*}

Interpolating the above (set $q=\infty-$) with \eqref{y5_4} (set $p=3-$) then gives \eqref{y5_3}.

\texttt{Step 2}. Let $\phi=(\phi^r,\phi^z)$ be the characteristic line associated with the velocity field $u+u_1$. We show that
\begin{align} \label{y5_6}
 \max_{0\le t\le \frac 1 {\log\log A}} \| \phi(t,\cdot) - \Phi(t,\cdot) \|_{\infty} \lesssim 2^{-\frac 7 6A+}.
\end{align}

Set $Y(t)=\phi(t)-\Phi(t)$. By Lemma \ref{y4}, we only need to consider the region $|x| \lesssim 2^{-A}$.
By \eqref{y5_t1_aa},
we have the estimate
\begin{align} \label{y5_6a}
 |u_2(t, \Phi(t,r,z))| \lesssim 4^{-A}, \qquad \forall\, 0\le t\le \frac 1 {\log\log A}, \, \forall\, \sqrt{r^2+z^2} \lesssim 2^{-A}.
\end{align}

Let $Y(t)= \phi(t)-\Phi(t)=(Y^r(t),Y^z(t))$. In order not to confuse the notation, we denote
\begin{align*}
& v=\begin{pmatrix}u^r,\, u^z
   \end{pmatrix},
   \quad
   v_1= \begin{pmatrix} u_1^r,\,
         u_1^z
        \end{pmatrix},
        \notag \\
&V= \begin{pmatrix} U^r,\,
    U^z
   \end{pmatrix},
   \quad
   v_2= \begin{pmatrix} u_2^r, \,
         u_2^z
        \end{pmatrix}.
\end{align*}

Then the equation for $Y$ takes the form
\begin{align} \label{y5_7}
 \frac d {dt} Y = (v+v_1)(\phi) - (v+v_1)(\Phi) + (v-V)(\Phi) - v_2(\Phi).
\end{align}

By \eqref{y5_1} and a simple energy estimate, we have
\begin{align}
 \| Du\|_{\infty} & \lesssim \| \omega\|_2 + \| \omega \|_{\infty} \log (10+ \| \omega\|_{H^2}) \notag \\
  & \lesssim 1+ \frac {\sqrt{\log A}} {\sqrt A} \cdot A   \lesssim \sqrt A \cdot \sqrt{\log A}. \notag
\end{align}

Since $u(t,0,0,z)=u^z e_z$, we have
\begin{align} \label{y5_7a}
 |\frac 1 {r} u^r(t,r,z)|& = \bigl|\frac {\bigl(u(t,x_1,x_2,z)-u(t,0,0,z)\bigr) \cdot e_{r} } r\bigr| \notag \\
 & \lesssim \|Du\|_{\infty} \lesssim \sqrt A \cdot \sqrt{\log A}.
\end{align}

By the incompressibility condition $\partial_r u^r = -\frac 1 r u^r -\partial_z u^z$, we obtain
\begin{align*}
 \| \partial_r u^r \|_{\infty} \lesssim \| Du\|_{\infty} \lesssim \sqrt A \cdot \sqrt{\log A}.
\end{align*}

Similarly we have the estimate for $\|\partial_z u^r\|_{\infty}$, $\|\partial_r u^z \|_{\infty}$, $\|\partial_z u^z \|_{\infty}$, and
hence
\begin{align}\label{y5_8}
 \| D v \|_{\infty} \lesssim \sqrt A \cdot \sqrt{\log A}.
\end{align}

By \eqref{y5_1}, \eqref{y5_3} and interpolation, we have
\begin{align}
 \max_{0\le t \le \frac 1 {\log\log A}} \| u(t)-U(t)\|_{\infty} & \lesssim \max_{0\le t\le \frac 1 {\log\log A}}
 (\|\omega(t)\|_2+\|W(t)\|_2)^{\frac 23} \|\omega(t)-W(t) \|_{\infty}^{\frac 13} \notag \\
 & \lesssim 2^{-A+} \cdot 2^{-\frac A 6+} \notag \\
 & \lesssim 2^{-\frac 76 A+}. \notag
\end{align}

Therefore
\begin{align}\label{y5_9}
 \max_{0\le t\le \frac 1 {\log\log A}} \| v(t)-V(t)\|_{\infty} \lesssim 2^{-\frac 76 A+}.
\end{align}

Plugging the estimates \eqref{y5_8}--\eqref{y5_9} into \eqref{y5_7} and using \eqref{y5_6a}, we have
\begin{align*}
 \frac d {dt} (|Y(t)|) & \lesssim \sqrt A \cdot \sqrt{\log A} \cdot |Y(t)| + 2^{-\frac 76 A+} + 4^{-A}.
\end{align*}

Integrating in time, we get
\begin{align*}
 \max_{0\le t\le \frac 1 {\log\log A}} |Y(t)| & \lesssim \int_0^{\frac 1 {\log\log A}}
 e^{\frac 1 {\log\log A} \sqrt A \cdot \sqrt{\log A}}
 \bigl( 2^{-\frac 76 A+} + 4^{-A} \bigr) ds \notag \\
 & \lesssim 2^{-\frac 76 A+}.
\end{align*}

Therefore \eqref{y5_6} is proved.

\texttt{Step 3}. We show that
\begin{align} \label{y5_10}
 &\| \partial_{rr} u^r(t)\|_{\infty} + \| \partial_{rz} u^r(t)\|_{\infty} + \| \partial_{zz} u^r(t)\|_{\infty} \notag \\
 & \qquad + \| \partial_{rr}u^z(t)\|_{\infty} + \| \partial_{rz} u^z(t)\|_{\infty}
 + \| \partial_{zz} u^z (t)\|_{\infty} \lesssim 2^{A+}, \quad \forall\,0\le t\le 1.
\end{align}
and
\begin{align} \label{y5_10a}
 &\| \partial_{rr} U^r(t)\|_{\infty} + \| \partial_{rz} U^r(t)\|_{\infty} + \| \partial_{zz} U^r(t)\|_{\infty} \notag \\
 & \qquad + \| \partial_{rr}U^z(t)\|_{\infty} + \| \partial_{rz} U^z(t)\|_{\infty}
 + \| \partial_{zz} U^z (t)\|_{\infty} \lesssim 2^{A+}, \quad \forall\,0\le t\le 1.
\end{align}
We shall only prove \eqref{y5_10} since the proof for \eqref{y5_10a} is essentially the same.

By a simple energy estimate, we have
\begin{align} \label{y5_11}
 \max_{0\le t\le 1} (\| D^2 u(t)\|_{\infty} +\|D\omega(t)\|_{\infty} )\lesssim 2^{A+}.
\end{align}
Write
\begin{align*}
 u= \begin{pmatrix} u^1,\, u^2,\, u^z\end{pmatrix}.
\end{align*}
Obviously
\begin{align*}
 u^1 (t,x_1,x_2,z) = \frac 1 r u^r x_1, \qquad u^2 = \frac 1 r u^r x_2.
\end{align*}
Since $u$ is axisymmetric, easy to check that
\begin{align*}
 \begin{pmatrix} u^1(t,x_1,x_2,z),\,u^2(t,x_1,x_2,z) \end{pmatrix}
 =\alpha(t,z) \begin{pmatrix} x_1,\,x_2\end{pmatrix} + O(r^2),
\end{align*}
where $\alpha(t,z)$ is a constant depending only on $(t,z)$. From this and \eqref{y5_11}, it is not
difficult to show that
\begin{align*}
 |(\partial_2 u^1)(t,0,x_2,z)| & \lesssim \| D^2 u \|_{\infty} \cdot |x_2| \notag \\
 & \lesssim 2^{A+} |x_2|,\quad\forall\, 0\le t\le 1, \, x_2 \in \mathbb R.
\end{align*}
By the Fundamental Theorem of Calculus, we then have for any $(x_1,x_2,z)$ ($r=\sqrt{x_1^2+x_2^2}$),
\begin{align}
 \bigl| \frac{(\partial_2 u^1)(t,x_1,x_2,z)} r \bigr| & \lesssim | \frac{ (\partial_2 u^1)(t,x_1,x_2, z)  -(\partial_2 u^1)(t,0,x_2,z)}
 r| + | \frac{(\partial_2 u^1)(t,0,x_2, z)} r| \notag \\
 & \lesssim 2^{A+}. \label{y5_13}
\end{align}
Denote $g=\frac 1r u^r$. Since
\begin{align*}
 u^1(t,x_1,x_2,z) = g(t,r,z) x_1, \quad r=\sqrt{x_1^2+x_2^2},
\end{align*}
differentiating in $x_2$ then gives us
\begin{align*}
 \partial_2u^1 = (\partial_r g) \cdot \frac{x_2 x_1} r.
\end{align*}
Therefore choosing $|x_1|\sim |x_2|\sim r$ and using \eqref{y5_13}, we obtain
\begin{align} \label{y5_14a}
 \|\partial_r (\frac 1 r u^r(t))\|_{\infty} = \| \partial_r g\|_{\infty} \lesssim 2^{A+}, \qquad \forall\, 0\le t\le 1.
\end{align}

By \eqref{x5_6}, we have
\begin{align*}
  \partial_{rr} u^r = - \partial_r ( \frac 1 r u^r) - \partial_{zz} {u^r} - \partial_z \omega^{\theta}.
\end{align*}
By \eqref{y5_11},
\begin{align} \label{y5_14b}
 &\|\partial_{zz} u^r\|_{\infty} = \|(\partial_{zz} u)\cdot e_r\|_{\infty} \lesssim \|D^2 u\|_{\infty} \lesssim 2^{A+}, \notag \\
 & \| \partial_z \omega^{\theta}\|_{\infty} = \| (\partial_z \omega) \cdot e_{\theta}\|_{\infty} \lesssim 2^{A+}.
\end{align}
Therefore by \eqref{y5_14a} and \eqref{y5_14b}, we get
\begin{align*}
 \max_{0\le t \le 1}\| \partial_{rr} u^r (t)\|_{\infty} \lesssim 2^{A+}.
\end{align*}

Similar to \eqref{y5_7a}, we have
\begin{align}
 |\frac 1 {r} (\partial_z u^r)(t,r,z)|& = \bigl|\frac {\bigl( (\partial_z u)(t,x_1,x_2,z)-(\partial_z u)(t,0,0,z)\bigr) \cdot e_{r} } r\bigr| \notag \\
 & \lesssim \|D^2u\|_{\infty} \lesssim 2^{A+}. \notag
\end{align}
Since $\omega^{\theta}=\partial_r u^z -\partial_z u^r$, we get
\begin{align*}
 \| \frac 1 r \partial_r u^z \|_{\infty} & \lesssim \| \frac{\omega^{\theta}} r \|_{\infty} + \| \frac{1} r \partial_z u^r \|_{\infty} \notag \\
 & \lesssim 2^{A+}.
\end{align*}
We then get
\begin{align*}
 \| \partial_{rr} u^z \|_{\infty} &\lesssim \| \Delta u^z \|_{\infty} + \| \frac 1 r \partial_r u^z \|_{\infty} + \| \partial_{zz} u^z \|_{\infty}
\notag \\
& \lesssim 2^{A+}, \qquad \forall\, 0\le t\le 1.
\end{align*}
We have proved that $\| \partial_{rr} u^r \|_{\infty}$ and $\| \partial_{rr} u^z \|_{\infty}$ are both under control. The rest of the terms
in \eqref{y5_10} are similarly estimated. We omit further details.

\texttt{Step 4}. Set $e(t)=(D\phi)(t)$, $E(t)=(D\Phi)(t)$, then obviously
\begin{align*}
& \partial_t E = (DV)(\Phi) E+ (Dv_1)(\Phi) E + (Dv_2)(\Phi) E, \notag \\
 & \partial_t e = (Dv)(\phi) e + (Dv_1)(\phi) e.
\end{align*}
Observe that
\begin{align*}
 Dv_1 = a(t) \begin{pmatrix} 1 \qquad 0\\
              0 \quad -2
             \end{pmatrix}.
\end{align*}
Set $q=E-e$. By Lemma \ref{y4}, we only need to control $q$ in the region $|x| \lesssim 2^{-A}$. In
this region we have
\begin{align*}
\| Dv_2(\Phi) \|_{\infty} \lesssim 2^{-A}.
\end{align*}
 The equation for $q$ takes the form
\begin{align*}
 \partial_t q & = ( (DV)(\Phi) - (DV)(\phi) ) E + ( (DV)(\phi) -(Dv)(\phi)) E \notag \\
 & \qquad + (Dv)(\phi) q +  a(t) \begin{pmatrix} 1 \qquad 0\\
              0 \quad -2
             \end{pmatrix} q \notag \\
             & \qquad + (Dv_2)(\Phi) E.
\end{align*}

By \eqref{y5_9}, \eqref{y5_10}, \eqref{y5_10a} and interpolation, we have
\begin{align} \label{y5_20}
 \max_{0\le t\le \frac 1 {\log\log A}} \| D(v-V)\|_{\infty} \lesssim 2^{-\frac 1 {12} A+}.
\end{align}

By \eqref{y5_10a},\eqref{y5_6}, \eqref{y5_0a}, \eqref{y5_8} and \eqref{y5_20}, we have
\begin{align*}
 \partial_t (|q|) & \lesssim \| D^2 V \|_{\infty} \cdot |\phi-\Phi| \cdot \| E\|_{\infty} + \| D(V-v)\|_{\infty} \cdot \| E\|_{\infty}
 \notag \\
 & \qquad + \| Dv \|_{\infty} \cdot |q| + B \cdot |q| + \| D v_2 (\Phi) \|_{\infty} \cdot \| E \|_{\infty} \notag \\
 & \lesssim 2^{A+} \cdot 2^{-\frac 76 A+} \cdot (\log\log\log A)^2 \notag \\
 & \qquad + 2^{-\frac 1 {12} A+} \cdot (\log\log\log A)^2 + (\sqrt A \cdot \sqrt{\log A}+B) |q| \notag \\
 & \qquad + 2^{-A} \cdot (\log\log\log A)^2.
\end{align*}
Integrating in time $t\le \frac 1 {\log\log A}$, we then obtain
\begin{align*}
 \max_{0\le t\le \frac 1 {\log\log A}} \| q(t)\|_{\infty} \lesssim 1.
\end{align*}
But this obviously contradicts \eqref{y3_t2}.
\end{proof}

The next proposition is the key to our construction in the 3D compactly supported data case. It is written
in the same style as in Lemma \ref{lem57}. The overall statement of the proposition is a bit long and over-stretched
due to some additional technical conditions pertaining to the 3D situation.
Nevertheless the structure of the proposition is the same as that in Lemma \ref{lem57}.
In short summary the main body of the proposition should read as "
Let $\omega_{-1}$ satisfy ...Then for any $0<\epsilon<\epsilon_0$, one can find $\omega_0$ with the properties ...
and $\delta_0$ such that for any $\omega_j$ with the properties ...,  the following hold true: ....".

\begin{prop} \label{y6}
Let $\omega_{-1} \in C_c^{\infty}(B(0,100) )$ be a given axisymmetric function
such that $\omega_{-1}=\omega_{-1}^{\theta} e_{\theta}$, $\omega_{-1}^{\theta}=\omega_{-1}^{\theta}(r,z)$ is scalar-valued
and for some $r_{-1}>0$, $0<R_0<\frac 1 {100}$,
\begin{align*}
\operatorname{supp}(\omega_{-1}^{\theta}) \subset\{(r,z):\; r>r_{-1},\, z\le -4R_0\}.
\end{align*}
Denote $u_{-1}= -\Delta^{-1} \nabla \times \omega_{-1}$ and
\begin{align*}
u_{-1}^*=\| u_{-1} \|_2.
\end{align*}

 Then for any $0<\epsilon\le \epsilon_0$ with $\epsilon_0=\epsilon_0(\omega_{-1})\ll R_0$  sufficiently small, we can find a smooth axisymmetric
 function $\omega_0=\omega_0^{\theta} e_{\theta}$
 (depending only on $(\epsilon,\omega_{-1})$)
 with the properties:
\begin{itemize}
 \item  $\omega_0\in C_c^{\infty}(B(0,100))$ and for some $r_0>0$,
\begin{align} \label{y6_t-1_1a}
 \operatorname{supp}(\omega_0^{\theta}(r,z)) \subset\{(r,z):\, r_0<r<\epsilon, \,  -\epsilon <z<\epsilon \}.
\end{align}
\item $\|\omega_0\|_{\infty} \exp(C \| \frac{\omega_0} r \|_{L^{3,1}} ) <\delta(\epsilon^2,\omega_{-1})$
(see \eqref{y1_t3});
\item denote $u_0= -\Delta^{-1} \nabla \times \omega_0$, then
\begin{align} \label{y6_t-1_1b}
\| u_0\|_2 <\epsilon u_{-1}^*<\frac 14 u_{-1}^*;
\end{align}
\end{itemize}

\noindent
and $\delta_0=\delta_0(\omega_{-1},\omega_0)\ll \epsilon $ sufficiently small
 such that for any smooth axisymmetric functions $\omega_j=\omega_j^{\theta}e_{\theta}$, $1\le j\le N$ (here $N\ge 1$ is arbitrary but finite)
  with the properties:
\begin{itemize}
\item $\omega_j\in C_c^{\infty}(B(0,100))$ and
$\operatorname{supp}(\omega_j^{\theta}) \subset\{(r,z):\, r>r_j, z>2R_0\}$ for some $r_j>0$.
\item for each $j\ge 1$, denote $f_j=\omega_{-1}^{\theta}+ \omega_0^{\theta}+\sum_{i=1}^{j-1} \omega_i^{\theta}$,
then
\begin{align*}
\| \omega_j^{\theta} \|_{\infty} \cdot \exp\bigl( C \|\frac{ \omega_j^{\theta}} r \|_{L^{3,1}} \bigr)<\delta_j,
\end{align*}
where $\delta_j=\delta(2^{-3j}\delta_0, f_j)$ as defined in \eqref{y1_t3};
\item denote $u_j = -\Delta^{-1} \nabla \times \omega_j$, then
\begin{align*}
\| u_j \|_2 < \frac {\epsilon} {2^{j+1}} u_{-1}^*;
\end{align*}
\end{itemize}

 the following hold true:

 Let  $\omega$ be the smooth solution to the axisymmetric system
\begin{align} \notag
 \begin{cases}
  \partial_t \left( \frac {\omega }r \right) + (u \cdot \nabla ) \left( \frac {\omega} r  \right) =0,
  \quad 0<t\le 1, \\
  u=-\Delta^{-1} \nabla \times \omega, \\
  \omega \Bigr|_{t=0} =(\omega_{-1}^{\theta} +\omega_0^{\theta} +\sum_{j=1}^N \omega_j^{\theta}) e_{\theta},
 \end{cases}
 \end{align}
then
\begin{enumerate}

\item for any $0\le t \le \epsilon$, we have the decomposition
\begin{align} \label{y6_t1}
\omega(t)=\omega_A(t)+\omega_B(t)+\omega_C(t),
\end{align}
where
\begin{align*}
&\operatorname{supp}(\omega_A(t)) \subset\{(r,z):\, z\le -4R_0+\sqrt{\epsilon} \}; \\
&\operatorname{supp}(\omega_B(t)) \subset\{(r,z):\, |z| \le \sqrt{\epsilon} \};\\
&\operatorname{supp}(\omega_C(t)) \subset\{(r,z):\, z\ge 2R_0-\sqrt{\epsilon} \};
\end{align*}
and $\omega_A(t=0)=\omega_{-1}$, $\omega_{B}(t=0)=\omega_0^{\theta} e_{\theta}$,
$\omega_C(t=0)=(\sum_{j=1}^N \omega_j^{\theta}) e_{\theta}$.

\item the  $L^{\infty}$ norm of $\omega_B$ and $\omega_C$ is uniformly small on the interval $[0,1]$:
\begin{align}
 \max_{0\le t \le 1} (\| \omega_B (t)\|_{L^{\infty}}+\|\omega_C(t)\|_{L^{\infty}}) \le \epsilon. \label{y6_t2}
\end{align}

\item the $\dot H^{\frac 32}$-norm of $\omega_B$ is inflated rapidly on the time interval $[0,\epsilon]$:
there exists $0<t_0^1=t_0^1(\epsilon,\omega_{-1},\omega_0)<\epsilon$, $0<t_0^2=t_0^2(\epsilon,\omega_{-1},
\omega_0)<\epsilon$, such that
\begin{align}
 &\| \omega_B(t=0)\|_{\dot H^{\frac 32}} <\epsilon, \notag \\
 &\| \omega_B(t)\|_{\dot H^{\frac 32}} >\frac 1{\epsilon}, \quad \text{for any $t_0^1\le t \le t_0^2$.} \label{y6_t3}
\end{align}

\item all $H^k$, $k\ge 2$ norms of $\omega_B$ can be bounded purely in terms of initial
data $\omega_0$ on the time interval $[0,\epsilon]$:
for any $k\ge 2$,
\begin{align} \label{y6_t4}
\max_{0\le t \le \epsilon} \| \omega_B(t) \|_{H^k} \le C(k,R_0,u^*_{-1}) \|\omega_0\|_{H^k}.
\end{align}
Note here the bound of $\|\omega_B\|_{H^k}$ is "almost local" in the sense that it depends only on $u^*_{-1}$ but
not on other higher Sobolev norms of $\omega_A$ or $\omega_C$. Similarly we have
\begin{align} \label{y6_t5}
\max_{0\le t \le \epsilon} \| \omega_A(t) \|_{H^k} \le C(k,R_0,u^*_{-1}) \| \omega_{-1} \|_{H^k}, \quad \forall\, k\ge 2.
\end{align}
\end{enumerate}
\end{prop}

\begin{proof}[Proof of Proposition \ref{y6}]
The nontrivial point is to find $\omega_0$ such that \eqref{y6_t3} is achieved. We first
show that a generic $\omega_0$ (i.e. satisfying the properties specified in \eqref{y6_t-1_1a}--\eqref{y6_t-1_1b})
is enough to make \eqref{y6_t1}, \eqref{y6_t2} and \eqref{y6_t4} hold.

By conservation of $\|u(t)\|_2$, we have
\begin{align}
\| u(t)\|_2 = \|u(0)\|_2
& \le u_{-1}^* + \sum_{j=1}^{\infty} u_{-1}^* \cdot 2^{-j} \notag \\
& \le 2u_{-1}^*, \quad \forall\, t\ge 0. \label{y6_1}
\end{align}

Let $\omega_L$ be the smooth solution to the axisymmetric system
\begin{align*}
\begin{cases}
\partial_t (\frac{\omega_L} r) + (u_L \cdot \nabla) (\frac{\omega_L} r) =0, \quad 0<t\le 1, \\
u_L = -\Delta^{-1} \nabla \times \omega_L, \\
\omega_L(t=0) = \omega_{-1}.
\end{cases}
\end{align*}

Obviously
\begin{align*}
\max_{0\le t\le 1} \|\omega_L(t)\|_{\infty} \lesssim_{\omega_{-1}} 1.
\end{align*}

By Lemma \ref{y1}, we have
\begin{align*}
\max_{0\le t\le 1} \| \omega(t) - \omega_L(t)\|_{\infty} \ll \epsilon
\end{align*}
and clearly
\begin{align*}
\max_{0\le t\le 1} \| \omega(t) \|_{\infty} \lesssim_{\omega_{-1}} 1.
\end{align*}

Interpolating the above with \eqref{y6_1} then gives
\begin{align*}
\max_{0\le t\le 1} \|u(t)\|_{\infty} \le c_1,
\end{align*}
where $c_1>0$ is a constant depending only on $\omega_{-1}$.

This shows that the support of $\omega(t)$ moves at a speed at most $c_1$. Since we can always
choose $\epsilon$ sufficiently small such that $c_1 \epsilon \ll \sqrt{\epsilon}$,
the decomposition \eqref{y6_t1} then obviously follows.

The inequality \eqref{y6_t2} is a simple consequence of Lemma \ref{y1}. To show \eqref{y6_t4}, we
note that for $0\le t\le \epsilon$, $\omega_B=\omega_B(t)$ solves the equation
\begin{align*}
\partial_t \omega_B+ \bigl((u_B+u_{ex}) \cdot \nabla\bigr) \omega_B=( \omega_B \cdot \nabla)(u_B+u_{ex}),
\end{align*}
where
\begin{align*}
&u_B(t)= -\Delta^{-1} \nabla \times \omega_B(t), \notag \\
&u_{ex}(t)= -\Delta^{-1} \nabla \times (\omega_A(t) + \omega_C(t)).
\end{align*}

Since for $0\le t\le \epsilon$ and $\epsilon$ sufficiently small,
\begin{align*}
&d(\operatorname{supp}(\omega_A(t)),\operatorname{supp}(\omega_B(t))) \ge R_0, \notag \\
&d(\operatorname{supp}(\omega_C(t)),\operatorname{supp}(\omega_B(t))) \ge R_0, \notag
\end{align*}
we can then write for $x\in \operatorname{supp}(\omega_B(t))$,
\begin{align} \label{y6_3}
u_{ex}(t,x) = \int_{R^3} K(x-y) (\omega_A(t,y)+\omega_C(t,y)) dy,
\end{align}
where the modified kernel $K(\cdot)$ satisfies
\begin{align*}
|(\partial^{\alpha} K)(x)|\lesssim_{R_0, \alpha} (1+| x |^2)^{-\frac{2+|\alpha|}2}, \quad\forall\, x\in \mathbb R^3,
\, |\alpha|\ge 0.
\end{align*}
Since $\omega(t)=\nabla \times u(t)$, we can rewrite \eqref{y6_3} as
\begin{align*}
u_{ex}(t,x)
&= \int_{\mathbb R^3} K(x-y) \omega(t,y) dy - \int_{\mathbb R^3} K(x-y) \omega_B(t,y)dy \notag \\
&= \int_{\mathbb R^3} K(x-y) \nabla \times u(t,y) dy -\int_{\mathbb R^3} K(x-y) \omega_B(t,y)dy \notag \\
&= \int_{\mathbb R^3} \tilde K(x-y) u(t,y)dy -\int_{\mathbb R^3} K(x-y) \omega_B(t,y)dy\notag \\
&=:u_{ex}^{(1)}(t,x)+u_{ex}^{(2)}(t,x).
\end{align*}
Obviously we only need to bound $u_{ex}^{(1)}$. Since
$|(\partial^{\alpha}\tilde K)(x)| \lesssim_{R_0,\alpha} (1+ |x|^2)^{-\frac{3+|\alpha|} 2}$, we
have
\begin{align*}
\| u_{ex}^{(1)}(t,\cdot)\|_{H^k} \lesssim_{k,R_0} \|u(t)\|_2 \lesssim_{k,R_0} u_{-1}^*, \quad \text{for any
$k\ge 0$}.
\end{align*}
The inequality \eqref{y6_t4} then easily follows from this and a simple energy estimate. Similarly one can prove
\eqref{y6_t5}.

It remains for us to show the existence of $\omega_0$ such that \eqref{y6_t3} hold. First we show that
it suffices to consider the following reduced system
\begin{align*}
\begin{cases}
\partial_t (\frac{W}r) + (U\cdot \nabla)(\frac W r)=0, \quad 0<t\le 1, \\
U=-\Delta^{-1} \nabla \times W, \\
W\Bigr|_{t=0} = \omega_{-1}+\omega_0.
\end{cases}
\end{align*}

By Lemma \ref{y1} and our assumptions on $\omega_j$, we have
\begin{align} \label{y6_5}
\max_{0\le t\le 1} \| \omega(t)-W(t)\|_{\infty} \lesssim \delta_0 \ll \epsilon.
\end{align}

Since $\max_{0\le t\le 1}\|u(t)\|_{\infty} \lesssim_{\omega_{-1}} 1$ and
$\max_{0\le t\le 1}\|U(t)\|_{\infty} \lesssim_{\omega_{-1}} 1$, we have
for some $c_2=c_2(\omega_{-1})>0$,
\begin{align*}
&\operatorname{supp}(W(t)) \subset B(0,100+c_2), \quad\forall\, 0\le t\le 1,\\
&\operatorname{supp}(\omega(t))\subset B(0,100+c_2), \quad\forall\, 0\le t\le 1.
\end{align*}
Therefore by \eqref{y6_5} and H\"older, we get
\begin{align} \label{y6_6}
\max_{0\le t\le 1} \| \omega(t)-W(t)\|_{2} \lesssim_{\omega_{-1}} \delta_0^{\frac 12}.
\end{align}
For $0\le t\le \epsilon$, we write the decomposition of $W$ as
\begin{align*}
W(t)=W_L(t)+W_R(t),
\end{align*}
where
\begin{align*}
&\operatorname{supp}(W_L(t)) \subset\{(r,z):\, z\le -4R_0+\sqrt{\epsilon} \}; \\
&\operatorname{supp}(W_R(t)) \subset\{(r,z):\, |z| \le \sqrt{\epsilon} \};
\end{align*}
and $W_L(t=0)=\omega_{-1}$, $W_R(t=0)=\omega_0$.

By \eqref{y6_t4} and a similar bound for $W_R(t)$, we have
\begin{align*}
\max_{0\le t\le \epsilon}\|\omega_B(t)-W_R(t)\|_{H^3} \lesssim_{\omega_{-1},R_0,u_{-1}^*} \|\omega_0\|_{H^3}.
\end{align*}
Interpolating the above with \eqref{y6_6} and choosing $\delta_0$ sufficiently small, we then get
\begin{align*}
\max_{0\le t \le \epsilon} \| \omega_B(t)-W_R(t)\|_{H^2} \le \epsilon.
\end{align*}
This shows that it suffices for us to inflate the $\|W_R(t)\|_{\dot H^{\frac 32}}$ norm.

To this end, let $W^1$ be the smooth solution to
\begin{align*}
\begin{cases}
\partial_t (\frac{W^1}r) + (U^1\cdot \nabla)(\frac {W^1} r)=0, \quad 0<t\le 1, \\
U^1=-\Delta^{-1} \nabla \times W^1, \\
W^1\Bigr|_{t=0} = \omega_{-1}+\tilde g_A e_{\theta},
\end{cases}
\end{align*}
where $\tilde g_A$ is the same as defined in \eqref{e_p63_1} and we shall take $A$ to be sufficiently large without
too much explicit mentioning. Eventually we shall take $\omega_0$ to be a suitable perturbation
of $\tilde g_A$ and let $W$ be the corresponding solution.

For $0\le t\le \frac 1 {\log\log A}$, we can decompose the solution $W^1$ as
\begin{align*}
W^1(t)= W^1_L(t)+W^1_R(t),
\end{align*}
where
\begin{align*}
\operatorname{supp}(W^1_L(t)) \subset \{(r,z):\, z\le -4R_0+ \frac 1 {\sqrt{\log\log A}} \}, \\
\operatorname{supp}(W^1_R(t)) \subset \{(r,z):\, |z| \le \frac 1 {\sqrt{\log\log A}} \}.
\end{align*}

The equation for $W^1_R$ takes the form
\begin{align*}
\begin{cases}
\partial_t(\frac{ W^1_R}r) + \bigl((U^1_R +U^1_L)\cdot \nabla \bigr) (\frac {W^1_R}r) = 0,\\
U^1_R=-\Delta^{-1} \nabla \times W^1_R, \\
U^1_L=-\Delta^{-1} \nabla \times W^1_L,\\
W^1_R(t=0)= \tilde g_A e_{\theta}.
\end{cases}
\end{align*}
Write
\begin{align*}
U^1_L=U_L^r e_r +U_L^z e_z.
\end{align*}
Let $\xi(t)$ solves the ODE
\begin{align*}
\begin{cases}
\frac d{dt} \xi(t) = -U_L^z(0,0,\xi(t)),\\
\xi(0)=0.
\end{cases}
\end{align*}
We can expand $U_L^1(t)$ near the point $(0,0,\xi(t))$ to get
\begin{align*}
U_L^1(t,x_1,x_2,z+\xi(t))=U_L^z(t,0,0,\xi(t))e_z + \underbrace{a(t)re_r-2a(t)ze_z}_{=:u_1(t,x_1,x_2,z)}+u_2(t,x_1,x_2,z),
\end{align*}
where for any $0\le t\le \frac 1 {\log\log A}$,
\begin{align*}
&|a(t)| \lesssim_{\omega_{-1},R_0} 1, \\
&|u_2(t,x)| \lesssim_{\omega_{-1},R_0} |x|^2,  \\
&|(Du_2)(t,x)| \lesssim_{\omega_{-1},R_0} |x|, \\
&|(D^2u_2)(t,x)| \lesssim_{\omega_{-1},R_0} 1, \qquad \forall\, x \in \mathbb R^3.
\end{align*}
Note that $u_2$ is axisymmetric without swirl, i.e. $u_2=u_2^r e_r +u_2^z e_z$.

Introduce $\Omega(t)=\Omega(t,x_1,x_2,z)$ such that
\begin{align*}
&\Omega(t,x_1,x_2,z):=W_R^1(t,x_1,x_2,z+\xi(t)), \\
&U_{\Omega}(t):= -\Delta^{-1} \nabla \times \Omega(t).
\end{align*}
It is then not difficult to check that the equation for $\Omega$ takes the form
\begin{align*}
\partial_t ( \frac{\Omega} r) +\bigl( (U_{\Omega}+u_1+u_2) \cdot \nabla \bigr)(\frac{\Omega} r) =0.
\end{align*}

Let $\Phi_\Omega=(\Phi^r_{\Omega},\Phi^z_{\Omega})$ be
the characteristic line associated with $U_{\Omega}+u_1+u_2$ and let $\tilde \Phi_{\Omega}$
be the corresponding inverse map. By Lemma \ref{y5}, for $A$ sufficiently large, we have
either
\begin{align} \label{y6_8a}
\max_{0\le t\le \frac 1 {\log\log A}} \| \Omega(t)\|_{\dot H^{\frac 32}} >\log\log\log A,
\end{align}
or
\begin{align} \label{y6_8b}
\max_{0\le t\le \frac 1 {\log\log A}} \| (D\tilde \Phi_{\Omega})(t)\|_{\infty} >\log\log\log A
\end{align}
must hold.

Now discuss two cases.

\texttt{Case 1}: \eqref{y6_8a} hold. In this case easy to check that
\begin{align*}
\max_{0\le t\le \frac 1 {\log\log A}} \| W_R^1(t)\|_{\dot H^{\frac 32}} \gtrsim \log\log\log A.
\end{align*}
Therefore we can just let $W(t)=W^1(t)$ with $\omega_0=\tilde g_A$.

\texttt{Case 2}: \eqref{y6_8b} hold. In this case we just need to apply a perturbation argument similar to
that in the proof of Proposition \ref{x6}. Easy to check that this case is also OK.

Concluding from the above two cases, the proposition is proved.
\end{proof}

We are now ready to complete the
\begin{proof}[Proof of Theorem \ref{thm4}]
We shall only sketch the proof for $\omega_0^{(g)}\equiv 0$. The construction of $\omega_0^{(p)}$ for the
general nonzero $\omega_0^{(g)}$ is a simple modification of the proof presented below. For example, one can
just take the first patch as $\omega_0^{(g)}$ and start the perturbation for $j\ge 2$.

We now begin the proof for $\omega_0^{(g)}\equiv 0$. For each integer $j\ge 1$, define $x^j_*=(0,0,\sum_{k=1}^j \frac 1 {2^k})$. Obviously for any
$j\ge 2$, we have
\begin{align*}
&|x^{j+1}_*-x^j_*|=\frac 1 {2^{j+1}}, \\
&|x^j_*-x^{j-1}_*|=\frac 1 {2^j}.
\end{align*}
We shall choose $x^j_*$ to be the center of the $j^{th}$ patch. So the distance between the nearest patches
is about $2^{-j}$. Define
\begin{align*}
x_*= \lim_{j\to \infty} x^j_*=(0,0,1).
\end{align*}
Our constructed solution will exhibit some additional regularity away from the limit point $x_*$.

Let $W^1$ be a smooth axisymmetric solution to the Euler equation (in vorticity form)
\begin{align*}
\begin{cases}
\partial_t (\frac{W^1}r) + (U^1 \cdot \nabla)(\frac{W^1} r)=0, \quad 0<t\le 1, x=(x_1,x_2,z),
r=\sqrt{x_1^2+x_2^2}, \\
U^1=-\Delta^{-1} \nabla \times W^1, \\
W^1\Bigr|_{t=0}=W^1_0=W^{1,\theta}_0 e_{\theta},
\end{cases}
\end{align*}
such that $W^{1,\theta}_0=W^{1,\theta}_0(r,z)$ is scalar-valued,
$\operatorname{supp}(W^{1,\theta}_0) \subset\{(r,z):\, r>r_0\}$ for some $r_0>0$,
$W^1(t)\in C_c^{\infty} (B(x_*^1, \frac 1 {2^{10}}))$  for any $0\le t\le 1$ and
\begin{align} \label{y7_1}
\| U^1(0,\cdot)\|_{H^{\frac 52}} +\max_{0\le t \le 1}  \|W^1(t,\cdot)\|_{\infty}\le \frac 1 {2^{100}}.
\end{align}
In view of the scaling symmetry ($\omega \to \omega_{\lambda}(t,x)=\lambda \omega(\lambda t, x)$)
and translation symmetry (in the axisymmetric case we just shift only along the $z$-axis so as to
keep axisymmetry) of the Euler equation, we can always find a nonzero $W^1$ satisfying the aforementioned
conditions by transforming an arbitrary compactly supported solution.

By repeated applying Proposition \ref{y6} (one needs to shift along the $z$-axis if necessary),  we can find
a sequence of smooth solutions $W^j$, $j\ge 2$, solving the equations
\begin{align*}
\begin{cases}
\partial_t (\frac{W^j}r) + (U^j\cdot \nabla)(\frac{W^j}r)=0, \quad 0<t\le 1, \\
U^j=-\Delta^{-1}\nabla \times W^j,\\
W^j\Bigr|_{t=0}=W^j_0=W^{j,\theta}_0 e_{\theta},
\end{cases}
\end{align*}
such that the following hold:

\begin{itemize}
\item $W^j_0 =(\sum_{k=1}^j f_k) e_{\theta}$, where $f_1=W^{1,\theta}_0$, and for $k\ge 2$,
$\operatorname{supp}(f_k) \subset \{(r,z):\, r>r_k\}$ for some $r_k>0$.

\item Define $F_k=f_k e_{\theta}$. Then for each $k\ge 1$, $F_k \in C_c^{\infty}(B(x_*^k, \frac 1{2^{10k}}))$.
Furthermore
\begin{align} \label{y7_2}
\| \Delta^{-1} \nabla \times F_k\|_{H^{\frac 52} } \le 2^{-100k},
 \quad\forall\, k\ge 1.
\end{align}

\item For any $j\ge 2$,
\begin{align} \label{y7_3}
\max_{0\le t\le 1} \| W^{j}(t,\cdot)-W^{j-1}(t,\cdot)\|_{\infty} \le 2^{-100j}.
\end{align}

\item For each $j_0\ge 2$, there exists $t_{j_0}^1$, $t_{j_0}^2$ with $0<t_{j_0}^1<t_{j_0}^2<2^{-j_0}$, such
that for any $j\ge j_0+2$, we have the decomposition:
\begin{align} \label{y7_3a}
W^j(t,x)= W_{<j_0}^j(t,x) + W_{j_0}^j(t,x) +W_{>j_0}^j (t,x), \quad \forall\, t\le t_{j_0}^2,
\end{align}
where $W_{<j_0}^j \in C_c^{\infty}(\mathbb R^3)$, $W_{j_0}^j \in C_c^{\infty}(\mathbb R^3)$,
$W_{>j_0}^j \in C_c^{\infty}(\mathbb R^3)$ satisfy
\begin{align*}
&\operatorname{supp}(W_{<j_0}^j) \subset \{x=(x_1,x_2,z):\; z\le \sum_{k=1}^{j_0-1}2^{-k} +\frac 18 \cdot {2^{-j_0}}\}, \\
&\operatorname{supp}(W_{j_0}^j) \subset \{ x=(x_1,x_2,z):\;  \sum_{k=1}^{j_0} 2^{-k} -\frac 18 \cdot 2^{-j_0} <z <\sum_{k=1}^{j_0}
2^{-k} + \frac 18\cdot 2^{-j_0} \};\\
&\operatorname{supp}(W_{>j_0}^j) \subset \{ x=(x_1,x_2,z): \; z>\sum_{k=1}^{j_0} 2^{-k} +\frac 14 2^{-j_0} \}.
\end{align*}
Here
\begin{align*}
&W_{<j_0}^j(t=0) =\sum_{k=1}^{j_0-1} F_k, \quad
W_{j_0}^j (t=0)=F_{j_0}, \notag \\
&W_{>j_0}^j(t=0)=\sum_{k=j_0+1}^j F_k.
\end{align*}
Furthermore
\begin{align}
& \| W_{j_0}^j (t,\cdot)\|_{\dot H^{\frac 32}(\mathbb R^3) } > j_0, \quad \forall\, t\in[t_{j_0}^1, t_{j_0}^2];\label{y7_4} \\
& \| W_{j_0}^j (t, \cdot)\|_{L^2(\mathbb R^3)}\le 2^{-100j_0}, \quad \forall\, t\le t_{j_0}^2. \label{y7_5} \\
& \max_{0\le t\le t_{j_0}^2} (\| W_{j_0}^j(t,\cdot)\|_{H^k(\mathbb R^3)} + \| W_{<j_0}^j (t,\cdot)\|_{H^k(\mathbb R^3)})
\le C_{j_0,k}<\infty,\quad\forall\, k\ge 2,  \label{y7_5a}
\end{align}
where $C_{j_0,k}$ is a constant depending only on $k$ and $(F_1,F_2,\cdots,F_{j_0})$.
\end{itemize}

We now show the existence of the solution $\omega$ as the limit of $W^j$, $j\to \infty$.  By $L^2$-conservation of velocity
and \eqref{y7_2}, we have
\begin{align}
 \max_{0\le t\le 1} \| U^j(t,\cdot)\|_2 & = \| U^j(0,\cdot)\|_2 \notag \\
 & \le \sum_{k=1}^{\infty} 2^{-100k} \le 2^{-99}, \quad \forall\, j\ge 1. \label{y7_6}
\end{align}

By \eqref{y7_1} and \eqref{y7_3},
\begin{align}
 \max_{0\le t\le 1} \| W^j(t,\cdot)\|_{\infty} \le \sum_{k=1}^j 2^{-100k} \le 2^{-99}, \quad \forall\, j\ge 1.
 \label{y7_7}
\end{align}

By \eqref{y7_6}, \eqref{y7_7} and interpolation, we then get
\begin{align}
 \sup_{j\ge 1} \max_{0\le t\le 1} \| U^j(t,\cdot)\|_{\infty} \lesssim 1. \label{y7_8}
\end{align}

Since $\operatorname{supp}(W^j(t,\cdot)) \subset B(0,2)$, \eqref{y7_8} then implies that
\begin{align} \label{y7_9}
 \operatorname{supp} (W^j(t,\cdot) ) \subset B(0, C_1), \quad \forall\, 0\le t\le 1,\, j\ge 1,
\end{align}
where $C_1>0$ is an absolute constant. By \eqref{y7_3} and \eqref{y7_9}, the sequence $W^j$ is Cauchy
in the space $C_t^0 C_x^0([0,1] \times \overline{B(0,C_1)})$ and hence converges to the limit solution $w$ in the same
space. By Sobolev embedding and interpolation, it is not difficult to check that $U^j$ also converges to
$u= -\Delta^{-1} \nabla \times \omega \in C_t^0 L_x^2 \cap C_t^0 C_x^{\alpha} ([0,1]\times \mathbb R^3)$ for any $0<\alpha<1$.
It follows easily that $\omega$ is the desired solution satisfying the first two statements in Theorem \ref{thm4}.

It remains for us to check the last two properties of $\omega$ in Theorem \ref{thm4}.

Fix any $j_0\ge 2$. By \eqref{y7_3a}, \eqref{y7_5a} and taking the limit $j\to \infty$,
we get the decomposition of $\omega(t,x)$ for $t\le t_{j_0}^2$ as
\begin{align} \label{y7_10}
\omega(t,x)=\omega_{<j_0}(t,x)+\omega_{j_0}(t,x)+\omega_{>j_0}(t,x),
\end{align}
where $\omega_{<j_0}(t) \in C_c^{\infty}(\mathbb R^3)$, $\omega_{j_0}(t)\in C_c^{\infty}(\mathbb R^3)$, $\omega_{>j_0}(t) \in C^0_c(\mathbb R^3)$
for $t\le t_{j_0}^2$, and
\begin{align*}
&\operatorname{supp}(\omega_{<j_0}) \subset \{x=(x_1,x_2,z):\; z\le \sum_{k=1}^{j_0-1}2^{-k} +\frac 18 \cdot {2^{-j_0}}\}, \\
&\operatorname{supp}(\omega_{j_0}) \subset \{ x=(x_1,x_2,z):\;  \sum_{k=1}^{j_0} 2^{-k} -\frac 18 \cdot 2^{-j_0} <z <\sum_{k=1}^{j_0}
2^{-k} + \frac 18\cdot 2^{-j_0} \};\\
&\operatorname{supp}(\omega_{>j_0}) \subset \{ x=(x_1,x_2,z): \; z>\sum_{k=1}^{j_0} 2^{-k} +\frac 14 2^{-j_0} \}.
\end{align*}
Furthermore
\begin{align}
& \| \omega_{j_0} (t,\cdot)\|_{\dot H^{\frac 32}(\mathbb R^3) } \ge  j_0, \quad \forall\, t\in[t_{j_0}^1, t_{j_0}^2];\label{y7_11a} \\
& \| \omega_{j_0} (t, \cdot)\|_{L^2(\mathbb R^3)}\le 2^{-100j_0}, \quad \forall\, t\le t_{j_0}^2. \label{y7_11b} \\
& \max_{0\le t\le t_{j_0}^2} (\| \omega_{j_0}(t,\cdot)\|_{H^k(\mathbb R^3)} + \| \omega_{<j_0} (t,\cdot)\|_{H^k(\mathbb R^3)})
\le C_{j_0,k}<\infty,\quad\forall\, k\ge 2,  \label{y7_11c}
\end{align}
where $C_{j_0,k}$ is a constant depending only on $k$ and $(F_1,F_2,\cdots,F_{j_0})$.
Now for any $y=(y_1,y_2,y_3) \ne x_*=(0,0,1)$, consider three cases. If $y_3\ge 1$, then in this case by our choice of
initial data and finite transport speed, we can find a small neighborhood $N_y$ of $y$ and $0<t_y<1$
 such that $\omega(t,x)=0$ for any $x \in N_y$ and $0\le t\le t_y$.
If $y_3<1$, then we can choose $j_0$ sufficiently large such that $y\in
\{x=(x_1,x_2,z):\; z< \sum_{k=1}^{j_0-1}2^{-k} +\frac 1{16} \cdot {2^{-j_0}}\}$. In this case we can just choose
$t_y=t_{j_0}^2$ and $N_y$ to be a small open neighborhood contained in
$\{x=(x_1,x_2,z):\; z< \sum_{k=1}^{j_0-1}2^{-k} +\frac 1{16} \cdot {2^{-j_0}}\}$.
By \eqref{y7_11c}
$\omega(t)\in C^{\infty}(N_y)$ for any $0\le t\le t_y$. Therefore statement (3) in Theorem \ref{thm4} is proved.

Finally we prove statement (4) in Theorem \ref{thm4}. For each integer $n\ge 1$, we shall take $j_n$ to be sufficiently
large and decompose $\omega$ according to \eqref{y7_10} with $j_0$ replaced by $j_n$. By a slight abuse
of notation we denote $t_n^1=t_{j_n}^1$, $t_n^2=t_{j_n}^2$ and $\omega_n=\omega_{j_n}$. Define
\begin{align*}
 & K_n = \overline{\{x\in \mathbb R^3:\; \omega_n(t,x) \ne 0 \text{ for some } 0\le t\le t_n^2 \}}, \notag \\
 & \Omega_n^1 = \{ x \in \mathbb R^3:\, \text{dist}(x, K_n) <\frac 1 {2^{100j_n}} \}, \notag \\
 & \Omega_n^2 = \{ x \in \mathbb R^3:\, \text{dist}(x, K_n)  <\frac 1 {1000} \cdot \frac 1 {2^{j_n}} \}.
\end{align*}
The inequality \eqref{thm4_t3a} follows from \eqref{y7_11a}.  To show \eqref{thm4_t3b}, we note that if $x\in \mathbb R^3 \setminus \Omega_n^2$
and $y \in K_n$, then
\begin{align*}
 |x-y| \gtrsim 2^{-j_n}.
\end{align*}
We can then write for $x\in \mathbb R^3 \setminus \Omega_n^2$,
\begin{align*}
 (|\nabla|^3 \omega_n) (t,x)  & = (|\nabla|^{-1} (-\Delta)^2 \omega_n)(t,x) \notag \\
 & = \int_{\mathbb R^3} K(x-y) \chi_{|x-y|\gtrsim 2^{-j_n}} \bigl( (-\Delta) \omega_n \bigr)(t,y) dy \notag \\
 & = \int_{\mathbb R^3} \tilde K(x-y) \omega_n (t,y) dy,
\end{align*}
where $\chi$ is a smooth cut-off function and we have integrated by parts in the last step. The modified kernel $\tilde K(\cdot)$
is smooth and obeys the bound
\begin{align*}
 |\tilde K(x) | \lesssim 2^{10 j_n} (1+| x|^2)^{-1}, \quad\forall\, x \in \mathbb R^3.
\end{align*}
Thus by \eqref{y7_11b} and taking $j_n$ sufficiently large, we obtain
\begin{align*}
&\max_{0\le t\le t_n^2} \| (|\nabla|^3 \omega_n)(t,x)\|_{L^2(\mathbb R^3\setminus \Omega_n^2)} \notag \\
\lesssim & 2^{10j_n} 2^{-100j_n} \le 1.
\end{align*}

Theorem \ref{thm4} is now proved.

\end{proof}

\section{Ill-posedness in Besov spaces} \label{sec_Bp2}
In this last section we settle the illposedness in the Besov case.

\begin{thm}\label{thm1_Bp2}
For any $\omega_0^{(g)} \in C_c^{\infty}(\mathbb R^2) \cap \dot H^{-1}(\mathbb R^2)$, any $\epsilon>0$,  and any $1<p< \infty$, $1<q\le \infty$,
we can find a $C^{\infty}$ perturbation $\omega_0^{(p)}:\mathbb R^2\to \mathbb R$
such that the following hold true:

\begin{enumerate}
 \item The perturbation is very small:
 \begin{align*}
   \| \omega_0^{(p)} \|_{L^1(\mathbb R^2)} + \| \omega^{(p)}_0\|_{L^{\infty}(\mathbb R^2)} + \| \omega^{(p)}\|_{\dot H^{-1}(\mathbb R^2)}
   + \| \omega_0^{(p)} \|_{B^{\frac 2p}_{p,q}(\mathbb R^2)} <\epsilon.
  \end{align*}
 \item Let $\omega_0=\omega_0^{(g)}+ \omega_0^{(p)}$. The initial velocity $u_0 = \Delta^{-1} \nabla^{\perp} \omega_0$ has regularity
  $u_0 \in L^2(\mathbb R^2) \cap B^{1+\frac 2p}_{p,q}(\mathbb R^2) \cap C^{\infty} (\mathbb R^2) \cap L^{\infty}(\mathbb R^2)$.
 \item There exists a unique classical solution $\omega = \omega(t)$ to the 2D Euler equation (in vorticity form)
 \begin{align*}
  \begin{cases}
   \partial_t \omega + (\Delta^{-1} \nabla^{\perp} \omega \cdot \nabla) \omega =0, \quad 0<t\le 1, \, x \in \mathbb R^2,\\
   \omega \Bigr|_{t=0} = \omega_0,
  \end{cases}
 \end{align*}
satisfying $\omega(t) \in L^1 \cap L^{\infty} \cap C^{\infty} \cap \dot H^{-1}$, $u=\Delta^{-1} \nabla^{\perp} \omega \in C^{\infty}
\cap L^2\cap L^{\infty}$ for each $0\le t \le 1$.

\item For any $0<t_0 \le 1$, we have
\begin{align} \label{thm1_Bp2_t3}
 \operatorname{ess-sup}_{0<t \le t_0} \| \omega(t,\cdot) \|_{\dot B^{\frac 2p}_{p,\infty}}  =+\infty.
\end{align}

\end{enumerate}

\end{thm}

\begin{proof}[Proof of Theorem \ref{thm1_Bp2}]
Again with out loss of generality we assume $\omega_0^{(g)}\equiv 0$.
We shall sketch the details and point out the important changes (as compared to the proof of Theorem \ref{thm1}).
The first crucial step is the local construction. Since
$B^{\frac 2p}_{p,q_1} \hookrightarrow B^{\frac 2p}_{p,q_2}$ whenever $q_1<q_2$, it suffices
for us to consider the case $B^{\frac 2p}_{p,q}$ with $1<q<p$. Fix such $p$ and $q$.
 We will prove the following

\textbf{Claim:}  for any small $\delta>0$, there exists a smooth initial data $\omega_0^{\delta} \in C_c^{\infty} (B(0,\delta))$
and $t_\delta \in (0,\delta)$ such that if $\omega^{\delta}$ is the smooth solution to
\begin{align*}
 \begin{cases}
  \partial_t \omega^{\delta} + (\Delta^{-1} \nabla^{\perp} \omega^{\delta} \cdot \nabla) \omega^{\delta} =0, \quad 0<t\le 1, \, x \in \mathbb R^2, \\
  \omega^{\delta} \Bigr|_{t=0} = \omega_0^{\delta},
 \end{cases}
\end{align*}
then the following hold:
\begin{itemize}
 \item $\| \omega_0^{\delta}\|_{B^{\frac 2p}_{p,q}} +
 \| \omega_0^{\delta} \|_{L^{\infty} } + \|\omega_0^{\delta} \|_{\dot H^{-1}} \le \delta$.
 \item $\operatorname{supp}(\omega^{\delta} (t,\cdot) ) \subset B(0,\delta)$ for any $0\le t\le \delta$.
 \item $\| \omega^{\delta} (t_{\delta}, \cdot ) \|_{\dot B^{\frac 2p}_{p,\infty}} > \frac 1 {\delta}$.
\end{itemize}

To prove the claim, we first take $A\gg 1$ and define one parameter of functions
\begin{align*}
 g_A^0 = \frac 1 {\log A} \sum_{A<k<A+\log A} \eta_k(x),
\end{align*}
where $\eta_k$ is the same as in \eqref{30_0a}. Easy to check that
\begin{align*}
 &\|g_A^0 \|_{\dot B^{\frac 2 r}_{r,1}} \lesssim 1, \quad \forall\, 1\le r\le \infty,\notag \\
 &\| g_A^0 \|_{\dot B^0_{\infty,\infty}} \lesssim  \| g_A^0\|_{\infty} \lesssim \frac 1 {{\log A}}. \notag
\end{align*}
Therefore by interpolation (choosing $r=p/q$)
\begin{align*}
 \| g_A^0 \|_{\dot B^{\frac 2p}_{p,q}} \lesssim \frac 1 {(\log A)^{2\epsilon_1}},
\end{align*}
where the exponent $\epsilon_1= \frac 12 (1-\frac 1q) \in (0,\frac 12)$.
 Now take
 \begin{align} \label{thm1_Bp2_hA_def}
  h_A = (\log A)^{\epsilon_1} g_A^0.
 \end{align}
Obviously we have
\begin{align} \label{thm1_Bp2_30}
 & \| h_A\|_{\infty} \lesssim \frac 1 {(\log A)^{1-\epsilon_1}}, \notag \\
 & \| h_A\|_{1} + \| h_A\|_{\dot H^{-1}} \lesssim 2^{-2A}, \notag \\
 & \| h_A \|_{\dot B^{\frac 2p}_{p,q}} \lesssim \frac 1 { (\log A)^{\epsilon_1}}.
\end{align}
Let $W_A$ be the smooth solution to the system
\begin{align} \label{thm1_Bp2_30a}
 \begin{cases}
  \partial_t W_A +  (\Delta^{-1} \nabla^{\perp} W_A \cdot \nabla) W_A=0, \quad t>0, \, x\in\mathbb R^2, \\
  W_A \Bigr|_{t=0}=h_A.
 \end{cases}
\end{align}
Define the forward characteristics $\phi_A$ such that
\begin{align} \label{thm1_Bp2_40}
 \begin{cases}
  \partial_t \phi_A (t,x) = ( \Delta^{-1} \nabla^{\perp} W_A)(t,\phi_A(t,x)), \\
  \phi_A(t=0,x)=x \in \mathbb R^2.
 \end{cases}
\end{align}
By following the proof of Proposition \ref{prop40} (or using Proposition
\ref{prop_gener_1}, we then have for $A$ sufficiently large,
\begin{align} \label{thm1_Bp2_40a}
M_A:= \max_{0\le t \le \frac 1{\log\log A}} \| (D\phi_A)(t,\cdot)\|_{\infty} \ge \log\log A.
\end{align}

Clearly we can find $0<t_A < \frac 1 {\log\log A}$ and $x_A$ such that
\begin{align*}
\| (D\phi_A)(t_A, x_A)\|_{\infty} > \frac 45 M_A.
\end{align*}
Denote $\phi_A(t,x_1,x_2)=(\phi_A^1(t,x_1,x_2), \phi_A^2(t,x_1,x_2))$. Without loss of generality we can assume
\begin{align*}
 | (\partial_2 \phi_A^2)(t_{A}, x_A) |>\frac 4 {5} M_A.
\end{align*}
By continuity we can find a small neighborhood $O_A=B(x_A,r_A)$ of $x_A$  such that
\begin{align}\label{thm1_Bp2_40b}
 | (\partial_2 \phi_A^2)(t_{A}, x) |>\frac 45 M_A, \quad \forall\, x \in O_A .
\end{align}

Depending on the location of $x_A$, we need to shrink $0<r_A<1$ slightly further and define an even function
$b\in C_c^{\infty}(\mathbb R^2)$ as follows. Fix a smooth radial cut-off function $\Phi_0 \in C_c^{\infty}(\mathbb R^2)$
such that $\Phi_0(x)=1$ for $|x|\le \frac 12$ and $\Phi_0(x)=0$ for $|x|>1$.
If $x_A=(0,0)$, we just define $b(x) = r_A^{-\frac 2p} \Phi_0(\frac x {r_A})$.
 If $x_A=(a_*,0)$ for some $a_*\ne 0$,
 then we choose $r_A>0$  such that $r_A \ll |x_A| $.
In this case we choose $b$ as an even function of $x_1$ and $x_2$ which takes the form
\begin{align*}
b(x) = {r_A}^{-\frac 2p} \Bigl( \Phi_0( \frac {x-x_A} {r_A})+ \Phi_0(\frac{x+x_A} {r_A}) \Bigr).
\end{align*}
The case $x_A=(0,c_*)$ is similar. Now if $x_A=(a_*,c_*)$ with $a_*\ne 0$ and $c_*\ne 0$, then
we choose $r_A \ll \min\{ |a_*|, |c_*|\}$ and define
\begin{align*}
b(x) =  {r_A}^{-\frac 2p} \sum_{\epsilon_1=\pm 1,\,\epsilon_2=\pm 1} \Phi_0 \bigl(
\frac {x -(\epsilon_1 a_*, \epsilon_2 c_* \bigr)}
{r_A}).
\end{align*}
Easy to check that $b$ is an even function of $x_1$ and $x_2$.

Therefore in all situations we can choose an even function $b\in C_c^{\infty} (\mathbb R^2)$ such that
\begin{align}
 &\|b\|_{L^p(O_A)} \sim 1, \notag \\
 & \| b\|_{L^p(\mathbb R^2)} \sim 1, \notag\\
 & \| b \|_{\dot B^0_{p,1}(\mathbb R^2)} \lesssim 1. \label{thm1_Bp2_100}
\end{align}
In the above the implied constants depend only on the definition of the smooth cut-off function $\Phi_0$
and thus can be made as absolute constants. To simplify later notations and discussions, we shall
still denote by $O_A$ the support of $b$ which are unions of even reflections of $O_A$ on the plane.
The last inequality in \eqref{thm1_Bp2_100} is due to the fact that
translation in the real domain is equivalent to phase modulation in the frequency domain and hence
\begin{align} \label{thm1_Bp2_101}
\| b \|_{\dot B^0_{p,1}} \lesssim \| {r_A}^{-\frac 2p} \Phi_0(\frac{\cdot}{r_A})\|_{\dot B^0_{p,1}} \lesssim
\| \Phi_0\|_{\dot B^0_{p,1}} \lesssim 1.
\end{align}

Now we consider two cases.

\texttt{Case 1}: $\max_{0\le t\le \frac 1 {\log\log A}}
\| W_A(t,\cdot) \|_{\dot B^{\frac 2p}_{p,\infty}} \ge \log \log\log\log A$.
In this case we set $\omega^{\delta}_0 = W_A$ with $A=A(\delta)$ chosen sufficiently large.
No particular work is needed in this case.

\texttt{Case 2}:
\begin{align} \label{thm1_Bp2_104}
 \max_{0\le t\le \frac 1 {\log\log A}} \| W_A(t,\cdot) \|_{\dot B^{\frac 2p}_{p,\infty}} <  \log\log\log\log A.
\end{align}

In this case we consider
\begin{align*}
 \tilde h_A = h_A +
 \underbrace{\frac 1 {\log\log\log A} \cdot k^{-\frac 2p} \sin (kx_1) \cdot b(x)}_{:=\beta(x)},
\end{align*}
where $b(x)$ was chosen as in \eqref{thm1_Bp2_100}.\footnote{Of
course a natural idea is to consider cutting off the high
frequencies of $b$, say replacing $b$ by $b_1:=P_{<N_1} b$ for some
sufficiently large $N_1$. This will simplify the computation of
$\dot B^{\frac 2p}_{p,1}$ norm of the perturbation $\sin(kx_1)
b_1(x)$  since for large $k\gg N_1$ the function $\sin(kx_1) b_1(x)$
will have frequency localized to $\{\xi:\, |\xi| \sim k \}$. However
the disadvantage of doing this is that $b_1$ is not compactly
supported in the $x$-space. This will bring some more unnecessary
technical complications in the gluing of patch solutions later. }
Once again we shall take the parameter $k$ sufficiently large.
Consider first $N\ll k$. Write
\begin{align*}
\sin(kx_1) b(x) = -\frac 1k \partial_{x_1} (\cos(kx_1) b(x))+ \frac 1k \cos (kx_1) \partial_{x_1} b(x).
\end{align*}
By Bernstein and the above identity, we get
\begin{align*}
N^{\frac 2p} \| P_{N} ( \sin (kx_1) b(x) ) \|_p
& \lesssim \frac {N^{1+\frac 2p}} k \| b\|_p + \frac 1 k  N^{\frac 2p}\| P_N( \cos (kx_1) \partial_{x_1} b) \|_p \notag \\
& \lesssim \frac {N^{1+\frac 2p}}  k \| b\|_p + \frac 1 k N^{\frac 2p} \| \partial_{x_1} b\|_{ p}.
\end{align*}
Summing over dyadic $N\ll k$ and letting $k$ be sufficiently large, we obtain
\begin{align*}
\frac 1 {k^{\frac 2p}}\sum_{N\ll k} N^{\frac 2p} \| P_{N} ( \sin (kx_1) b(x) ) \|_p \lesssim 1.
\end{align*}

Next consider $N\gg k$. By frequency localization, observe
\begin{align*}
 &  P_{N} ( \sin(kx_1) b ) = P_{N} ( \sin (kx_1) \tilde P_N b),
 \end{align*}
where $\tilde P_N$ is a fattened frequency projector adapted to the regime $|\xi|\sim N$.
Clearly by taking $k$ sufficiently large, we have
\begin{align*}
 & \frac 1 {k^{\frac 2p}} \sum_{N \gg k} N^{\frac 2p}\| P_{N} ( \sin(kx_1) b(x) ) \|_p \notag \\
 \lesssim & \frac 1 {k^{\frac 2p}} \sum_{N\gg k} N^{\frac 2p} \| \tilde P_N b \|_p
 \lesssim k^{-\frac 2p} \| b\|_{\dot B^{\frac 2p}_{p,1}} \lesssim 1.
 \end{align*}
In the intermediate regime $N\sim k$, there are finitely many such dyadic $N$ and we have
\begin{align*}
\frac 1 {k^{\frac 2p}}\sum_{N \sim k} N^{\frac 2p}\| P_{N} ( \sin(kx_1) b(x) ) \|_p \lesssim \| b\|_p \lesssim 1.
\end{align*}

Summing over all cases, we have proved
\begin{align*}
\frac 1{k^{\frac 2p}} \| b(x) \sin(kx_1) \|_{\dot B^{\frac 2p}_{p,1}} \lesssim 1.
\end{align*}
Therefore
\begin{align*}
 \| \beta \|_{\dot B^{\frac 2p}_{p,1}} \lesssim \frac 1 {\log\log\log A}.
\end{align*}
By a similar analysis, we also have
\begin{align}
\| \beta \|_{ B^{\frac 2p}_{r,1}} =O(1), \quad \forall\, p\le r \le \infty. \label{thm1_Bp2_104_11a}
\end{align}
Here and below we adopt the same big O notation as described in the paragraph after \eqref{prop20_16}.
Denote $e_1=(1,0)$. Then
\begin{align*}
 \| |\nabla|^{-1} \beta \|_2^2 & \lesssim \frac 1 {k^{\frac 4p}} \int_{\mathbb R^2} \frac 1 {|\xi|^2} | \hat b(\xi+ke_1)-\hat b(\xi-ke_1) |^2 d\xi
\notag \\
& \lesssim k^{-\frac 4p} ( \| x b(x)\|_1^2 +\|b\|_2^2) \notag \\
& = O(k^{-\frac 4p}).
\end{align*}
Therefore
\begin{align}
 \| \beta \|_{\dot H^{-1}}  = O(k^{-\frac 2p}). \notag
\end{align}

By \eqref{thm1_Bp2_30} and choosing $k$ sufficiently large, we then have
\begin{align}
 & \| \tilde h_A\|_{\infty} \lesssim \frac 1 {(\log A)^{1-\epsilon_1}}, \notag \\
 & \|\tilde  h_A\|_{1} + \| \tilde h_A\|_{\dot H^{-1}} \lesssim 2^{-2A}, \notag \\
 & \| \tilde h_A \|_{\dot B^{\frac 2p}_{p,q}} \lesssim \frac 1 { (\log A)^{\epsilon_1}} +\frac 1 {\log\log\log A}. \notag
\end{align}
Let $W_A^1$ be the smooth solution to the equation
\begin{align*}
 \begin{cases}
  \partial_t W_A^1 + (\Delta^{-1} \nabla^{\perp} W_A^1 \cdot \nabla) W_A^1 =0, \quad 0<t\le 1, \\
  W_A^1 \Bigr|_{t=0} = \tilde h_A.
 \end{cases}
\end{align*}

Let $\eta=W_A^1-W_A$ where $W_A$ is the solution to \eqref{thm1_Bp2_30a}. Then $\eta$ satisfies
\begin{align*}
 \begin{cases}
  \partial_t \eta + (\dpp \eta \cdot \nabla) W_A^1 +(\dpp W_A\cdot \nabla) \eta =0, \\
  \eta \Bigr|_{t=0} =\beta.
 \end{cases}
\end{align*}
Now
\begin{align*}
 \partial_t ( \| |\nabla|^{-1} \eta \|_2^2) & \lesssim \| \dpp \eta \|_2 \cdot \| W_A^1 \|_{\infty} \cdot \| |\nabla|^{-1} \eta\|_2
 \notag \\
 & \qquad + \left| \int (\dpp W_A \cdot \nabla) ( \Delta \Delta^{-1} \eta) \cdot \Delta^{-1} \eta dx \right| \notag \\
 & \lesssim \| |\nabla|^{-1} \eta \|_2^2 \cdot ( \| W_A^1 \|_{\infty} + \| \mathcal R_{ij} W_A \|_{\infty}).
\end{align*}
Hence
\begin{align} \label{thm1_Bp2_m21_1a}
 \max_{0\le t \le 1} \| (|\nabla|^{-1} \eta) (t,\cdot)\|_2 = O(k^{-\frac 2p}).
\end{align}
Take $r \in(p,\infty)$. By \eqref{thm1_Bp2_104_11a} and local wellposedness in $B^{\frac 2p}_{r,1}$,
\begin{align}
 \max_{0\le t\le 1} (\| W_A^1(t,\cdot)\|_{B^{\frac 2p}_{r,1}}+  \| W_A(t,\cdot)\|_{B^{\frac 2p}_{r,1}} )=O(1).
 \label{thm1_Bp2_m21_1b}
\end{align}
Interpolating the above with \eqref{thm1_Bp2_m21_1a}, we get
\begin{align}
 \max_{0\le t\le 1} \| \eta(t,\cdot)\|_{B^{0}_{\infty,1}} =O(k^{-\alpha}). \label{thm1_Bp2_m21_2}
\end{align}
Here and below we denote by the general notation $X=O(k^{-\alpha})$  if the quantity $X \lesssim k^{-\alpha}$ for some
$\alpha>0$. The value of $\alpha$ does not play much role in the analysis as long as $\alpha>0$.

Let $\Phi_A$ be the characteristic line associated with $W_A^1$, i.e.
\begin{align*}
 \begin{cases}
  \partial_t \Phi_A(t,x)= (\dpp W_A^1)(t,\Phi_A(t,x)), \\
  \Phi_A (0,x)= x \in \mathbb R^2.
 \end{cases}
\end{align*}

Set $J(t)=(D\Phi_A)(t)-(D\phi_A)(t)$. Then
\begin{align*}
 \partial_t J &= (\mathcal R W_A^1) (\Phi_A) J + \Bigl( \mathcal R (W_A^1-W_A) \Bigr)(\Phi_A) D\phi_A \notag \\
 & \qquad + \Bigl( (\mathcal R W_A)(\Phi_A) - (\mathcal R W_A)(\phi_A) \Bigr) D\phi_A.
\end{align*}

Using \eqref{thm1_Bp2_m21_2} and the above equation, it is easy to check
\begin{align}
 \max_{0\le t \le 1} \left(
 \| (D\Phi_A)(t,\cdot)- (D\phi_A)(t,\cdot)\|_{\infty} +
 \| \Phi_A (t,\cdot)- \phi_A (t,\cdot)\|_{\infty}\right)=O(k^{-\alpha}). \label{thm1_Bp2_m21_3}
\end{align}

Let $W_A^2$, $W_A^3$ be the smooth solutions to the \emph{linear} equations
\begin{align*}
 \begin{cases}
  \partial_t W_A^2 + (V_A \cdot \nabla) W_A^2 =0, \quad t>0, \\
  W_A^2 \Bigr|_{t=0} = g_A,
 \end{cases}
\end{align*}
\begin{align*}
 \begin{cases}
  \partial_t W_A^3 + (V_A \cdot \nabla )W_A^3 =0, \quad t>0, \\
  W_A^3 \Bigr|_{t=0} = \beta,
 \end{cases}
\end{align*}
where $V_A(t,x)=(\Delta^{-1} \nabla^{\perp} W_A^1)(t,x)$. Obviously,
\begin{align*}
 W_A^1= W_A^2 + W_A^3.
\end{align*}

We first show that
\begin{align}
 \max_{0\le t\le 1}\| W_A(t,\cdot) - W_A^2(t,\cdot) \|_{\dot B^{\frac 2p}_{p,\infty}} =  O(k^{-\alpha}).
 \label{thm1_Bp2_m21_5}
\end{align}
By \eqref{thm1_Bp2_m21_1b}, it is easy to check
\begin{align}
 &\max_{0\le t\le 1} \|  D^2 W_A^2(t,\cdot) \|_{p} = O(1), \quad \text{if $1<p\le 2$, } \notag \\
 &\max_{0\le t\le 1} \|  D W_A^2(t,\cdot) \|_{p} = O(1), \quad \text{if $2<p<\infty$. }\label{thm1_Bp2_m21_6}
\end{align}
On the other hand, by the Fundamental Theorem of Calculus, we have
\begin{align*}
 W_A^2-W_A  & = g_A(\tilde \Phi_A) - g_A(\tilde \phi_A) \notag \\
 & = \int_0^1 (Dg_A)(\tilde \phi_A + s(\tilde \Phi_A- \tilde \phi_A ))ds (\tilde \Phi_A-\tilde \phi_A).
\end{align*}
Here $\tilde \Phi_A$, $\tilde \phi_A$ denote the inverse map of $\Phi_A$ and $\phi_A$ respectively.
By \eqref{thm1_Bp2_m21_3} and H\"older, we then get\footnote{Alternatively one can derive the estimate in an ``Eulerian'' way, i.e. directly derive
an $L^p$ estimate from the equation.}
\begin{align*}
 &\max_{0\le t\le 1} \| W_A^2(t,\cdot) -W_A(t,\cdot)\|_{p} \notag \\
\lesssim & \max_{0\le t\le 1} \| W_A^2(t,\cdot) -W_A(t,\cdot)\|_{\infty}=
 O(k^{-\alpha}).
\end{align*}
Interpolating this with \eqref{thm1_Bp2_m21_6} then yields \eqref{thm1_Bp2_m21_5}.

By \eqref{thm1_Bp2_m21_5} and \eqref{thm1_Bp2_104}, we only need to show
\begin{align*}
 \| W_A^3(t_A, \cdot) \|_{\dot B^{\frac 2p}_{p,\infty}} \gg \log\log\log A.
\end{align*}

For this we need to introduce $W_A^4$ which solves the linear equation
\begin{align*}
 \begin{cases}
  \partial_t W_A^4 + (U_A \cdot \nabla )W_A^4 =0, \quad 0<t\le 1, \\
  W_A^4 \Bigr|_{t=0} = \sin(kx_1) b(x)=:W_{4,0},
 \end{cases}
\end{align*}
where $U_A=\dpp W_A$.

We shall \emph{not} directly compare $W_A^3$ with $k^{-\frac 2p} \frac 1 {\log\log\log A} W_A^4$ and
run a perturbation argument in $\dot B^{\frac 2p}_{p,\infty}$. Instead we will carry out
an indirect argument as follows.

We first analyze the structure of $W_A^4$. Write $\tilde \phi_A = (\tilde \phi_A^1,\tilde \phi_A^2)$ and
\begin{align*}
 W_A^4(t_A,x) & = W_{4,0}(\tilde \phi_A(t_A,x) ) \notag \\
 &=  \sin( k \tilde \phi_A^1 (t_A,x)) \cdot b(\tilde \phi_A(t_A,x)).
\end{align*}

Now consider
\begin{align} \label{thm1_Bp2_110}
 F(\xi) = \int_{\mathbb R^2} \sin (k\tilde \phi_A^1 (t_A,x)) b(\tilde \phi_A(t_A,x))
  e^{-ix\cdot \xi} dx.
\end{align}

By a change of variable $x\to \phi_A(t_A,x)$ in \eqref{thm1_Bp2_110} (and recall that the map is volume-preserving), we get
\begin{align*}
F(\xi)
&= \frac 1{2i} \int_{\mathbb R^2} b(x)  \cdot e^{-i \phi_A(t_A,x) \cdot \xi +ikx_1} dx \notag \\
&\quad - \frac 1{2i} \int_{\mathbb R^2}  b(x)  \cdot e^{-i \phi_A(t_A,x) \cdot \xi -ikx_1} dx
\end{align*}

Consider the phase $\phi_A(t_A, x)\cdot \xi +kx_1$. Write
\begin{align*}
 (D\phi_A)(t_A,x) \xi + k \begin{pmatrix} 1 \\ 0 \end{pmatrix}
 &= (D \phi_A)(t_A, x)  \Bigl(   \xi + k  \bigl( (D \phi_A)(t_A,x) \bigr)^{-1} \begin{pmatrix} 1 \\ 0 \end{pmatrix} \Bigr).
\end{align*}
Since $\bigl((D\phi_A)(t_A,x) \bigr)^{-1}=\operatorname{adj} ( (D\phi_A)(t_A,x) )$, by \eqref{thm1_Bp2_40a} and
\eqref{thm1_Bp2_40b}, we get
\begin{align*}
\frac 1 {C_1} \le \frac 1 {M_A } | \bigl( (D \phi_A)(t_A,x) \bigr)^{-1} \begin{pmatrix} 1 \\ 0 \end{pmatrix} \Bigr) | \le C_1,
\quad\forall\, x \in O_A,
\end{align*}
where $C_1>0$ is an absolute constant. Now if $|\xi| \ge 2C_1 \cdot kM_A$ and $x\in O_A$, then
\begin{align}
 &| (D\phi_A)(t_A,x) \xi + k \begin{pmatrix} 1 \\ 0 \end{pmatrix}  | \notag \\
\gtrsim &\; | \Bigl( (D\phi_A)(t_A,x) \Bigr)^{-1}|^{-1} \cdot | \xi + k
\bigl( (D \phi_A)(t_A,x) \bigr)^{-1} \begin{pmatrix} 1 \\ 0 \end{pmatrix}|
\notag \\
\gtrsim & \;M_A^{-1} \cdot |\xi|.  \label{thm1_Bp2_113}
\end{align}

Similarly if $|\xi| \le \frac 1 {2C_1} k M_A$ and $x\in O_A$, then
\begin{align}
 &| (D\phi_A)(t_A,x) \xi + k \begin{pmatrix} 1 \\ 0 \end{pmatrix}  | \notag \\
\gtrsim &\;  M_A^{-1} \cdot k M_A \gtrsim k. \label{thm1_Bp2_115}
\end{align}

This shows that $F(\xi)$ is essentially localized to $|\xi| \sim kM_A$.
By \eqref{thm1_Bp2_113}, \eqref{thm1_Bp2_115} and a stationary phase argument (note that the derivatives of $\phi_A$ are independent
of $k$!), we have
\begin{align*}
 \| P_{\gg kM_A} W_A^4(t_A,\cdot)\|_{p} + \| P_{\ll kM_A} W_A^4(t_A,\cdot) \|_{p}
 = O(k^{-\alpha}).
\end{align*}

 Consequently
\begin{align}
  \| P_{\sim kM_A} W_A^4(t_A,\cdot) \|_p
& \ge \| W_A^4(t_A,\cdot) \|_p -  \| P_{\gg kM_A} W_A^4(t_A,\cdot)\|_{p} -\| P_{\ll kM_A} W_A^4(t_A,\cdot) \|_{p}
\notag \\
& \ge \| \sin(kx_1) b(x) \|_p - O(k^{-\alpha}), \notag
\end{align}
where in the last step we have used $L^p$-conservation. Take an integer $m$ such that $10p>2m>p$, obviously
for $k$ sufficiently large,
\begin{align*}
 \| \sin(kx_1) b(x)\|_p^p &\ge \int_{\mathbb R^2} (\sin(kx_1))^{2m} |b(x)|^p dx \notag \\
 & \gtrsim_p \int_{\mathbb R^2} (1- \cos(2k x_1))^{m} |b(x)|^p dx \notag \\
 & \gtrsim_p \|b\|_p^p + \sum_{1\le j\le m} (-1)^j\binom{m}{j} \int_{\mathbb R^2} (\cos(2k x_1))^j |b(x)|^p dx \notag \\
 & \gtrsim_p \| b\|_p^p + O(k^{-\alpha}) \gtrsim \| b\|_p^p \gtrsim 1.
\end{align*}
Therefore
\begin{align}
 \| P_{\sim kM_A} W_A^4(t_A,\cdot )\|_{p} \gtrsim 1. \label{thm1_Bp2_m21_10}
 \end{align}

Now set
\begin{align*}
 \eta_1 = W_A^4- k^{\frac 2p} \cdot (\log\log\log A) \cdot W_A^3.
\end{align*}
Clearly,
\begin{align*}
 \begin{cases}
  \partial_t \eta_1 + ((U_A-V_A)\cdot \nabla )W_A^4 + (V_A\cdot \nabla) \eta_1 =0, \\
  \eta_1 \Bigr|_{t=0}=0.
 \end{cases}
\end{align*}
By \eqref{thm1_Bp2_m21_2} (to control $U_A-V_A$) and a similar argument as in the derivation of
\eqref{thm1_Bp2_m21_1a}, we have
\begin{align*}
 \max_{0\le t\le 1} \| |\nabla|^{-1} \eta_1 (t,\cdot)\|_2 =O(k^{-\alpha}).
\end{align*}
Since $\|\eta_1\|_1+\|\eta_1\|_{\infty} =O(1)$, interpolation then gives
\begin{align*}
 \max_{0\le t\le 1} \| \eta_1(t,\cdot)\|_p = O(k^{-\alpha}).
\end{align*}
By \eqref{thm1_Bp2_m21_10}, we then obtain
\begin{align}
 k^{\frac 2p }\| P_{\sim kM_A} W_A^3(t_A,\cdot )\|_{p} \gtrsim \frac 1 {\log\log\log A}. \notag
 \end{align}
Clearly,
\begin{align*}
 \| W_A^3(t_A, \cdot) \|_{\dot B^{\frac 2p}_{p,\infty}} & \gtrsim (kM_A)^{\frac 2p} \| P_{\sim k M_A} W_A^3(t_A,\cdot) \|_p \notag \\
 & \gtrsim \frac{M_A^{\frac 2p}} { \log\log\log A} \notag \\
 &\gtrsim \frac{ (\log\log A)^{\frac 2p} }{ \log\log\log A} \gg \log\log\log A.
\end{align*}
This settles Case 2.

We have proved the claim in the local construction step.

To finish the proof of Theorem \ref{thm1_Bp2} we need to appeal to a version of Proposition \ref{prop10}
and build a solution in the form
\begin{align*}
\omega(t,x) = \sum_{j=1}^{\infty} \omega_j(t,x),
\end{align*}
where each $\omega_j$ has compact support and $\operatorname{dist}(\operatorname{supp}(\omega_j),
\operatorname{supp}(\omega_k))=:R_{jk}\gg 1$ for $j\ne k$. Furthermore for any $n>1$ we can find
$0<t_n<\frac 1n$ and $j_n$ such that
\begin{align} \label{thm1_Bp2_130}
\| \omega_{j_n}(t_n, \cdot )\|_{\dot B^{\frac 2p}_{p,\infty}}>n.
\end{align}
Due to the nonlocal nature of the Besov norm $\| \cdot \|_{\dot B^{\frac 2p}_{p,\infty}}$, we have to control
the interactions of the patches $\omega_j$. For this we will need to use the convexity inequality:
if $1<r<\infty$ and $x,y \in \mathbb C^d$, then
\begin{align} \label{convex_ineq1}
 |x+y|^r \ge |x|^r + r |x|^{r-2} \bar x \cdot y, \qquad \forall\, x ,\, y \in \mathbb C^d.
\end{align}

Now fix any dyadic $N\ge 2$.  By the convexity inequality above, we have for any $j$,
\begin{align*}
 \|P_N \omega\|_p^p  & = \int_{\mathbb R^2} | P_N \omega_j + \sum_{k\ne j} P_N \omega_k |^p dx \notag \\
 & \ge \| P_N \omega_j\|_p^p + p \sum_{k\ne j} \int_{\mathbb R^2} | P_N \omega_j|^{p-2}( P_N \omega_j) \cdot P_N \omega_k dx.
\end{align*}

Observe that for any $m\ge 1$, $N\ge 2$,
\begin{align*}
\| P_N \omega_m \|_{L^p(\{x:\, \operatorname{dist}(x,\operatorname{supp}(\omega_m))>2\})} \lesssim N^{-100} \| P_N \omega_m\|_p.
\end{align*}

By this and the fact $R_{jk}\gg 1$ for $j\ne k$, we have
\begin{align*}
&\sum_{k\ne j} \left| \int_{\mathbb R^2} | P_N \omega_j|^{p-2}( P_N \omega_j) \cdot P_N \omega_k dx \right| \notag \\
\lesssim & \sum_{k\ne j} N^{-100} \|\omega_j\|_p^{p-1} \cdot \| \omega_k\|_p  \lesssim N^{-100},
\end{align*}
where we need to use the fact $\sum_{k=1}^{\infty} \|\omega_k\|_p \lesssim 1$ which can be easily accommodated into the construction.
Clearly for any $j$,
\begin{align*}
 \| P_N \omega \|_p^p & \ge \| P_N \omega_j \|_p^p -\operatorname{const}\cdot N^{-100}.
\end{align*}
From this and \eqref{thm1_Bp2_130}, it is then easy to check \eqref{thm1_Bp2_t3} holds.

\end{proof}

\begin{thm} \label{thm2_Bp2}
 For any $\omega_0^{(g)} \in C_c^{\infty}(\mathbb R^2) \cap \dot H^{-1}(\mathbb R^2)$ which is odd in $x_1$, any $\epsilon>0$,  and any $1<p< \infty$, $1<q\le \infty$,
we can find a  perturbation $\omega_0^{(p)}:\mathbb R^2\to \mathbb R$
such that the following hold true:

\begin{enumerate}
 \item $\omega_0^{(p)}$ is compactly supported, continuous and
 \begin{align*}
   \| \omega^{(p)}_0\|_{L^{\infty}(\mathbb R^2)} + \| \omega^{(p)}\|_{\dot H^{-1}(\mathbb R^2)}
   + \| \omega_0^{(p)} \|_{B^{\frac 2p}_{p,q}(\mathbb R^2)} <\epsilon.
  \end{align*}
 \item Let $\omega_0=\omega_0^{(g)}+ \omega_0^{(p)}$.  Corresponding to $\omega_0$,
 there exists a unique time-global solution $\omega = \omega(t)$ to the 2D Euler equation
satisfying $\omega(t) \in  L^{\infty} \cap \dot H^{-1}$. Furthermore $\omega \in C_t^0 C_x^0$ and
$u=\Delta^{-1} \nabla^{\perp} \omega \in C_t^0 L_x^2 \cap C_t^0 C_x^{\alpha}$ for any $0<\alpha<1$.

\item $\omega(t)$ has additional local regularity in the following sense: there exists $x_* \in \mathbb R^2$ such
that for any $x\ne x_*$, there exists a neighborhood $N_x \ni x$, $t_x >0$ such that
$w(t, \cdot) \in C^{\infty} (N_x)$ for any
$0\le t \le t_x$.

\item For any $0<t_0 \le 1$, we have
\begin{align} \label{thm2_Bp2_t1}
 \operatorname{ess-sup}_{0<t \le t_0} \| \omega(t,\cdot) \|_{\dot B^{\frac 2p}_{p,\infty} } =+\infty.
\end{align}

More precisely, there exist $0<t_n^1<t_n^2 <\frac 1n$, open precompact sets
$\Omega_n^1$, $\Omega_n^2$ with $\Omega_n^1\subset \overline{\Omega_n^1} \subset \Omega_n^2$, $n=1,2,3,\cdots$ such that
\begin{itemize}
\item $\omega(t) \in C^{\infty}(\Omega_n^2)$ for all $0\le t \le t_n^2$;
\item $\omega(t,x)\equiv 0$ for any $x\in \Omega_n^2 \setminus \Omega_n^1$, $0\le t\le t_n^2$.
\item Define $\omega_n(t,x)=\omega(t,x)$ for $x\in \Omega_n^1$, and $\omega_n(t,x)=0$ otherwise. Then
$\omega_n \in C^\infty_c(\mathbb R^2)$, and for some dyadic $N_n \ge 2$,
 \begin{align} \label{thm2_Bp2_t3a}
 N_n^{\frac 2p} \| (P_{N_n} \omega_n)(t,\cdot) \|_{L^p(\mathbb R^2)} >n, \quad \forall \, t_n^1\le t\le t_n^2,
\end{align}
and
\begin{align} \label{thm2_Bp2_t3b}
&\| ( P_{N_n} \omega_n)(t,\cdot)\|_{L^p(x\in \mathbb R^2 \setminus \Omega_n^2)} \notag \\
& \quad + \| \bigl(P_{N_n} (\omega-\omega_n) \bigr) \|_{L^p( \Omega_n^2)}\le \frac 1 {N_n^{100}},
\quad\forall\, 0\le t\le t_n^2.
\end{align}
\end{itemize}

\end{enumerate}

\end{thm}

\begin{proof}[Proof of Theorem \ref{thm2_Bp2}]
Again WLOG assume $\omega_0^{(g)} \equiv 0$.
 We shall only sketch the needed modifications (compared to the proof of Theorem \ref{thm2} and repeating some of the steps
 in Theorem \ref{thm1_Bp2}). In the local construction step,
 we take same $h_A$ as in \eqref{thm1_Bp2_hA_def}.  We then need to modify  Lemma \ref{lem55}  (with initial data $h_A$)
only slightly, namely \eqref{lem55_2} is replaced by
  \begin{align*}
  \max_{0\le t\le \frac 1 {\log\log\log A}} \| (D\Phi)(t,\cdot) \|_{\infty} > \log\log\log\log A.
 \end{align*}
 The proof stays essentially the same.

The next step in the construction is to modify Lemma \ref{lem57}. In \eqref{lem57_1a} and the sentence before
\eqref{lem57_2}, we replace the $H^1$ norm
by $\dot B^{\frac 2p}_{p,q}$ ($q=1+$) norm. The inequality \eqref{lem57_2} is modified as
\begin{align} \label{thm2_Bp2_20a}
 \max_{0\le t\le t_0} \| P_{>\epsilon^{-\frac 1{10}}}  \omega_0(t,\cdot)\|_{\dot B^{\frac 2p}_{p,\infty}} >\frac 1{\epsilon}.
\end{align}
Also it should be noted that we need to choose $\epsilon \ll R_0$. One can then easily check
\begin{align} \label{thm2_Bp2_20b}
 &\max_{0\le t\le t_0}\Bigl( \| P_N \omega_0 (t,\cdot)\|_{L^p(x:\, \operatorname{dist}(x, \operatorname{supp}(\omega_0))\gtrsim R_0) }
 \notag \\
 & \quad + \| P_N (\omega-\omega_0)(t,\cdot)\|_{L^p(x:\, \operatorname{dist}(x, \operatorname{supp}(\omega_0))\lesssim R_0) }\Bigr)
 \lesssim_m N^{-m},
 \quad \forall\, N>\epsilon^{-\frac 1{10}}.
\end{align}

The last step is to glue the patch solutions. This is essentially the same as the proof of Theorem \ref{thm2}.
The inequalities \eqref{thm2_Bp2_20a} and \eqref{thm2_Bp2_20b} then imply \eqref{thm2_Bp2_t3a}
and \eqref{thm2_Bp2_t3b} respectively. To show  \eqref{thm2_Bp2_t1} from \eqref{thm2_Bp2_t3a}--\eqref{thm2_Bp2_t3b}, we just
decompose $\omega$ as
\begin{align*}
\omega(t,x) = \omega_n(t,x) + g_n(t,x).
\end{align*}
By \eqref{convex_ineq1},
\begin{align*}
 \| P_{N_n} \omega\|_p^p \ge \| P_{N_n} \omega_n \|_p^p - p \left| \int_{\mathbb R^2}  |P_{N_n} \omega_n|^{p-2} (P_{N_n} \omega_n)
 \cdot P_{N_n}g_n dx \right|.
\end{align*}
By construction, we have for some $R_n>0$, $\operatorname{dist}(\operatorname{supp}(\omega_n), \operatorname{supp}(g_n)) >3R_n$, and
\begin{align*}
 &\| P_{N_n} g_n \|_{L^p( x \in \mathbb R^2:\, \operatorname{dist}(x,\operatorname{supp} (\omega_n)) \le R_n )} \lesssim \frac 1 {N_n^{100}}, \notag \\
 &\| P_{N_n} \omega_n \|_{L^p( x \in \mathbb R^2:\, \operatorname{dist}(x,\operatorname{supp} (\omega_n)) > R_n )} \lesssim \frac 1 {N_n^{100}}.
\end{align*}
Clearly,
\begin{align*}
 \| P_{N_n} \omega\|_p^p \ge \| P_{N_n} \omega_n \|_p^p - \frac {\operatorname{const}} {N_n^{100}}.
\end{align*}
Thus  \eqref{thm2_Bp2_t1} follows.

\end{proof}

The last two theorems are on the illposedness of 3D Euler in Besov
spaces. We omit the proof since it mimics the ones made in preceding
sections.

\begin{thm} \label{thm3_BP2}
Consider the 3D incompressible Euler equation in vorticity form:
\begin{align} \label{thm3_Bp2_1}
 \begin{cases}
  \partial_t \omega + (u\cdot \nabla) \omega = (\omega \cdot \nabla)u, \quad 0<t\le 1, \; x=(x_1,x_2,z) \in \mathbb R^3;\\
  u=-\Delta^{-1} \nabla \times \omega,\\
  \omega \Bigr|_{t=0} =\omega_0.
 \end{cases}
\end{align}

For any axisymmetric vorticity $\omega_0^{(g)} \in C_c^{\infty}(\mathbb R^3)$, any $\epsilon>0$,  and any $1<p< \infty$, $1<q\le \infty$,
we can find a $C^{\infty}$ perturbation $\omega_0^{(p)}:\mathbb R^3\to \mathbb R^3$
such that the following hold true:

\begin{enumerate}
 \item The perturbation is very small:
 \begin{align*}
   \| \omega_0^{(p)} \|_{L^1(\mathbb R^3)} + \| \omega^{(p)}_0\|_{L^{\infty}(\mathbb R^3)}
   + \| \omega_0^{(p)} \|_{B^{\frac 3p}_{p,q}(\mathbb R^3)} <\epsilon.
  \end{align*}
 \item Let $\omega_0=\omega_0^{(g)}+ \omega_0^{(p)}$. Let $u_0$ be the velocity corresponding to the initial vorticity $\omega_0$.
 We have $u_0 \in B^{\frac 3p+1}_{p,q} (\mathbb R^3) \cap C^{\infty} (\mathbb R^3)\cap L^{\infty} (\mathbb R^3)$.
\item Corresponding to $\omega_0$, there exists a unique  solution $\omega = \omega(t)$ to \eqref{thm3_Bp2_1} on the whole time
interval $[0,1]$ such that
\begin{align*}
 \sup_{0\le t \le 1} (\| \omega(t,\cdot)\|_{L^\infty} + \| \omega(t,\cdot) \|_{L^1} )<\infty.
\end{align*}
Moreover $\omega \in C^{\infty}$ and $u \in C^{\infty}$ so that the solution is actually classical.

\item For any $0<t_0 \le 1$, we have
\begin{align*}
 \operatorname{ess-sup}_{0<t \le t_0} \| \omega(t,\cdot) \|_{\dot B^{\frac 3p}_{p,\infty} (\mathbb R^3)} =+\infty.
\end{align*}

\end{enumerate}

\end{thm}

\begin{thm} \label{thm4_Bp2}
For any axisymmetric vorticity $\omega_0^{(g)} \in C_c^{\infty}(\mathbb R^3)$, 
any $\epsilon>0$,  and any $1<p< \infty$, $1<q\le \infty$,
we can find a  perturbation $\omega_0^{(p)}:\mathbb R^3\to \mathbb R^3$
such that the following hold true:

\begin{enumerate}
 \item $\omega_0^{(p)}$ is compactly supported, continuous and
 \begin{align*}
  \| \omega^{(p)}_0\|_{L^{\infty}(\mathbb R^3)}+\| \omega_0^{(p)} \|_{B^{\frac 3p}_{p,q}(\mathbb R^3)} <\epsilon.
  \end{align*}
 \item Let $\omega_0=\omega_0^{(g)}+\omega_0^{(p)}.$
 Corresponding to $\omega_0$ there exists a unique solution $\omega = \omega(t,x)$ to the Euler equation \eqref{thm3_1} on the time
 interval $[0,1]$ satisfying
 \begin{align} \label{thm4_Bp2_t1}
 &\sup_{0\le t\le 1} \|\omega(t,\cdot)\|_{\infty}<\infty, \notag \\
 &\operatorname{supp}(\omega(t,\cdot))\subset\{x, \: |x| < R\}, \quad \forall\, 0\le t\le 1,
 \end{align}
where $R>0$ is some constant. Furthermore $\omega \in C_t^0 C_x^0$ and
$u \in C_t^0 L_x^2 \cap C_t^0 C_x^{\alpha}$ for any $\alpha<1$.

\item $\omega(t)$ has additional local regularity in the following sense: there exists $x_* \in \mathbb R^3$ such
that for any $x\ne x_*$, there exists a neighborhood $N_x \ni x$, $t_x >0$ such that $w(t) \in C^{\infty} (N_x)$ for any
$0\le t \le t_x$.

\item For any $0<t_0 \le 1$, we have
\begin{align} \label{thm4_Bp2_t2}
 \operatorname{ess-sup}_{0<t \le t_0} \| \omega(t,\cdot) \|_{\dot B^{\frac 3p}_{p,\infty}(\mathbb R^3) } =+\infty.
\end{align}
%

\end{enumerate}

\end{thm}

\end{document}